\documentclass{article}
\usepackage {geometry} 
\usepackage{graphicx} 
\usepackage[english]{babel}
\usepackage{appendix}
\usepackage{xcolor}
\usepackage{anyfontsize}
\usepackage{setspace}
\usepackage{amsthm}
\usepackage{amsmath}
\usepackage{amssymb}
\usepackage{amsthm}
\usepackage{mathabx}
\usepackage{mathtools}
\usepackage{stmaryrd}
\usepackage{hyperref}
\usepackage{enumitem}
\usepackage{wrapfig}
\usepackage{wrapstuff}
\usepackage{tikz-cd}
\usepackage{comment}
\usepackage{mathtools}
\usepackage[
    left = \flqq{},%
    right = \frqq{},%
    leftsub = \flq{},%
    rightsub = \frq{} %
]{dirtytalk}

\setlist{itemsep=0pt}

\theoremstyle{plain}
\newtheorem{theorem}{Theorem}[subsection]
\newtheorem*{theorem*}{Theorem}
\newtheorem{proposition}[theorem]{Proposition}
\newtheorem*{proposition*}{Proposition}
\newtheorem{lemma}[theorem]{Lemma}
\newtheorem{corollary}[theorem]{Corollary}
\theoremstyle{definition}
\newtheorem{definition}[theorem]{Definition}
\newtheorem*{definition*}{Definition}
\newtheorem{remark}{Remark}

\newtheorem{example}[theorem]{Example}
\newtheorem*{example*}{Example}
\newtheorem{examples}[theorem]{Examples}
\newtheorem{construction}[theorem]{Construction}

\DeclareMathOperator{\Hom}{Hom}
\DeclareMathOperator{\Isom}{Isom}
\DeclareMathOperator{\Fib}{Fib}
\DeclareMathOperator{\Bc}{Bc}
\DeclareMathOperator{\Aut}{Aut}

\DeclareMathOperator{\GL}{GL}
\DeclareMathOperator{\Perm}{Perm}
\DeclareMathOperator{\Spec}{Spec}
\DeclareMathOperator{\Gal}{Gal}
\DeclareMathOperator{\End}{End}
\DeclareMathOperator{\FEnd}{FEnd}
\DeclareMathOperator{\Sym}{Sym}
\DeclareMathOperator{\Ker}{Ker}
\DeclareMathOperator{\Imm}{Im}
\DeclareMathOperator{\core}{core}
\DeclareMathOperator{\pr}{pr}
\DeclareMathOperator{\im}{im}
\DeclareMathOperator{\ev}{ev}
\DeclareMathOperator{\eq}{eq}
\DeclareMathOperator{\coev}{coev}
\DeclareMathOperator{\prj}{prj}
\DeclareMathOperator{\Lim}{lim}
\DeclareMathOperator{\Nil}{Nil}
\DeclareMathOperator{\Deck}{Deck}
\DeclareMathOperator{\Der}{Der}
\DeclareMathOperator{\colim}{colim}
\DeclareMathOperator{\id}{id}
\DeclareMathOperator{\Sh}{\textbf{Sh}}
\DeclareMathOperator{\PSh}{\textbf{PSh}}
\DeclareMathOperator{\Sch}{\textbf{Sch}}

\DeclareMathOperator{\FSep}{\bf{FSep}}
\DeclareMathOperator{\FTCSet}{\bf{FTCSet}}
\DeclareMathOperator{\FCSet}{\bf{FCSet}}
\DeclareMathOperator{\FSet}{\bf{FSet}}
\DeclareMathOperator{\SET}{\bf{Set}}
\DeclareMathOperator{\Fun}{\bf{Fun}}
\DeclareMathOperator{\CFun}{\bf{CFun}}

\DeclareMathOperator{\Cov}{\bf{Cov}}
\DeclareMathOperator{\Alg}{\bf{Alg}}

\DeclareMathOperator{\Grp}{\bf{Grp}}
\DeclareMathOperator{\Top}{\bf{Top}}
\DeclareMathOperator{\Rng}{\bf{Rng}}
\DeclareMathOperator{\EF}{\bf{EF}}
\DeclareMathOperator{\Vect}{\bf{Vect}}
\DeclareMathOperator{\Rep}{\bf{Rep}}

\DeclareMathOperator{\LCS}{\bf{LCS}}
\DeclareMathOperator{\Mod}{\bf{Mod}}
\DeclareMathOperator{\Comod}{\bf{CoMod}}
\DeclareMathOperator{\Bund}{\bf{Bund}}
\DeclareMathOperator{\Fet}{\bf{F\acute{e}t}}
\DeclareMathOperator{\FeAl}{\bf{F\acute{e}Al}}
\DeclareMathOperator{\LS}{\bf{LS}}

\begin{document}

\begin{titlepage}
    \begin{center}
        \vspace*{3cm}
            
        \Huge
        \textbf{\'Etale Fundamental Groups}
            
        \vspace{0.5cm}
        \LARGE
            A geometric and topological approach to fundamental groups in algebraic geometry
            
        \vspace{1.5cm}
            
        \textbf{Loris De Vos}
            
        \vfill

        \large
        A thesis presented for the degree of\\
        Master in Mathematics $2025$-$2026$
            
        \vspace{0.8cm}
        \includegraphics[width=0.4\textwidth]{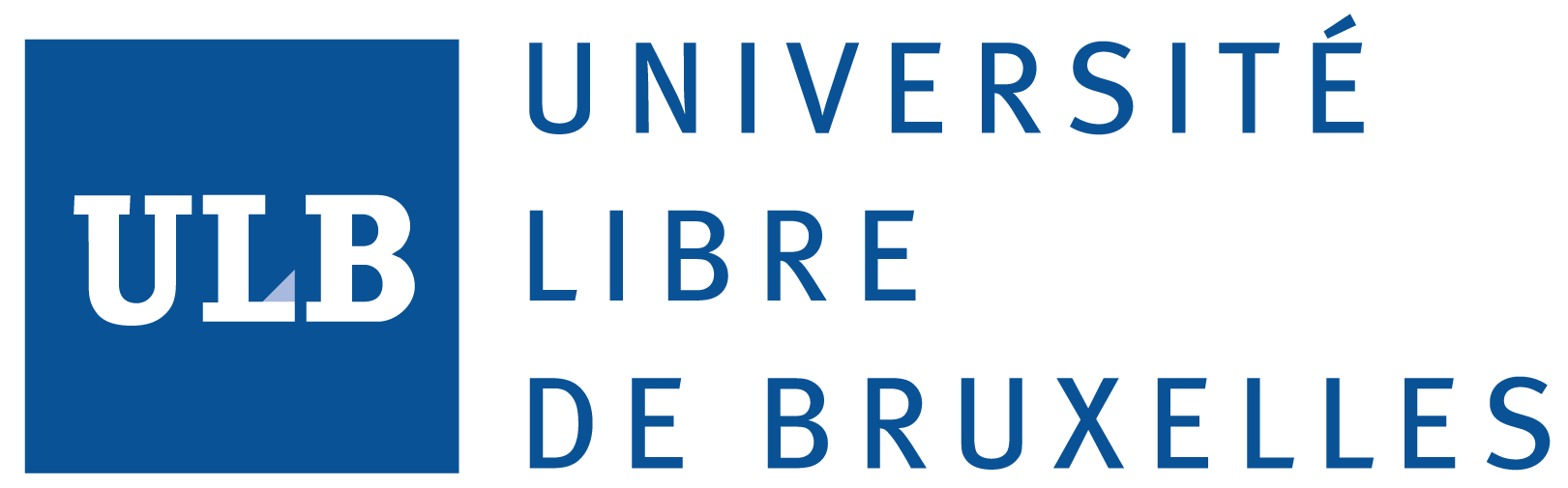}  

         \vspace{0.8cm}
        \large
        Facult\'e des Sciences\\
        Supervisor: Prof. Dr. Joost Vercruysse
    \end{center}
\end{titlepage}

\newpage

\begin{abstract}
    \noindent This thesis explores the notion of fundamental groups across three mathematical settings. 
    We begin with the classical topological theory of covering spaces, highlighting its structural analogy with Galois theory. 
    We then follow Grothendieck in transporting these ideas to algebraic geometry.
    The inadequacy of the Zariski topology motivates the \'etale topology, from which the \'etale fundamental group is constructed and compared to its topological counterpart via transcendental methods. 
    Finally, we linearise the theory through Tannakian duality, where fundamental groups are recovered as automorphism groups of fibre functors on certain monoidal categories, a framework broad enough to encompass \'etale, topological, and motivic Galois groups alike.
\end{abstract}

\newpage
\tableofcontents
\newpage

\section*{Introduction}
\addcontentsline{toc}{section}{Introduction}

The aim of this thesis is threefold. 
It is first to give an overview of Grothendieck's interpretation of topological covering spaces and the  fundamental group of a topological space.
Then, to show how it led to the realisation of similar concepts within the category of schemes.
And finally, to link this approach to the classical results of Tannaka and Krein.\\

The fundamental group of a topological space is an object which occupies a central place in mathematics and, more precisely, in algebraic topology because it provides one of the first ways of translating geometric phenomena into algebraic language.
In the first chapter of this dissertation, we will study the fundamental group through the scope of the theory of covering spaces. 
This approach is well-suited for the developments appearing later in the thesis because it reveals a deep analogy between topology and Galois theory.
Indeed, covering spaces behave in many respects like field extensions: automorphism groups of coverings resemble Galois groups, universal covers play the role of the algebraic closures, and the classification of coverings mirrors the classical correspondence between intermediate field extensions and subgroups of its Galois group.
We will then discuss a generalisation by \textit{Krull} of classical Galois theory for infinite field extensions, which takes advantage of the topology endowed on profinite groups. 
This perspective emphasises the fact that the essential structure underlying classical Galois theory is not tied specifically to fields, 
but rather to discrete sets equipped with an action of a Galois group.
From there, we will explain how covering spaces fit into this broader framework by showing a similar reformulation for the classification of covering spaces. 
The first chapter will end with a triple interpretation of covering spaces, as locally constant sheaves and as representations of the fundamental group.
This chapter will serve as both a review of classical material and a conceptual foundation for the remainder of the thesis.\\


One of Grothendieck’s deepest insights was that an analogy similar to the one between Galois theory and covering spaces could be transported into algebraic geometry.
But the difficulty in doing so lies in the fact that schemes do not possess a sufficiently rich topology. 
Indeed, for the Zariski topology, many schemes become topologically simply connected, making the resulting fundamental group trivial.
Consequently, a new notion of topology had to be introduced. 
It led to the \textit{\'etale topology}, for which coverings are no longer given by open subsets, but by families of \'etale morphisms.
One possible axiomatic approach we could have taken in this work was to study the \textit{Galois category} of finite \'etale morphisms over connected schemes. 
The \'etale fundamental group would then appear as the automorphism group of a certain fibre functor. 
This is done in \cite{SGA} and in \cite{Lenstra}.
In this second chapter, we chose to follow a more direct approach, also presented in \cite{Szamuely},
where the goal is to construct the algebraic counterpart of the theory developed for topological spaces.
This will allow us to define an object analogous to the fundamental group of topological spaces called the \textit{\'etale fundamental group}.
This can only be done after a thorough investigation of \'etale morphisms.
Here the reader will contemplate how our approach of the fundamental group through covering spaces from the first chapter bears its fruits.
The techniques we use to describe the \'etale fundamental group are almost equivalent to the ones used to describe its topological counterpart.
For the sake of accuracy, we then present an approach that allows us to compare our two analogue fundamental groups.
It uses techniques called transcendental methods. 
These show that the \'etale fundamental group of a connected scheme of finite type over $\mathbb{C}$, is the profinite completion of the topological fundamental group of the analytification of that scheme, provided we fix compatible base points.
We finish by giving a few structure theorems for the \'etale fundamental group, further deepening the bridge between algebraic geometry, topology and arithmetic.\\

We end our investigation by a linearisation of the above concepts, replacing coverings by linear algebraic data. 
This idea originates from the work of Tadao Tannaka and Marc Krein, and goes as follows; 
rather than studying a (locally) compact group directly, one studies the monoidal (tensor) category formed by its representations and attempts to recover the group from this category and an associated fibre functor alone.
Later, Grothendieck and Saavedra Rivano extended the theory to affine group schemes. 
The third chapter aims to introduce the basic machinery of Tannakian categories and explain how we can define a Tannakian
fundamental group scheme via the automorphism group of the fibre functor. 
One of the remarkable aspects of this theory is its unifying character, \'etale fundamental groups, topological fundamental groups (and actually differential Galois groups and motivic Galois groups) all arise naturally within the Tannakian framework. 
This highlights the \textit{motivic} character of the Tannakian fundamental group.

\subsubsection*{Notations and Conventions}

\begin{itemize}
    \item[] Most of the notations we use in this thesis are defined in the Appendices, for example every schemes are assumed to be separated.
    \item[] Categories are usually represented by calligraphic capital letters such as $\mathcal{C}$ or bold abbreviations such as $\SET$, $\Rng$, $\Sch$, $\dots$
    \item[] When considering group actions, we always assume that they act on a set from the left.
    \item[] All rings are considered to be commutative.
    \item[] We denote the algebraic closure of a field $k$ by $\overline{k}$, this is not to be confused with the geometric point $\overline{s}$ when $s\in S$ for $S$ a scheme.
    \item[] And in the context of algebraic geometry all schemes are assumed to be Noetherian.
    \item[] When $A$ is a local ring, we denote its (unique) maximal ideal with a fraktur symbol, such as $\mathfrak{p}$ or $\mathfrak{m}$.
    \item[] When $f:X\rightarrow Y$ is a morphism of schemes, $f^{\#}: A \rightarrow B$ is the induced homomorphism of rings for the affine morphism $f:U\rightarrow V$ with $\Spec\: B = U\subset X$ and $\Spec\:A = V\subset Y$ such that $f(U)\subset V$.
    \item[] Other notations are defined in the Appendices.
\end{itemize}

\newpage
\section{Topological Covering Spaces and Galois Theory}

The fundamental group is a core object from algebraic topology; it is an algebraic invariant for topological spaces. 
It is usually defined via homotopy of closed paths,
and many powerful tools can be used to compute them, such as the \textit{Seifert-Van Kampen theorem} or \textit{deformation retracts}.
However, in this thesis we will focus on a method which lies within the theory of covering spaces.
In this chapter, we describe how fundamental groups arise as automorphism groups of certain covering spaces and present theorems classifying these covering spaces.
Interestingly enough, the theory of covering spaces is in some sense parallel to the theory of separable field extension; it has a certain ‘‘Galois flavour" to it.\\

\subsection{The Fundamental Group in Topology}

\noindent We start by defining what homotopies and homotopy groups are.
Let $X$ be a topological space, two paths on $X$ are said to be \textit{homotopic} if it is possible to continuously deform one into another. 
Formally:
\begin{definition}
    A \textbf{homotopy of paths} in $X$ starting at $x_0$ and ending at $x_1$ is a family of paths $\gamma_t:[0,1]\rightarrow X$ indexed by $I$ such that 
    \begin{itemize}
        \item[$a)$] the endpoints are fixed : $\gamma_t(0) = x_0$,  $\gamma_t(1)=x_1$ , and 
        \item[$b)$] the map $H:I\times[0,1]\rightarrow X$ ; $(t,s)\mapsto \gamma_t(s)$ is continuous.
    \end{itemize}
    Two paths $f$ and $g$ are said to be \textbf{homotopic} if there is a homotopy of paths between them and we write $f\simeq g$.
\end{definition}

\noindent It is a straightforward verification to see that the relation ``$\gamma_1$ is homotopic to $\gamma_2$" is an equivalence relation.
Let's denote $\Sigma(X,x)$ the set of loops in $X$ based at $x$, now path-homotopy of loops defines an equivalence relation on $\Sigma(X,x)$.

\begin{definition}
    A continuous map $f:X\rightarrow Y$ is called a \textbf{homotopy equivalence} if there exists a continuous map $g:Y\rightarrow X$ such that $g\circ f \simeq  \id_X$ and $f\circ g\simeq \id_Y$.  
    In this case, the spaces $X$ and $Y$ are said to be \textbf{homotopy equivalent}.
\end{definition}

\begin{definition}
    The \textbf{fundamental group} of $X$ based at $x$, denoted by $\pi_1(X,x)$, is the set of homotopic classes of loops based at $x$. 
\end{definition}

\noindent The fundamental group is also often called the first homotopy group. 
The subscript $1$ in the fundamental group notation hints at the existence of higher homotopy groups $\pi_n(X,x)$, which are similarly considered as sets of homotopy classes based on maps $f:(S^n,s)\rightarrow(X,x)$. 
These higher homotopy groups ‘‘detect" higher-dimensional ‘‘holes" and the situation becomes much more complicated. 
In this thesis, we focus on $\pi_1(X,x)$, which already gives us a very rich theory.\\

\noindent The set of loops equipped with the homotopy equivalence relation is indeed a group with its operation being the concatenation of paths denoted by $\bullet$ and its inverses correspond to reversing the direction of paths. Given two paths $\gamma_1,\gamma_2 : [0,1] \rightarrow X$ with $\gamma_1(1) = \gamma_2(0)$ we define their concatenation $\gamma_1\bullet\gamma_2 : [0,1]\rightarrow X$ by 
\begin{equation*}
  \gamma_1\bullet\gamma_2 =
    \begin{cases}
      \gamma_1(2t) & \text{for $t\in [0,\frac{1}{2}]$}\\
      \gamma_2(2t-1) & \text{for $t\in [\frac{1}{2}, 1]$}.
    \end{cases}       
\end{equation*}

\noindent We say that a topological space is \textbf{simply-connected} if and only if its fundamental group is trivial.
\begin{remark}
    Let $f:(X,x)\rightarrow (Y,y)$ be a continuous map between pointed topological spaces such that $f(x)=y $, then we denote by $f_*:\pi_1(X,x)\rightarrow\pi_1(Y,y)$ the induced homomorphism, obtained by composition, $f_*([\gamma]) := [f \circ\gamma]$ for any $[\gamma]\in \pi_1(X,x)$.
\end{remark}
\noindent As expected, the fundamental group is a topological invariant :
\begin{proposition}
    Homeomorphic spaces have isomorphic fundamental groups.
    If $f:X\rightarrow Y$ is a homeomorphism, then $f_ *:\pi_1(X,x)\rightarrow \pi_1(Y,f(x))$ is an isomorphism.
\end{proposition}
\begin{proof}
    If $f$ is a homeomorphism, then $(f^{-1})_*\circ f_* = (f^{-1}\circ f)_* = (id_X)_* = id_{\pi_1(X,x)}$. Similarly, $f_* \circ (f^{-1})_*$ is the identity on $\pi_1(Y,f(x))$.
\end{proof}
\noindent Similarly, one proves the following.
\begin{corollary}
    Two homotopic equivalent spaces have the same fundamental group.
    If $f:X\rightarrow Y$ is a homotopy equivalence, then $f_ *:\pi_1(X,x)\rightarrow \pi_1(Y,f(x))$ is an isomorphism.
\end{corollary}

\noindent One thing to notice is that if $X$ is not path-connected, the fundamental groups based at points in different path-connected components don't have any relation to each other. 
Therefore, it is customary to consider path-connected topological spaces when studying fundamental groups.
The next proposition tells us that in a path-connected space, the fundamental group is insensitive to changes of base points.
Thus, when this is clear, we will sometimes omit to write the base point.
\begin{proposition}
\label{fund_grou_base_pt}
    If $X$ is a path-connected topological space, then $\pi_1(X, x_0)$ and $\pi_1(X, x_1)$ are isomorphic for any $x_0,x_1\in X$.
\end{proposition}
\begin{proof}
    Since two points $x_0$ and $x_1$ may be joined by a path $\gamma$ starting at $x_0$ and ending at $x_1$, $\pi_1(X,x_0)$ is non-canonically isomorphic to $\pi_1(X,x_1)$ via $\delta \mapsto \gamma\bullet \delta \bullet \gamma^{-1}$.
\end{proof}

\begin{example}
Of course, since $\mathbb{R}^n$ is contractible, any loop is homotopic to the constant loop, so $\pi_1(\mathbb{R}^n,x) = \{1\}$ for any base point $x\in\mathbb{R}^n$.
\end{example}

\noindent Computing non-trivial fundamental groups such as $\pi_1(S^1)$ is already harder. An alternative is to use the theory of covering spaces and lifting of paths.
The calculation of $\pi_1(S^1)$ however, has many applications. It allows one to give a proof of the \textit{Fundamental Theorem of Algebra}, to show the \textit{Brouwer's fixed-point theorem in dimension $2$} or even to prove the \textit{Borsuk-Ulam theorem in dimension $2$} (these theorems with their proofs can be found in \cite{Hatcher} as Theorem $1.8$, $1.9$ and $1.10$ respectively).\\

\noindent In the language of categories, the fundamental group is a covariant functor which assigns to each pointed topological space $(X,x)$ its fundamental group based at $x$, and to each pointed continuous map $f:(X,x)\rightarrow(X',x')$ with $f(x)=x'$ its induced homomorphism $f_{*}:\pi_1(X,x)\rightarrow \pi_1(X',f(x))$.\\

\subsection{Covering Spaces and their Galois structure}
\label{fund_grp_cov_spaces}

In this section, we introduce the theory of covering spaces. 
Then we show that it exhibits a structure that behaves similarly to that of field extensions. 
Later on, we will prove the existence of a so-called universal covering that is deeply related to the fundamental group.

\label{covering_spaces}
\begin{definition}
\label{def_cover}
    A \textbf{covering} of a topological space $X$ consists of a topological space $Y$ and a map $p:Y\rightarrow X$ with the following property : for any $x\in X$ there exists a neighbourhood $U$ of $x$ such that 
    $$p^{-1}(U) = \bigsqcup_{y\in p^{-1}(x)} V_y  \text{ and } p\vert_{V_y}:V_y\rightarrow U \text{ is a homeomorphism}$$
    for each $y\in p^{-1}$. The sets $V_y$ are called sheets of the covering. We define a morphism between the covers $p:Y\rightarrow X$ and $q:Z\rightarrow X$ to be a continuous map $f:Y\rightarrow Z$ making the following diagram commute,
 \[
    \begin{tikzcd}[column sep=small]
        Y \arrow[rd,"p"'] \arrow[rr,"f"]   &   &Z \arrow[ld,"q"]\\
          & X      
    \end{tikzcd}
    \]
\end{definition}

\noindent This makes a category over the topological space $X$, whose objects are covers of $X$ and morphisms are morphisms of covers of $X$, called the category of covers of $X$.\\

\begin{remark}
\label{connected_constant_sheets}
    In the following, we will assume that $X$ is (path)-connected, otherwise we can simply consider coverings of connected components individually. The definition yields that each fibre $p^{-1}(x)$ is a discrete set for any $x\in X$, since each sheet $V_y$ contains a unique point of $p^{-1}(x)$, namely $y$. If this set is finite for any $x\in X$, then the number of sheets $\#p^{-1}(x)$ is constant over $U$ and (since $X$ is connected) constant over $X$.\\
    \noindent For the purposes of our thesis, we further assume that all our topological spaces are locally path-connected and semi-locally simply-connected. 
    These assumption can be understood as the technical bedrock for covering space theory to hold.
\end{remark} 
\noindent However, as our first example shows, we do not require that the above covering space must be connected :

\begin{examples}
\label{example_covering}\textbf{ }
\begin{itemize}
    \item[$a)$] Let $I$ be a non-empty discrete topological space and form the topological product $X\times I$. The projection on the first component $p:X\times I\rightarrow X$ turns $X\times I$ into a cover of $X$ with fibre $I$. We call this the \textbf{trivial cover}.

    \item[$b)$] Consider the circle $S^1$, with a map $p:\mathbb{R}\rightarrow S^1$ defined by $p(t)=e^{2\pi i t}$. This map is clearly a connected covering of $S^1$, locally $S^1$ is homeomorphic to $\mathbb{R}$, and it is even an infinitely sheeted covering. An insightful picture is to think about the covering as an infinite helix above the circle.
\end{itemize} 
\end{examples}
\noindent In this sense, every covering space can be characterised as being locally trivial. And if $X$ is connected, the fibres of $p$ are all homeomorphic to the same discrete space $I$.\\

\noindent Of course the composition of two coverings maps is still a covering map where the number of sheets might be greater. We also have the following proposition about composition of coverings. 
We say that a cover is \textit{connected} is the covering space is connected.

\begin{proposition}
\label{composition_covers}
    Let $q:Z \rightarrow X$ be a connected cover and $f:Y \rightarrow Z$ be a continuous map. If the composition $q\circ f:Y \rightarrow X$ is a cover, then so is $f:Y \rightarrow X$.
\end{proposition}
\begin{proof}
    We first show that for any open neighbourhood $V$ of $z\in Z$, its preimage by $f$ is a disjoint union of open sets such that each of them map homeomorphically onto $V$.\\
    Let $x= q(z)$ and $U$ be an open neighbourhood of $x$ such that its preimages by $q$ and $p=q\circ f$ are $q^{-1}(U) = \bigsqcup_i V_i$ and $p^{-1}(U) = \bigsqcup_j W_j$ as in the definition of covering maps.
    So, for each $W_j \subset Y$, its image $f(W_j)$ is a connected subset of $Z$ mapped by $q$ onto $U$, hence there is some $i$ such that $f(W_j)\subset V_i$. 
    Which is actually a homeomorphism since both $f(W_j)$ and $V_i$ are homeomorphically mapped onto $U$ by $q$. 
    This shows that the preimage of any $V_i$ by $f$ is a disjoint union of $W_j$'s.\\
    Now we show that $f$ is surjective. First notice that the argument above implies that $f(Y)$ is open in $Z$. Since $Z$ is connected it is enough to now prove that $Z\backslash f(Y)$ is also open.
    If $z\in Z\backslash f(Y)$ and $U$ is a neighbourhood of $x=q(z)$ as above, then the whole component $V_i$ of $q^{-1}(x)$ containing $z$ must be disjoint from $f(Y)$, otherwise the whole component would have to be contained in $f(Y)$ which is a contradiction. This shows that $f$ is surjective, concluding the proof.
\end{proof}

\noindent Powerful tools in the study of covering spaces are the following lifting lemmas:

\begin{lemma}[Lifting (Homotopy) Property]
\label{unique_lift}
    Let $p:Y\rightarrow X$ be a cover, $y$ a point of $Y$ and $x=p(y)$.
    \begin{itemize}
        \item Given a path $\gamma:[0,1]\rightarrow X$ starting at $x$, there is a unique lift $\widetilde{\gamma}:[0,1]\rightarrow Y$ with $\widetilde{\gamma}(0)=y$ and $p\circ\widetilde{\gamma}=\gamma$.
        \item For each homotopy $h:I\times [0,1]\rightarrow X$ such that $h(t,0) = x$ for all $t\in I$, there is a unique lift $\widetilde{h}:I\times [0,1]\rightarrow Y$ such that $\widetilde{h}(t,0) = y$ for all $t\in I$.
    \end{itemize}
\end{lemma}
\begin{proof}
    See \cite{Hartshorne} (Prop.1.30 and 1.34).
\end{proof}

\begin{lemma}[Lifting Criterion]
\label{lift_criterion}
    Let $p:Y\rightarrow X$ be a connected covering and $f:Z\rightarrow X$ be a continuous map. Then a lift $\widetilde{f}: Z \rightarrow Y$ of $f$ exists if and only if $f_*(\pi_1(Z,z))$ is a subgroup of $p_*(\pi_1(Y,y))$ for $(y,z)\in Y\times Z$.
\end{lemma}
\begin{proof}
    See \cite{Hartshorne}(Prop. 1.33).
\end{proof}

\noindent The following lemma is the universal property for continuous maps with the quotient topology.
\begin{lemma}
\label{univ_ppty_lift}
    Let $X$ and $Y$ be topological spaces and $\sim$ be an equivalence relation on $Y$.
    Let $p:Y\rightarrow X$ be a continuous map such that for all $y,y'\in Y$, $y\sim y' \Rightarrow p(y)=p(y')$.
    Then there exists a unique continuous map $\overline{p}: Y/\sim \rightarrow X$ with $p=\overline{p} \circ \pi$ for $\pi:Y\rightarrow Y/\sim$ the natural projection
\end{lemma}
\begin{proof}
    See \cite{munkrs}, Theorem $II.\:22.2$.
\end{proof}

\noindent As the above lifting properties suggest, there is an intimate connection between covering maps and fundamental groups. The key to further understanding this connection is a natural action on each fibre of a covering by the fundamental group of the base called the \textit{monodromy action}.

\begin{theorem}[The Monodromy Action]
    Let $p:Y\rightarrow X$ be a connected covering and $x\in X$. There is a transitive left action of $\pi_1(X,x)$ on the fibre $p^{-1}(x)$ called the \textbf{monodromy action}, given by $[\gamma]\cdot y = \widetilde{\gamma}_y(1)$ for $y\in p^{-1}(x)$, $[\gamma]\in\pi_1(X,x)$ and $\widetilde{\gamma}_y$ the unique lift of $\gamma$ in $Y$ starting at $y$.
\end{theorem}
\begin{proof}
    Let any $y\in p^{-1}(x)$, then the path-lifting property \ref{unique_lift} associates to every loop $\gamma$ based at $x$ in $X$ a unique lift which is a path $\widetilde{\gamma}_y$ starting at $y$ and finishing at $\widetilde{\gamma}_y(1)\in p^{-1}(x)$.
    By lifting of homotopies we have that if $\gamma_1 \sim \gamma_2$, then $\widetilde{\gamma_1}_y \sim \widetilde{\gamma_2}_y$, so $\widetilde{\gamma}_y(1)$ depends only on the path-class of $\gamma$ ($[\gamma]\cdot y$ is now well defined).
    Let us prove that this is indeed a group action.
    Since the constant path is the unique lift of the constant loop $[c] \cdot y = \widetilde{c_y}(1) = y$.
    Suppose $\gamma$ and $\delta$ are two loops based at $x$ and let $z = [\gamma]\cdot y = \widetilde{\gamma_y}(1)$, then $[\delta]\cdot ([\gamma]\cdot y) = \widetilde{\delta_z}(1)$. 
    And $\widetilde{\delta_z}\bullet \widetilde{\gamma_y}$ being the lift of $\delta\bullet\gamma$ at $y$, we associativity as follows: 
    $$([\delta]\bullet[\gamma])\cdot y = [\delta\bullet\gamma]\cdot y = (\widetilde{\delta_z}\bullet\widetilde{\gamma_y})(1) = \widetilde{\delta_z}(1) = [\delta]\cdot ([\gamma]\cdot y)$$
    And finally since $Y$ is path-connected, any two points $y, y'$ in $p^{-1}(x)$ are joined by a path $\omega$ in $Y$. 
    Setting $\gamma = p \circ \omega$ ($\omega$ becomes a loop when its endpoints are collapsed together by $p$), we see that $\omega$ is actually a lift of $\gamma$ and therefore by monodromy action $[\gamma]\cdot y = y'$. The action is transitive on $p^{-1}(x)$.
\end{proof}

\noindent Below we prove a theorem that is a direct consequences of the lifting properties.

\begin{theorem}
\label{property_morphism_covers}
    Let $p:Y\rightarrow X$ and $q:Z\rightarrow X$ be two coverings. 
    \begin{itemize}
        \item[$a)$] If two morphism of coverings from $p$ to $q$ agree at one point of $Y$, then they are equal.
        \item[$b)$] Fixing $x\in X$, any morphism between the coverings $p$ and $q$ restrict to a $\pi_1(X,x)$-equivariant map from $p^{-1}(x)$ to $q^{-1}(x)$ (with respect to the monodromy action).
    \end{itemize}
\end{theorem}
\begin{proof}
    A morphism $\varphi:Y\rightarrow Z$ of covers can also be seen as a lift of $p:Y\rightarrow X$. Thus uniqueness follows from uniqueness of the lifting property proving $a)$.\\
    Suppose $\varphi:Y\rightarrow Z$ is a morphism of covers (from $p$ to $q$).
    By definition, $\varphi$ maps $p^{-1}(x)$ to $q^{-1}(x)$.
    Now, given $x\in X$, $y\in p^{-1}(x)$ and $[\gamma]\in\pi_1(X,x)$, we need to show that $\varphi([\gamma]\cdot y)= [\gamma]\cdot\varphi(y)$.
    Let $\widetilde{\gamma}_y$ be the lift of $\gamma$ to a path in $Y$ starting at $y$, and consider the path $\varphi\circ\widetilde{\gamma}_y$ in $Z$.
    Its initial point is $\varphi\circ\widetilde{\gamma}_y(0) = \varphi(y)$ and it satisfies $q\circ\varphi\circ\widetilde{\gamma}_y = p\circ\widetilde{\gamma}_y$, so $\varphi\circ\widetilde{\gamma}_y$ is the lift of $\gamma$ to $Z$ starting at $\varphi(y)$. Therefore, by the monodromy action,
    $$[\gamma]\cdot\varphi(y) = (\varphi\circ\widetilde{\gamma}_y)(1) = \varphi(\widetilde{\gamma}_y(1)) = \varphi([\gamma]\cdot y).$$
\end{proof}

\noindent A clever way of retrieving new covering spaces from existing ones, is to consider their quotient by some group action. Let us define the appropriate group action needed in this case :

\begin{definition}
\label{properly_discontinuous}
    A group $G$ acting on a topological space $X$ by homeomorphisms (here we may understand by automorphisms) is said to act \textbf{properly discontinuously} if for each $x\in X$, there exist an open neighbourhood $U$ of $x$ such that $(g\cdot U)\cap U = \emptyset$ for all $g\in G$ and $g\neq 1$.
    Properly discontinuous actions are also called covering space actions.
\end{definition}

\begin{proposition}
\label{quotient_cover}
    Let $Y$ be a connected topological space and let $G$ act on $Y$ properly discontinuously. Then the quotient map $p_G : Y \rightarrow Y/G$ is a covering.
\end{proposition}
\begin{proof}
    The map $p_G$ is surjective and moreover each $x\in Y/G$ has an open neighbourhood of the form $U = p_G(V)$ with $V$ such that $V\cap (g\cdot V)=\emptyset$ for all $g\in G$ with $g\neq 1$. 
    And $U$ satisfies the conditions of Definition \ref{def_cover} so that $Y$ becomes a covering of $Y/G$.
\end{proof}

\noindent We now introduce a particular group which has the property of acting properly discontinuously on covering spaces.

\begin{definition}
    Let $p:Y\rightarrow X$ be a covering. A \textbf{deck transformation} is a homeomorphism $d:Y\rightarrow Y$ such that the following diagram commutes.
    \[
    \begin{tikzcd}[column sep=small]
                Y \arrow[rd,"p"'] \arrow[rr,"d"]   &   &Y \arrow[ld,"p"]\\
                                        & X                        
    \end{tikzcd}
    \]
    Together with composition, the set of deck transformations forms a group denoted $\Deck(p)$ or $\Aut(Y|X)$  $($the second notation may be familiar to the reader, as it resembles the notation used to describe Galois groups$)$.
\end{definition}

\noindent Notice that for any $x\in X$, the automorphism group $\Aut(Y|X)$ sends all of $p^{-1}(x)$ to itself. We further assume that $\Aut(Y|X)$ acts on $Y$ by the left.

\begin{corollary}
\label{corollary_property_morphism_covers}
    A homeomorphism $d\in \Aut(Y|X)$ of a connected covering $p:Y\rightarrow X$ having a fixed point must be trivial.
\end{corollary}
\begin{proof}
    This is immediate from Theorem \ref{property_morphism_covers} if we take one (auto)morphism of covers to be the deck transformation $d$ and another one to be the identity.
\end{proof}

\begin{proposition}
    If $p:Y\rightarrow X$ is a connected covering, the action of $\Aut(Y|X)$ on $Y$ is properly discontinuous.
\end{proposition}
\begin{proof}
    Let $y\in Y$ and set $x = p(y)$. Let $U$ be a connected open neighbourhood of $x$ such that $p^{-1}(U)$ is a disjoint union of open sets $V_i$ as in the Definition \ref{def_cover}, where one of these opens, let's say $V_i$, contains $y$. 
    For $d\in \Aut(Y|X)$ non-trivial, $d$ maps $V_i$ isomorphically onto some $V_j$ by definition of a deck transformation. Since $Y$ is connected, we can apply Corollary \ref{corollary_property_morphism_covers} and since $d\neq id_Y$, we must have $i\neq j$. This satisfies the condition of properly discontinuity.
\end{proof}

\noindent One more thing can be said about the group $G$, acting properly discontinuously on the connected topological spaces $Y$ and the associate quotient map $p_G:Y\rightarrow Y/G$ from Proposition \ref{quotient_cover}:

\begin{proposition}
\label{G_is_aut_group}
    If $G$ is a group acting properly discontinuously on a connected space $Y$, the automorphism group of the covering $p_G:Y\rightarrow Y/G$ is precisely $G$.
\end{proposition}
\begin{proof}
    Since $G$ acts by homeomorphism on $Y$ and $p$ sends any $y\in Y$ to $p(y)\in Y/G$, then $p(y) = p(g\cdot y)$ by properly discontinuity of the action of $G$ onto $Y$, therefore $G$ is a subgroup of $\Aut(Y|(Y/G))$. 
    Now, let $\phi\in\Aut(Y, (Y/G))$, and $y\in Y$. Since the fibres of $p_G$ are exactly the orbits of $G$ (we quotient $Y$ by it), we may find $g\in G$ such that $\phi(y) = g\cdot y$. 
    Now applying Corollary \ref{corollary_property_morphism_covers} to the homeomorphism $\phi\circ g^{-1}$, we get $g=\phi$.
\end{proof}

\noindent From the previous propositions, it is now obvious that for a connected covering $p:Y\rightarrow X$, 
we have the following decomposition $p = \overline{p} \circ p_{\Aut(Y|X)} : Y\xrightarrow{p_{\Aut(Y|X)}} Y/ \Aut(Y|X) \xrightarrow{\overline{p}} X$. Where $p_{\Aut(Y|X)}$ and $\overline{p}$ are continuous maps and $p_{\Aut(Y|X)}$ is a now a \textit{subcovering} of $p$.

\begin{definition}
    We say that a covering $p:Y\rightarrow X$ is \textbf{Galois} or \textbf{Normal} if $Y$ is connected and if the map $\overline{p}$ defined above is a homeomorphism.
\end{definition}

\noindent The previous proposition already gives us an example of Galois coverings.
These types of coverings will turn out to be important in the classification of covering spaces (Theorem \ref{clasification_of_covers}), playing a similar role to Galois field extensions in the Main Theorem of Galois. 

\begin{remark}
\label{Galois_free}
    If $Y$ is a connected Galois covering of $X$, then $\Aut(Y|X)$ acts freely on $p^{-1}(x)$ for any $x\in X$.
    This follows by Corollary \ref{corollary_property_morphism_covers}, indeed for any two points $y$ and $y'$ in $p^{-1}(x)$, there is an automorphism $\phi\in \Aut(Y|X)$ mapping $p(y)\mapsto p(y')$.
    Suppose that there exists another such automorphism $\phi'$.
    Then we can construct a map $d:= \phi^{-1}\circ \phi'$, fixing $p$. 
    By the corollary it must be the identity and thus $\phi=\phi'$, 
    meaning that the action is not only transitive but also free.
\end{remark}

\noindent Finally, we will give equivalent characterisations of Galois covers. But first, we need two lemmas

\begin{lemma}
\label{induced_subgroup}
    Let $p:Y\rightarrow X$ be a connected covering and $x\in X$. For each $y\in p^{-1}(x)$, the stabiliser of $y$ under the monodromy action is $p_*(\pi_1(Y,y)) \leq \pi_1(X,x)$. And as $y$ varies over the fibre $p^{-1}(x)$, the set of induced subgroups $p_*(\pi_1(Y,y))$ is exactly one conjugacy class in $\pi_1(X,x)$.
\end{lemma}
\begin{proof}
    Let $y\in p^{-1}(x)$ be arbitrary, and suppose that $[\gamma]$ is in the stabiliser of $y$.
    This means that $\widetilde{\gamma}_y(1) = [\gamma]\cdot y = y$ (by the monodromy action), hence $\widetilde{\gamma}$ is actually a loop in $Y$ and thus represents an element in $\pi_1(Y,y)$. And by uniqueness of lifts we have that the induced map $p_* : \pi_1(Y,y) \rightarrow \pi_1(X,x$ sends $[\widetilde{\gamma}_y]$ to $[\gamma]$, so $p_*(\pi_1(Y,y))$.
    Conversely, if $[\gamma] \in p_*(\pi_1(Y,y))$, then there is a loop $\delta$ in $Y$ based at $y$ such that $p_*([\delta]) = [\gamma]$, which means that $p\circ \delta$ and $\gamma$ are homotopic equivalent.
    Denote $\gamma' = p\circ \delta$, then $\delta = \widetilde{\gamma'}_y$ by uniqueness of lifts and $[\gamma]\cdot y = [\gamma']\cdot y = \widetilde{\gamma'}_y(1) = \delta(1) = y$, which means that $[\gamma]$ is in the stabiliser of $y$.\\
    The last statement follows from Theorem \ref{Cameron_1_3} on the classification of transitive $G$-spaces by conjugacy classes subgroups of $G$ and quotients by stabilisers. As $y$ varies over $p^{-1}(x)$, the set of stabilisers of $y$ is exactly one conjugacy class in $\pi_1(X,x)$ (recall that the monodromy action is a transitive action).
\end{proof}

\noindent The following lemma is actually a criterion for deciding whether the morphism of coverings of Theorem \ref{property_morphism_covers} is actually an isomorphism.

\begin{lemma}
\label{isomorphism_conjugate}
    Let $p:Y\rightarrow X$ and $q:Z\rightarrow X$ be two connected coverings.\\
    $Y$ and $Z$ are isomorphic if and only if, for some $x\in X$ the conjugacy classes of subgroups of $\pi_1(X,x)$ induced by $q_1$ and $q_2$ are the same.
\end{lemma}
\begin{proof}
    First assume that $Y$ and $Z$ are isomorphic. 
    For any $x\in X$, this isomorphism restrict to a $\pi_1(X,x)$-equivariant bijection from $p^{-1}(x)$ to $q^{-1}(x)$ as in Theorem \ref{property_morphism_covers}. 
    So the sets of stabilisers of points in $p^{-1}(x)$ and in $q^{-1}(x)$ are in the same conjugacy class subgroup as Theorem \ref{Cameron_1_3} suggests. 
    And since these groups are exactly the induced subgroups $p_*(\pi_1(Y,y))$ and $q_*(\pi_1(Z,z))$ form Lemma \ref{induced_subgroup}, they must be conjugate.\\
    Now suppose that $p$ and $q$ induce the same conjugacy class of subgroups for some $x\in X$. 
    Choosing $y\in Y$ arbitrarily, by Lemma \ref{induced_subgroup}, there is some $z\in Z$ such that $p_*(\pi_1(Y,y)) = q_*(\pi_1(Z,z))$.
    By the Lifting Criterion, there exist  homomorphism of coverings $f : Z \rightarrow Y$ and $g: Y\rightarrow Z$ (since they are essentially lifts) with $f(z) = y$ and $g(y)=z$. 
    The composites $f\circ g$ and $g\circ f$ are covering homomorphism fixing $z$ and $y$ respectively, so by Corollary \ref{corollary_property_morphism_covers}, they are identities, there is an isomorphism of coverings between $p$ and $q$.
\end{proof}

\noindent Finally we get to the following characterisation,

\begin{theorem}
\label{caracterize_Galois_cov}
    Let $p:Y\rightarrow X$ be a connected covering.\\
    The following conditions are equivalent :
    \begin{itemize}
        \item[$a)$] $p$ is a Galois covering.
        \item[$b)$] $Y$ is connected and $\Aut(Y|X)$ acts transitively on the fibre $p^{-1}(x)$ for any $x\in X$.
        \item[$c)$] The stabiliser $p_*(\pi_1(Y,y))$ is a normal subgroup of $\pi_1(X,x)$ for any $y\in p^{-1}(x)$.
    \end{itemize}
\end{theorem}
\begin{proof}
    $a) \Leftrightarrow b)$ ; 
    by definition, the underlying set of $Y/\Aut(Y|X)$ is the $\Aut(Y|X)$-orbits of $Y$. 
    So the map $\overline{p}$ is one to one precisely when each such orbits is equal to a whole fibre of $p$ or equivalently that $\Aut(Y|X)$ acts transitively on each fibre. 
    And since $X$ is connected, all fibres are homeomorphic to the same discrete space.

    $b) \Leftrightarrow c)$ ; fix $x\in X$.
    Combining Lemmas \ref{induced_subgroup} and \ref{isomorphism_conjugate}, we have that
    $\Aut(Y|X)$ acts transitively on $p^{-1}(x)$ if and only if the subgroups $p_*(\pi_1(Y,y))$ are the same for all $y\in Y$. And since they are conjugacy classes, it is equivalent to saying that $p_*(\pi_1(Y,y))$ is a normal subgroup of $\pi_1(X,x)$.\\
\end{proof}

\noindent We now touch upon the classification of topological covering spaces, which is surprisingly similar to the Galois correspondence of field extensions, here, Galois covers play the role of Galois field extensions, the Galois groups are automorphism groups of covers, and as we will see later, universal coverings are separable closures.


\begin{theorem}[Main Classification of Coverings]
\label{clasification_of_covers}
    Let $X$ be a connected topological space, $p:Y\rightarrow X$ be a Galois cover and $G$ the group automorphisms of $p$. 
    If $q:Z\rightarrow X$ is a connected cover such that $p = q\circ f$ for $f:Y\rightarrow Z$, then $f$ is a Galois cover and $Z\cong Y/H$ (where $H=\Aut(Y|Z)$ is a subgroup of $G$).
    The maps $H\mapsto Y/H$ and $Z\mapsto \Aut(Y|Z)$ induce a bijection between the subgroups of $G$ and the intermediate covers $Z$ as above.
    The cover $q:Z\rightarrow X$ is Galois if and only if $\Aut(Y|Z)=H$ is a normal subgroup of $G$ (in that case, $\Aut(Z|X) = G /H$.
\end{theorem}
\begin{proof} 
    By hypothesis, the map $p$ factors as the composite $\overline{p}\circ p_H$ with $\overline{p}= Y/H\rightarrow X$ and $p_H:Y\rightarrow Y/H$ for $H$ a subgroup of $G$.
    Since $p$ and $p_H$ are coverings (Proposition \ref{quotient_cover}), then $\overline{p}$ is continuous.
    Moreover, $p$ is a covering, so by definition, there exists an open set $U$ in $X$ such that $p^{-1}(U) \cong U\times I$ and the subgroup $H$ of $G$ acts on $I$.
    So we have that $\overline{p}^{-1}(U)$ is open subset of $Y/H$ with $\overline{p}^{-1}(U)=\cong U\times J$ where $J$ is the set of $H$-orbits of $I$.
    Hence $\overline{p}$ is a cover since locally (for a sufficiently small open set) it is a trivial cover.\\
    Now suppose that $q:Z\rightarrow X$ is a connected cover with $p=q\circ f$ as above. Then by Proposition \ref{composition_covers}, $f:Y\rightarrow Z$ is a covering map as well. 
    To see that it is moreover a Galois cover, by the characterisation \ref{caracterize_Galois_cov} of Galois coverings, we need to check that $\Aut(Y|Z)$ acts transitively on a fibres of $f$.
    Let $z\in Z$ and $y_1,y_2\in f^{-1}(z)$. 
    Then $y_1,y_2\in p^{-1}(q(z))$ because $f^{-1}(z) \subset p^{-1}( q(z))$.
    Since $p$ is a Galois covering, there exists $\phi\in\Aut(Y|X)$ such that $y_1 = \phi(y_2)$. 
    We are done if we show that actually $\phi\in\Aut(Y|Z)$, which is equivalent to saying that the set $\{y\in Y \mid f(y)=f(\phi(y))\}$ is equal to the whole of $Y$.
    But this follows from Theorem \ref{property_morphism_covers}, so we have that $f = f\circ\phi$ and $\{y\in Y \mid f(y)=f(\phi(y))\} = Y$ and $\phi\in \Aut(Y|Z)$.
    The above two constructions being inverse to each other, only the last statement is left to prove.\\
    One implication is easy. Since $p$ is a Galois covering, if $H$ is a normal subgroup of $G$, then $G/H$ acts transitively on fibres of $q$ in $Z=Y/H\subset Y$. 
    Making $q:Z\rightarrow X$ a Galois cover. 
    And by definition of Galois covers and Proposition \ref{G_is_aut_group}, $Z/(G/H) \cong X$ and so $\Aut(Z|X)\cong G/H$.\\
    For the other implication, assume that $q:Z\rightarrow X$ is a Galois covering.
    What we want to show is that each $\phi\in \Aut(Y|X)$ induces an automorphism $\varphi$ of $Z$ over $X$, which in term induces a group homomorphism $h : G \rightarrow \Aut(Z|X)$ given by $\phi\mapsto\varphi$. 
    Let $y\in Y$ with $x = p(x)= (q\circ f)(y)$. Since $q$ is a covering, $f(y)$ and $f(\phi(y))$ are in the same fibre $q^{-1}(x)$ of $q$.
    $q$ being moreover Galois, there is an automorphism $\varphi\in\Aut(Z|X)$ with $\varphi(f(y))=f(\phi(y))$ by transitivity of automorphism groups of Galois covers.
    Our group homomorphism $h$ is well-defined since $\varphi$ is unique, for if $\psi\in\Aut(Z|X)$ was another such automorphism, by Theorem \ref{property_morphism_covers}, $\psi\cong\varphi$.
    By the same argument, the maps $\varphi\circ f$ and $f\circ \phi$ are equal making the following diagrams commute
    \[
    \begin{tikzcd}
        Y \arrow[r,"\phi"] \arrow[d, "f"'] & Y \arrow[d,"f"] \\
        Z \arrow[r,"\varphi"] \arrow[d,"q"'] & Z \arrow[d,"q"]\\
        X \arrow[r, "\id"'] & X
    \end{tikzcd}
    \]
    And since the kernel of the morphism $h$ is simply $H = \Aut(Z|X)$, we have that indeed $H$ is a normal subgroup of $G$.
\end{proof}

\noindent We end our discussion on covering spaces and actions with the following two results linking deck transformations and monodromy actions. 
Below, we denote for $H$ a subgroup of $G$, the \textit{normaliser} of $H$ in $G$ as $N_G(H) = \{g\in G\mid gH=Hg\}$.
\begin{theorem}
\label{th_link_aut_fund}
Let $p:Y\rightarrow X$ be a covering map and fix $y\in Y$ and $x = p(y)$. Then we have the following isomorphism
    $$\Aut(Y|X) \cong \frac{N_{\pi_1(X,x)}(p_*(\pi_1(Y,y)))}{p_*(\pi_1(Y,y))}  .$$
\end{theorem}
\begin{proof}
    See \cite{Intro_manifolds}, Theorem $12.7$
\end{proof}

\begin{corollary}
\label{co_link_aut_fund}
    For $X$, $Y$, $x$ and $y$ as in Theorem \ref{th_link_aut_fund}. If the covering $p:Y\rightarrow X$ is a Galois covering, then 
    $$\Aut(Y|X) \cong \frac{\pi_1(X,x)}{p_*(\pi_1(Y,y))}.$$
\end{corollary}
\begin{proof}
    Follows from Theorem \ref{th_link_aut_fund} since $p_*(\pi_1(Y,y))$ is a normal subgroup of $\pi_1(X,x)$ (Theorem \ref{caracterize_Galois_cov}).
\end{proof}

\begin{remark}\textbf{ }
    \begin{itemize}
        \item[$a)$] Recall from Theorem \ref{th_link_aut_fund} and Corollary \ref{co_link_aut_fund} that the Main Classification of Coverings for $p:Y\rightarrow X$ Galois can be restated in terms of fundamental groups $\pi_1(X,x)$ and induced subgroups $p_*(\pi_1(Y,y))$ as follow:

        \hspace*{2mm} \textbf{Theorem :} \textit{Let $X$ be a connected topological space and $x\in X$ is any base point.
        There is a one-to-one correspondence between isomorphism classes of coverings of $X$ and conjugacy classes of subgroups of $\pi_1(X,x)$.
        The correspondence associates each covering $p:Y\rightarrow X$ with the conjugacy class of its induced subgroup $p_*(\pi_1(Y,y))$.}

        \item[$b)$] From the classification of covering spaces, it is now direct that if $\pi_1(X,x)$ is abelian, every covering of $X$ is a Galois covering.\\
    \end{itemize}
\end{remark}

\subsection{Universal Cover and Fundamental Group}
\label{universal_cover_section}

In the previous section, we discussed the concept of subcovers, or covers that cover other covers. 
We now introduce a new type of covering space, which can be thought of as the covering of all coverings, inherently called the universal cover.

\begin{definition}
    Let $X$ be a path-connected, locally simply-connected topological space and pick a base point $x\in X$. The covering $q: \widetilde{X}_x \rightarrow X$ is called an \textbf{universal cover} if $\widetilde{X}_x$ is simply-connected.
\end{definition}

\noindent It is important to note that universal covers do not exists in general. 
But for sufficiently nice topological spaces (path-connected, simply path-connected) they always exists,
for more on this, we refer to \cite{Hatcher} page $63$.

\begin{construction}
    For $X$ path-connected and locally simply-connected, there is a way to construct a \textbf{simply-connected cover} $\widetilde{X}_x$ of $X$. \\
    The points of $\widetilde{X}_x$ are homotopy classes of paths on $X$ starting from $x$. 
    That way, $\widetilde{X}_x$ has a canonical point $\widetilde{x}$ corresponding to the constant path on $x$.
    This construction is motivated by the fact that $\widetilde{X}_x$ has to be simply-connected, thus any two points of $\widetilde{X}_x$ can be joined by a unique homotopy class of paths.
    The projection $q: \widetilde{X}_x \rightarrow X$ is defined as follows:
    take a point $\widetilde{y}\in \widetilde{X}_x$ consider its representing path $\gamma:[0,1]\rightarrow X$ starting at $x$ and ending at $y$, and set $q(\widetilde{y})=\gamma(1)=y$.\\
    Simply-connected covers constructed that way are endowed with the following topology: 
    consider as a basis of open neighbourhoods of $\widetilde{y}\in \widetilde{X}_x$, the sets $\widetilde{U}$.
    To define $\widetilde{U}$, we start with a simply-connected neighbourhood $U$ of $q(\widetilde{y})$ in $X$, and if $\gamma:[0,1]\rightarrow X$ is a path, representing $\widetilde{y}$, we define $\widetilde{U}_{\widetilde{y}}$ to be the set of homotopy classes of paths obtained by composing the homotopy class of $f$ with the homotopy class of some path $\delta:[0,1]\rightarrow X$ with $\delta(0) = y$ and $\delta([0,1])\subset U$. 
    That way, $\widetilde{U}$ can be thought of as obtained by ‘‘continuing homotopy classes of paths arriving at $y$ to other points of $U$".
    Given two neighbourhoods $\widetilde{U}_{\widetilde{y}}$ and $\widetilde{V}_{\widetilde{y}}$ of $\widetilde{y}$, there exists $\widetilde{W}_{\widetilde{y}}$ for some simply-connected neighbourhood $W$ of $y$, contained in $U\cap V$.
\end{construction}

\noindent Another consequence of the lifting criterion is that a universal cover $\widetilde{X}$ of $X$ is unique up to isomorphism.

\begin{definition}
    Let $X$ be a topological space and $x\in X$. We define the \textbf{fibre functor $\Fib_x$} from the category of coverings spaces $\Cov(X)$ of $X$ to the category $\pi_1(X,x)$-$\SET$ of sets equipped with a $\pi_1(X,x)$-action, by sending a cover $p:Y\rightarrow X$ to the fibre $p^{-1}(x)$.
\end{definition}

\noindent We immediately check that $\Fib_x$ is a functor, morphism of covers carry over to the fibres by Theorem \ref{property_morphism_covers} and the composition condition is similarly verified.

\begin{proposition}
\label{representable_fibre_functor}
    Let $X$ be a path-connected, locally simply-connected topological space and pick a base point $x\in X$. 
    The functor $\Fib_x :\Cov(X) \rightarrow \pi_1(X,x)$-$\SET$ is representable by the universal cover $q:\widetilde{X}_x\rightarrow X$.
\end{proposition}
\begin{proof}
    We want to show that for a cover $p:Y\rightarrow X$, each point $y\in p^{-1}(x)$ corresponds to a morphism $\varphi_y :\widetilde{X}_x\rightarrow Y$ of covers.
    We define the map $\varphi_y$ as follows : it takes a point $\widetilde{y}\in \widetilde{X}_x$ represented by a path $f:[0,1]\rightarrow X$ starting at $x$, lifts this path to a path $\widetilde{f}$ in $Y$ starting at $y$ and sends  $\widetilde{y}$ to $\widetilde{f}(1)$.
    The lifting properties unsure that this map is well defined for any choice of $y\in Y$. And by construction $\varphi_y$ is a morphism of covers. \\
    The map  $y\mapsto\varphi_y$ is a bijection between $p^{-1}(x)$ and $\Hom_X(\widetilde{X}_x,Y)$, which has for inverse, the map sending a morphism $\phi \in \Hom_X(\widetilde{X}_x,Y)$ to the evaluation $\phi(\widetilde{x})$, the element $\widetilde{x}$ being represented by the constant path.
    Finally, the bijection we obtain is functorial : given a morphism of cover $h: Y\rightarrow Z$ : $y\mapsto z$ of $X$, the induced map $\Hom_X(\widetilde{X}_x, h): \Hom_X(\widetilde{X}_x,Y)\rightarrow \Hom_X(\widetilde{X}_x,Z)$ maps $\varphi_y$ to $\varphi_z$.
\end{proof}

\begin{proposition}
\label{unive_cov_galois}
    The universal cover $q : \widetilde{X}_x\rightarrow X$ is a Galois cover.
\end{proposition}
\begin{proof}
    Since $\widetilde{X}_x$ is connected, by Theorem \ref{caracterize_Galois_cov} it is enough to prove that the group $\Aut(\widetilde{X}_x,X)$ acts transitively on $q^{-1}(x)$.
    Recall that since the functor $\Fib_x$ is representable by a universal covering (\ref{representable_fibre_functor}), each point $\widetilde{y}\in\ q^{-1}(x)$ correspond to a morphism of covers $\varphi_{\widetilde{y}}:\widetilde{X}_x\rightarrow\widetilde{X}_x$, sending $\widetilde{x}$ to $\widetilde{y}$ ($\varphi_{\widetilde{y}}$ being a cover means that it's a surjective morphism).
    We now check that for any $\widetilde{y}\in q^{-1}(x)$, $\varphi_{\widetilde{y}}$ is an automorphism making $\Aut(\widetilde{X}_x,X)$ act transitively on $p^{-1}(x)$. 
    Now since composition of covers are covers, we can reapply Proposition \ref{representable_fibre_functor} to the cover $q\circ \varphi_{\widetilde{y}}$ and again to each point $\widetilde{z}\in\varphi^{-1}_{\widetilde{y}}(\widetilde{x})$ there is a corresponding morphism of covers $\varphi_{\widetilde{z}}:\widetilde{X}\rightarrow\widetilde{X}$ and we set $\varphi_{\widetilde{z}}(\widetilde{x}) = \widetilde{z}$ 
    (recall that $\widetilde{x}$ is the point in $\widetilde{X}$ represented by the constant path, it can be seen as the lift of the base point $x$ in $X$). \\
    Since, $\varphi_{\widetilde{z}}$ is again a cover compatible with $q\circ \varphi_{\widetilde{y}}$, we have $q\circ\varphi_{\widetilde{y}}\circ\varphi_{\widetilde{z}} = q$, 
    and because $\widetilde{z}\in \varphi_{\widetilde{y}}^{-1}(\widetilde{x})$, $\varphi_{\widetilde{y}}\circ\varphi_{\widetilde{z}}(\widetilde{x}) = \widetilde{x}$, is the identity on $\widetilde{X}_x$ since agreeing at one point implies agreeing everywhere (\ref{property_morphism_covers}). 
    Finally recall that $\varphi_{\widetilde{z}}$ is surjective (as a cover), hence $\varphi_{\widetilde{y}}$ is injective, and thus an automorphism.
\end{proof}

\subsubsection{$\pi_1$ as the Group of Deck Transformations} 

\noindent The following theorem describes the connection between universal covers and fundamental groups.
Recall that in group theory, for a group $G$, its opposite group, denoted $G^{op}$ is the group with the same elements as $G$ but with group operation defined by $(x,y)\mapsto y\cdot x$ (of course, $G\cong G^{op}$). In that case, left actions become right actions and vice versa.

\begin{theorem}
\label{aut_iso_fund}
    There is a natural isomorphism between $\Aut(\widetilde{X}_x|X)^{op}$ and $\pi_1(X,x)$.
\end{theorem}
\begin{proof}
    Let's endow $\widetilde{X}_x$ with the following right $\pi_1(X,x)$-action : let $\widetilde{y}\in \widetilde{X}_x$ and $[\gamma]\in\pi_1(X,x)$ with $f$ the representative of $\widetilde{y}$. 
    We define the right action by considering the homotopy class of the concatenation $f \bullet \gamma$.
    Now the group action map $\phi_{\gamma}:\widetilde{X}_x \rightarrow \widetilde{X}_x$ is continuous and an automorphism of covers (since pre-concatenating with a loop doesn't affect the endpoints).
    Recall from our definition that $\Aut(\widetilde{X}_x|X)$ acts on $\widetilde{X}_x$ by the left (by evaluation).
    This defines a group morphism $\pi_1(X,x)\rightarrow \Aut(\widetilde{X}_x|X)^{op}$ sending any loop homotopy $[\gamma]$ to its group action map $\phi_{\gamma}\in \Aut(\widetilde{X}_x|X)$ which is an injective morphism since for a homotopy class that is not the one of the constant loop the associated group action map moves the element $\widetilde{x}\in\widetilde{X}_x$ associated with the homotopy class of the constant loop.\\
    It is left to show that this morphism is surjective.
    Consider $\phi\in \Aut(\widetilde{X}_x|X)$ and $\widetilde{y}\in\widetilde{X}_x$ corresponding to the homotopy class of some path $f$ in $X$. Then $\phi(\widetilde{y})$ is the point corresponding to the homotopy class of some other path $g$ in $X$ such that their endpoints coincide. 
    This means that the concatenation $f^{-1}\bullet g$ is a closed path around $x$ in $X$, and denote by $\delta$ its homotopy class in $\pi_1(X,x)$. Now do notice that the previously defined map $\phi_{\delta}$ sends then $\widetilde{y}$ corresponding to $[f]$ to $[f]\bullet\delta = [f\bullet (f^{-1}\bullet g)] = [g]$.
    Hence the automorphism $\phi^{-1}\circ\phi_{\delta}$ sends $\widetilde{y}$ corresponding to the class $[f]$ to itself and by Proposition \ref{property_morphism_covers}, there exist an element $\delta\in\pi_1(X,x)$ such that $\phi_{\delta} = \phi$.
\end{proof}

\begin{corollary}
\label{free_action_universal}
    The fundamental group $\pi_1(X,x)$ acts freely on the universal covering $\widetilde{X}$ of $X$.
\end{corollary}
\begin{proof}
    Indeed, since $\Aut(\widetilde{X}_x|X)^{op}$ and $\pi_1(X,x)$ and since two automorphisms of $\Aut(\widetilde{X}_x|X)$ agreeing at one point being identical by Theorem \ref{property_morphism_covers}, we have that $\Aut(\widetilde{X}_x|X)$ acts freely on $\widetilde{X}$, and so does $\pi_1(X,x)$.
\end{proof}

\noindent The previous Theorem \ref{aut_iso_fund} also echoes with the Theorem  \ref{th_link_aut_fund} and Corollary \ref{co_link_aut_fund} above.
It tells us that $p_*(\pi_1(\widetilde{X},y)) = 1$. 
And also that one does not distinguish between the monodromy action of $\pi_1(X,x)$ and the properly discontinuous and transitive action of $\Aut(\widetilde{X} |X)$ on $p^{-1}(x)$, they are essentially the same. 
To tie everything together, we have the following corollary.

\begin{corollary}
\label{fibre_univ_cov}
    Let $X$ be a path-connected locally simply-connected topological space, and $q:\widetilde{X}\rightarrow X$ the universal cover. 
    Then $\Fib_x(\widetilde{X}) = \pi_1(X,x)^{op}$.
\end{corollary}
\begin{proof}
    By Proposition \ref{representable_fibre_functor}, $q^{-1}(x)$ is in one-to-one correspondence with $\Aut(\widetilde{X}|X)$ and by Theorem \ref{aut_iso_fund}, $\Aut(\widetilde{X}|X) \cong \pi_1(X,x)^{op}$.
\end{proof}


\noindent In continuation of Example \ref{example_covering}, we compute the fundamental group of the circle by means of lifting of paths and the universal covering map:

\begin{theorem}
    $\pi_1(S^1) \cong \mathbb{Z}$
\end{theorem}
\begin{proof}
    $p:\mathbb{R}\rightarrow S^1: t\mapsto e^{2\pi i t}$ for $t\in[0,1]$ is clearly the universal cover since $\mathbb{R}$ is simply-connected.
    Notice that for any $x\in S^1$, we have that the fibre at that point is $\mathbb{Z}$ (a point in each sheet of the covering).
    Now consider the map $\Gamma:\mathbb{Z}\rightarrow \pi_1(S^1)$ given by $n\mapsto [\gamma_n]$
    where $\gamma_n$ is a loop in $S^1$ such that $\gamma_n(s) = e^{2\pi i ns}$.
    By the path-lifting property \ref{unique_lift}, there is a lift $\widetilde{\gamma}_n$ in $\mathbb{R}$ such that $\widetilde{\gamma}_n(s) = ns$.
    Considering the ``helix picture'' we described above, we can think of $\widetilde{\gamma}_n$ as the path that starts at $0$ and wraps $|n|$ times around the helix in the appropriate direction.\\
    This allows us to think of $\Gamma$ as a map sending $n$ to the homotopy class of a loop $p\widetilde{f}$ for some path $\widetilde{f}$ in $\mathbb{R}$ such that $\widetilde{f}(0)=0$ and $\widetilde{f}(1) = n$. 
    Since $\mathbb{R}$ is simply-connected $\widetilde{\gamma}_n \simeq \widetilde{f}$. So $p(\widetilde{f}) \simeq p(\widetilde{\gamma}_n) = \gamma_n$ ($[p(\widetilde{f})] = [\gamma_n]$.
   Consider the transformation $r_k:x \mapsto x+k$ in $\mathbb{R}$. Then we can define the path $\widetilde{\gamma}_n\cdot r_n(\widetilde{\gamma}_m)$ in $\mathbb{R}$ from $0$ to $n+m$. Recall that under the projection, $p(r_k(\widetilde{\gamma}_m)) = p(\widetilde{\gamma}_m)$. 
   Thus as above $\Gamma$ sends the path $\widetilde{\gamma}_n\cdot r_n(\widetilde{\gamma}_m)$ defined by the endpoint $n+m$ to its image under $p$, which is the homotopy class of $\gamma_n\cdot\gamma_m$. So $\Gamma(n+m) = \Gamma(n)\cdot\Gamma(m)$, making $\Gamma$ a morphism of groups.\\
   Consider a loop $f$ in $S^1$ with base point $x\in S^1$ that represents an element of $\pi_1(S^1)$. 
   Again, by the path-lifting property, there exist a unique lift $\widetilde{f}$ that starts at $0$ and ends at $n$ since $p(\widetilde{f}(1)) = f(1) = x$ and $p^{-1}(x) = \mathbb{Z}$. Therefore by our reconsiderations above, we have that $\Gamma(n) = [p(\widetilde{f})] = [f]$, so $\Gamma$ is surjective.
   Now suppose that $\Gamma(n) = \Gamma(m)$, implying that $\gamma_n \simeq\gamma_m$. Let $h$ be the homotopy between $\gamma_n$ and $\gamma_m$. By the homotopy lifting property \ref{unique_lift}, there is a unique homotopy lift $\widetilde{h}$ between the paths $\widetilde{\gamma}_n$ and $\widetilde{\gamma}_m$. Moreover it implies that $\widetilde{\gamma}_n \simeq \widetilde{\gamma}_m$ which means that $n=m$, proving that $\Gamma$ is injective.
\end{proof}

\noindent We finish our introduction on covering spaces and topological fundamental groups by invoking a result about compact manifolds.
\begin{proposition}
    A compact manifold is homotopy equivalent to a CW complex.
\end{proposition}
\begin{proof}
    See \cite{Hatcher}, Corollary $A.12$.
\end{proof}

\noindent And as consequence of the Van Kampen Theorem, we have the following application for CW complexes:

\begin{theorem}
    The fundamental group of a CW complex is finitely presented.
\end{theorem}
\begin{proof}
    See \cite{Hatcher}, Chapter $1.2$, Applications to Cell Complexes.
\end{proof}

\noindent From this, for every compact manifold $M$, its fundamental group $\pi_1(M)$ is finitely presented.\\

\subsection{A Reformulation of Classical Galois Theory}

In this section, we present a reformulation of the famous classification of field extensions by Galois-correspondence. 
It was an idea originated by Grothendieck who recognised that the Galois correspondence was really a statement about group action and sets, 
we shall refer to such realisations with the French expression ‘‘\`a la Grothendieck". 
This reformulation could be seen as a genesis for this thesis since this idea motivated Grothendieck to develop an analogous theory for the classification of topological covering spaces.
With these two reformulations at hand, Grothendieck, in his more than influential \textit{S\'eminaires de G\'eom\'etrie Alg\'ebrique} \cite{SGA} developed the theory of finite \'etale covering for algebraic spaces, which sits above the intersection between field extensions and covering spaces.\\
Grothendieck's idea was to study field extensions through the prism of category theory, by showing the existence of a category equivalence between separable field extensions and sets on which the \textit{absolute Galois group} acts in a certain way.
For that, we need a stronger result within the theory of Galois, related to infinite field extensions and profinite groups. This theorem was proven by \textit{W. Krull} in the 1930's, for that he introduced a topological structure on Galois groups.\\

\noindent In the following, the requisite background for classical Galois theory is assumed. 
However, for the sake of clarity, we recall the Main Theorem of Galois theory, which provides the definitive classification of separable field extensions.

\begin{theorem}[Main Theorem of Galois Theory]
\label{galois_finite_field}
    Let $L|k$ be a finite Galois extension with Galois group $G$. The maps $H\mapsto L^{H}$ and $M \mapsto \Gal(L|M)$ yield an inclusion-reversing bijection between the set of subgroups of $G$ and the set of intermediate fields between $L$ and $k$ $($for $H$ a subgroup of $G$ and $M$ an intermediate field between $L$ and $k)$. Moreover
    the extension $L|M$ is always Galois and 
    the extension $M|k$ is Galois if ans only if $H$ is a normal subgroup of $G$ ($\Gal(M|k)\cong G/H$).
\end{theorem}
\begin{proof}
    The material needed and the proof itself of the main theorem of Galois can be found in \cite{Lang_algebra} (Chapter $VI.$ $1$) .
\end{proof}

\noindent From classical Galois theory for field extensions, we obtain the following lemma (mostly from the primitive element theorem, Theorem $4.6$ in \cite{Lang_algebra})

\begin{lemma}
\label{embed_finite_subextension}
    Let $K|k$ be a Galois field extension (possibly infinite). Then each finite subextension of $K|k$ can be embedded in a finite Galois subextension of $K|k$ called the Galois closure.
\end{lemma}
\begin{proof}
    By the \textit{theorem of the primitive element}, each finite subextension is simple (of the form $k(\alpha)$) and we may embed $k(\alpha)$ into the splitting field of a the minimal polynomial of $\alpha$ which is Galois over $k$.
\end{proof}

\subsubsection{Profinite Groups}

As the name suggests, the notion of a topological group is one of the most natural bridges between group theory and topology.
A topological group $G$ is a topological space for which the underlying set is also a group. We also require that the group operations and the inversion maps are continuous.
A homomorphism of topological groups is defined to be a continuous group homomorphism $G\rightarrow H$.
Together they form a category.
Since topological groups are still groups, they can act on all kinds of sets, below we give a necessary and sufficient condition to determine whether an action of a topological group on a discrete topological space is continuous or not.

\begin{definition}
    We say that an action of a topological group $G$ on a topological space $X$ is a \textbf{continuous group action} if the map $m :G\times X\rightarrow X ; (g,x) \mapsto g\cdot x$ is a continuous map.
\end{definition}

\begin{lemma}
\label{continuous_open_stabiliser}
    Let $X$ be a discrete topological space, and $G$ a topological group acting on $X$. We say that $G$ has a continuous action on $X$ if and only if for each $x\in X$, the stabiliser $G_x$ is open.
\end{lemma}
\begin{proof}
    Suppose that $G_x$ is open. 
    Since $X$ has the discrete topology, every $x$ is open and\\
    $U_x = \{(g,y)\in G\times X \mid gy=x\}$ is their preimage by our action.
    And
    \begin{equation*}
        U_x = \bigsqcup_{y\in X} \{(g,y), g\in G\mid gy = x\}
            = \bigsqcup_{y\in G\cdot x}\{(g,y), g\in G\mid gy=x\}
    \end{equation*}
    since if $y$ is not in the $G$-orbit of $x$, $\{(g,y) \mid gy = x\} = \emptyset$.
    Now, for all $y\in G\cdot x$, take $g_y\in G$ such that $g_yx = y$, then 
    \begin{equation*}
        U_x  = \bigsqcup_{y\in G\cdot x} \{ (g,y), g\in G\mid gg_yx=x\}
            = \bigsqcup_{y\in G\cdot x} \{y\} \times G_x g_y
    \end{equation*}
    which is open by assumption. The previously defined action $m:G\times X\rightarrow X$ is then continuous.
    Now suppose that the above map is continuous, then $G_x$ is the preimage of $x$ by the composition $m\circ \iota_x$ where $\iota_x : g\mapsto  (g,x)$ continuously.
    Hence, for every $x\in X$, $G_x$ is open by assumption
\end{proof}

\noindent An important class of topological groups are profinite groups. 
Profinite objects arise in category theory as the inverse limit of an inverse system of finite objects.

\begin{definition}
    A \textbf{profinite group} is a group that is isomorphic to the inverse limit of an inverse system of discrete finite groups.
    An inverse system consists of a directed set $(I, \leq)$, an indexed family of finite groups $\{G_i \mid i \in I\}$ and a family of homomorphisms $\{\phi_{ij} :G_j\rightarrow G_i \mid i,j\in I , i\leq j\}$ such that $\phi_{ii}$ is the identity map and the collection satisfied $\phi_{ij}\circ\phi_{jk} = \phi_{ik}$ whenever $i\leq j\leq k$. The inverse limit is then the set,
    $$\varprojlim G_i = \Bigl\{(g_i)_{i\in I} \in \prod_{i\in I}  G_i \mid \phi_{ij}(g_j)=g_i \text{ for all } i\leq j   \Bigr\}$$
\end{definition}

\begin{remark}
\label{prof_are_top}
    We can endow a profinite group with a natural topology as follows:
    let $G$ be an inverse limit of the system of finite groups $(G_{\alpha} , \phi_{\alpha\beta})$ and endow $G_{\alpha}$ with the discrete topology, their product with the product topology and the subgroup $G\subset \prod G_{\alpha}$ with the subspace topology.\\
    With this in mind, we see that the projection homomorphisms $\pr_i:\varprojlim G_i \rightarrow G_i$ are continuous, and their kernels form a basis of open neighbourhoods of $1$ in $G$ (the last statement holds since for every non-trivial element $g$, there exists some coordinate $i$ such that $\pr_i(g)\neq 1$ by definition of the inverse limit).\\
    And by translations, we can then obtain a basis of open normal subgroups of neighbourhoods of any elements $g\in G$ (note that translations are homeomorphisms)
\end{remark}

\noindent Another famous characterisation of profinite group is the following: 

\begin{theorem}
\label{profinite_groups_other}
    Profinite group are equivalently groups for which the underlying set is compact, totally disconnected and Hausdorff.
\end{theorem}
\begin{proof}
    See \cite{Ribes} (Theorem $2.1.3$)
\end{proof}

\begin{corollary}
\label{closed_profinite_subgroup}
    Closed subgroups of profinite groups are again profinite.
\end{corollary}
\begin{proof}
    Since closed subsets of a compact Hausdorff space are precisely the compact subsets and subsets of totally disconnected Hausdorff spaces are again totally disconnected and Hausdorff. Closed subgroups are again profinite by the above characterisation.
\end{proof}

\begin{corollary}
\label{open_subgroup_profinite}
    A subgroup of a profinite groups is open if and only if it is closed and of finite index.
\end{corollary}
\begin{proof}
    Every open subgroup $U$ in a topological group is closed. Indeed, for topological groups, translations are homeomorphisms, so $U\mapsto gU$ is a homeomorphism and since the complement of $U$ is given by the disjoint union of open cosets $gU$, it is also closed. 
    By compactness, the number of these cosets must be finite hence the subgroup is of finite index.
    Conversely, a closed subgroup of finite index is open, being the complement of the finite disjoint union of its cosets, which are closed by homeomorphism.
\end{proof}

\begin{lemma}
\label{trivial_intersection_normal_subgroups}
    Let $G$ be a profinite group. 
    Then the intersection of all open normal subgroups of $G$ is trivial.
\end{lemma}
\begin{proof}
    Since $G$ is Hausdorff, for any element $g\in G$, there exist two open sets $U$ and $V$ in $G$ with $1\in U$ and $g\in V$ such that $U\cap V = \emptyset$.
    But since $U$ contains $1$, from the basis of normal open neighbourhoods $\{U_i\mid i\in I\}$ (from Remark \ref{prof_are_top}), there exists an open normal subgroup $U_i\subset U$ which contains $1$.
    Thus $U_i\cap V=\emptyset$, so $g\notin \bigcap_{i\in I}U_i$.
    Hence the only element that lies in every neighbourhood of $1$ is $1$ itself.
\end{proof}

\begin{lemma}
\label{lemma_profinite}
    Let $G$ be a profinite group, and $H\subset G$ a closed subgroup.
    \begin{itemize}
        \item[$a)$] The intersection of the open subgroups of $G$ containing $H$ is exactly $H$.
        \item[$b)$] Given an open subgroup $V'\subset H$, there is an open subgroup $V\subset G$, with $V\cap H = V'$
    \end{itemize}
\end{lemma}
\begin{proof}
    For the first statement, given $g\in G\backslash H$, we shall find an open subgroup of $G$ containing $H$ but not $g$.
    Since $H$ is closed, and $G$ is Hausdorff, there exists an open neighbourhood $U$ of $g$ such that $U\cap H = \emptyset$. 
    From Remark \ref{prof_are_top}, by translation we can take an open subgroup $N$ in $H$, normal in $G$ such that $gN$ is a basis element of the neighbourhood system at $g$, this gives us $gN\cap H = \emptyset$.
    Then the open subgroup $\pr_N^{-1}(\pr_N(H))$ is such a subgroup, with $\pr_N: G\rightarrow G/N$ the natural projection.\\
    For the second statement, we use that $V'$ is closed in $G$ and by the first statement, we write both $V'$ and $H$ as the intersection of the open subgroups of $G$ containing them.
    But since $[H:V']$ is finite (by the above corollaries), we find finitely many open subgroups $V_1,...,V_n$ in $G$ containing $V'$ but not $H$ such that $V'= V_1\:\cap\:...\cap\: V_n\cap H$. Thus $V = V_1\:\cap...\cap\:V_n$ proves the last statement.
\end{proof}

\noindent Sometimes, we want to endow a group with a profinite group structure (in the sense of Theorem \ref{profinite_groups_other}) even when it is not possible. 
For example, the group $\mathbb{Z}$ cannot be seen as a profinite group; we cannot endow it with a topology that is simultaneously compact, totally disconnected and Hausdorff. 
When this is the case, we can ‘‘force" a profinite group structure by considering the profinite completion of a group.

\begin{definition}
    Given an arbitrary group $G$, the \textbf{profinite completion} $\hat{G}$ of $G$ is defined as the inverse limit of the groups $G/N$, where $N$ runs through the normal subgroups in $G$ of finite index (these normal subgroups are partially ordered by inclusion, which makes them into an inverse system of natural homomorphisms between the quotients).
\end{definition}

\begin{remark}
\label{finite_generated_prof_grp}
While it is not always the case that the profinite completion of a finitely presented group is finitely presented, 
it is the case that the profinite completion of a finitely generated group is topologically finitely generated (see \cite{Ribes}, Section $2.12.2$).
We say that a profinite group is topologically finitely generated if it contains a
dense finitely generated subgroup.
\end{remark}

\noindent The next proposition realises Galois groups of Galois extensions as profinite groups

\begin{proposition}
\label{Galois_groups_are_profinite}
    Let $K|k$ be a Galois field extension (possibly infinite). The Galois groups of finite Galois subextensions of $K|k$ together with the homomorphisms $\phi_{ML}:\Gal(M|k)\rightarrow \Gal(L|k)$ form an inverse system whose inverse limit is isomorphic to $\Gal(K|k)$.
\end{proposition}
\begin{proof}
    By the main theorem of Galois theory, we can quickly check that the above data forms an inverse system, $\phi_{ML}$ are the canonical surjections and we have $\phi_{NL}=\phi_{ML}\circ\phi_{NM}$ for $N$ a finite Galois subextension containing $M$. We only have to prove the isomorphism part of the statement.\\
    Let us define a group homomorphism 
    \begin{equation*}
        \begin{split}
            \phi : \Gal(&K|k)\longrightarrow \prod_{L}  \Gal(L|k)\textbf{ } ; \textbf{ }
            \sigma \longmapsto \prod_{L} \sigma\vert_L
        \end{split}
    \end{equation*}
    For $L$, finite Galois extensions of $k$.
    We first prove that $\phi$ is injective by contraposition: if $\sigma$ does not fix an element $\alpha$ of $k_s\subset K$ (for $k_s$ the separable closure as in Definition \ref{separable_closure}), then its restriction to a finite Galois subextension $L$ containing $k(\alpha)$  (which always exists via Lemma \ref{embed_finite_subextension}) is nontrivial.
    And then show that $\phi$ is also surjective. 
    From the \textit{main theorem of Galois theory}, the image of $\phi$ is contained inside $\varprojlim \Gal(L|k)$ since the chaining inclusions make the correct sequences. 
    Now, if we take an element $(\sigma_L)\in \varprojlim \Gal(L|k)$ and define a $k$-automorphism $\sigma$ of $K$ with $\sigma(\alpha) = \sigma_L(\alpha)$, $L$ some finite Galois extension containing $k(\alpha)$. $\sigma$ is well-defined since by definition our sequence $(\sigma_L)$ forms a compatible system of automorphisms. Finally, by construction $\sigma$ maps to $(\sigma_L)\in \varprojlim \Gal(L|k)$, so $Im (\phi) = \varprojlim \Gal(L|k)$.
\end{proof}

\noindent From the previous proposition and by the definition of profinite groups, we can indeed conclude that Galois groups are profinite.
And as long as one does not specify the ground field, the opposite is also true. 
This was proven in two steps by \textit{C. Jordan} : finite symmetric groups arise as Galois groups for some field extension (see Example $VI.\:2.4$ in \cite{Lang_algebra}). And then in \cite{profinite_waterhouse}, \textit{C. Waterhouse}, showed that profinite groups are always realisable as the Galois group of some (infinite) field extension.
The case for which the base field is fixed to be the rational numbers $\mathbb{Q}$ however, remains an open problem, famously referred to as the \textit{Inverse Galois Problem}.\\

\noindent Now we can prove the main theorem of Galois theory for infinite field extensions (this was our main reason to introduce topological groups):

\begin{theorem}[Krull]
    Let $K|k$ be a Galois extension (possibly infinite) with Galois group $G$. If $L$ is a subextension of $K|k$, then $\Gal(K|L)$ is a closed group of $G$.
    The maps $H\mapsto K^H$ and $L\mapsto \Gal(K|L)$ yield an inclusion-reversing bijection between the set of closed subgroups of $G$ and the set of intermediate fields between $K$ and $k$ (for $H$ a closed subgroup of $G$ and $M$ an intermediate field between $K$ and $k$).
    And a subextension $L|k$ is Galois if and only if $\Gal(L|k)$ is normal in $G$ (in that case, $\Gal(L|k) \cong G/H$).
\end{theorem}
\begin{proof}
    Let $L|k$ be a finite separable extension in $K|k$.
    By Lemma \ref{embed_finite_subextension}, we can embed $L|k$ in a finite Galois extension $M|k$ in $K|k$.
    Then $\Gal(M|k)$ is a finite quotient of $\Gal(K|k) = G$ since the latter is profinite and $\Gal(M|L)$ is a subgroup of $\Gal(M|k)$ by the main theorem of Galois theory \ref{galois_finite_field}. Now consider $U_L$ the inverse image of $\Gal(M|L)$ by the natural projection $\pr_M:G \rightarrow \Gal(M|k)$. $U_L$ is then open because $\Gal(M|k)$ has the discrete topology.\\
    Notice that $U_L = \Gal(K|L)$. 
    Indeed, each element of $U_L$ fixed $L$ since it comes from $\sigma\in\Gal(M|L)$ by inverse of the natural projection $\pr_M$ so $U_L\subseteq \Gal(K|L)$. 
    Conversely considering the projection $\pr_M$, $\pr_M(\Gal(K|L))\subseteq \Gal(M|L)$, so $\Gal(K|L) \subseteq \pr_M^{-1}(\Gal(M|L)) = U_L$. 
    Making $\Gal(K|L)$ a closed subgroup of $G$ by Lemma \ref{open_subgroup_profinite}.\\
    Now let $L|k$ be an arbitrary separable subextension (possibly infinite) of $K|k$.
    Then $L|k$ can be written as a union of finite separable subextensions $L_{\alpha}|k$ for which the Galois groups $\Gal(L_{\alpha}|k)$ are open in $G$ so also closed as shown above.
    Since their intersection is precisely $\Gal(K|L)$, we have that $\Gal(K|L)$ is a closed subgroup of $G$.\\
    In the opposite direction, let $H$ be a closed subgroup of $G$.
    Then $H$ is a subgroup of $\Gal(K|L)$ for $L = K^H$. We are going the show the other inclusion: 
    let $\sigma\in \Gal(K|L)$ and pick a fundamental open neighbourhood $U_M$ of the identity in $\Gal(K|L)$ (so the kernel of the projection $\pr_M$), in fact it is also a subgroup corresponding to a Galois extension $M|L$ ($U_M = \Gal(M|L)$).
    And $H\subset \Gal(K|L)$ surjects into $\Gal(M|L)$ by $\pr_M\vert_{\Gal(K|L)}:\Gal(K|L)\rightarrow \Gal(M|L)$ (denote this restricted projection by $\prj$).
    Indeed, if $\prj_M(H)\subsetneq \Gal(M|L)$, then $\prj_M(H)$ would be a subgroup fixing a field $N = M^{\prj_M(H)}$ with $L\subset N\subset M$. But this is a contradiction since each element of $M\backslash L$ is moved by an element of $H$ (recall it only fixes $L$), so $\prj_M(H) = \Gal(M|L)$.
    In particular some element in $H$ must map to the same element in $\Gal(M|L)$ by $\prj_M$ as $\sigma$ for example.
    Hence $H$ contains an element of $\sigma U_M$ and since $U_M$ was chosen arbitrarily, $\sigma$ is in the closure of $H$, but $H$ is closed so $\sigma \in H$ and we have proved that $\Gal(K|L) \subseteq H$.\\
    Finally, we have the following equivalences just in the proof of the classical main theorem of Galois theory :
    \begin{equation*}
    \begin{split}
        \text{$L|k$ is Galois} &\Leftrightarrow  \forall\sigma\in G, \sigma(L) = L,\\
        &\Leftrightarrow \text{there exist a map } .\vert_L:G\rightarrow\Gal(L|k), \text{with $\ker(.\vert_L) = \Gal(K|L)$},\\
        &\Leftrightarrow  \text{by isomorphism theorem } \Gal(L|k) \cong \frac{G}{\Gal(K|L)}.
    \end{split}
    \end{equation*}
\end{proof}

\noindent We can now start to show Grothendieck's reformulation of Galois theory which focuses on the absolute Galois group and sets equipped with a continuous and transitive action.

\begin{definition}
\label{separable_closure}
    We define the \textbf{separable closure} of $k$ in $\overline{k}$ as the compositum of all separable extensions of $k$. And we call $\Gal(k)=\Gal(k_s|k)$ the \textbf{absolute Galois group}.
\end{definition}

\begin{remark}
    It is quite standard to show that the separable closure $k_s$ is a separable extension of $k$.
    Indeed, if $M$ and $L$ are two separable extensions, we have that $LM = M(\alpha_1,...,\alpha_n)$ for $\alpha_1,...\alpha_n$ the separable elements of $L$. Since the $\alpha_i$ are separable over $k$ they are separable over $M$, so $LM|M$ is a separable extension, as well as $M|k$ by hypothesis. So by the property of tower extensions, $LM|k$ is a separable extension.
\end{remark}

\noindent Fix a base field $k$ and consider the separable closure $k_s$. Let $L$ be a finite separable extension of $k$.
From classical Galois theory, we know that $L$ has finitely many $k$-homomorphisms into $\overline{k}$ whose images are actually contained in $k_s$. 
So we may consider a finite set $\Hom_k(L,k_s)$ (with $|\Hom_k(L,k_s)|=[L:k]$) and endow it with the action of $\Gal(k)$ given by $(g,\phi)\mapsto g\circ \phi$ for $g\in\Gal(k)$ and $\phi\in\Hom_k(L,k_s)$. \\


\noindent And from our previous discussion on topological spaces, we have the following result.

\begin{corollary}
    The action of $\Gal(k)$ on $\Hom_k(L,k_s)$ as above is continuous and transitive.
\end{corollary}
\begin{proof}
    The stabiliser $U$ of an element $\phi\in\Hom_k(L,k_s)$ consists of the elements of $\Gal(k)$ fixing $\phi(L)$. And by the main theorem of Galois for infinite extensions $U$ is open, hence $\Gal(k)$ acts continuously on $\Hom_k(L,k_s)$ by the above lemma since $\Hom_k(L,k_s)$ has the discrete topology.
    $L|k$ being separable, by the primitive element theorem, $L$ is generated by a primitive element $\alpha$ with minimal polynomial $f$. Each $\phi\in\Hom_k(L,k_s)$ is given by mapping $\alpha$ to a root of $f$ in $k_s$. Since $\Gal(k)$ permutes these roots transitively , the $\Gal(k)$-action on $\Hom_k(L,k_s)$ is transitive.
\end{proof}

\begin{corollary}
\label{Galois_normal_reformulation}
    Let $L|k$ be a finite separable extension, then $\Hom_k(L,k_s)$ as a $\Gal(k)$-set is isomorphic to the left coset space of some open subgroup in $\Gal(k)$.
    If $L|k$ is Galois, this coset space is in fact a quotient by an open normal subgroup.
\end{corollary}
\begin{proof}
    The first statement is immediate from the above corollary and Theorem \ref{Cameron_1_3} $a)$  with the open subgroup being the stabiliser $U$ of any element $\phi\in\Hom_k(L,k_s)$.\\
    By the main theorem of Galois for infinite extensions, the coset $U/\Gal(k)$ becomes a quotient if and only if $U$ is a normal subgroup.
\end{proof}

\noindent Before stating the reformulation of Galois theory by Grothendieck, we have to introduce the functor responsible for the equivalence of categories. 
This functor is given by $L \mapsto \Hom_k(L,k_s)$ for $L$ a finite separable extension of $k$ and $\varphi \mapsto \varphi^*$ for $\varphi:L\rightarrow M$ a $k$-homomorphism and $\varphi^* :\Hom_k(M,k_s)\rightarrow\Hom_k(L,k_s)$ a $\Gal(k)$-equivariant map defined by pre-composition with $\varphi$.
Thus the functor $\Hom_k(\_,k_s)$ is a contravariant functor from the category $\FSep_k$ of \textit{finite separable extensions of $k$} to the category $\Gal(k)$-$\FTCSet$ of \textit{finite sets equipped with a $\Gal(k)$ transitive and continuous action}.


\begin{theorem}[Reformulation of Galois Theory]
\label{Grot_reformulation_Galois}
    Let $k$ be a field with separable closure $k_s$. 
    The above contravariant functor $\Hom_k(\_,k_s)$
    gives an anti-equivalence of categories between $\FSep_k$ and $\Gal(k)$-$\FTCSet$.
    And Galois extensions give rise to $\Gal(k)$-sets isomorphic to some finite quotient of $\Gal(k)$.
\end{theorem}

\begin{proof}
    We show that any continuous transitive finite left $\Gal(k)$-set $T$ is isomorphic to some $\Hom_k(L,k_s)$ for $L$ a finite separable extension of $k$.
    Let $t\in T$, $H_t = \Gal(k)_t$ be the stabiliser subgroup and denote $L=k_s^{H_t}$ the subfield of $k_s$ fixed by $H_t$.
    Remark that with this notation by classical Galois theory, $H_t = \Gal(k_s|L)$.
    We now define a map of $\Gal(k)$-sets $\Hom_k(L,k_s) \rightarrow T$ by the rule $g\circ \iota \mapsto g\cdot t$ with $\iota\in \Hom_k(L,k_s)$ the natural inclusion and any $g\in \Gal(k)$. 
    This map is well defined since the stabiliser $U_{\iota}$ of $\iota$ is isomorphic to the stabiliser $H_t$ of $t$.
    Indeed 
    \begin{equation*}
    \begin{split}
        U_{\iota} &= \{g\in \Gal(k) \mid g\circ \iota\ = \iota\}\\
                               &= \{g\in \Gal(k)\mid g\vert_{\iota(L)}=id_{\iota(L)}\}\\
                               &= \Gal(k_s|\iota(L)) = \Gal(k_s|L) \cong H_t
    \end{split}         
    \end{equation*}
    And since we know that $\Gal(k)$ acts transitively both on $T$ and $\Hom_k(L,k_s)$, we can apply Theorem \ref{Cameron_1_3} to obtain get $T\cong \Gal(k)/H_t \cong \Gal(k)/U_{\iota} \cong \Hom_k(L,k_s)$.
    Hence $\Hom_k(\_,k_s)$ is an essentially surjective functor.\\
    To prove that the functor is fully faithful, we want to show that for two finite separable extensions $L$ and $M$ of $k$, the set of $k$-homomorphism $\Hom_k(L,M)$ corresponds bijectively to the set of $\Gal(k)$-maps $\Hom_k(M,k_s)\rightarrow \Hom_k(L,k_s)$.
    Let $\varphi_1$ and $\varphi_2$ be $k$-homomorphism $L\rightarrow M$ and suppose that $f\circ \varphi_1 = f\circ \varphi_2$ for all $f:M\rightarrow k_s$.
    Take any $a\in L$, we want to show that $\varphi_1(a) = \varphi_2(a)$ in $L$.
    For all $f:M\rightarrow k_s$, $f(\varphi_1(a)) = f(\varphi_2(a))$, so $\varphi_1(a)$ and $\varphi_2(a)$ have the same image under embeddings $f :M\rightarrow k_s$.
    Thus $\varphi_1(a) = \varphi_2(a)$ implies that $\varphi_1=\varphi_2$. This shows injectivity.
    Let $\psi:\Hom_k(M,k_s) \rightarrow \Hom_k(L,k_s)$ be a $\Gal(k)$ equivariant map.
    Pick $j\in \Hom_k(M,k_s)$ and set $i=\psi(j)$ two $k$-embeddings.
    For $g$ an element of $U_j$ the stabiliser of $j$ in $\Gal(k)$, we have $g\cdot i = g\cdot \psi(j) = \psi(g\cdot j) = \psi(j) = i$ so $U_j \subset U_i$ (the stabiliser of $i$ in $\Gal(k)$).
    Hence $\Gal(k_s| j(M)) \subset \Gal(k_s|i(L))$, thus $i(L)\subset j(M)$ as subfields of $k_s$, and so $j^{-1}\circ i:L\rightarrow M$ is a well-defined $k$-homomorphism. We denote $\varphi:=j^{-1}\circ i$.
    Since $\Gal(k)$ acts transitively on $\Hom_k(M,k_s)$, then for any $f\in\Hom_k(M,k_s)$, $f = g\cdot j$ for some $g\in\Gal(k)$.
    And $\psi(f) = \psi(g\cdot j) = g\cdot \psi(j) = g\cdot i$ while $f\circ \varphi = (g\cdot j) \circ (j^{-1}\circ i) = g\cdot i$, so $\psi(f) = f\circ \varphi$ for any $f\in\Hom_k(M,k_s)$. We have shown essential surjectivity.\\
    The last statement is now direct from Corollary \ref{Galois_normal_reformulation}.


\end{proof}

\subsubsection{\'Etale and Separable Algebras}

We end this section with an important generalisation of the previous reformulation theorem, 
which makes use of \textit{\'etale algebras}, a natural replacement for finite separable extensions (for commutative algebras). 
It allows one to consider sets with a continuous but not necessarily transitive $\Gal(k)$-action.

\begin{definition}
    A finite-dimensional commutative $k$-algebra $A$ is said to be \textbf{\'etale} (over $k$) if it is isomorphic to a finite direct product of separable extensions of $k$.
\end{definition}

\begin{remark}
    Morphisms of finite \'etale $k$-algebras are simply morphisms of $k$-algebras. From this we can consider the category of finite \'etale $k$-algebras $\FeAl_k$.
    And just as before, for $A$ an \'etale $k$-algebra, the $\Gal(k)$-action on the separable closure $k_s$ induces a left action on the set of $k$-algebra homomorphisms $\Hom_k(A,k_s)$. Allowing us to consider the functor $\Hom_k(\_,k_s)$ with domain $\FeAl$ 
\end{remark}

\begin{theorem}
\label{equivalence_catgory_etale_algebra}
    Let $k$ be a field with separable closure $k_s$. 
    Then above functor $\Hom_k( , k_s)$ mapping a finite \'etale $k$-algebra $A$ to the finite set $\Hom_k(A,k_s)$ gives an anti-equivalence between the categories $\FeAl_k$ and $\Gal(k)$-$\FCSet$.
    Here separable field extensions give rise to sets with transitive $\Gal(k)$-actions and
    Galois extensions to $\Gal(k)$-sets isomorphic to quotients of $\Gal(k)$.
\end{theorem}
\begin{proof}
    This theorem follows from the previous one by noticing that, given a finite decomposition $A = \prod_i L_i$ into a product of fields and an element $\phi \in \Hom_k(A,k_s)$, the maps $\phi$ induces the injection of exactly one $L_i$ in $k_s$.
    Indeed, if $\phi(L_i)\neq 0$, being a field extension, $L_i$ injects into $k_s$,
    and on the other hand, a product $L_i\times L_j$ cannot inject into $k_s$ since $k_s$ has no zero-divisors.
    Thus we have the decomposition $\Hom_k(A,k_s) \cong \bigsqcup_i \Hom_k(L_i,k_s)$.
    For a similar reason, given another finite \'etale $k$-algebra $A' = \prod_jL_i'$, a morphism $\varphi:A\rightarrow A'$ is given by a collection of morphisms $\varphi_{ij}:L_i\rightarrow L_j'$, for each $i,j$. 
    And these correspond bijectively to morphisms of the corresponding $\Gal(k)$-sets by the previous theorem.
\end{proof}

\noindent \'Etale algebras will also play an important role later in this thesis. 
For that reason, we give another characterisation of finite \'etale $k$-algebras.
For that, we first need the next lemma.

\begin{lemma}
\label{finite_dimensional_direct_sum_reduced}
    A finite-dimensional commutative algebra $A$ over a field $k$ is isomorphic to a direct sum of finite field extensions of $k$ if and only if $A$ is reduced.
\end{lemma}
\begin{proof}
    The ‘‘only if" part is obvious, one just has to notice that fields are already reduced. \\
    To prove the other direction, by decomposing $A$ into a finite direct product of indecomposable $k$-algebras, we may assume that $A$ is indecomposable.
    Under this restriction $A$ can have no idempotent elements other than $0$ and $1$, indeed, if $e\neq 0,1$ is an idempotent, then $A \cong Ae \times A(1-e)$ would be a non-trivial direct product decomposition since $e(1-e)=e-e^2=0$.
    The lemma follows if we show that every non-zero element $x\in A$ is invertible making it a field.
    Since $A$ is finite-dimensional, it is Artinian, i.e. the descending chain of ideals $(x)\supset (x^2)\supset\: ... \supset (x^n)\supset\:...$ must stabilise.
    And thus for some $m$, we have $x^n = x^{n+1}y$ for some $y\in A$.
    By iteration we get $x^n = x^{n+i}y^i$ for $i\in\mathbb{N}$ until $x^n = x^{2n}y^n$, thus $(x^ny^n)=(x^ny^n)^2$, an idempotent. 
    If $x^ny^n=0$, then $x^n = (x^n)(x^ny^n) = 0$ which is a contradiction since $x\neq 0$ by assumption. So $x^ny^n=1$ and $x$ is invertible. $A$ is reduced.
\end{proof}

\begin{theorem}
\label{etale_algebras}
    Let $A$ be a finite-dimensional commutative $k$-algebra. Then the following are equivalent :
    \begin{itemize}
        \item[$a)$] $A$ is \'etale
        \item[$b)$] $A\otimes_k \overline{k} \cong \overline{k}^{\oplus^n}$
        \item[$c)$] $A\otimes_k\overline{k}$ is reduced.
    \end{itemize}
\end{theorem}
\begin{proof}
    Let us first show that $a)\Rightarrow b)$:\\
    We may restrict to finite separable extensions $L$ of $k$.
    Then of course, we have $L = k[x]/(f)$ for some polynomial $f(x)=\prod_i(x-\alpha_i)$ in $\overline{k}$.
    We can conclude by the Chinese Remainder Theorem as follows :
    $$L\otimes_k\overline{k} \cong \overline{k}[x]/(f) = \overline{k}[x]/(x-\alpha_1)...(x-\alpha_n)\cong \prod_{i=1}^n\overline{k}[x]/(x-\alpha_i) \cong \prod_{i=1}^n\overline{k}.$$
    Now to prove that $b)\Rightarrow a)$, consider $A_r = A/\Nil(A)$, then $A_r$ is reduced and the previous lemma implies that $A_r$ is the direct sum of finite field extensions of $k$.
    Since $\overline{k}$ is reduced as a field, each $k$-homomorphisms $A\rightarrow \overline{k}$ factors through $A_r$ and hence to one of its finite field extension decomposition factors $L$.
    From Galois theory, we know that $|\Hom_k(L,\overline{k})|\leq |[L:k]|$ with equality if and only if the extension is separable, 
    whence $\Hom_k(A,\overline{k})$ has at most $\dim_k(A)$ elements with equality if and only if $A = A_r$ and $A$ is \'etale. \\
    To see that equality holds, observe that $\Hom_k(A,\overline{k}) \cong \Hom_{\overline{k}}(A\otimes_k\overline{k},\overline{k})$ are bijective as sets.\\
    Indeed given a $k$-algebra homomorphism $A\rightarrow\overline{k}$, tensoring by $\overline{k}$ and composing with the multiplication map gives a $\overline{k}$-homomorphism $A\otimes_k\overline{k}\rightarrow \overline{k}\otimes_k\overline{k}\rightarrow \overline{k}$.
    Conversely, we construct an inverse map, the natural inclusion $k\rightarrow \overline{k}$ induces a $k$-homomorphism $A\cong A\otimes_kk\rightarrow A\otimes_k\overline{k}$ which composed by a given $\overline{k}$-homomorphism $A\otimes_k\overline{k}\rightarrow\overline{k}$ gives a map from the right hand side to that on the left hand side.
    The assumption of $b)$ implies that the set $\Hom_{\overline{k}}(A\otimes_k\overline{k}, \overline{k})$ has $\dim_{\overline{k}}(A\otimes_k\overline{k})$ elements and actually $\dim_{\overline{k}}(A\otimes_k\overline{k}) = \dim_k(A)$, proving the equivalence.\\
    Finally, $b)\Rightarrow c)$ is immediate, again fields are automatically reduced.\\
    And $c) \Rightarrow b)$ follows by applying the previous lemma to the $\overline{k}$-algebra $A\otimes_k\overline{k}$.
\end{proof}

\noindent We have discussed so far a generalisation of separable field extension through \'etale algebras, and seen that it allows to refine the classification field extensions. 
Another plausible and useful generalisations of separable field extension is the notion of \textit{separable algebras}, it can be considered a broader generalisation than the \'etale case since it allows to consider non-commutative algebras.

\begin{definition}
    An algebra $A$ over a field $k$ is called \textbf{separable} if there exists a map $\Delta : A\otimes_k A\rightarrow A$ with $m(\Delta(a))=a$ and $a\Delta(b) = \Delta(ab) = \Delta(a)b$, for $a,b\in A$ and $m$ the usual multiplication of $A$.
\end{definition}

\begin{proposition}
\label{commut_separable_is_etale}
    Commutative separable algebras  and finite \'etale algebras coincide over fields.
\end{proposition}
\begin{proof}
    Corollary $4.5.8$ in \cite{Sep_al} says that such a commutative $k$-algebra $A$ is a separable algebra if and only if $A$ is isomorphic to a finite direct sum of fields $L_1 \oplus ... \oplus L_n$ where each $L_i$ is a finite separable field extension of $k$.
    But this is exactly the definition of finite \'etale $k$-algebras.
\end{proof}

\noindent For $m:A\otimes_k A \rightarrow A$ the multiplication map of a $k$-algebra $A$, denote its kernel $I$.
Notice that if $A$ is a commutative algebra, $A\otimes_kA$ is a commutative algebra in its own right, and $I$ is an ideal in this algebra because it is an $A$,$A$-bimodule : for all $x\in I$, $(a\otimes b)x = axb \in I$ for $a,b\in A$.

\begin{lemma}
\label{lemma_splitting_exact_sequence_separable}
    An algebra $A$ over a commutative ring is separable if and only if this short exact sequence of $A$,$A$-bimodules splits
    $$0\xrightarrow[]{} I \xrightarrow[]{i} A\otimes_kA \xrightarrow[]{m} A\rightarrow 0$$
    for $i$ the inclusion.
\end{lemma}
\begin{proof}
    We use the splitting lemma (in \cite{Lang_algebra}, Proposition $3.2$).
    We can define a map $\Delta: A\rightarrow A\otimes A$ so that $m\circ \Delta = \id_A$. 
    And  $\Delta$ is an $A$,$A$-bimodule map if and only if $a\Delta(b) = \Delta(ab)= \Delta(a)b$. 
    But this is precisely the definition of a separable algebras. \\
\end{proof}

\subsection{Another Classification of Covering Spaces}
\label{equivalence_fibre_functor}

\noindent Similarly as in Galois theory, there is a formulation ‘‘\textit{\`a la Grothendieck}" of the classification of covering spaces in terms of equivalence of categories via a fibre functor.

\begin{theorem}
\label{covering_reformulation_1}
    Let $X$ be a connect topological space and $x\in X$ a base point. The fibre functor $\Fib_x$ induces an equivalence of the categories between $\Cov(X)$ and $\pi_1(X,x)$-$\SET$.
    Connected covers correspond to $\pi_1(X,x)$-sets with transitive action 
    and Galois covers correspond to coset spaces of normal subgroups.
\end{theorem}
\begin{proof}
    We check that $\Fib_x$ is fully faithful and essentially surjective.\\
    Let $p:Y\rightarrow X$ and $q:Z\rightarrow X$ be two covers of $X$.
    We can assume that $Y$ and $Z$ are connected since (otherwise, we can just run the following argument on pairs of connected components).
    We want to show that the $\pi_1(X,x)$-equivariant map $\phi\in\Hom(\Fib_x(Y),\Fib_x(Z))$ comes from the unique morphism of covers $f: Y\rightarrow X$.
    Consider the maps $\varphi_y: \widetilde{X} \rightarrow Y$ and $\varphi_z:\widetilde{X}\rightarrow Z$ corresponding to the points $y\in \Fib_x(Y)$ and $z:=\phi(y)$ respectively as in Proposition \ref{representable_fibre_functor}.
    Since the universal covering $r:\widetilde{X}\rightarrow X$ is Galois (\ref{unive_cov_galois}) and by the above classification of covering spaces, the maps $\varphi_y$ and $\varphi_z$ are Galois coverings and $Y\cong \widetilde{X}/H_y$ and $Z\cong \widetilde{X}/H_z$ for each stabilisers $H_y=\Aut(\widetilde{X}_x|Y)$ of $y$ and $H_z=\Aut(\widetilde{X}_x|Z)$ of $z$.
    So we can write $\phi : \Fib_x (\widetilde{X}/H_y) \rightarrow \Fib_x(\widetilde{X}/H_z)$.
    Since this map is $\pi_1(X,x)$-equivariant, it induces a map $\Fib_x (\widetilde{X})/H_y \rightarrow \Fib_x(\widetilde{X})/H_z$,
    which implies that there is a map between the stabilisers $H_y\rightarrow H_z$ via $\phi$.
    Hence, $\varphi_z: \widetilde{X}_x\rightarrow Z$ induces a map 
    $$\widetilde{X}_x / H_y \rightarrow \widetilde{X}_x/H_z \xrightarrow[]{\sim} Z.$$
    Composing this with the homeomorphism $Y\cong \widetilde{X}_x/H_y$, gives us the desired map$f:Y\rightarrow Z$, thus $\Fib_x$ is indeed fully faithful.\\
    To show essentially surjective, consider $S$ any $\pi_1(X,x)$-set, we want to find a covering map whose fibre is exactly $S$.
    Assume that $\pi_1(X,x)$ acts transitively on $S$ and take $s\in S$. Let $H=\pi_1(X,x)_s$ be the stabiliser of $s\in S$, again by Theorem \ref{Cameron_1_3}, we have that $S\cong \pi_1(X,x)/H$ the coset. 
    Thus we wish to construct a covering map whose fibre is this exact coset.
    Consider $q:\widetilde{X}\rightarrow X$ the universal cover, by Proposition \ref{unive_cov_galois} it is Galois, thus $X \cong \widetilde{X}/\pi_1(X,x)$.
    Let $Y:= \widetilde{X}/H$ and $p:Y\rightarrow X$ a covering space, then we have on the quotients that $\widetilde{X}/H\rightarrow \widetilde{X}/\pi_1(X,x)$ is a covering map.
    Recall that from Corollary \ref{fibre_univ_cov}, we have $\Fib_x(\widetilde{X}/H) =\Fib_x(\widetilde{X})/H \cong \pi_1(X,x)/H \cong S$. 
    For $S$ not transitive, we decompose $S$ into disjoint union of orbits (into transitive components) and define a covering $p_i:Y_i\rightarrow X$ for each orbit. 
    Then, taking the fibre of the disjoint union yields a $\pi_1(X,x)$-equivariant bijection from the disjoint union of these fibres to the disjoint union of orbits of $S$, and thus a $\pi_1(X,x)$-equivariant bijection to $S$ itself.\\
    Finally the last statement of the theorem follows from the previous argument combined with the above classification of covering spaces.
\end{proof}

\noindent We can state a refinement of the previous theorem.
It applies the theorem to our previous discussion on Galois groups being profinite. 

\begin{remark}
    Earlier, we have shown that Galois groups of Galois extensions are profinite, however fundamental groups of topological spaces are generally not profinite; a prime example of this is the fundamental group of the circle $\pi_1(S^1) = \mathbb{Z}$ which cannot be seen as profinite.
\end{remark}

\noindent Notice that this refinement is much closer to Grothendieck's formulation of Galois theory for finite field extensions (Theorem \ref{Grot_reformulation_Galois}).
Since the subsequent chapter will mostly focus on \textit{}{finite} coverings (over algebraic spaces), this refinement is quite fitting.
The ‘‘new" tool we use here is the profinite completion of groups defined in the previous section.

\begin{definition}
    We say that a covering space $p:Y\rightarrow X$ is $\textbf{finite}$ if it has finite fibres. For a connected topological space, these fibres have the same cardinality, called the \textbf{degree} of the covering.
\end{definition}

\noindent In order to prove the next theorem, we first need the following lemma.

\begin{lemma}
    Let $G$ be a group and $H$ a subgroup of finite index. Then there exists a normal subgroup $N$ of $G$ contained in $H$
\end{lemma}
\begin{proof}
    Suppose the index of $H$ in $G$ is $[G:H]=n$, and consider the representation of $G$ on the left coset space of $H$, so $\rho:G\rightarrow \Sym(G/H) = S_n$ is a group homomorphism whose kernel $N=\ker\rho$ is $\bigcap_{g\in G} gHg^{-1} = \core_G(H) \subseteq H$. (or exercise $I$ $9.(a)$ in \cite{Lang_algebra})
\end{proof}

\begin{theorem}
\label{covering_reformulation_2}
    Let $X$ be a connect topological space and $x\in X$ a base point. 
    The functor $\Fib_x$ induces an equivalence between the category of finite covers of $X$ and the category of finite continuous $\widehat{\pi_1(X,x)}$-sets.
    Connected covers correspond to finite $\widehat{\pi_1(X,x)}$-sets with transitive action and Galois covers to coset spaces of open normal subgroups.
\end{theorem}
\begin{proof}
    First, notice that for a finite connected cover $p:Y\rightarrow X$, from the classification of coverings, the action of $\pi_1(X,x)$ on $p^{-1}(x)$ factors through finite quotients, 
    so $\widehat{\pi_1(X,x)}$ also acts on $p^{-1}(x)$.
    By orbit stabiliser and since the covering is finite, the stabiliser of each point $y\in Y$ is a subgroup of finite index. 
    Thus, by the above lemma, it must contain a normal subgroup of finite index.
    Since those normal subgroups of finite index are the kernels of the projection maps $\pr_N$, they must be open, and being contained in the stabiliser of $y$, the stabiliser can be seen as a union of cosets of an open normal subgroup.
    Hence, this stabiliser is open, and by Lemma \ref{continuous_open_stabiliser}, this means that the action is continuous.
    Conversely, a continuous action of $\widehat{\pi_1(X,x)}$ on a finite set factors through continuous actions of $\pi_1(X,x)$ on finite quotients and by Theorem \ref{covering_reformulation_1}, gives rise to a finite cover.
    The last two statements are proven similarly as in the proof of Theorem \ref{covering_reformulation_1}, by factoring through quotients.\\
\end{proof}

\subsection{Classification of Locally Constant Sheaves}

The theories of \textit{locally constant sheaves} and covering spaces are closely related. 
In this section, we show that they are, in fact, equivalent. 
The reader may find all the necessary preliminaries to sheaves in Appendix \ref{review_sheaves}. Let us now introduce the main objects of this section.


\begin{definition}
    Let $X$ be a topological space and $S$ be a discrete set. We define a sheaf of sets $\mathcal{F}_S$ on $X$ by taking $\mathcal{F}_S(U)$ to be the set of continuous functions $U\rightarrow S$. But since $S$ is discrete, the sections of $\mathcal{F}_S$ are just constant functions, thus $\mathcal{F}_S(U) = S$ and we call this sheaf the \textbf{constant sheaf} with values in $S$.
    A sheaf $\mathcal{F}$ on $X$ is then said to be \textbf{locally constant} if the restriction $\mathcal{F}\vert_U$ is a constant sheaf for all open $U$ in some open covering of $X$.
\end{definition}

\noindent We will study locally constant sheaves through the lens of another sheaf construction. 

\begin{construction}
    Given a morphism of topological spaces $p:Y\rightarrow X$, we define the presheaf $\mathcal{F}_Y$ of sets on $X$, called the \textbf{presheaf of sections} of $p$ as follows:
\begin{itemize}
    \item[$a)$] for any open subset $U\subset X$, let $\mathcal{F}_Y(U)$ be the set of sections of $p$ over $U$, and
    \item[$b)$] for any two open subsets $V\subset U\subset X$, the restriction map $\rho_{UV}^{\mathcal{F}_Y}$ is defined by restricting the sections to $V$,
\end{itemize}
\end{construction}

\begin{remark}
\label{fibres_are_stalks}
    A first interesting fact is that for sheaves of sections of covering maps, fibres and stalks at a point $x\in X$ are the same. 
    Indeed let $\phi: \mathcal{F}_{Y,x}\rightarrow p^{-1}(x)$ be the map that sends the representative $[s]_x$ to its image $s(x)$. This map is well-defined since equivalent sections agree over an open neighbourhood of $x$.
    When $\phi([s]_x) = \phi([t]_x)$ in $p^{-1}(x)$, since $p$ is a local homeomorphism, there exist an open neighbourhood $V\subset X$ of $x$ such that $s(V)=t(V)$, and since they are sections, $s=t$ near $x$, so $[s]_x=[t]_x$.
    For any $y\in p^{-1}(x)$, we have that $p\vert_V:V\rightarrow U$ is a homeomorphism, with $s:=(p\vert_V)^{-1}$ a section such that $\phi([s]_x) = s(x) = y$. 
\end{remark}

\begin{proposition}
\label{sheaf_of_sections_if_of_cover}
    The presheaf $\mathcal{F}_Y$ of sections of $p$ is indeed a sheaf. Moreover, if $p:Y \rightarrow X$ is a covering, then $\mathcal{F}_Y$ is a locally constant sheaf. $\mathcal{F}_Y$ is a constant sheaf if and only if $p:Y\rightarrow X$ is a trivial cover.
\end{proposition}
\begin{proof}
    Since sections are continuous maps over an open $U\subset X$, the glueing property of sheaves is immediately verified.
    Suppose now that $p:Y\rightarrow X$ is a covering. For any $x\in X$ take a connected open neighbourhood $V$ of $x$ such that the cover $p$ over $V$ is trivial. So over $V$ the cover is isomorphic to $V\times S$, for $S$ the fibre at $x$.
    Let $s\in\mathcal{F}_Y(V)$ be a local section of $p$ over $V$. 
    Its image is a connected open subset of $Y$, and is isomorphically sent to $V$ by $p$.
    So $im(s)$ must be one of the connected components of $p^{-1}(V) = V\times S$. And since the images of two different sections are isomorphic to two different connected components of $p^{-1}(V)$, sections of $p$ over $V$ correspond bijectively to points of the fibre $S$. Thus the restriction $\mathcal{F}_Y\vert_V$ is the constant sheaf with value $S$.
    And we get that $\mathcal{F}_Y$ is a constant sheaf iff the connected open set $V$ is the whole of $X$.
\end{proof}

\noindent We can now define a functor from the category of covering spaces of $X$ to the category of locally constant sheaves on $X$, which sends an object $Y$ to $\mathcal{F}_Y$ and a morphism $\phi :Y\rightarrow Z$ to a morphism $\varphi : \mathcal{F}_Y \rightarrow \mathcal{F}_Z$. 
Which is a morphism of sheaves, $\varphi$ sends a section $s\in \mathcal{F}_Y(U)$ to a section $\phi\circ s\in \mathcal{F}_Z(U)$ and this is compatible with the restriction maps.\\

\noindent In order to define a functor in the reverse direction, we introduce the notion of \textit{\'etale spaces}.

\begin{definition}
\label{etale_space}
    Let $\mathcal{F}$ be a presheaf on a topological space $X$. 
    The \textbf{\'etale space} of $\mathcal{F}$ is the disjoint union of its stalks over $X$, denoted $X_{\mathcal{F}} = \bigsqcup_{x\in X}\mathcal{F}_x$ where points are represented by a pair $(U,s)$ and are denoted $s_x$. 
    And for any open set $U \subset X$, a section $s \in \mathcal{F}(U)$ defines a map $i_s : U \rightarrow X_{\mathcal{F}}$ sending each $x\in U$ to the germ $s_x\in\mathcal{F}_x$.
    We can now endow the \'etale space with an \textbf{\'etale topology}, which is the coarsest topology for which all sets of the form $i_s(U)$ are open for all $U$ and $s$ (note this is not the same topology as the \'etale topology we encounter in Chapter \ref{nex_topology}).
\end{definition}

\noindent Notice that when $\mathcal{F}$ is a locally constant sheaf, the \'etale space is actually a covering space and we consider the following \textbf{\'etale covering map} $p_{\mathcal{F}}:X_{\mathcal{F}}\rightarrow X$ sending a germ to the point over which it is defined ($s_x\mapsto x$).
Indeed, for $U\subset X$ open, $p_{\mathcal{F}}$ is continuous since $p_{\mathcal{F}}^{-1}(U) = \bigcup_{s\in\mathcal{F}(U)} i_s(U)$ is open (from the \'etale topology).
Let $\{U_i\mid i\in I\}$ be an open affine covering of connected subsets of $X$ such that $\mathcal{F}\vert_U$ is the constant sheaf with values in $S$. Then for all $x\in U$, the fibre $\mathcal{F}_x=S$, and from Remark \ref{fibres_are_stalks}, $p_{\mathcal{F}}^{-1}(x)\cong S$, so $p_{\mathcal{F}}^{-1}(U)$ is homeomorphic to $U\times S$. Observe that $p_{\mathcal{F}}$ is always a local homeomorphism but only a covering when $\mathcal{F}$ is locally constant.

\begin{remark}
\label{Sheaves_X_equiv_etale_X}
    Sheaf theory is quite vast. The knowledgeable reader might recognise the connection above with the theory of \textit{bundles} and \textit{overcategories}. A standard reference is \cite{sheaves_geometry_logic}, where in chapter II, the authors define the category of \'etale spaces over $X$ as the full subcategory of the category $\Bund(X) = \Top/X$ consisting of \'etale bundles. And here \'etale bundles are essentially the same thing as \'etale projections.
    We refer to Corollary $6.3$ of that chapter which establishes an equivalence between the category of \'etale spaces over $X$ and the category of sheaves over $X$.
\end{remark}

\noindent Finally, a morphism of sheaves $\varphi : \mathcal{F} \rightarrow \mathcal{G}$ induces maps between stalks $\phi_x:\mathcal{F}_x \rightarrow \mathcal{G}_x$ for all $x\in X$ and hence a map of \'etale spaces $\phi: X_{\mathcal{F}} \rightarrow X_{\mathcal{G}}$.

\begin{lemma}
\label{morphism_of_etal_space}
    Let $\mathcal{F}$ be a locally constant sheaf on $X$,the above map $\phi : X_{\mathcal{F}} \rightarrow X_{\mathcal{G}}$ is a morphism of coverings of $X$.
\end{lemma}
\begin{proof}
    It is straightforward to see that $\phi$ is compatible with the projections $p:Y\rightarrow X$ and $q:Z\rightarrow X$.
    And we show that $\phi$ is continuous for the \'etale topology:
    pick an open subset $U\subset X$, a section $t\in\mathcal{G}(U)$ and a point $x\in U$.
    We have $i_t(U) \cap \mathcal{G}_x = t_x$ a point in $X_{\mathcal{G}}$.
    Now each preimage $s_x \in \Phi^{-1}(i_t(U))$ lies in $\mathcal{F}_x$ (this should be clear from our above construction).
    And each of these $s_x\in\Phi^{-1}(i_t(U))$ comes from a section $s\in\mathcal{F}(V)$ with $x\in V\subset U$ that can be chosen so small such that $\phi(s) = t\vert_V$.
    So $s_x\in i_s(V)\subset X_{\mathcal{F}}$ with $\Phi(i_s(V)) = i_t(V)$ (because we chose $V$ such that $\phi(s) = t\vert_V$).
    So every point in the preimage has an open neighbourhood, the preimage is open, we have a continuous morphism of \'etale spaces.
\end{proof}

\noindent We now have the reverse functor and are ready to prove the main theorem of this section.

\begin{theorem}
\label{LCS_equiv_Cov}
    The categories of covering spaces of $X$ and the category of locally constant sheaves on $X$ are equivalent. 
    Moreover, trivial coverings correspond to constant sheaves.
\end{theorem}
\begin{proof}
    Let $F: \LCS(X) \rightarrow \Cov(X)$ and $G:\Cov(X)\rightarrow\LCS(X)$ be functors as above.
    In order to show the equivalence of categories, we will show that the functors are quasi-inverses (that $\alpha:\id_{\Cov(X)} \Rightarrow F\circ G$ and $\beta:\id_{\LCS(X)}\Rightarrow G\circ F$ are natural isomorphisms). \\
    By the previous discussion, consider the following \'etale covering
    $p_{\mathcal{F}_Y}:X_{\mathcal{F}_Y}\rightarrow X$, then the fibre at $x$ is $p_{\mathcal{F}_Y}^{-1}(x) \cong \mathcal{F}_{Y,x}$ (by Remark \ref{fibres_are_stalks}), and since $\mathcal{F}_Y$ is a sheaf of sections, $\mathcal{F}_{Y,x} \cong p^{-1}(x)$ for the covering $p:Y\rightarrow X$.
    Thus for any $Y\in\Cov(X)$, the morphism $\alpha_Y$ sends an element $y\in Y$ to the corresponding element in $\mathcal{F}_{Y,p(y)}$.;                                                                                              
    And for any $\mathcal{F}\in \LCS(X)$, the other morphism $\beta_{\mathcal{F}}$ sends an element $s\in\mathcal{F}(U)$, for $U$ an open subset of $X$ to the local section $i_s$ it defines, which is an element of $\mathcal{F}_{X_{\mathcal{F}}}(U)$.
    Thus, we have that the following diagrams commute :
    \[
    \begin{tikzcd}
        Y \arrow[r,"f"] \arrow[d, "\alpha_Y"']
        & Z \arrow[d,"\alpha_{Z}"] \\
        X_{\mathcal{F}_{Y}} \arrow[r,"X_{\mathcal{F}_f}"']
        &  X_{\mathcal{F}_{Z}}
    \end{tikzcd}
    \hspace{1cm}
    \begin{tikzcd}
        \mathcal{F} \arrow[r,"g"] \arrow[d, "\beta_{\mathcal{F}}"']
        & \mathcal{G} \arrow[d,"\beta_{\mathcal{G}}"] \\
        \mathcal{F}_{X_{\mathcal{F}}} \arrow[r,"\mathcal{F}_{X_{g}}"']
        &  \mathcal{F}_{X_{\mathcal{G}}}.
    \end{tikzcd}
    \]
    What is left to show is that $\alpha_Y$ and $\beta_{\mathcal{F}}$ are actually isomorphisms. For that it is enough to show that their restriction over a suitable open covering of $X$ are isomorphisms.
    So let $\{U_i\mid i\in I\}$ be an open covering of $X$ such that $\mathcal{F}\vert_{U_i}$ is constant for every $i\in I$.
    Replacing $X$ by $U_i$, we may henceforth assume that $\mathcal{F} = \mathcal{F}_S$ is the constant sheaf with values in $S$ a discrete set.
    So $X_{\mathcal{F}_S} \cong X\times S$ (as in what follows Definition \ref{etale_space}) and since we know that the sheaf of section of the trivial cover $X\times S\rightarrow X$ is the constant sheaf defined by $S$ (by Proposition \ref{sheaf_of_sections_if_of_cover}), what have shown that $\beta_{\mathcal{F}}$ is an isomorphism.
    Over such an open covering of X, a covering map restrict to a trivial covering map, and similarly as above, we obtain that $\alpha_Y$ is an isomorphism (locally, which is enough).
\end{proof}

\noindent Combining the previous theorem with Theorem \ref{covering_reformulation_1} we get the following,

\begin{corollary}
    The category of locally constant sheaves on $X$ is equivalent to the category of sets endowed with a left-action of $\pi_1(X,x)$.
\end{corollary}

\begin{theorem}
\label{equi_R_mod_loc_cst_sh}
    Let $X$ be a connected and locally simply-connected topological space, with $x\in X$.
    There is an equivalence between the category of locally constant sheaves of $R$-modules on $X$, $\LCS_R(X)$ and the category of $R[\pi_1(X,x)]$-modules, $\Mod_{R[\pi_1(X,x)]}$.
\end{theorem}
\begin{proof}

    Let $\mathcal{F}$ be a locally constant sheaf of $R$-modules on $X$. Then $\forall x\in X$, $\mathcal{F}_x$ is also a $R$-module since it is a direct limit of $R$-modules (see Theorem \ref{direct_limit_R_module}).\\
    Moreover, if $\mathcal{F}$ is a locally constant sheaf, we saw that its stalk at $x$ can be seen as the fibre of the \'etale projection $p_{\mathcal{F}} : X_{\mathcal{F}}\rightarrow X$ at $x$, $\mathcal{F}_x \cong p_{\mathcal{F}}^{-1}(x)$. From the previous sections, this means exactly that $\forall x\in X$, there is an action of $\pi_1(X,x)$ on $\mathcal{F}_x$.
    And this action is actually compatible with the $R$-module structure :
    The addition law on $\mathcal{F}$ is a morphisms $\varphi :\mathcal{F}\times \mathcal{F}\rightarrow \mathcal{F}$ of sheaves. It induces a morphism  $\varphi_x : (\mathcal{F}\times \mathcal{F})_x \rightarrow \mathcal{F}_x$ of stalks for all $x\in X$. 
    And since the product sheaf $\mathcal{F}\times \mathcal{F}$ is defined by $(\mathcal{F}\times\mathcal{F})(U) =\mathcal{F}(U)\times\mathcal{F}(U)$ for all open subset $U\subset X$, its stalk over $x$ is $\mathcal{F}_x\times \mathcal{F}_x$, so $\varphi_x$ is then the addition law on $\mathcal{F}_x$.
    As in Lemma \ref{morphism_of_etal_space}, consider the map $\Phi : X_{\mathcal{F}}\times X_{\mathcal{F}}\rightarrow X_{\mathcal{F}}$ between \'etale sets which is a morphism of covers when $\mathcal{F}$ is a locally constant sheaf.
    We know from the theory of covering spaces that $\Phi$ is a $\pi_1(X,x)$-equivarient map on the fibres.
    Thus, the restricted morphisms $\varphi_x$ on stalks are also $\pi_1(X,x)$-equivarient for all $x\in X$. (Addition and group action commute).\\
    Similarly, the multiplication law on $\mathcal{F}$ is a morphism $R\times \mathcal{F \rightarrow \mathcal{F}}$ and since $R$ is a ring, we can just consider for all open subset $U\subset X$ the maps $\mathcal{F}(U)\rightarrow\mathcal{F}(U): s\mapsto r\cdot s$ for any fixed $r\in R$. 
    By the same argument as above, the multiplication law on $\mathcal{F}$ induces a multiplication law on $\mathcal{F}_x$ for all $x\in X$ which is again $\pi_1(X,x)$-equivarient (as a restriction of $\Phi : X_{\mathcal{F}}\rightarrow X_{\mathcal{F}}$).
    Since the multiplication and addition on $\mathcal{F}_x$ commute with the action of $\pi_1(X,x)$, $\mathcal{F}_x$ can be considered an $R[\pi_1(X,x)]$-module.
    And since morphisms of sheaves are defined for sections of open subsets, similarly as above, we can construct a morphism of $R[\pi_1(X,x)]$-modules. 
    Hence we can consider the functor $F_x:\LCS_R(X) \rightarrow \Mod_{R[\pi_1(X,x)]}$ sending $\mathcal{F}$ to its stalk $\mathcal{F}_x$.\\
    In the reverse direction, let $M$ be an $R[\pi_1(X,x)]$-module and give it the usual discrete topology.
    We first construct a cover : 
    $$q_M : M\times_{\pi_1(X,x)}\widetilde{X} \rightarrow X , (m,\widetilde{y})\mapsto p(\widetilde{y})$$ 
    with $p:\widetilde{X}\rightarrow X$ the universal cover at $x$.
    We check that this is indeed a cover,
    note that $M\times_{\pi_1(X,x)}\widetilde{X}$ is the fibre product consisting of the elements $(m,\widetilde{y})$ 
    such that 
    $(m , \widetilde{y}) \sim (\gamma m,\gamma\widetilde{y}) = \gamma\cdot (m,\widetilde{y})$ for $\gamma\in\pi_1(X,x), m\in M$ and $\widetilde{y}\in\widetilde{X}$.
    From now on, let's write $\pi_1$ in place of $\pi_1(X,x)$.
    Let $U$ be an open neighbourhood of $x$ in $X$. The preimage of $U$ by $q_M$ is $M\times_{\pi_1}p^{-1}(U)$ and since $p$ is actually a cover of $X$, $p^{-1}(U) = \bigsqcup_{y\in p^{-1}(x)} V_{y}$ with $V_{y}$ homeomorphic to $U$ for all $y\in p^{-1}(x)$, and we have 
    $$q_M^{-1}(U) =  M\times_{\pi_1}\bigsqcup_{y\in p^{-1}(x)}V_{y}.$$
    But since $X$ is connected and locally simply-connected, $p^{-1}(x)$ has an action of $\pi_1$, so any $\gamma\in\pi_1$ sends an element of $V_y$ to its corresponding element in $V_{\gamma y}$ and since the elements of our fibre product are equivalent classes defined by the above diagonal relation, we have 
    $q^{-1}_M(U) = M\times_{\pi_1}V$
    (with $V$ being any of the copies $V_{y}$).
    Finally, we have that $q_M(M\times_{\pi_1}V) = p(V)$ which is homeomorphic to $U$, making it a covering space.
    Notice that $M\times_{\pi_1}\widetilde{X}$ is actually an $R$-module, with an $R$-action defined on the first component as follow: $r\cdot(m,\widetilde{y}) = (rm,\widetilde{y})$. The $R$-action is well-defined since it commutes with the group action (coming from the quotient relation) and this holds since $M$ is an $R[\pi_1(X,x)]$-module :
    $$r\cdot(\gamma m,\widetilde{y}) = (r\gamma m, \widetilde{y}) = (\gamma rm,\widetilde{y})= r\cdot(m,\gamma\widetilde{y})$$
    for $\gamma \in \pi_1(X,x)$.
    From the constructed cover $q_M:M\times_{\pi_1}\widetilde{X}\rightarrow X$, we can consider the locally constant sheaf of sections $\mathcal{F}_M$ of $q_M$.
    For every open subset $U\subset X$, $\mathcal{F}_M(U)$ is the set of sections $s:U\rightarrow M\times_{\pi_1}\widetilde{X}$ of $q_M$ that is continuous on $U$.
    Such a section $s$ as its image completely in $M\times_{\pi_1}p^{-1}(U)$ which is a sub $R$-module. We can thus define a multiplication and addition on $\mathcal{F}_M(U)$ as follow,
    for all $y\in U$ and $s,y\in\mathcal{F}_M(U)$, let $(s+t)(y) = s(y)+t(y)$ and $(rs)(y) = r\cdot s(y)$. This makes $\mathcal{F}_M$ into a sheaf of $R$-modules.
    From this we can consider the functor $G:\Mod_{R[\pi_1]}\rightarrow \LCS_R(X)$ sending an $R[\pi_1]$-module $M$ to the sheaf of sections $\mathcal{F}_M$.
    This is indeed a functor since taking a morphism of $R[\pi_1]$-modules gives rise to a morphism of the resulting covering spaces as above which will define a morphism between the respective sheaves of local sections. By construction, the composition condition also holds.\\
    In a similar fashion as in the proof of Theorem \ref{LCS_equiv_Cov}, we now show that these functor are quasi-inverses.
    We want to show that the natural morphisms $\alpha :FG \Rightarrow \id_{\pi_1\Mod}$ and $\beta: GF\Rightarrow \id_{\LCS_R}$ are actually functorial isomorphisms.\\
    Consider the map $\varphi:M\rightarrow \mathcal{F}_{M,x}$ for $x\in X$ that sends an element $m\in M$ to the equivalence class $[m, \widetilde{x_0}]$ for $\widetilde{x_0}$ the lift of a base point $x_0\in X$ (recall $\mathcal{F}_{M,x}\cong q_M^{-1}(x)$). 
    Notice that this map is an isomorphism, indeed,
    if $[m,\tilde{y}] = [m',\tilde{y}']$, then there exist $\gamma\in\pi_1$ such that $[\gamma m,\gamma\tilde{y}] = [m',\tilde{y}']$. 
    And since $\pi_1$ acts freely on $\widetilde{X}$ by Corollary \ref{free_action_universal}, $\gamma = 1$ and thus $m=m'$. 
    And let $[m,\gamma\tilde{y}]\in q_M^{-1}(x)$, then $\gamma^{-1}m\in M$ is its preimage.
    Thus, for $M$ a $R[\pi_1]$-module, the morphism $\alpha_M$ sends an element $m\in M$ to its corresponding element in $\mathcal{F}_{M,x}$.\\    
    Let $\mathcal{F}\in \LCS_R$ and consider the $R[\pi_1]$-module $\mathcal{F}_x$. We can construct a covering map $q_{\mathcal{F}_x} : \mathcal{F}_x\times_{\pi_1}\widetilde{X}\rightarrow X$. But since the universal covering is a Galois covering, we have that $\widetilde{X}/\pi_1\cong X$ and thus $$\mathcal{F}_x\times_{\pi_1}\widetilde{X} \cong \bigsqcup_{x\in X} \mathcal{F}_x = X_{\mathcal{F}}.$$
    So the sheaf of sections of the above covering space must be $\mathcal{F}$ again from our discussion in Remark \ref{Sheaves_X_equiv_etale_X}.
    And we have the map $\beta_{\mathcal{F}}:\mathcal{F}\rightarrow \mathcal{F}_{\mathcal{F}_x}$ sending a section $s\in \mathcal{F}(U)$ to the corresponding section in $\mathcal{F}_{\mathcal{F}_x}$ is an isomorphism.\\
    We have thus showed, for any $M\in\Mod_{\pi_1}$ and $\mathcal{F}\in \LCS_R$, that $\alpha_M$ and $\beta_{\mathcal{F}}$ are isomorphism and moreover these isomorphisms are functorial in $M$ and $\mathcal{F}$ respectively since the following diagrams commute,
    \[
    \begin{tikzcd}
        M \arrow[r,"f"] \arrow[d, "\alpha_M"']
        & N \arrow[d,"\alpha_{N}"] \\
        \mathcal{F}_{M,x} \arrow[r,"\mathcal{F}_{f_x}"']
        &  \mathcal{F}_{N,x}
    \end{tikzcd}
    \hspace{1cm}
    \begin{tikzcd}
        \mathcal{F} \arrow[r,"g"] \arrow[d, "\beta_{\mathcal{F}}"']
        & \mathcal{G} \arrow[d,"\beta_{\mathcal{G}}"] \\
        \mathcal{F}_{\mathcal{F}_x} \arrow[r,"\mathcal{F}_{g_{x}}"']
        &  \mathcal{F}_{\mathcal{G}_x}.
    \end{tikzcd}
    \]

\end{proof}

\noindent Finally, we can apply the identification of representation of groups with group-ring modules described below to the above theorem and obtain the next corollary. 

\begin{theorem}
    Let $k$ be a field.
    The category of $k[G]$-modules and the category of representations of $G$ are isomorphic.
\end{theorem}
\begin{proof}
    Given a $G$-representation over $k$, $(V,\rho)$, we can define a $k[G]$-module structure on $V$ by $u_g\cdot v := \rho(g)(v)$ for all $v\in V$ and $g\in G$ (so $V$ becomes a free vector space over $G$). 
    Precisely, $u_g\cdot(u_h\cdot v) = \rho(g)\cdot \rho(h)(v) = \rho(gh)(v) = u_{gh}\cdot v = (u_gu_h)\cdot v$ and $u_e\cdot v= \rho(e)(v)=\id_V(v) = v$.
    Given a morphism $f:V\rightarrow W$ of $G$-representation over $k$,  (it is $G$-equivariant), we obtain a morphism of $k[G]$-modules as follow. From above, any element of $V$ can be written as a sum element over the generating group.
    $$f \Bigl( (\sum_{g\in G} a_gu_g)\cdot v \Bigr) = f \Bigl( \sum_{g\in G} a_g(u_g\cdot v) \Bigr) \stackrel{\text{linearity}}{=} \sum_{g\in G} a_g f(u_g\cdot v) \stackrel{G\text{-equiv}}{=} \sum_{g\in G} a_g u_gf(v).$$
    This makes a functor $F :\Rep_k(G) \rightarrow\Mod_{k[G]}$.\\
    In the reverse direction, consider a $k[G]$-module $M$, we can define a morphism $\rho: G\rightarrow \End(M)$ given by $\rho(g)(m) = u_g\cdot m$. This makes $(M,\rho)$ into a representation of $G$ over $k$.
    Consider a morphism $f':M\rightarrow N$ of $k[G]$-modules, similarly as above, we can see it is also a morphism of $G$-representations over $k$.
    This makes a functor $G:\Mod_{k[G]} \rightarrow \Rep_k(G)$.
    Finally, since our constructions are mutual inverses, we have that $F\circ G (M) = M$ and $G\circ F(V)=V$ for all $M\in \Mod_{k[G]}$ and $V\in \Rep_k(G)$. Yielding an isomorphism of categories.\\
\end{proof}

\begin{corollary}
\label{corollary_Rpi_mod}
    Let $k$ be a field. 
    The category of locally constant sheave of $k$-modules $\LCS_k$ is equivalent to the category of representations of $\pi_1(X,x)$ over $k$.
\end{corollary}

\noindent This equivalence is quite strong as it is also an equivalence that stands for the category of complex local systems.

\begin{definition}
    A \textbf{complex local system} on $X$ is a locally constant sheaf of finite-dimensional complex vector spaces. 
    If $X$ is connected, all stalks must have the same dimension, which is called the dimension of the local system.
\end{definition}

\noindent Hence the category of local systems $\LS_X$ on $X$ is equivalent to the category of locally constant sheaf $\LCS_{\mathbb{C}}(X)$.
The following corollary, is known as the \textit{monodromy representation of local systems}.

\begin{corollary}
\label{local_systems_representations}
    Let $X$ be a connected an locally simply-connected topological space with $x\in X$. 
    The category of complex local systems on $X$ is equivalent to the category of finite-dimensional representations of $\pi_1(X,s)$. 
\end{corollary}
\begin{proof}
    Follows from Theorem \ref{equi_R_mod_loc_cst_sh} and Corollary \ref{corollary_Rpi_mod}.
\end{proof}

\noindent Local systems are a powerful tool from algebraic topology. 
They also arise naturally in differential geometry when one tries computes solutions of linear differential equations as the following example shows.

\begin{example}
    Let $D\subset \mathbb{C}$ be a connected open disc.
    Consider over $D$  homogeneous $n$-th order linear differential equation
    $$y^{(n)} +f_1y^{(n-1)} +\:...+ f_{n-1}y'+f_ny = 0$$
    where the $f_i$ are holomorphic functions on $D$.
    Consider the local holomorphic solutions of the differential equation over each open subset $U\subset D$. As $\mathbb{C}$-linear combinations of local solutions of the differential equation over $U$ are again local solutions, they form a $\mathbb{C}$-vector space $\mathcal{S}(U)$.
    And due to a classical theorem from linear differential equations (Theorem $11.2$ in \cite{Forster}), each point of $D$ has an open neighbourhood $U$ where the $\mathbb{C}$-vector space $\mathcal{S}(U)$ has a finite basis $x_1,...,x_n$.
    This actually makes the sets of local solutions of the differential equation a sheaf $\mathcal{S}\subset \mathcal{O}^n$, a sheaf of complex vector spaces.
    And since the restriction of the above basis $x_i$ to smaller open sets still form a basis for the solutions, we can conclude that the sheaf $\mathcal{S}$ is locally constant, hence a local system of dimension $n$.
\end{example}

\begin{remark}
    But local systems have a deeper connection to differential geometry. Notably, the \textit{Riemann-Hilbert correspondence} says that the category of holomorphic flat connections on a complex manifold $X$ is equivalent to the category of complex local systems on $X$.\\
    And with the above equivalence, this correspondence partial answers the $21$st Hilbert problem which asks about the existence of a linear differential equation having a prescribed monodromy group.
\end{remark}

\newpage
\section{Covers and Fundamental Groups for Schemes}

In the last chapter, we discussed how the theory of covering spaces provided a nice geometric interpretation of the fundamental group of topological spaces.
In this chapter, we explain in what manner(s) the Zariski topology is too coarse. 
Then we go through a study of \textit{\'etale morphisms} in order to define \'etale coverings,
which will allow us to define the \textit{\'etale fundamental group} similarly as in the previous chapter, denoted by $\pi_1^e$. 
Finally, we compare our new fundamental group with its topological counterpart by \textit{analytifying} non-singular projective schemes over $\mathbb{C}$. 
We also present a few structural properties of $\pi_1^e(X)$.\\

\subsection{The \'Etale Topology}
\label{nex_topology}

As one knows, we usually endow an algebraically variety over an algebraic closed field or a scheme with the Zariski topology. 
But this topology is not well-suited for the methods of algebraic topology. 
In the 1940's, Andr\'e Weil conjectured  that the number of solutions to a set of polynomials over a finite field might be explained by a cohomology theory for which the Lefshetz formula for fixed points would hold (roughly), giving motivation to the study of cohomology of algebraic spaces (see Appendix \ref{Sheaf_cohomology}).
But the cohomology groups computed using the Zariski topology (the Z-cohomology) are all zero for an important class of spaces (this follows from \textit{Grothendieck's vanishing theorem}).
This later led to the definition of the \textit{\'etale topology} of a scheme by Grothendieck, which in turn gave birth to the theory of \textit{\'etale cohomology}. 
It is thus quite awkward to motivate the need for the \'etale topology without mentioning strong ideas from cohomology 
(which fall outside the scope of this thesis).
Nonetheless, we will see that for our purposes, the Zariski topology is also ill-suited. 

\begin{theorem}[Grothendieck's vanishing theorem]
    If $X$ is an irreducible topological space, then $H^r(X,\mathcal{F}) = 0$ for all constant sheaves and all $r>0$. 
\end{theorem}
\begin{proof}
    See \cite{milneEtCoH}, Theorem $1.1$.
\end{proof}

\noindent And the $\check{C}ech$ $cohomology$ $groups$ are also zero.\\
The Zariski topology on an irreducible scheme $X$ is quite coarse; it has very few open sets which are additionally dense, and all its closed sets are finite. 
From this only, we already understand that some key notions of algebraic topology fail; 
it is impossible to ‘‘zoom in" on a point, and 
there is no notion of a small open ball centred at a point that usual topology is based on.\\
Even better, we can show that a very wide range of irreducible varieties and in fact every irreducible schemes are contractible (simply-connected) when endowed with the Zariski topology.


\begin{definition}
    Let $X$ be a topological space.
    A \textbf{cofinite subset} of a set $X$ is a subset $A$ whose complement is a finite set.
    The \textbf{cofinite topology} is a topology whose open subsets are precisely all cofinite subsets and the empty set.
\end{definition}

\noindent It is immediately clear that in dimension $1$, the Zariski topology is cofinite.

\begin{lemma}
    Let $X$ be a set of cardinality at least $2^{\aleph_0}$ endowed with the cofinite topology.
    Then $X$ is contractible.
\end{lemma}
\begin{proof}
    The assumption is introduced simply to require that $\lvert X \rvert = \lvert [0,1]\rvert$ (in fact this forces $X$ to be path-connected).
    Now from \cite{set_theory} $\S 10.35$ , we have the following result : $\lvert X\times [0,1]\rvert = \max\{\lvert X\rvert , \lvert [0,1]\rvert \} = \lvert X\rvert$ (since they are both infinite sets).
    So, we may consider an injection $f:X\times [0,1] \hookrightarrow X$.
    Such an injection is continuous by the definition of the cofinite and product topologies. 
    Then for a fixed $x_0\in X$, we define a homotopy from the constant map to the identity map :
    \begin{equation*}
    H(x,t)=
    \begin{cases}
        x, & t=0\\
        f(x,t), & t\in (0,1)\\
        x_0, & t=1
    \end{cases}
    \end{equation*}
    continuity of the homotopy follows from the fact that the inverse image of any point in $X$ is closed in the product topology (and since for the cofinite topology, closed set are finite sets, it is enough to check only for points). 
    This proves that $X$ is contractible.
\end{proof}

\noindent In set theory, the \textit{cardinality of the continuum} is the cardinality of the set of real numbers $\mathbb{R}$.
It is an infinite cardinal number and it is denoted as $2^{\aleph_0}$, where $\aleph_0$ represents the cardinality of the set of natural numbers $\mathbb{N}$.

\begin{remark}
    Notice that it is not enough to assume that $X$ is infinite.
    For example, the set of rational numbers $\mathbb{Q}$ with the cofinite topology is not path-connected, which then implies that it cannot be contractible.
\end{remark}

\noindent Thus, an algebraic curve over an uncountable base field (e.g. $\mathbb{C}$ or even $\mathbb{R}$) has trivial fundamental groups. 
We can go even further, and claim that every irreducible schemes are contractible.
To show this, we give a preliminary result about irreducible schemes and \textit{generic points}. 
They are points that are non-necessarily closed and thus in the Zariski topology, their closure is the whole space.


\begin{lemma}
\label{irreducible_generic_point}
    Let $X$ be an irreducible scheme. Then $X$ has a unique generic point.
\end{lemma}
\begin{proof}
    If $X=\Spec\:R$, then the nilradical of $R$ ($\sqrt{(0)}$) is the generic point. 
    Indeed $\Spec\:R$ is irreducible if and only if $\sqrt{(0)}$ is a prime ideal : 
    $X$ is irreducible if and only if the intersection of two non-empty distinguished opens $D(f)$ and $D(g)$ which is $D(fg)$ is non-empty. But this is equivalent to saying that $f$ and $g$ are not nilpotent then $fg$ is not nilpotent, hence that $\sqrt{(0)}$ is a prime ideal. 
    And since the nilradical is the intersection of all prime ideals, it is contained in every point of $\Spec\:R$, so $\overline{\sqrt{(0)}} = \Spec\:R$.
    It is unique because if another prime ideal $P$ were to be a generic point, the $V(P) = V(\sqrt{(0)})$ and thus $P = \sqrt{(0)}$.\\
    Now in general, if $X$ is irreducible and $U=\Spec\:R$ is an affine open subset, then any generic point of $X$ is a generic point of $U$. Conversely if $P$ is a generic point of $U$ there cannot be another affine open subset $V\subset X$ that does not contain $P$, since if it were the case, $V\cap U$ it would be empty, not containing the generic point of $U$.
\end{proof}

\noindent It is now easy to show that
\begin{lemma}
    Every irreducible scheme is contractible.
\end{lemma}
\begin{proof}
    Let $X$ be an irreducible scheme and denote by $\eta$ its unique generic point. Then we have the following homotopy 
    \begin{equation*}
    H(x,t)=
    \begin{cases}
        x, & t=0\\
        \eta, & t \neq 0
    \end{cases}
    \end{equation*}
    which is again continuous since $\eta$ is a generic point, every point of $X$ is either in $\eta$ or in its boundary. 
    And which implies that the identity map is homotopic to the constant map. Hence $X$ is contractible.
\end{proof}

\noindent While the Zariski topology is ideal in many cases, mainly for keeping track of algebraic data (such as prime/maximal ideals, dimensionality, localisation, nilpotents, homomorphisms, ...) 
we have seen that it fails at discerning geometric properties topologically. 
So, how to proceed from there ?\\
Since there is already a large theory of algebraic geometry built with the Zariski topology, it is rather clever that we don't discard all of it.
To tackle this issue, Grothendieck had the brilliant idea to consider not only open sets, but at the same time, open immersions. 
Thus, introducing a new kind of topology (that can now even be defined on a whole category, see Appendix \ref{Groth_topo}).
One can further define the notion of a sheaf for a category. \\

\noindent In our case, recall that in the theory of coverings from algebraic topology, we are interested in certain morphisms, those which are locally homeomorphic, so a good guess would be to consider morphisms of schemes that ‘‘look like" local isomorphisms.
We will see later that these are naturally \'etale morphisms.
Although the \'etale topology is defined through the notion of \'etale morphisms, we will not define what they are yet. 
For now, it is only important to keep in mind that they are the algebraic analogue of local isomorphisms. All of this is explained in the next section.

\begin{definition}
\label{def_etale_covering}
    Let $X$ be a scheme. We call an \'etale covering of $X$ a family of morphisms  $\{f_i:X_i\rightarrow X\}_{i\in I}$ of schemes such that $f_i$ is \'etale for all $i\in I$ and $X = \cup_i f_i(X_i)$. \\
    The \textbf{\'etale topology} is a Grothendieck topology on the category of schemes such that for $X$ a scheme, $\mathcal{G}(X)$ is an \'etale covering of $X$.
\end{definition}

\noindent The category of schemes allows many different Grothendieck topologies for different purposes. 
In this thesis, we will focus on the \'etale topology.
To establish that it is indeed the correct notion, we can compare the outcome \'etale results with 
results from the analytification of a scheme of finite type over $\mathbb{C}$ (see Section \ref{comparaison_theorem} and Appendix \ref{analytification}), which has the induced complex topology, 
a topology standard in the context of algebraic topology and the fundamental groups.
The initial comparison theorem is the following, due to Serre and can be found in \cite{Serre} (Theorem $1$, Chapitre $12$).

\begin{theorem}
    For every coherent sheaf $\mathcal{F}$ on $X$ a projective scheme and for all $r \in \mathbb{N}$, the following  homomorphism is a bijection
    $$\varepsilon : H^r(X,\mathcal{F}) \rightarrow H^r(X^{an},\mathcal{F}^{an}).$$
\end{theorem}

\subsection{On \'Etale Morphisms of Schemes}
\label{morphisms_of_schemes}

As motivated previously, to define a fundamental group over schemes, our first step will be to introduce \textit{\'etale morphisms}, the core of the \'etale topology.
But before doing so, we need to introduce important notions such as \textit{finiteness}, \textit{flatness}, \textit{sheaves of differentials} and \textit{smoothness} that are directly related to morphisms of schemes. 
Such notions are called relative notions.
\noindent As Vakil suggests in \cite{Vakil}, there are a gazillion finiteness conditions on morphisms of schemes. 
The ones we are immediately interested in are (locally of) \textit{finite type}, and (locally of) \textit {finite presentation}.

\subsubsection{Finite morphisms}

In a vacuum, the following finiteness conditions ensure that the dimensions of the fibres of our morphisms do not degenerate (that they stay finite).

\begin{definition}
    We say that a morphism of schemes $f:X\rightarrow Y$ is \textbf{locally of finite type} if $Y$ has an affine open covering $\{V_i = \Spec\:B_i\}$ so that $f^{-1}(V_i)$ has itself an affine open covering $\{U_{ij} = \Spec\:A_{ij}\}$ with $A_{ij}$ finitely generated $B_i$-algebras.\\
    Moreover, a morphism $f$ is said to be of \textbf{finite type}, if in addition, each $f^{-1}(V_i)$ can be covered by a finite number of $U_{ij}$.\\
    And a morphism $f$ is said to be \textbf{finite} if each $f^{-1}(V_i)=U_i$ an open affine subscheme of $X$.
\end{definition}

\begin{remark}
\label{loc_finite_pres_loc_finite_type_noetherian}
    Usually it is require that smooth or \'etale morphism are not only locally of finite type, but locally of finite presentation.
    A ring $A$ is said to be a \textbf{finitely presented}  $B$-algebra if 
    $$A \cong B[x_1,...,x_n]/(r_1(x_1,...,x_n), ..., r_j(x_1,...,x_n)).$$
    A morphism $f:X\rightarrow Y$ is \textbf{locally of finite presentation} if for each affine open set $\Spec\:B$ of $Y$, $f^{-1}(\Spec\:B)=\cup_i\Spec\:A_i$ with $B\rightarrow A_i$ finitely presented.\\
    Now notice that if we are in the Noetherian setting (so let $B$ be Noetherian) then being finitely presented is the same as being of finite type.
    And since we are restricting ourselves to the Noetherian case, in the following we will not consider morphisms locally of finite presentation.
\end{remark}

\begin{examples}\textbf{ }
\begin{itemize}
    \item[$a)$] By definition, finite morphisms are affine.
    \item[$b)$] And for a scheme $X$ and a field $k$, if the morphism of schemes $X\rightarrow \Spec\:k$ is finite, then $X$ is an affine scheme.
\end{itemize}
\end{examples}

\begin{proposition}[Base change]
\label{base_change_locally_finite}
    Let $f:X\rightarrow Y$ be a morphism of schemes, locally of finite type. And consider the following fibre product diagram for any morphism $g:Y'\rightarrow Y$with $X'= X\times_YY'$,
    \[
    \begin{tikzcd}
        X'  \arrow[r,"f'"] \arrow[d, "\rho_X"'] & Y' \arrow[d,"g"] \\
        X \arrow[r,"f"'] &  Y.
    \end{tikzcd}
    \] 
    Then the morphism $f': X'\rightarrow Y'$ is also locally of finite type.
\end{proposition}
\begin{proof}
    Let $\{V'_i = \Spec\:B'_i\}$ be an affine open covering of $Y'$, which maps into some open affine covering $V_i = \Spec\:B_i$ of $Y$.
    Since $f$ is locally of finite type, $f^{-1}(V_i)$ has an affine open covering $\{U_{ij}=\Spec\:A_{ij}\}$ such that $A_{ij}$ are finitely generated $B_i$-algebras.\\
    Now $\pi_X^{-1}(U_{ij})_{V'} = U_{ij}\times_V V' = \Spec\:A_{ij}\otimes_{\Spec\:B_i}\Spec\:B'_i$ are affine opens and thus $f^{-1}(V_i)\times_VV' \subset X'$ has an open affine covering by the above affine opens which are now finitely generated $B'_i$-algebras. Thus $f'$ is locally of finite type.
\end{proof}

\begin{proposition}[Composition]
\label{composition_locally_finite_type}
    If $f:X\rightarrow Y$ and $g:Y\rightarrow Z$ are morphisms of schemes locally of finite type, then $g\circ f :X\rightarrow Z$ is again locally of finite type.
\end{proposition}
\begin{proof}
    Follows by transitivity of finite generation;
    if $A$ is a finitely generated $B$- algebra and $B$ is a finitely generated $C$-algebra, then $A$ is a finitely generated $C$-algebra.
\end{proof}

\begin{lemma}
    In the Noetherian setting, schemes of finite type over a field are of finite dimension.
\end{lemma}
\begin{proof}
    Let $X$ be a scheme of finite type over $k$, then $X$ can be covered by affine sets $\Spec\: A_i$ where the $A_i$ are finitely generated $k$-algebras.
    This means that each chain of closed subscheme with the strict inclusion is finite, therefore the supremum of the length of chains of strictly increasing irreducible closed sets of $X$ is also finite.
\end{proof}

\begin{proposition}
    Let $f:X\rightarrow Y$ be a morphism of schemes of finite type. Then for all $q\in Y$, the fibre scheme $X_y$ has finite dimension. 
\end{proposition}
\begin{proof}
    By base change, the fibre of a morphism of schemes of finite type is again of finite type over the base change, which is now the residue field.
    Hence by the above lemma, its fibres at any point are finite-dimensional.
\end{proof}

\begin{proposition}
\label{finite_is_closed}
    Let $f:X\rightarrow Y$ be a finite morphism of Noetherian schemes. Then $f$ is closed.
\end{proposition}
\begin{proof}
    Since $f$ is finite, locally, consider 
    $Y = \Spec\:B$ and $X=\Spec\:A$ 
    with $A$ a finitely generated $B$-algebra. 
    Pick generators $a_1,...,a_n \in A$ and define a surjection $B[x_1,...,x_n] \rightarrow A$, by $x_i\mapsto a_i$.\\
    So we can consider the following embedding $\Spec\: A \hookrightarrow Proj\:B[x_0,...,x_n] = \mathbb{P}_B^n$. 
    Thus we obtain a factorisation $X\hookrightarrow \mathbb{P}_Y^n \rightarrow Y$, so $f$ is projective (locally).
    Projective morphisms are proper (\cite{Hartshorne}, Theorem $4.9$) and 
    properness is local on the base (\cite{Hartshorne}, Corollary $4.8$). So $f$ is a proper morphisms which is closed since it is universally closed. 
\end{proof}

\begin{proposition}
\label{closed_is_finite}
    Closed immersions are finite 
\end{proposition}
\begin{proof}
    Let $f:X\rightarrow Y$ be a closed immersion. 
    Then locally, it looks like $\Spec\: A/I \rightarrow \Spec A$. 
    Since $A/I$ is finitely generated as an $A$-module, it follows that $f$ is finite.
\end{proof}

\subsubsection{Flat Morphisms}

In \cite{Red_Book}, David Mumford begins his section on flat morphisms with the following quote:
\textit{The concept of flatness is a riddle that comes out of algebra, but which
technically is the answer to many prayers}.\\
And indeed, flatness is quite an intuitive notion in algebra while being quite mysterious in the geometric realm.
At the level of modules, flatness is a property about the extension of scalar functor. In broad terms, it claims that the tensor product behaves nicely.
Now, at the level of schemes, it is again a relative notion that encapsulates the concept of a ‘‘continuous family of schemes". 
By that we mean that the fibres of a \textit{flat} morphism of  schemes does not vary too wildly.

\begin{definition}
    An $R$-module $M$ is called \textbf{flat} if tensoring with $M$ over $R$ as a functor :\\
    $-\otimes_RM:Mod_R\rightarrow Mod_{R'}$ is an exact functor. 
    A morphism of rings $f: B \rightarrow A$ is said to be \textbf{flat} if $A$ is flat as a $B$-module.

\end{definition}

\noindent An additive functor $F:Mod_R\rightarrow Mod_{R'}$ is called an \textit{exact functor} if it sends (short) exact sequences of $R$-modules to (short) exact sequences of $R'$-modules. 
We make the distinction that a functor $F$ is \textit{left exact} if it sends a short exact sequence
$0 \rightarrow N \rightarrow N' \rightarrow N'' \rightarrow 0$ to a sequence that is exact at $F(N)$ and $F(N')$. Conversely, $F$ is \textit{right exact} if it sends the short exact sequence above to a sequence exact at $F(N')$ and $F(N'')$.

\begin{remark}
\label{restriction_on_showing_exactness}
Note that the functor $-\otimes_RM:Mod_R \rightarrow Mod_R$ is always right exact;
the extension of scalar functor is a left adjoint of the restriction of scalar functor. 
Left adjoint functors preserve all colimits, and in particular they preserve cokernels (much like kernels detect injectivity by trivial kernel, a map is surjective if and only if its cokernel is trivial). 
So the extension of scalar functor actually preserves surjectivity and thus is a right exact functor.
Hence, to show the exactness of the extension of scalar functor it is enough to show only left exactness (that it preserves injections).
\end{remark}

\noindent We now show a few important properties of flatness, such as compatibility with base change and localisation, which are crucial in algebraic geometry to make sense of pulling back structures and the interplay between local and global settings.

\begin{proposition}
    Projective modules are flat.
\end{proposition}
\begin{proof}
    Consider an injective morphism $g: N \rightarrow N'$ of $R$-modules and $P$ a projective $R$-module. We want to show that $P\otimes_R g:P\otimes_R N'\rightarrow P\otimes_R N'$ is injective.\\
    Let $\sum_j p_j \otimes n_j \in ker(P \otimes_R g)$ and since $P$ is projective, let $(e_i,f_i)_{i\in I}$ be a dual basis of $P$.\\
    For any $i\in I$, apply $f_i\otimes_R N'$ to $\sum_j p_j \otimes g(n_j)=0$, we obtain 
    $$0 = \sum_jf_i(p_j)g(n_j) = g(\sum_jf_i(p_j)n_j).$$
    And since $g$ is injective, $\sum_jf_i(p_j)n_j = 0$ for all $i\in I$.
    $$\sum_j p_j\otimes n_j = \sum_{i,j} e_if_i(p_j)\otimes n_j = \sum_i e_i \otimes (\sum_j f_i(p_j)n_j) = 0.$$
    Hence $P\otimes_R g$ is injective, and the conclusion holds by Remark \ref{restriction_on_showing_exactness}.
\end{proof}

\begin{corollary}
\label{free_modules_are_flat}
    Free modules $($which are projective$)$ are flat modules.
\end{corollary}

\begin{proposition}
\label{transitivity_flatness}
    If $M$ is a flat $A$-module and $N$ is a flat $M$-module, then $N$ is also a flat $A$-module.
\end{proposition}
\begin{proof}
    This follows from the following natural isomorphism for any $N$-modules $P$,
    $$N\otimes_A P \cong  N\otimes_M (M\otimes_A P).$$
\end{proof}

\begin{proposition}[Base change]
\label{bas_change_flat_mod}
    Let $A\rightarrow B$ be a ring map and $M$ be a flat $A$-module. Then the base change, $M\otimes_AB$, a $B$-modules is also flat.
\end{proposition}
\begin{proof}
    Suppose we have the following exact sequence of $B$-modules
    $...\rightarrow N\rightarrow N' \rightarrow N' \rightarrow ...$.
    Since $M\otimes_AB\otimes_B N = M\otimes_A N$ for all modules in the above sequence and $M$ is flat as an $A$-module, the sequence 
    $$...\rightarrow M\otimes_AB\otimes_B N\rightarrow M\otimes_AB\otimes_B N' \rightarrow M\otimes_AB\otimes_B N'' \rightarrow ...$$ 
    is exact.
\end{proof}

\noindent The next two results about flatness of modules we show are tied to localisation, and they hint at the fact that the notion of flatness is well-defined for sheaves of modules.

\begin{proposition}
\label{localization_as_tensor}
    Let $S$ be a multiplicative closed subset in $R$ and $M$ and $R$-module. Then $S^{-1}M \cong M\otimes_R S^{-1}R$.
\end{proposition}
\begin{proof}
    Consider the $R$-module homomorphism $\varphi :S^{-1}M\rightarrow M\otimes_R S^{-1}R$ which sends an element $\frac{m}{s} \mapsto m\otimes\frac{1}{s}$.
    We can verify that since for $m'\in M$ and $s'\in S$ with $\frac{m}{s}=\frac{m'}{s'}$, we have $u(ms'-m's)=0$, we get any $u\in S$,
    $m\otimes\frac{1}{s}-m'\otimes\frac{1}{s'} = u(ms'-m's)\otimes\frac{1}{uss'} = 0$.\\
    Now considering the universal property of the tensor product, we obtain the map \\
    $\psi : M\otimes_R S^{-1}R \rightarrow S^{-1}M$ such that $m\otimes\frac{r}{s}\mapsto\frac{rm}{s}$.
    And again, for $r'\in R$ and $s'\in S$ with $u(rs'-r's)=0$, we get $u(rms'-r'ms)=0$, so $\frac{rm}{s}=\frac{r'm}{s'}$.\\
    Since $\varphi$ and $\psi$ are inverse to each other, they are isomorphisms.
\end{proof}

\noindent From the last proposition, we can obtain that flatness is preserved under localisation.We go further and show the more restrictive following proposition.

\begin{proposition}[Localisation]
\label{flatness_at_all_localizations}
    An $R$-module $M$ is flat if and only if $M_\mathfrak{p}$ is flat as an $R_\mathfrak{p}$-module for all $\mathfrak{p}\in R$.
\end{proposition}
\begin{proof}
    Suppose that $M$ is a flat $R$-module and let 
    $$0\rightarrow N\rightarrow N'\rightarrow N''\rightarrow 0$$ 
    be any short exact sequence of $R_{\mathfrak{p}}$-modules.\\
    Since $R_{\mathfrak{p}}$-modules are canonically $R$-modules ($R\rightarrow R_{\mathfrak{p}}, r\mapsto \frac{r}{1}$), 
    $$0\rightarrow M\otimes_R N\rightarrow M\otimes_R N'\rightarrow M\otimes_R N''\rightarrow 0 $$
    is also exact.
    And because $M\otimes_R N \cong M_{\mathfrak{p}}\otimes_{R_{\mathfrak{p}}}N$ (recall $S^{-1}M\cong M\otimes_RS^{-1}R$, so $M_{\mathfrak{p}}\otimes_{R_{\mathfrak{p}}}N\cong M\otimes_R R_{\mathfrak{p}}\otimes_{R_{\mathfrak{p}}} N\cong M\otimes_R N$), we have that $M_{\mathfrak{p}}$ is a flat $R_{\mathfrak{p}}$-module for any prime ideal $\mathfrak{p}$ in $R$.\\
    Suppose that $M_{\mathfrak{p}}$ is a flat $R_{\mathfrak{p}}$-module for any $\mathfrak{p}\in R$. Let 
    $$0\rightarrow N\rightarrow N'\rightarrow N''\rightarrow 0$$ 
    be a short exact sequence of $R$-modules. And since $M\otimes_R -$ is a right exact functor, we have
    \begin{equation}
    \label{(1)}
        ...\rightarrow L \rightarrow M\otimes_RN\rightarrow M\otimes_RN'\rightarrow M\otimes_RN''\rightarrow 0.
    \end{equation}
    We want to know if $L=0$ (remember that it is enough to show that $L_{\mathfrak{p}}=0$ for all $\mathfrak{p}\in R$).
    Since localisation is exact, we have the following short exact sequence
    $$0\rightarrow N_{\mathfrak{p}}\rightarrow N'_{\mathfrak{p}}\rightarrow N''_{\mathfrak{p}}\rightarrow 0$$
    And since $M_{\mathfrak{p}}$ are flat $R_{\mathfrak{p}}$-modules, the following sequences is also short exact,
    $$0\rightarrow M_{\mathfrak{p}}\otimes_{R_{\mathfrak{p}}} N_{\mathfrak{p}}\rightarrow M_{\mathfrak{p}}\otimes_{R_{\mathfrak{p}}}N'_{\mathfrak{p}}\rightarrow M_{\mathfrak{p}}\otimes_{R_{\mathfrak{p}}}N''_{\mathfrak{p}}\rightarrow 0$$
    And notice that by localising (\ref{(1)}) and using the fact that $(M\otimes_RN)_{\mathfrak{p}} \cong M_{\mathfrak{p}}\otimes_{R_{\mathfrak{p}}}N_{\mathfrak{p}}$ we get
    $$0\rightarrow L_{\mathfrak{p}}\rightarrow M_{\mathfrak{p}}\otimes_{R_{\mathfrak{p}}} N_{\mathfrak{p}}\rightarrow M_{\mathfrak{p}}\otimes_{R_{\mathfrak{p}}}N'_{\mathfrak{p}}\rightarrow M_{\mathfrak{p}}\otimes_{R_{\mathfrak{p}}}N''_{\mathfrak{p}}\rightarrow 0$$
    So $L_{\mathfrak{p}}$ must be zero for all $\mathfrak{p}\in R$.\\
\end{proof}

\begin{definition}
    A quasi-coherent sheaf $\mathcal{F}$ on a scheme $X$ is said to be \textbf{flat at $p\in X$} if $\mathcal{F}_p$ is a flat $\mathcal{O}_{X,p}$-module. It is flat over $X$ if it is flat at all $p\in X$.\\
    Let $f:X\rightarrow Y$ be a morphism of schemes, and let $\mathcal{F}$ be an $\mathcal{O}_X$-module. We say that $\mathcal{F}$ is \textbf{flat} over $Y$ at $p\in X$ (or that $f$ is flat at $p$), if the stalk $\mathcal{F}_p$ is a flat $\mathcal{O}_{Y,f(p)}$-module. We say that $\mathcal{F}$ is \textbf{flat } over $Y$ is if it is flat at every point $p\in X$. Moreover, we say that $X$ is \textbf{flat} over $Y$ if $\mathcal{O}_X$ is.
\end{definition}

\begin{remark}
    Any scheme $X$ over a field $k$ is flat. 
    If $k$ is a field, then any $k$-module is actually a free module, hence a flat module. Schematically, we have that $X$ is flat at a point $x$ if $\mathcal{O}_{X,x}$ is flat, and since it is a $k$-algebra, it is then flat by the preceding argument.\\
    This can be understood geometrically as well, flatness means that the fibres of a morphism $X\rightarrow\Spec\:k$ ‘‘vary continuously'', and $\Spec\:k$ being a single point, it is obviously the case.
\end{remark}

\noindent Many flatness results about modules carry over easily to the case of morphisms of schemes. 
We give a few below.

\begin{proposition}[Base change]
\label{base_change_flat_morphism}
    Let $f: X\rightarrow Y$ be a morphism of schemes and let $\mathcal{F}$ be an $\mathcal{O}_X$-module which is flat over $Y$. 
    And consider the following fibre product diagram for any morphism $g:Y'\rightarrow Y$with $X'= X\times_YY'$,
    then the pull-back sheaf $\rho_X^*\mathcal{F}$ is flat over $Y'$.
\end{proposition}
\begin{proof}
    Recall that the pull-back sheaf is defined as $\pi_X^*\mathcal{F} = \pi_X^{-1}\mathcal{F}\:\otimes_{\pi_X^{-1}\mathcal{O}_X}\:\mathcal{O}_{X\times_YY'}$, and that from Appendix \ref{Annexe_scheme}, we have the following for the structure sheaf of the fibre product : 
    $\mathcal{O}_{X\times_YY'} \:\cong\: \pi_X^{-1}\mathcal{O}_X \: \otimes_{(f\circ \pi_X)^{-1}\mathcal{O}_Y}\: \pi_{Y'}^{-1}\mathcal{O}_{Y'}$. 
    Thus we have
    $$\pi_X^*\mathcal{F} = \pi_X^{-1}\mathcal{F}\:\: \otimes_{(f\circ \pi_X)^{-1}\mathcal{O}_Y}\: \pi_{Y'}^{-1}\mathcal{O}_{Y'}$$
    And now since $\mathcal{F}$ is flat over $Y$, by simple base change (from proposition \ref{bas_change_flat_mod}), $\pi_X^*\mathcal{F}$ becomes flat over $Y'$.
\end{proof}

\begin{proposition}[Composition]
\label{composition_flat}
    Let $f:X\rightarrow Y$ and $g:Y\rightarrow Z$ be morphisms such that $X$ is flat over $Y$ and  $Y$ is flat over $Z$.
    Then $X$ is flat over $Z$.
\end{proposition}
\begin{proof}
    Follows from Proposition \ref{transitivity_flatness}
\end{proof}

\begin{proposition}
\label{flat_is_open}
    Let $f:Y\rightarrow Y$ be flat morphism of schemes locally of finite type. Then $f$ is universally open and in particular it is open.
\end{proposition}
\begin{proof}
    See \cite{EGA}, Th\'eor\`eme $2.4.6$.
\end{proof}

\noindent Here \textit{universally open} simply means that $f$ is open under any base change, so obviously it implies that $f$ is open.

\begin{example}Open embeddings are flat:\\
    Let $f:X\rightarrow Y$ be an open embedding of schemes. This means that $X$ is isomorphic to an open subscheme $U$ of $Y$ (that $X$ and $U$ are homeomorphic and that the induced map $f^{\#}:\mathcal{O}_Y\vert_U \rightarrow f_*\mathcal{O}_X$ is an isomorphism).
    From this we have that maps between stalks are isomorphisms as well and so at any point of $X$, our map $f$ is flat.
\end{example}

\noindent The next class of morphism of schemes is a weaker kind of flat morphisms.

\begin{definition}
    A morphism $f:X\rightarrow Y$ is said to be \textbf{locally free} if the direct image sheaf $f_*\mathcal{O}_X$ is a  locally free as an $\mathcal{O}_Y$-module.
\end{definition}

\begin{remark}
\label{locally_free_are_flat_in_noetherian}
    It is easy and useful to see that locally free morphisms are flat.  
    Indeed if $f_*\mathcal{O}_X$ is locally free as an $\mathcal{O}_Y$-module, then locally on $Y$, $f_*\mathcal{O}_X$ is free, hence flat by Corollary \ref{free_modules_are_flat}.
    Thus $f$ is now a flat morphism.\\
    On the other hand, the converse is also true in the Noetherian setting.
    Recall that in this case morphisms of finite type are equivalently of finite presentations (Remark \ref{loc_finite_pres_loc_finite_type_noetherian}), then by Corollary $24.4.7$ in \cite{Vakil}, flatness and locally freeness are the same thing.
\end{remark}

\begin{proposition}[Composition]
    The composition of two locally free morphisms of schemes is again locally free.
\end{proposition}
\begin{proof}
    Let $f:X\rightarrow Y$ and $g:Y\rightarrow Z$ be locally free morphisms.
    Then $f_*\mathcal{O}_X$ is locally free as a $\mathcal{O}_Y$-module.
    And again $g_*(f_*\mathcal{O}_X)$ is a locally free $\mathcal{O}_Z$-module making $g\circ f$ a locally free morphism.
\end{proof}

\noindent Finally, we quote the following theorem by Grothendieck (appearing in  \textit{\'El\'ements de G\'eom\'etrie Alg\'ebrique IV}), a key result in the quest of understanding flatness,
\begin{corollary}
    Soient $Y$ un pr\'esch\'ema localement noeth\'erien, $f:X\rightarrow Y$ un morphisme propre, $y\in Y$ un point tel que $f$ soit ouvert en tous les points de $f^{-1}(y)$. 
    Alors la fonction $z\mapsto dim(f^{-1}(z))$ est constante dans un voisinage de $y$.
\end{corollary}
\begin{proof}
    See \cite{EGA}, Corollaire $14.2.5$.
\end{proof}


\noindent Translating everything, and combining it with the fact that flat morphism of finite type are open (Proposition \ref{flat_is_open}), we get 

\begin{theorem}
    Let $f:X\rightarrow Y$ be a proper, flat morphism of schemes locally of finite type. Then for $y\in Y$, the map $y\mapsto dim(X_y)$ is locally constant.
\end{theorem}

\subsubsection{The Sheaf of K\"ahler Differentials}

The motivation behind K\"ahler differentials is to generalise the intuitive concept of differential forms from differential geometry. 
One important aspect of the sheaf of K\"ahler differentials is that it is a quasi-coherent sheaf whose fibre at any point is precisely the dual of the tangent space at that point. 
It will also be central in the definition of smooth and \'etale morphisms. 
There are several ways of defining the module of K\"ahler differentials, 
one of them is by means of derivations. 

\begin{definition}
    Suppose $A$ is a $B$-algebra.
    We define a \textbf{$B$-linear derivation} on $A$ to be a $B$-module homomorphism $d:A \rightarrow M$ to an $A$-module $M$ satisfying the Leibniz rule : $d(fg) = f\:dg + g\:df$. 
    We denote the set of such $B$-linear derivations by $\Der_B(A,M)$
\end{definition}

\begin{definition}
    The \textbf{module of K\"ahler differentials} is defined to be the $A$-module $\Omega_{A/B}$ for which there is a universal derivation $d:A\rightarrow \Omega_{A/B}$. 
    The composition with $d$ provides, for every $A$-module $M$ an $A$-module isomorphism $\Hom_A(\Omega_{A/B}, M) \xrightarrow[]{\cong} \Der_B(A,M)$.\\
    Now one way to construct $\Omega_{A/B}$ and $d$ is by constructing a free $A$-module as linear combinations of the symbol $da$ for each $a\in A$ and imposing the following relations :
    \begin{itemize}
        \item  $da+da' = d(a+a')$ $\forall a'\in A$ (additivity),
        \item $d(aa') = ada' +a'da$ (Leibniz law) , and
        \item  $db = 0$ $\forall b\in B$ (triviality of pull-back).
    \end{itemize}
    Notice that $d:A\rightarrow\Omega_{A/B}$ is $B$-linear. And the universal derivation sends $a$ to $da$.
\end{definition}

\noindent From the Leibniz law of differentials, we can easily deduce the quotient law and the chain rule for differentials.

\begin{examples}\textbf{ }
\label{example_kahler}
\begin{itemize}
    
    \item[$a)$] If $A = k[x_1,...,x_n]$ and $B=k$, then $\Omega_{A/B} = Adx_1 \oplus ...\oplus Adx_n$ a free $A$-module of rank $n$.\\
    Since for all $f\in A$, $df = \sum_{i=1}^n \frac{\partial f}{\partial x_i} dx_i$, $\Omega_{A/B}$ is generated by $dx_1, ..., dx_n$ as an $A$-module. There being no further relations among the differentials, $\Omega_{A/B} = Adx_1 \oplus ...\oplus Adx_n$.

    \item[$b)$] If $A =S^{-1}B$ with $S$ a multiplicative closed subset of $B$, then $\Omega_{A/B} = 0$.\\
    Any element of $A$ can be written as $a=\frac{b}{s}$ for some $b,s\in B$ and by the quotient rule, $da = \frac{sdb-bds}{s^2}=0$ via the triviality of pull-back.
    
    \item[$c)$] If $L|k$ is a finite separable field extension, then $\Omega_{L/k} = 0$.\\
    Since $L|k$ is a separable extension, for all $\alpha\in L$, there exists a polynomial $p\in k[x]$ such that $p(\alpha)=0$ and $p'(\alpha)\neq 0$. So by the chain rule we have $0 = d(p(\alpha)) = p'(\alpha)d\alpha$, so $d\alpha=0$ for all $\alpha\in L$.
\end{itemize}
\end{examples}

\noindent In the last example the converse is also true and it can be shown using separable algebras. 
But before we need the following lemma : 
\begin{lemma}
    If a finitely generated ideal $I$ in a commutative ring $R$ has $I^2=  I$, then $I = pR$ for $p\in R$ an idempotent.
\end{lemma}
\begin{proof}
    The lemma is almost immediate following Nakayama's Lemma \ref{Nakayama_lemma}.
    We apply it to $M = I$ and we get that there exist $i\in I$ such that $im =m$ for all $m\in I$. 
    It follows that $i^2 = i$ and $I = iR$, and just let $p=i$.
\end{proof}

\begin{proposition}
    Let $A$ be a commutative $k$-algebra. If $\Omega_{A/k} = 0$ then $A$ is a separable $k$-algebra.
\end{proposition}
\begin{proof}
    We show this later on (Theorem \ref{second_def_Kalher}) but, $\Omega_{A/k} = 0$ implies that $I = I^2$ for $I$ the kernel of $\delta$ which is the multiplication map for our $k$-algebra $A$.
    When $k$ is a field and $A$ is finite-dimensional, notice that every ideal of $A\otimes A$ is finitely generated, so we can apply the above lemma and consider the multiplication by the idempotent $p$ for $p\in A\otimes A$ such that $I = p\cdot (A\otimes A)$.
    This gives an $A$,$A$-bimodule map $\cdot\: p: A\otimes A\rightarrow I$ and we have $(\cdot \:p)\circ i = \id_I$ so by the splitting lemma, the exact sequence in Lemma \ref{lemma_splitting_exact_sequence_separable} splits making $A$ separable.
\end{proof}

\noindent So by combining these results, and remembering that separable commutative algebras coincide with \'etale algebras (Proposition \ref{commut_separable_is_etale}) we have the following theorem.

\begin{theorem}
\label{separable_zero_kahler}
    If $L|k$ is finite algebraic field extension, then $\Omega_{L,k}=0$ if and only if $L|k$ is separable.
\end{theorem}

\noindent As it was the case with flat modules, the module of K\"ahler differentials is also compatible with base change and localisation:

\begin{proposition}[Base change]
\label{Base_change_kahler_module}
    Let $B\rightarrow B'$ be any ring map and let the base change of $f$ by $B\rightarrow B'$ be the ring map $B'\rightarrow A\otimes_B B'$ (we often write $A' = A\otimes_B B'$).
    Then there is an isomorphism $\Omega_{A/B}\otimes_B B'\cong \Omega_{A'/B'}$ of $A'$-modules
\end{proposition}
\begin{proof}
    See \cite{Hartshorne}, Proposition $II.\: 8.2A$.
\end{proof}

\begin{proposition}[Localisation]
\label{localization_kahler_module}
    Let $S$ be a multiplicative closed subset of $A$, then $\Omega_{S^{-1}A/B} = S^{-1}\Omega_{A/B}$.
\end{proposition}
\begin{proof}
    See \cite{Hartshorne}, Proposition $II.\: 8.2A$.
\end{proof}

\noindent We also give a useful exact sequence:

\begin{proposition}
\label{exact_seq_differential_module}
    Let $C\rightarrow B \rightarrow A$ be homomorphisms of rings.
    Then there is a natural exact sequence of $A$-modules
    $$\Omega_{B/C} \otimes_B A \xrightarrow{a\otimes db\:\mapsto\: adb} \Omega_{A/C} \xrightarrow{da\:\mapsto\:da} \Omega_{A/B} \rightarrow 0$$
\end{proposition}
\begin{proof}
    See \cite{Hartshorne}, Proposition $II.\: 8.3A$.
\end{proof}

\noindent And the following construction shows that there is a direct relation between the module of K\"ahler differentials and the tangent space. 

\begin{construction}
\label{construction_tangent_space_differential}
    Recall that the tangent space at a point $p$ of an algebraic variety can be thought of as assigning to a polynomial $f$ its linear term in the Taylor expansion at $p$. Since the derivation $df$ of a polynomial $f$ evaluated at a point is exactly that linear term, we define the tangent space of $X$ at $p$ to be $T_pX := V(\{df \mid f\in I(X)\})$.\\
    Denote $R =A(X)$ a $k$-algebra. From example \ref{example_kahler} and by the fact that for all $f\in I(X)$, $f=0$ implies $df=0$, we have, by imposing sufficient conditions on the generators of $I(X)$ (suppose they are $f_1,...,fr$) 
    $$\Omega_{R/k} = Rdx_1 \oplus ... \oplus Rdx_n \Big/ \Bigl< \sum_j \frac{\partial f_i}{\partial x_j} dx_j \mid i=1,...,r\Bigr>.$$
    Taking the fibre at a closed point $p\in X$ (corresponding to a maximal ideal $\mathfrak{m}\trianglelefteq A(X)$), we set $\kappa(p)= R/\mathfrak{m}$ the residue field and by evaluation at $p$ and extension of scalar, we have
    $$\Omega_{R/k}\otimes_R \kappa(p) = \kappa(p)dx_1 \oplus ... \oplus \kappa(p)dx_n \Big/ \Bigl< \sum_j \frac{\partial f_i}{\partial x_j}(p) dx_j \mid i=1,...,r\Bigr>$$
    Letting $x=(x_1,...,x_n)$ be the coordinates in $\mathbb{A}^n$, and $y = x-p$ be the shifted coordinates such that $p$ becomes the origin.
    By Taylor expansion of $f$ at $p=0$, the linear term of $f_i\in I(X)$ in these coordinates is $\sum_{j=1}^n \frac{\partial f_i}{\partial x_j}(p)\cdot y_j$. 
    And since the tangent space at $p$ is the zero locus of these linear terms, we can actually say that $\Omega_{R/k}\otimes_R \kappa(p)$ is the dual $(T_pX)^*$ of the tangent space of $X$ at $p$.
\end{construction}

\noindent Unfortunately, the above construction is not intrinsic, it crucially depends on the choice of embedding $B\rightarrow A$, and the choice of generators $x_i$ and $f_i$.
There is an alternative definition of $\Omega_{A/B}$, that fixes this issue.
Consider the following diagonal homomorphism (induced by diagonal maps of affine schemes)  $\delta : A \otimes_B A \rightarrow A$, given by multiplication.
We define the ideal $I = \ker\:\delta$ of $A \otimes_B A$. 

\begin{theorem}
\label{second_def_Kalher}
    $I/I^2$ inherits a structure of $A$-modules. 
    Define a map $d:A\rightarrow I/I^2$ by $da = 1\otimes a-a\otimes 1$.
    Then $I/I^2 \cong \Omega_{A/B}$ as $A$-modules
\end{theorem}
\begin{proof}
    First note that $A\otimes_B A$ is an $A$-algebra in two ways (by multiplication on the left and right factor), so $I$ is an $A$-submodule.
    The ideal $I$ is generated by elements of the form $1\otimes a- a\otimes 1$. Indeed, if $\delta(\sum x_i\otimes y_i)= \sum x_iy_i=0$, then 
    $$ \sum_i x_i\otimes y_i = \sum_i (x_i \otimes y_i - x_iy_i\otimes 1) = \sum_i x_i(1\otimes y_i -y_i \otimes 1).$$
    And $I^2$ is an $A$-submodule of $I$ generated by all products $x_1x_2\otimes y_1y_2$ with $x_1\otimes y_1,x_2\otimes y_2\in I$ for both multiplication on the left and on the right. So $I/I^2$ is an $A$-module. \\
    Let's check that the map $d: A\rightarrow I/I^2$, $a\mapsto 1\otimes a - a \otimes 1$ is a universal derivation.
    Additivity is evident, and if $b\in B$, then $1\otimes b-b\otimes 1 = 0$ since remember tensoring is done over $B$.
    It is left to show the Leibniz rule :
    \begin{equation*}
    \begin{split}
        d(aa') - ada'-a'da &=  (1\otimes aa' -aa' \otimes 1) - (a\otimes a'-aa'\otimes 1) -(a' \otimes a - a'a \otimes 1) \\
        & = 1\otimes aa' - a\otimes a' -a'\otimes a + a'a \otimes 1\\
        & = (1\otimes a - a\otimes1)(1\otimes a'- a'\otimes 1) \in I^2
    \end{split}
    \end{equation*}
    Since the maps $a\mapsto 1\otimes a$ and $a\mapsto a\otimes 1$ from $A$ to $A\otimes_B A$ are $B$-linear, then so is their difference $d$.\\
    It is left to check that the induced map $\phi : \Omega_{A/B}\rightarrow I/I^2$ sending $da$ to $1\otimes a-a\otimes 1$ is an isomorphism. 
    We can define the inverse $\psi :I/I^2\rightarrow\Omega_{A/B}$ sending $\sum x_i\otimes y_i$ to $\sum x_idy_i$. \\
    It is $B$-linear and an inverse to $\phi$:
    $$\psi\circ\phi: da \mapsto\psi(1\otimes a-a\otimes 1) = 1da - ad1 = da, \textit{ and}$$
    $$\phi\circ\psi : \sum_ix_i\otimes y_i \mapsto \phi(\sum_ix_idy_i) = \sum_i x_i(1\otimes y_i - y_i\otimes 1) = \sum_ix_i\otimes y_i$$
    $\phi$ is clearly injective, and it is surjective as well since every element of $I/I^2$ is a linear combination of $1\otimes a-a\otimes 1$ so an image of $da$ under $\phi$.
\end{proof}

\noindent With the above construction and analogous description, $\Omega_{A/B}$ and the tangent space become intrinsic notions, we use them to define the sheaf of K\"ahler differentials on schemes.
Let $f: X\rightarrow Y$ be a morphism of schemes (assumed to be separated) and consider the diagonal morphism $\Delta : X\rightarrow X\times_Y X$.
Remember it is a closed embedding, so we can consider the ideal sheaf $\mathcal{I}$ of of $X$ in $Y$ just as in Lemma \ref{ideal_sheaf}.

\begin{definition}
    We define the \textbf{sheaf of K\"ahler differentials} $\Omega_{X/Y}$ of $X$ over $Y$ as the pull-back $\Delta^*(\mathcal{I}/\mathcal{I}^2)$ on $X$. 
\end{definition}

\noindent Note that it is a quasi-coherent sheaf (see Appendix \ref{annexe_quasi_coherent}).
We can verify that in the affine case, it is consistent with our previous definition of module of relative differentials.
Let $U=\Spec\:A$ and $V=\Spec\:B$ be an affine open subsets of $X$ and $Y$ respectively such that $f(U) \subset V$, then $U\times_V U$ is an affine open subset of $X\times_Y X$ isomorphic to $\Spec\:(A\otimes_B A)$.
$\Delta(X)\cap (U\times_V U)$ is the closed subscheme defined by the kernel of the associated diagonal homomorphism $\delta :A\otimes _B A\rightarrow A$.
That way we see that $\mathcal{I}/\mathcal{I}^2$ is the sheaf of modules associated to the $A\otimes_BA$-module $I/I^2$. 
And the restricted sheaf of relative differentials $\Omega_{U/V}$ is the sheaf associated to $I/I^2$ as an $A$-module, so $\Omega_{X/Y}\vert_{U/V} \cong \widetilde{\Omega_{A/B}}$.
In the language of quasi-coherent sheaves, we have constructed a sheaf $\Omega_{X/Y}$ on $X$ whose fibre $\Omega_{X/Y}(P) = \Omega_{X/Y}\otimes_{\mathcal{O}_{X,P}} \kappa(P)$ at $P$ is now a finite-dimensional vector space over $\kappa(P)$ that is exactly the dual of the tangent space $T_PX$.

\begin{example}
\label{open_embeddings_zero_kahler}
    Let $U\hookrightarrow X$ be an open embedding of schemes, then $\Omega_{U/X} = 0$\\
    Indeed, we can check locally, take an affine open $V=\Spec\:A\subset X$ and suppose that $U\cap V = \Spec\: B$. Since $U$ is an open subset, then $B$ is a localisation of $A$ so $B=S^{-1}A$ for a multiplicative set $S\subset A$. And as we saw in Example \ref{example_kahler} $b)$, $\Omega_{A/B}=0$ thus on each affine pieces $\Omega_{\Spec\:S^{-1}A/\Spec\:A} = 0$ so globally $\Omega_{U/X} = 0$.
\end{example}

\noindent Similarly to flatness, algebraic properties of the module of K\"ahler differentials carry over to the sheaf of relative differentials.

\begin{proposition}[Base change]
\label{base_change_Kahler_sheaf}
    Let $f:X\rightarrow Y$ be a morphism of schemes and consider the following fibre product diagram for any morphism $g:Y'\rightarrow Y$ with $X' = X\times_Y Y'$,
    \[
    \begin{tikzcd}
        X'  \arrow[r,"f'"] \arrow[d, "\rho_X"']
        & Y' \arrow[d,"g"] \\
        X \arrow[r,"f"']
        &  Y
    \end{tikzcd}
    \]
    Then $\Omega_{X'/Y'} \cong \rho_X^*(\Omega_{X/Y})$.
\end{proposition}
\begin{proof}
    Follows from Proposition \ref{Base_change_kahler_module}.
\end{proof}

\noindent There is an exact sequence similar to the one in Proposition \ref{exact_seq_differential_module}, for morphisms of schemes.
\begin{proposition}
\label{exact_seq_sheaves_differential}
    Let $f:X\rightarrow Y$ and $g:Y\rightarrow Z$ be morphisms of schemes. Then there is an exact sequence of sheaves on $X$,
    $$f^*\Omega_{Y/Z} \rightarrow \Omega_{X/Z} \rightarrow \Omega_{X/Y} \rightarrow 0.$$
\end{proposition}
\begin{proof}
    Follows from Proposition \ref{exact_seq_differential_module}.
\end{proof}

\subsubsection{Regular Schemes and Smooth Morphisms}

Smoothness is a central and intuitive notion in geometry. In algebraic geometry, this notion is called \textit{regularity} and it is closely related to the dimension of the tangent space.
We begin by defining the tangent space of a scheme at a point and convince ourselves that it is reasonably defined.

\begin{definition}
    The \textbf{Zariski cotangent space} of a local ring $(A,\mathfrak{m})$ is the vector space $\mathfrak{m}/{\mathfrak{m^2}}$ over the residue field $A/\mathfrak{m}$ and its dual is the \textbf{Zariski tangent space}.
    Now if $X$ is a scheme, the \textbf{Zariski cotangent space} $T_pX^*$ of $X$ at $p\in X$ is the cotangent space of the local ring $\mathcal{O}_{X,p}$ (and similarly for the \textbf{Zariski tangent space} $T_pX$).
    Elements of the cotangent space are called differentials (yes, as defined previously) and elements of the tangent space are called tangent vectors.
\end{definition}

\noindent One way to see that this is the right definition is the following:

\begin{construction}
    Since we have defined tangent spaces on varieties, 
    let $X$ be a variety and $p=0$ be the origin in $X$ corresponding to the maximal ideal $\mathfrak{m} \trianglelefteq A(X)$.
    Consider the $k$-linear map 
    $$\varphi : \mathfrak{m} \rightarrow \Hom_k(TpX,k) \textbf{   } ;\textbf{ }f\mapsto df\vert_{T_pX}.$$
    By definition of the tangent space, this map is well-defined and surjective since linear maps on $T_pX$ must be combinations of the coordinate functions $dx_1,...,dx_n$
    We now show that $\ker\:\varphi = \mathfrak{m}^2$.\\
    Notice that $W := \{dg \mid g\in I(X)\}$ is a linear subspace of $k[x_1,...,x_n]$ of dimension $t$, and its zero locus $T_pX$ has dimension $n-t$. 
    So the space of linear forms vanishing on $T_pX$, (i.e. $\{ dg \mid dg\vert_{T_pX} = 0\}$) has again dimension $t$ and also contains $W$. Thus $W = \{ dg \mid dg\vert_{T_pX} = 0\}$.\\
    Now if $f\in \ker\:\varphi$, we know that $df \in W$, so there exists a polynomial $g\in I(X)$ such that $df=dg$, where $dg$ is a linear combination of the linear forms defining $W$. So $f-g$ contains no constant term (we are at the origin) and no linear term. This means exactly that $f-g\in\mathfrak{m}^2$. 
    Notice here that we have first understood $f\in \ker\:\varphi$ as the restriction of a formal polynomial, defining a variety.
    We managed to say a lot of things about $f$ as a formal polynomial and ended up with two polynomials $f$ and $g$ that are equal when restricted to $X$, hence we could write that $f = f-g$ with $f$ on the left hand side being the restriction of $f$ to $X$. 
    While this is confusing, we rather explain it this way than to introduce new notation.\\
    And since $\mathfrak{m}^2$ is generated by products $fg$ for $f,g\in\mathfrak{m}$, we have $\varphi(fg) =f(p)dg\vert_{T_pX} + g(p)df\vert_{T_pX} = 0$ since $f(p)=g(p) = 0$.\\
    By the isomorphism theorem, we can conclude that indeed $\mathfrak{m}/\mathfrak{m}^2 \cong T_pX^*$ (the dual of $T_pX$).
\end{construction}

\noindent We are now able to state the definition of regularity for local rings, 

\begin{definition}
    We say that a local ring $(R,\mathfrak{m})$ is \textbf{regular} if the dimension of the local ring is equal to the dimension of the Zariski cotangent space, i.e. $dim\:R = dim_\kappa\: \mathfrak{m}\big/\mathfrak{m}^2$, with $\kappa$ the residue field $R/\mathfrak{m}$.
    A scheme $X$ is called \textbf{regular} at a point $p$ if the local ring $\mathcal{O}_{X,p}$ is regular. Otherwise, $X$ is said to be singular at $p$.
\end{definition}

\noindent Recall that the sheaf of K\"ahler differentials encodes the data of tangent spaces of a scheme over a base scheme.
Then it would make sense to characterise regularity of a scheme through its sheaf of K\"ahler differentials.

\begin{proposition}
\label{regularity_kahler}
    Let $X$ be an irreducible scheme of finite type over $k$ an algebraically closed scheme. 
    Then $\Omega_{X/k}$ is a locally free sheaf of rank $n$ if and only if $X$ is a regular scheme over $k$.
\end{proposition}
\begin{proof}
    See \cite{Hartshorne} Theorem $8.15$.  
\end{proof}

\noindent The following result gives us a necessary condition for regularity.

\begin{theorem}[Auslander–Buchsbaum]
    Any regular local ring is a unique factorisation domain (UFD).
\end{theorem}
\begin{proof}
    See \cite{Auslander_Buchsbaum}.
\end{proof}
\noindent And since, by definition, any UFD is an integral domain, regular local rings are integral domains. 
Then any point where $2$ irreducible components of a variety meet must be a singular point.

\begin{remark}
\label{normal_regular}
    From commutative algebra (Proposition $4.20$ in \cite{Eisenbud}), we have that UFD's are normal rings. 
    So for a scheme $X$, regularity at $P$ implies normality, i.e. we say that a scheme $X$ is normal at $P$ if the local ring $\mathcal{O}_{X,P}$ is normal: an integrally closed integral domain.\\
    But for schemes of degree $1$, so curves, the other implication is also true.
    Indeed, normal local rings are actually also discrete valuation rings (DVR) and in dimension $1$, regularity for a local ring is equivalent to being a DVR (again, see $11.1$, $11.2$ \cite{Eisenbud}).
    This means that for curves, regularity and smoothness are the same.
\end{remark}

\noindent Our next concern is that regularity does not refer to a base scheme.
For this reason, we introduce the notion of \textit{smooth morphisms} of schemes and say that a scheme $X$ is smooth over $k$ if $X \rightarrow \Spec\:k$ is a smooth morphism. 
One misconception is that smooth morphisms in algebraic geometry are the counterpart of smooth morphisms from differential geometry.
This is not the case, the correct picture is the one below, 
\begin{equation}
\label{relation_morphism}
\begin{split}
    \text{Submersions} & \longleftrightarrow \text{Smooth morphisms}\\
    \text{Local Isomorphism} & \longleftrightarrow \text{\'Etale morphisms}\\
    \text{Immersions} & \longleftrightarrow \text{Unramified morphisms}
\end{split}
\end{equation}
\noindent As a quick reminder, in differential geometry, a submersion is a smooth map of manifolds $f:M\rightarrow N$ whose differential $Df_p: T_pM\rightarrow T_{f(p)}N$ is surjective at any point $p\in M$. 
An immersion is similar but requires that the differential is injective at any points.
And a smooth map is called a local isomorphism if it is an immersion and a submersion at the same time.

\begin{definition}
    A morphism of schemes $f:X\rightarrow Y$ is \textbf{smooth of relative dimension $n$} if it is flat, locally of finite type and $\Omega_{X/Y}$ is locally free of rank $n$.
\end{definition}

\noindent Here, of relative dimension $n$ only means that all non-empty fibres have the same dimension $n$.
Flatness ensures that the fibres of the morphism are of the same dimension,
and $\Omega_{X/Y}$ is locally free of rank $n$ specifies what the dimensions of those fibres are; this last condition is not surprising considering Proposition \ref{regularity_kahler}. 
Again, smooth morphisms of schemes are also compatible with base change and composition.

\begin{proposition}[Base change]
\label{base_change_smooth}
    Let $f:X\rightarrow Y$ be a smooth morphisms of schemes of relative dimension $n$ and consider the following fibre product diagram for any morphism $g:Y'\rightarrow Y$ with $X'=X\times_Y Y'$, 
    Then the morphism $f': X'\rightarrow Y'$ is also smooth of relative dimension $n$.
\end{proposition}
\begin{proof}
    By Proposition \ref{base_change_flat_morphism}, \ref{base_change_locally_finite} and \ref{base_change_Kahler_sheaf} $f'$ is smooth again.
\end{proof}

\begin{proposition}[Composition]
\label{composition_smooth}
    If $f:X\rightarrow Y$ is smooth of relative dimension $n$, and $g:Y\rightarrow Z$ is smooth of relative dimension $m$, then $g\circ f :X\rightarrow Z$ is smooth of relative dimension $n+m$.
\end{proposition}
\begin{proof}
    See \cite{Hartshorne}, Proposition $10.1.\: (c)$.
\end{proof}

\begin{remark}
A nice way to see that smooth morphisms are indeed linked to submersions is by reasoning as follows,
let $f:X\rightarrow Y$ be a smooth morphism of regular schemes of relative dimension $n=dim\:X-dim\:Y$ over $k$ an algebraically closed field.
By Proposition \ref{exact_seq_sheaves_differential}, we get the exact sequence of sheaves : 
$$f^*\Omega_{Y/k} \rightarrow \Omega_{X/k} \rightarrow \Omega_{X/Y} \rightarrow 0.$$ 
Tensoring with the residue field $\kappa(P)$ for $P\in X$, we obtain a sequence of vector spaces over $\kappa(P)$,
$$f^*\Omega_{Y/k} \otimes \kappa(P) \xrightarrow{\phi} \Omega_{X/k}\otimes \kappa(P) \rightarrow \Omega_{X/Y}\otimes \kappa(P) \rightarrow 0.$$ Recall that since $X$ and $Y$ are regular schemes over $k$ and that $\Omega_{X/Y}$ is locally free of rank $n$, the now vector spaces above, have dimensions $\dim\:Y$, $\dim\:X$ and $n$ respectively.
By right-exactness, we have that $\dim (\im\phi) = \dim X - n = \dim Y$, and by the rank-nullity theorem, $\dim \ker \phi = 0$ so the map on the left is injective.
Finally, dualising it, we obtain (as seen in Construction \ref{construction_tangent_space_differential}), that the differential map $Df_P: T_PX\rightarrow T_PY$ at any point $P\in X$, coming from $f^*\Omega_{Y/k}\otimes \kappa(P)\rightarrow \Omega_{X/k}\otimes \kappa(P)$ is actually surjective.
\end{remark}

\subsubsection{\'Etale Morphisms}

Finally, we arrive at the main definitions of this section.

\begin{definition}
    A morphism of schemes $f: X\rightarrow Y$ is said to be \textbf{\'etale} if it is flat, locally of finite type and $\Omega_{X/Y}=0$. 
\end{definition}

\noindent Notice the resemblance to the definition of  smoothness for morphisms of schemes. 
The only difference being that the sheaf of relative differential need not only to be locally free, but it has to be the zero sheaf.
This can be interpreted as a ‘‘hidden" condition for \'etaleness, requiring that the morphism must be at the same time smooth and unramified (see relation picture (\ref{relation_morphism}))

\begin{definition}
    Let $f:X\rightarrow Y$ be a morphism of schemes locally of finite type, and $P\in X$. $f$ is said to be \textbf{unramified} at $P$ if the residue field of $X$ at $P$ is a separable field extension of the residue field of $Y$ at $f(P)$. And $f^{\#}(\mathfrak{m}_{f(P)})\mathcal{O}_{X,P} = \mathfrak{m}_P$ for $\mathfrak{m}_P$ and $\mathfrak{m}_{f(P)}$ are local rings of the local rings $\mathcal{O}_{X,P}$ and $\mathcal{O}_{Y,f(P)}$.
\end{definition}


\noindent We mention the following characterisation of unramified morphism through fibres.
\begin{proposition}
\label{unramified_carac}
    Let $f: X\rightarrow Y$ be a morphisms of finite type, then the following are equivalent:
    \begin{itemize}
        \item[$a)$] $f$ is unramified
        \item[$b)$] for all $Q\in Y$, the fibre $X_Q \rightarrow \Spec\:\kappa(Q)$ over $X$ is unramified.
        \item[$c)$] for all $Q\in Y$, the fibre $X_Q$ is the spectra of a finite separable $\kappa(Q)$-algebras.
        \item[$d)$] for all $Q\in Y$, the fibre $X_Q$ is a finite disjoint union of $\Spec\: k_i$, where $k_i$ are finite separable field extensions of $\kappa(Q)$ (or the spectrum of a finite \'etale $\kappa(s)$-algebra).
    \end{itemize}
\end{proposition}
\begin{proof}
    See \cite{milneEtCoH}, Proposition $3.2$.
\end{proof}

\begin{proposition}
    A morphism of schemes $f:X\rightarrow Y$ locally of finite type is said to be \textbf{unramified} at $P\in X$ if and only  $\Omega_{X/Y} = 0$ at $P$.
\end{proposition}
\begin{proof}
    See \cite{milneEtCoH} Proosition $3.5$.
\end{proof}
\noindent Thus, from Example \ref{open_embeddings_zero_kahler}, we have that open embeddings are unramified (in fact, any immersion is unramified) and with the previous proposition, we can reformulate our definition for \'etale morphisms of schemes.

\begin{definition}
\label{etale=flat+unramified}
    A morphism of schemes $f:X\rightarrow Y$ is said to be \'etale if it is equivalently flat and unramified.
\end{definition}
\noindent From this definition, it is clear that \'etale morphisms must behave like local isomorphisms, being at the same times submersions and immersions.

\begin{examples}
    From previous examples, open embeddings are \'etale.
\end{examples}

\begin{proposition}[Base change]
\label{base_change_etale}
    Let $f: X\rightarrow Y$ be an \'etale morphism of schemes and consider $f' : X' \rightarrow Y'$ the base change of $f$.
    Then $f'$ is \'etale again.
\end{proposition}
\begin{proof}
    Since smooth morphisms are stable under base change (\ref{base_change_smooth}), we need to show that unramified morphisms are stable under base change.
    But they clearly are since by Proposition \ref{base_change_Kahler_sheaf} and by the fact that the pull-back of a zero sheaf is again a zero sheaf.
\end{proof}

\begin{proposition}[Composition]
\label{composition_etale}
    Let $f:X\rightarrow Y$ and $g:Y \rightarrow Z$ be two \'etale morphisms of schemes. Then $g\circ f:X\rightarrow Z$ is also \'etale.
\end{proposition}
\begin{proof}
    By Propositions \ref{composition_locally_finite_type}, \ref{composition_flat} 
    and since $\Omega_{X/Y}$ and $f^*\Omega_{Y/Z}$ are zero sheave, by Proposition \ref{exact_seq_sheaves_differential}, $\Omega_{X/Z}= 0$.
\end{proof}

\begin{proposition}
\label{clopen_immersion}
    Let $f:X\rightarrow Y$ be a finite \'etale morphism,. 
    Then the diagonal morphism $\Delta : X\rightarrow X\times_YX$ coming from $f$ is an open and closed immersion (in particular $\Delta$ is an isomorphism of $X$ onto its image).
\end{proposition}
\begin{proof}
    Since $f$ is finite, it is separated and hence the diagonal morphism $\Delta : X\rightarrow X\times_YX$ is a closed immersion, 
    it is left to show that it is an open map. 
    Consider $U$ an open subscheme of $X\times_YX$ such that $\Delta:Y\rightarrow U$ is a still a closed immersion.
    Let $\mathcal{I}$ be the sheaf of ideals on $X\times_YX$, then restricted to $\Delta(X)$, the quotient $\mathcal{I}/\mathcal{I}^2\cong \Omega_{X/Y}$. \\
    Since we are in the Noetherian setting, $X\times_YX$ is Noetherian, making the ideal sheaf $\mathcal{I}$ a coherent sheaf, so we can apply Nakayama's lemma \ref{Nakayama_lemma}. 
    Thus obtaining that $\mathcal{I}=0$ on an open subset $U$ of $X\times_YX$, restricting the closed immersion $\Delta(X) \hookrightarrow X\times_YX$ to an isomorphism $\Delta(X) \cong U$. So $\Delta(X)$ is both open and closed in $X\times_YX$.
\end{proof}

\noindent Notice in this case, $X$ is isomorphic to a union of connected component of $X\times_YX$.

\begin{example}
    A morphism of affine schemes $\Spec\:B \rightarrow \Spec\: A$ is finite \'etale if and only if $B$ is a finite projective $A$-module and for every point $P\in \Spec\:A$, $B\otimes_A\kappa(P)$ is a finite separable algebra over the residue field.
\end{example}

\begin{example}
\label{example_etale_over_field}
    Let $k$ be a field.
    A scheme $X=\Spec\:A$ over $k$ is finite \'etale if and only if $A$ is a finite \'etale $k$-algebra.
    Since $X$ is an affine scheme over $k$ it is automatically finitely flat, so it is \'etale if and only if $\Omega_{X/k} = 0$ which is equivalent to asking that $\Omega_{A/k}=0$. By Theorem \ref{separable_zero_kahler} this is similar to proving that $A$ is a finite \'etale $k$-algebra.
    Moreover, by the following equation:
    $$\Spec\: A = \Spec \Big( \prod_{i=1}^r L_i\Big) = \coprod_{i=1}^r \Spec\:L_i$$
    if $X$ is connected, $A$ is actually a single separable field extension of $k$.
\end{example}

\subsection{Finite \'Etale Covers and Algebraic Fundamental Groups of Schemes}

Remember that we are forgoing a Noetherian investigation; all our rings and schemes are to be Noetherian. 
In that case, we may replace flat by locally free and we can restrict to working with morphisms of finite type. 
Every scheme is defined over a base scheme $S$.
And since \'etale morphisms are insensitive to base change, we can consider the \textit{geometric fibres} instead of the classical fibres. One reason to consider such geometric fibres is the following construction.

\begin{construction}
    Let $f:X\rightarrow S$ be a finite \'etale morphism of schemes. 
    A \textbf{geometric point} $\overline{s}$ on $S$ is a morphism of schemes $\overline{s}:\Spec\:\Lambda \rightarrow S$, such that the image of $\overline{s}$ is a point $s\in S$ and where $\Lambda$ is an algebraically closed field extension of the residue field $\kappa(s)$.
    The \textbf{geometric fibre} $X_{\overline{s}}$ of $f$ over $\overline{s}$ is defined as the fibre product $X\times_S\Spec\:\Lambda$.
    It is often more useful to work with geometric fibres since it results in working with schemes defined over algebraically closed fields.
    Using Theorem \ref{etale_algebras}, we can actually conclude that the geometric fibre $X_{\overline{s}}$ has finitely many points.
    Indeed if $S = \Spec\:A$ and $X=\Spec\:B$ locally with $B$ a finite $A$-algebra, then $X_{\overline{s}} = \Spec\:(B\otimes_A\Lambda)$ with $B\otimes_A\Lambda$ a finite  ($n$-)dimensional $\Lambda$-algebra, invoking Theorem \ref{etale_algebras} we obtain that $B\otimes_A\Lambda \cong \Lambda \times ...\times\Lambda$ ($n$ times). 
    Therefore $X_{\overline{s}} = \Spec\:\Lambda \:\sqcup \:...\sqcup \:\Spec\:\Lambda$, a scheme consisting of $n$ points.
\end{construction}

\noindent So, in light of Proposition \ref{unramified_carac} and Corollary \ref{etale=flat+unramified}, we have an analogue definition. 

\begin{definition}
    A finite morphisms of schemes $f:X\rightarrow S$ is said to be \textbf{finite \'etale} if it is locally free (flat) and if each fibre $X_s$ of $f$ is the spectrum of a finite \'etale $\kappa(s)$-algebra (unramified).
\end{definition}

\noindent Notice from Propositions \ref{finite_is_closed} and \ref{flat_is_open}, that the image of a finite \'etale morphisms is both open and closed.

\begin{remark}
    Resuming our discussion over geometric fibres, we get from Theorem \ref{etale_algebras} that the fibres $X_s$ are spectra of a finite \'etale $\kappa(s)$-algebra if and only if its geometric fibres are of the form $\Spec\:(\Lambda \times ...\times \Lambda)$, or in other words that they are finite disjoint union of points defined over $\Lambda$ (when $f$ is a finite \'etale morphism).
This is the main reason behind our preference for the geometric fibres.
\end{remark}

\noindent In the following, we provide a theory of Galois ‘‘\`a la Grothendieck" for finite \'etale covers analogous to the theory developed for the classification of field extensions and of covering spaces in the first chapter. 

\subsubsection{Finite \'Etale Coverings}
\label{finite_etale_coverings}

\begin{definition}
    A morphism $f:X\rightarrow S$ of schemes is a \textbf{finite \'etale cover} if $f$ is a surjective finite \'etale morphism. 
    And $f$ is called a \textbf{trivial finite \'etale cover} if $X$ as a scheme over $S$ is isomorphic to a disjoint union of copies of $S$, and the map $f$ restricts to identity on each component.
\end{definition}

\begin{example}
    The simplest case would be to check how finite \'etale coverings would look like over $S=\Spec\:k$ for $k$ an algebraically closed field.
    Let us consider $X$ any finite \'etale covering over $S$. Since $f$ is finite and \'etale, we know that $X$ is the spectrum of a finite \'etale $k$-algebra $A$. 
    And by Theorem \ref{etale_algebras}, we have that $A \cong A\otimes_k\overline{k} \cong k^{\oplus^n}$ because $k$ is algebraically closed.
    So $X = \Spec\:k\: \sqcup \:...\sqcup\:\Spec\:k$ ($n$ times), meaning that $\Spec\:k$ only has trivial finite \'etale covers which makes sense since $S$ is just a point. 
\end{example}

\noindent In order to make more sense of trivial coverings, we show that finite \'etale coverings are locally trivial, we want the situation to be analogous to that of topological coverings.
Notice that the restriction of a cover $Y\rightarrow X$ above an open subset $U\subset X$ is none but the fibre product $Y\times_XU$, the proposition below then proves our claim.

\begin{proposition}
\label{etale_cover_locally_trivial}
    Let $f:X\rightarrow S$ be an affine surjective morphism of schemes with $S$ connected. Then $f$ is a finite \'etale cover if and only if, there is a finite, locally free and surjective morphism $g:Y\rightarrow S$ such that $X\times_SY$ is a trivial cover of $Y$.
\end{proposition}
\begin{proof}
    Suppose that $g$ is finite, locally free and surjective and $X\times_SY$ is a trivial cover of $Y$.
    As $g$ is locally free, each point in $S$ is contained in an affine open subset $U=\Spec\:A$ over which $g$ restricts to a morphism $\Spec\:C\rightarrow\Spec\:A$ with $C$ an $A$-algebra that is finitely generated and free as an $A$-module.
    If $f$ restricts to $\Spec\:B\rightarrow \Spec\:A$ over $U$, then we have the following base change $\Spec(B\otimes_AC) \rightarrow \Spec\:B$ where $B\otimes_AC$ is a finitely generated free $C$-module, hence also a finitely generated free $A$-module.
    On the other hand, $\Spec(B\otimes_AC)$ is isomorphic to a finite direct sum of $\Spec\:B$ since it is a trivial cover, and this is only possible if $B$ is finitely generated free $A$-module.
    This shows that $f$ must be finite and locally free.
    Let $\overline{s}:\Spec\:\Lambda \rightarrow Y$ be a geometric point of $Y$. By composition with $g$ it yields a geometric point of $S$.
    The geometric fibres $X_{(\overline{s}\circ g)} := X\times_S \Spec\:\Lambda$ and $(X\times_SY)_{\overline{s}}:= X\times_SY\times_Y\Spec\:\Lambda$ are isomorphic and the latter is a finite direct sum of copies of $\Spec\:\Lambda$ since $X\times_SY$ is a finite \'etale covering of $Y$.
    And since $g$ is surjective, we have that all the fibres of $f$ are indeed finite \'etale $\kappa(s)$-algebras (invoking Theorem \ref{etale_algebras}). Thus $f$ is a finite \'etale cover.\\
    In the other direction, suppose that $f$ is a finite \'etale cover.
    $S$ being a connected scheme, then all fibres of $f$ have the same cardinality say $n$.
    We now proceed by induction on $n$.\\
    If $n=1$, then since $f$ is finite \'etale, it is also an isomorphism. So it is possible to find a finite locally free surjective morphism $g:Y\rightarrow S$ such that $S\times_SY$ is a trivial cover of $Y$.\\
    For $n>1$, we consider the base change $X\times_SX\rightarrow X$ 
    (notice that $(X\times_SX)_x \cong X_s$ for $x$ a point lying above $s$, so $|(X\times_SX)_x|=n$ for all $x\in X$ as well), 
    by Proposition \ref{clopen_immersion}, the diagonal map $\Delta$ induces a section of the base change map. In fact, $X\times_SX$ is the disjoint union of $\Delta(X)$ with some open and closed subscheme $X'$ of $X$, with $X'$ being the complement of $\Delta(X)$.
    Now consider the inclusion map $X'\rightarrow X\times_SX$ and the projection map $X\times_SX\rightarrow X$, these maps are both finite \'etale since respectively, open embeddings are \'etale and finite \'etale morphisms are preserved under base change.
    Then, by Proposition \ref{composition_etale}, the composition $X'\rightarrow X$ is also finite \'etale.
    By construction, we have that the fibres of this morphism has cardinality $n-1$.
    Indeed, since $X\times_SX = \Delta(X) \:\sqcup X'$, we also have that $(X\times_SX)_x = \Delta(X)_x \:\sqcup X'_x$ and because $\Delta(X)\cong X$, $|\Delta(X)_x|=1$, hence $|X'_x|=n-1$.
    So the induction hypothesis yields a finite, locally free and surjective morphism $g':Y\rightarrow X$ such that $X'\times_XY$ is isomorphic to the disjoint union of $n-1$ copies of $Y$. But then $(X\times_SX)\times_SY \cong X\times_SY$ is the disjoint union of $n$ copies of $Y$.
    Lastly, notice that the composition $g=f\circ g'$ is also finite, locally free and surjective since $g'$ and $f$ are.
\end{proof}

\begin{definition}
    We can actually build a category whose objects consists of finite \'etale covering $X\rightarrow S$ over a base scheme $S$ and whose morphisms are morphisms of schemes that are compatible over $S$. 
    It is a full subcategory of the category of relative schemes over $S$ and we denote it $\Fet_S$.
\end{definition}

\begin{proposition}
\label{reduced_carry_above}
    Let $S$ be a connected scheme and $p:X\rightarrow S$ is a morphisms in $\Fet_S$. 
    If $S$ is a reduced scheme, then $X$ is a reduced scheme as well.
\end{proposition}
\begin{proof}
     We may assume that $X= \Spec\:  B$ and $S= \Spec\:A$.
    Let $P$ be a point of $S$ (a prime ideal).
    By assumption, the localisation $A_P$ has no nilpotents and we will prove that the same holds for $C_P = B\otimes_A A_P$.
    This is sufficient since the local ring of points of $X$ lying above $P$ are localisations of $C_P$ and $p$ is surjective.\\
    Let $P_1,..P_r$ be the set of minimal prime ideals in $A_P$.
    Their intersection is the nilradical (see \cite{Eisenbud} Corollary $2.12$), so the ideal of nilpotents in $A_P$, then  the natural map $A_P\rightarrow \prod(A_P/P_i)$ is injective.\\
    Now, since $p$ is a finite \'etale cover, it is also locally free, so $C_P$ is a free $A_P$-module. 
    So tensoring the above map with $C_P$, we obtain an injective map $C_P\rightarrow \prod (C_P/P_iC_P)$.\\
    Each $\Spec\:C_P/P_iC_P$ is finite \'etale over $\Spec\:A_P/P_i$ (by base change). 
    So by considering this map, we can reduce to the case where $A_P$ is an integral domain.
    Denoting by $K_P$ its fraction field, we get that the following map $\Spec\: C_P\otimes_{A_P}K_P\rightarrow \Spec K_P$ is finite \'etale cover by base change.    
    And the natural map $K_P \rightarrow C_P\otimes_{A_P}K_P$ induced by the inclusion $A_P\rightarrow K_P$ is injective since $p$ is locally free.
    So $C_P\otimes_{A_P}K_P$ is a finite product of fields (by \'etaleness) which is clearly reduced (reduced rings are stable under finite product), hence $C_P$ ha no nilpotents.
\end{proof}

\noindent From now on, the reader might notice the resemblance of certain of the following results with lemmas and propositions from Section \ref{covering_spaces} and the theory of field extensions.

\begin{lemma}
\label{etale_between_composition}
    Let $f:X\rightarrow Y$ and $g:Z\rightarrow X$ be morphisms of schemes. If $f\circ g$ and $f$ are finite \'etale, then so is $g$.
\end{lemma}
\begin{proof}
    By Proposition \ref{clopen_immersion}, the diagonal map $X\rightarrow X\times_S X$ is finite as it is a closed immersion and since $X$ is isomorphic to a union of connected components in $X\times_SX$ it is also \'etale.
    Considering the fibre product of $Z$ over $X$ with $Y$ via $g$, we get a morphism $\Gamma_g : Z\rightarrow Z\times_YX$ (the "graph" of $g$) which is again finite and \'etale.
    The projection $p_2:Z\times_YX \rightarrow X$ is again finite \'etale (being the base change of $f\circ g:Z\rightarrow Y$).
    So the composition $p_2\circ\Gamma_g = g$ is finite \'etale.
\end{proof}

\begin{proposition}
\label{sections_etale_cover}
    Let $f:X\rightarrow S$ be a morphism in $\Fet_S$, and $s: S\rightarrow X$ be a section of $f$.
    Then $s$ induces an isomorphism of $S$ with an open and closed subscheme of $X$.
\end{proposition}
\begin{proof}
    By Lemma \ref{etale_between_composition}, $s$ is a finite \'etale morphism. 
    Now since $s$ is finite, it is closed (\ref{finite_is_closed}), so its image is closed as well. Since $s$ is flat and locally of finite type, the image of $s$ is open as well (\ref{flat_is_open})
    And since sections are injective morphisms, the proof is done.
\end{proof}

\noindent In particular, if $S$ is connected, then the section $s$ maps $S$ isomorphically onto a whole connected component of $X$ (this is similar to what happened in Chapter $1$ for coverings of connected topological spaces).
The next result is of utmost importance as it is analogous to Proposition \ref{property_morphism_covers}.
And if you recall, it as used profusely throughout the first chapter. 

\begin{proposition}
\label{proporty_mono_etale_cover}
    Let $g:Z\rightarrow S$ be a morphism of schemes with $S$ a connected scheme and $h_1,h_2: Z \rightarrow X$ are two morphisms of schemes over $S$, with $f:X\rightarrow S$ an \'etale morphism with $h_1\circ \overline{z} = h_2\circ \overline{z}$ for some geometric point $\overline{z}: \Spec\: \Lambda\rightarrow Z$. Then $h_1 = h_2$.
\end{proposition}
\begin{proof}
    Consider the fibre product $Z\times_SX$, since \'etaleness is preserved under base change, we may assume that $Z=S$. 
    Now, morphisms $h_1$ and $h_2$ become sections of $f$, so it is enough to check that 
    if two sections of a finite \'etale cover of a connected scheme $S$ coincide at a geometric point, then they are equal. 
    This follows from Proposition \ref{sections_etale_cover} since each section is an isomorphism of (now) $S$ onto a connected component of $X$, thus it is determined by the image of a geometric point.
\end{proof}

\noindent For a morphism of schemes $f:X\rightarrow S$, we define the group of scheme automorphisms of $X$ preserving $f$ to be $\Aut(X|S)$. 
Notice that there is a natural left action of $\Aut(X|S)$ on the geometric fibre $X_{\overline{s}} = X\times_S\Spec\;\Lambda$ (coming from the base change of its action on $X$).

\begin{corollary}
\label{no_fixed_points}
    Let $f: X\rightarrow S \in \Fet_S$. Then the nontrivial elements of $\Aut(X|S)$ act without fixed points on each geometric fibre. Hence $\Aut(X|S)$ is finite.
\end{corollary}
\begin{proof}
    Just as in the proof of Corollary \ref{corollary_property_morphism_covers}, the first statement is immediate from Proposition \ref{proporty_mono_etale_cover} by taking $Z=X$, $h_1 = id$ and $h_2 = \phi\in \Aut(X|S)$. So it follows that the action of an elements of $\Aut(X|S)$ on the underlying set of geometric fibres is faithful. And since geometric fibres of finite \'etale morphisms are finite sets, and $\Aut(X|S)$ acts faithfully on such fibres, then $\Aut(X|S)$ is a subgroup of $\Sym(X_{\overline{s}})$ a finite group.
\end{proof}

\noindent Again, one important construction is that of making new finite \'etale covers from existing ones. This is done by considering quotients of group actions as follow, and then checking that they are indeed schemes.

\begin{construction}
    Let $f: X\rightarrow S$ be an \textit{affine} surjective morphism of schemes, and $G\leq \Aut(X|S)$ a finite subgroup. 
    We define the ringed space $X/G$ for which the underlying topology is the quotient topology of $X$ by the action of $G$. And the morphism of ringed spaces $p :X\rightarrow X/G$ for which the underlying continuous map is the natural projection.
    And we define the structure sheaf of $X/G$ to be the subsheaf $(p_*\mathcal{O}_X)^G$ of $G$-invariant elements in $p_*\mathcal{O}_X$.
\end{construction}

\begin{proposition}
    The ringed space $X/G$ above is indeed a scheme, the morphism $p$ is affine and surjective, and $f:X\rightarrow S$ factors as $f = g\circ p$ for $g:X/G \rightarrow S$ an affine morphism.
\end{proposition}
\begin{proof}
    Using the affine assumption on $f$, we may assume that $S=\Spec\:A$ and $X=\Spec\:B$ and $f$ comes from a ring homomorphism $\lambda: A\rightarrow B$.
    It is then sufficient to prove that the ringed space $X/G$ is isomorphic to the spectrum of $B^G$, the ring of $G$-invariants of $B$, let's denote it $X^G$. \\
    This will also imply the rest of the claim, that the projection $p$ is surjective.
    Note that $B^G\subset B$ is an integral extension, indeed every $b\in B$ is a root of the monic polynomial $\prod (x-\sigma(b)) \in B^G[x]$, where $\sigma$ runs over the elements of $G$. So by Proposition \ref{Lang_1_10}, every point $\mathfrak{p}$ in $X^G$ has a preimage $\mathfrak{q}$ in $\Spec\:B$.\\
    To identify $X/G$ with $X^G$ it is enough to identify them as sets, as closed subsets $V(I)\subset X$ induce closed subsets $V(I^G) \subset X^G$. 
    By the surjectivity hypothesis on $p$ we just need to show that the fibres of of $X\rightarrow \Spec\: B^G$ are the $G$-orbits of $X$.
    By contraposition, assume that we have two $G$-orbits $\{g(P)\mid g \in G\}$ and $\{g(Q)\mid g\in G\}$ for $P,Q\in X$ that are lying above the same point $J$ in $\Spec\:B^G$ (so $g(P)\cap B^G = J = g(Q)\cap B^G$).
    Notice that the fibre $X_{\kappa(J)} = \Spec\:(B\otimes_{B^G}\kappa(J))$ is zero dimensional, indeed by base change, $\overline{B}=B\otimes_{B^G}\kappa(J)$ is now a finite-dimensional $\kappa(J)$-vector space, which means that it is Artinian as well implying that it has zero Krull dimension \ref{Artin_zero_dim}.
    So any point $g(P)$ and $g(Q)$ induces maximal ideals $g(\overline{P})$ and $g(\overline{Q})$ of $\overline{B}$ with $\bigcap g(\overline{P}) = 0 = \bigcap g(\overline{Q})$.
    And using the Chinese Remainder Theorem, we fin an element $\overline{b}\in \overline{B}$ such that $\overline{b}\in g(\overline{P})$ but $\overline{b}\notin g(\overline{Q})$ for all $g\in G$ which gives us our contradiction.\\
    Finally, to show that $(p_*\mathcal{O}_X)^G \cong \mathcal{O}_{X^G}$, notice that the first sheaf is quasi-coherent as the kernel of the following morphism of quasi-coherent sheaves,
    $$p_*\mathcal{O}_X \rightarrow \bigoplus_{g\in G} p_*\mathcal{O}_X \: , \: s\mapsto(..., g(s)-s, ...)$$
    Thus it is enough to check the isomorphism on the ring of sections over $X^G$, and in both cases these are exactly $B^G$.    
\end{proof}

\begin{proposition}
\label{quotient_finite_etale_cover}
    Let $f: X\rightarrow S \in \Fet_S$, with $X$ being connected, and $G\leq \Aut(X|S)$ a finite subgroup. 
    Then $p:X\rightarrow X/G$ and $g:X/G \rightarrow S$ are finite \'etale covers.
\end{proposition}
\begin{proof}
    Using Lemma \ref{etale_between_composition} it is sufficient to only prove that $g:X/G\rightarrow S$ is a finite \'etale cover.
    We apply Proposition \ref{etale_cover_locally_trivial} to obtain a base change $X\times_SY\rightarrow Y$ such that $X\times_SY$ is a trivial cover of $Y$, so $X\times_SY\cong F\times Y$ for $F$ a finite set.
    There is a natural action of $G$ on $X\times_SY$ coming from the base change of the action of $G$ on $X$, thus we have the following $(X\times_SY)/G\cong (F/G)\times Y$.\\ 
    Observe that we also have $(X\times_SY)/G\cong (X/G)\times_SY$. 
    Indeed, just as in the previous proof, consider a small affine neighbourhood $U:=\Spec\:A$ of a point $s\in S$ with $\Spec\:B$ and $\Spec\:C$ its preimages in $X$ and $Y$ respectively. 
    The isomorphism we want to prove then translates to $B^G\otimes_AC \cong (B\otimes_AC)^G$, and this holds for $U$ sufficiently small as $C$ is then a free $A$-module with trivial $G$-action.
    Finally obtain that $(X/G)\times_G Y \cong (F/G)\times Y$ which is again a finite disjoint union of copies of $Y$. And we conclude by applying Proposition \ref{etale_cover_locally_trivial} in the other direction.
\end{proof}

\noindent Before stating the analogue to the theorem of Galois classification of topological covers, we need to introduce a new type of finite \'etale cover: the \textit{finite \'etale Galois cover}.
\begin{definition}
    A connected object $f:X\rightarrow S \in \Fet_S$ is said to be a \textbf{finite \'etale Galois cover} or a \textbf{Galois object} if $\Aut(X|S)$ acts transitively on the underlying set of geometric fibres.
\end{definition}

\begin{remark} \textbf{ }
\label{remark_Galois}
\begin{itemize}
    \item[$a)$] Notice that $X\rightarrow S$ is a finite \'etale Galois cover if and only if the canonical morphism $X/\Aut(X|S) \rightarrow S$ is an isomorphism.
    This ties in well with the way we defined topological Galois covering spaces in the first chapter.\\
    Indeed, first notice that for $X$ a connected scheme, the evaluation map $\Aut(X|S)\rightarrow X_{\overline{s}}$ mapping $\sigma\mapsto\sigma(\overline{x})$ is injective (this follows from Proposition \ref{proporty_mono_etale_cover}).
    Then the evaluation is actually an isomorphism if and only if $X$ is Galois covering. 
    Now the isomorphism we want to prove follows by noticing that $S_{\overline{s}}\cong (X/\Aut(X|S))_{\overline{s}}$ in the category of finite sets since they are now both singletons, and that $(X/\Aut(X|S))_{\overline{s}} \cong X_{\overline{s}}/\Aut(X|S)$ from the proof of Proposition \ref{quotient_finite_etale_cover}.
    \item[$b)$] Similarly to Remark \ref{Galois_free}, for $X$ a finite \'etale Galois cover, the group $\Aut(X|S)$ acts also freely on $\Fib_{\overline{s}}(X)$. This is true since in this case, 
    Corollary \ref{no_fixed_points} plays the same role as Corollary \ref{corollary_property_morphism_covers} in chapter $1$.\\
\end{itemize}
\end{remark}

\begin{theorem}
\label{Galois_correpondance_etale_cov}
    Let $f:X\rightarrow S$ be a finite \'etale Galois cover.
    If $g: Z\rightarrow S$ is a connected finite \'etale cover such that $f = g \circ h$ for $h:X\rightarrow Z$, then $h$ is a finite \'etale Galois cover and $Z\cong X/H$ (where $H$ is a subgroup of $\Aut(X|S)$).
    That way we get a bijection between subgroups of $\Aut(X|S)$ and intermediate covers $Z$ as above.
    The cover $g:Z\rightarrow Y$ is Galois if and only if $H$ is a normal subgroup of $\Aut(X|S)$ (in which case, $\Aut(Z|S) \cong \Aut(X|S) / H$).
\end{theorem}
\begin{proof}
    This proof is similar to that of Theorem \ref{clasification_of_covers}.
    Here, Lemma \ref{etale_between_composition},
    Propositions \ref{proporty_mono_etale_cover}, \ref{quotient_finite_etale_cover} and Remark \ref{remark_Galois}
    are analogous to
    Proposition \ref{composition_covers},
    Theorem \ref{property_morphism_covers},
    Proposition \ref{quotient_cover} and  \ref{G_is_aut_group} and Proposition \ref{caracterize_Galois_cov}.\\
\end{proof}

\noindent Finally, we prove a useful lemma which says that any finite \'etale cover of the separable closure scheme $\overline{X}$ for a certain class of schemes $X$ over $k$ come from finite separable extensions of $k$.
And then, we prove a key proposition that generalises Lemma \ref{embed_finite_subextension}, remember, it stated that every finite separable field extension can be embedded in a finite Galois extension and there is a smallest such extension called the Galois closure. In relation to the former, we write:
\begin{itemize}
    \item[]$\overline{X}:= X\times_{\Spec\:k}\Spec\:k_s$ for $X$ a scheme over $k$ and $k_s$ the separable closure, and
    \item[] $X_L:= X\times_{\Spec\:k}\Spec \:L$ for $L$ a finite extension of $k$ contained in $k_s$.
\end{itemize}
\noindent Of course, by base change, we have that $X$ is a finite \'etale cover of $S$ if and only $\overline{X}$ and $X_L$ are respectively finite \'etale covers of $\overline{S}$ and $S_L$. 

\begin{lemma}
\label{etale_cover_closure_come_from_extension}
    Given a finite \'etale cover $\overline{f}:\hat{Y}\rightarrow \overline{X}$ with $X$ a quasi-compact scheme over $k$, 
    there is a finite extension $L|k$ contained in $k_s$ and a finite \'etale cover $Y_L$ of $X_L$ so that $\hat{Y} \cong Y_L\times_{\Spec\:L}\Spec\:k_s$.\\
    And similarly, elements of $\Aut(\hat{Y}|\overline{X})$ come from $\Aut(Y_L|X_L)$.
\end{lemma} 
\begin{proof}
    Since $X$ is quasi-compact, it has a finite covering by open affine subschemes.
    Let $\{U_i\mid 1\leq i\leq n\}$ be a finite open covering of $X$ such that $U_i=\Spec\:A_i$. \\
    For each such $i$ associate an affine open $\overline{U}_i:= \Spec\:A_i\otimes_kk_s$
    in $\overline{X}$ so that the preimage via $\overline{f}$ of each $\overline{U}_i$ in $\hat{Y}$ is an affine open subscheme of the form $\Spec\:\hat{B}_i$, where $\hat{B}_i$ is a finitely generated $A_i\otimes_kk_s$-module, note that this holds since $\overline{f}$ is a finite morphism.
    It is a standard fact that we can write a finitely generated $A_i\otimes_kk_s$-module as $(A_i\otimes_kk_s)[x_1,...,x_m]/I$ where $I$ is the ideal generated by the polynomials $f_1,...,f_r$.
    For $1\leq j \leq r$, we have $f_j =\sum_t c_{jt}x^t$ with coefficients $c_{jt} = \sum_la_{jtl}\otimes\lambda_{jtl}\in A\otimes_kk_s$, and since there are finitely many polynomials with finitely many monomials, there are finitely many coefficients with finite tensor factorisation.\\
    If $L$ is the finite extension of $k$ generated by the finite elements $\lambda_{jtl}$, then the coefficients of the polynomials $f_j$ are all contained in $A_i\otimes_k L$
    And hence $\hat{B}_i \cong (( A_i\otimes_kL)[x_1,...,x_m]/I)\otimes_L k_s$.
    Moreover, such an isomorphism holds for all $1\leq i\leq n$
    And by a similar reasoning, we can see that the isomorphism showing the compatibility of the $\hat{Y}\vert_{U_i}$ over the overlaps $U_i\cap U_j$ can be defined by polynomials involving only finitely many coefficients.
    Thus an extension $L$ that is so large that it contains all of the coefficients above satisfies the requirements of the lemma.
\end{proof}

\begin{proposition}
\label{Galois_embed_etale_cover}
    Let $f:X\rightarrow S\in \Fet_S$, with $X$ being connected. There is a morphism $p:Y\rightarrow X$ such that $f\circ p : Y\rightarrow S$ is a Galois object in $\Fet_S$ and 
    moreover every morphisms over $S$ from a Galois object to $X$ factors through $Y$, we call $Y$ the Galois closure.
\end{proposition}
\begin{proof}
    First we fix a geometric point $\overline{s}:\Spec\:\Lambda\rightarrow S$ and let $F=\{\overline{x}_1,...,\overline{x}_n\}$ be the finite set of $\Spec\:\Lambda$-points of the geometric fibre $X_{\overline{s}}$.
    We consider the n-fold product scheme $X^n = X\times_S ...\times_S X$ and a geometric point $\overline{x} = (\overline{x}_1, ..., \overline{x}_n)$. $\overline{x}$ is induced by an ordering of the $\overline{x}_i$'s, that is because of the bijection $X^n\times_S\Spec\:\Lambda \rightarrow (X\times_S\Spec\:\Lambda)^n$.\\
    Let $Y$ be a connected component of $X^n$ containing the image of $\overline{x}$
    and let $p:Y\rightarrow X^n\rightarrow X$ be the composition of the embedding of $Y$ into $X^n$ and the projection onto the first coordinate $X$. Since embeddings and trivial coverings are finite \'etale, by composition (Proposition \ref{composition_etale}), $f\circ p :Y\rightarrow S$ is a finite \'etale covering.\\
    Let $p_{ij} : X^n \rightarrow X_i\times_SX_j$ for $1\leq i <j\leq n$ be the projection onto the product of the $i$-th and $j$-th factors. Let $\Delta(X)$ be the diagonal of $X$ in $X\times_S X$ of image $p_{ij}(X)$. By Proposition \ref{clopen_immersion}, $\Delta(X)$ is open and closed and therefore so is its inverse image by a projection $p_{ij}$. 
    Denote $\Delta^{-1} := \bigcup_{i,j} p_{ij}^{-1}(\Delta(X))$.
    As $Y$ is connected, if $Y\cap \Delta^{-1} \neq \emptyset$, then $Y\subset \Delta^{-1}$, but this cannot occur since $\overline{x} \notin \Delta^{-1}$. Whence $Y\cap \Delta^{-1}$ is empty and each point in the geometric fibre $Y_{\overline{s}}$ is represented by an $n$-tuple with distinct elements.\\
    We now show that $Y$ is Galois over $S$. 
    Every permutation $\sigma$ of the $\overline{x}_i$ induces an automorphism $f_{\sigma}$ of $X^n$ over $S$ permuting the components. 
    If $f_{\sigma}(\overline{x}) \in Y_{\overline{s}}$, since $f_{\sigma}(Y)\cap Y \neq \emptyset$ and $Y$ is connected, then $f_{\sigma}(Y) = Y$ and  $f_{\sigma}\in \Aut(Y|S)$.
    Thus $\Aut(Y|S)$ act transitively on one geometric fibre, making $Y\rightarrow S$ Galois in $\Fet_S$.\\
    Finally, let $q:Z\rightarrow X$ be a morphism in $\Fet_S$ with $Z\rightarrow S$ a Galois object and let us choose a preimage $\overline{z}$ of $\overline{x}_1$. Then by Theorem \ref{Galois_correpondance_etale_cov}, $q$ is in particular a surjective morphism of covers.
    Composing $q$ with distinct elements of $\Aut(Z|S)$, we get $n$ maps $q_1, ..., q_n : Z\rightarrow X$ such that $q_i\circ \overline{z} = \overline{x}_i$ for every $\overline{x}_i\in F$.
    From these $q_1,...,q_n$, we get a morphisms of schemes $Z\rightarrow X^n$ over $S$ that factors through $P$ since it maps $\overline{y}$ to $\overline{x}$, and $Z$ is connected.
\end{proof}

\subsubsection{\'Etale Group Schemes}
\label{etale_group_scheme}

Another useful point of view is to consider schemes with the extra structure of a group. 
This yields a nice parallel between finite \'etale covers as group schemes and the notion of \textit{principal homogeneous spaces}, that we will explain below.
We take this opportunity to introduce the concept of group schemes a little further, since we will review it in the next chapter.

\begin{definition}
\label{group_scheme_def_1}
    Let $S$ be a scheme. A \textbf{group scheme} over $S$ is a morphism of schemes $p:G\rightarrow S$ that has a section $e:S\rightarrow G$,
    together with multiplication $m:G\times_SG\rightarrow G$ and inverse $i:G\rightarrow G$ morphisms over $S$ satisfying the following associative laws commutative diagrams: 
    \[
    \begin{tikzcd}
        G\times_S G \times_S G \arrow[r,"\id\times m"] \arrow[d, "m\times \id"'] & G\times_S G \arrow[d,"m"] \\
        G\times_S G \arrow[r,"m"'] &  G
    \end{tikzcd}
    \hspace{1cm}
    \begin{tikzcd}
        G\times_S G \arrow[dr, "m"]  & G\times_S S \arrow[l,"\id\times e"'] \arrow[d,"\cong"] \\
        S\times_S G \arrow[u,"e\times\id"] \arrow[r,"\cong"'] &  G
    \end{tikzcd}
    \]
    \\
    and
    \[
    \begin{tikzcd}[column sep=small]
        & G\times_S G \arrow[rr, "\id\times i"] && G\times_S G \arrow[dr,"m"] \\
        G\arrow[rr, "p"] \arrow[ur,"d"] \arrow[dr,"d"'] && S \arrow[rr,"e"] && G\\
        & G\times_S G \arrow[rr,"i\times\id"'] &&  G\times_S G \arrow[ur,"m"']
    \end{tikzcd}
    \]
    Where $d:G\rightarrow G\times_SG$ is the map obtained by pulling back the identity on $G$ along itself (the diagonal morphism), it sends elements $g\in G$ onto $(g,g) \in G\times_S G$.
    We say that $G$ is finite, flat, unramified if the morphism $p:G\rightarrow S$ is respectively finite, flat or unramified. And a morphism between group schemes $G_1$ and $G_2$ over $S$ is a morphism of schemes $G_1\rightarrow G_2$ over $S$.
\end{definition}

\begin{example}\textbf{ }
\begin{itemize}
    \item[$a)$] We can put a group scheme structure on the general linear group $\GL_n$. \\
    By considering the affine space $\mathbb{A}^{2n+1}$, with coordinates $(x_{ij},t)$ where $x_{ij}$ are the entries of a $n\times n$ matrix and $t$ is an additional variable.
    We define a closed affine subscheme in $\mathbb{A}^{2n-1}$ as $\GL_n \cong V(\det(x_{ij})\cdot t -1)$.
    This is an isomorphism since $t=1/\det(A)$ is the inverse of the determinant.
\end{itemize}
    
\end{example}

\begin{remark}
\label{Hopf_algebra_group_scheme}
The case of affine group schemes is even more interesting, in the following we show an anti-equivalence between the category of affine group schemes over $k$ and that of commutative Hopf algebras.\\
Let $G$ be an affine group scheme over $k$ and consider its coordinate ring $\mathcal{O}(G)$, it is an object from the opposite category of commutative $k$-algebras. 
But this objects carries additional structure, recall that the coordinate algebras of cartesian products are given by the tensor product of coordinate algebras and that the coordinate algebra of a one point scheme over $k$ has $k$ as coordinate ring.\\
Hence the coordinate algebra of affine group schemes comes equipped with the following algebra morphisms from the affine group scheme structure, $\Delta: \mathcal{O}(G)\rightarrow \mathcal{O}(G\times_S G) = \mathcal{O}(G)\otimes \mathcal{O}(G)$, $\epsilon : \mathcal{O}(G)\rightarrow k$ and $\Sigma: \mathcal{O}(G)\rightarrow \mathcal{O}(G)$.\\
Moreover, $\Delta$ and $\epsilon$ are also algebra morphisms.
From the commutative diagram of Definition \ref{group_scheme_def_1}, we obtain the following commutative diagrams
\[
    \begin{tikzcd}
        \mathcal{O}(G)\otimes \mathcal{O}(G) \otimes \mathcal{O}(G)   & \mathcal{O}(G)\otimes \mathcal{O}(G) \arrow[l,"\id\otimes \Delta"'] \\
        \mathcal{O}(G)\otimes \mathcal{O}(G) \arrow[u, "\Delta\otimes \id"]  &  \mathcal{O}(G) \arrow[u,"\Delta"'] \arrow[l,"\Delta"]
    \end{tikzcd}
    \hspace{1cm}
    \begin{tikzcd}
        \mathcal{O}(G)\otimes \mathcal{O}(G) \arrow[r,"\id\otimes\epsilon"] \arrow[d,"\epsilon\otimes\id"']& \mathcal{O}(G)\otimes k \\
        k\times \mathcal{O}(G) &  \mathcal{O}(G)\arrow[u,"\cong"'] \arrow[l,"\cong"] \arrow[ul, "\Delta"']
    \end{tikzcd}
    \]
    \\
    and
    \[
    \begin{tikzcd}[column sep=small]
        & \mathcal{O}(G)\otimes \mathcal{O}(G) \arrow[rr, "\id\otimes \Sigma"] && \mathcal{O}(G)\otimes \mathcal{O}(G) \arrow[dr,"m"] \\
        \mathcal{O}(G)\arrow[rr, "p"] \arrow[ur,"\Delta"] \arrow[dr,"\Delta"'] && k \arrow[rr,"\epsilon"] && \mathcal{O}(G)\\
        & \mathcal{O}(G)\otimes \mathcal{O}(G) \arrow[rr,"\Sigma\otimes\id"'] &&  \mathcal{O}(G)\otimes \mathcal{O}(G) \arrow[ur,"m"']
    \end{tikzcd}
    \]
    But this is exactly the structure of a Hopf algebra (see Definition \ref{def_hopf_alg}). 
    The construction being duals, one sees that $(G, m, e,i)$ is an affine group scheme over $k$ if and only if $\mathcal{O}(G)$ with $\Delta$, $\epsilon$ and $\Sigma$ becomes a commutative Hopf algebra.
    
\end{remark}

\noindent We have defined group schemes over $S$ as group objects from the category of schemes over $S$. 
But sometimes this definition is not practical, for example, to define a group scheme one would have to give a scheme $G$, then one needs to define the morphisms $m,i$ and $e$ and finally that they respect the commutative diagrams.
There is an incentive to describe a group scheme as a scheme whose points form a group, it can be done as follows:

\begin{construction}
Suppose we have a scheme $X$ over $S$. 
The notion of points of $X$ is not as transparent as in the case of sets or topological spaces.
If $T$ is a scheme over $S$, then we define a $T$-valued point of $X$, as a morphism $x : T\rightarrow X$ over $S$ and denote the set of such points by $X(T)$. 
In the particular case where $S = \Spec\: k$ and $T : \Spec\:K$ where $K|k$ is a field extension, we refer to such points of $X$ as $K$-rational points.
\end{construction}

\begin{definition}
    A functor $F:\mathcal{C}\rightarrow \Grp$ is said to be a \textbf{representable group-valued functor} if and only if the functor $U\circ F :\mathcal{C}\rightarrow \SET$ is representable where $U:\Grp \rightarrow \SET$ is the forgetful functor.
    Now the category of representable group-valued functors and natural transformations between them is denoted as $[\mathcal{C}, \Grp]$
\end{definition}

\noindent The following result is quite powerful. 
It says that we can understand group schemes as functors and it concludes our above construction by endowing spaces of the form $X(T)$ with a group structure for $X$ and $T$ schemes over $S$.

\begin{proposition}
\label{group_scheme_are_representable_functors}
    There is an equivalence between the category of group schemes over $S$ and the category $[\Sch_S, \Grp]$ of representable group-valued functors from $\Sch_S$ to $\Grp$.
\end{proposition}
\begin{proof}

Consider the category $\Sch_S$ of schemes over $S$ and write $\mathcal{C} = \Fun(\Sch_S, \SET)$ the category of representable functors.
Now for any scheme $X\in \Sch_S$, we define the representable functor $h_X = \Hom_{\Sch_S}(\_,X)$ in $\mathcal{C}$. 
Sending $X$ to $h_X$ gives a covariant functor $h:\Sch_S \rightarrow \mathcal{C}$ that is moreover fully faithful (it is actually a Yoneda embedding, see Proposition \ref{yoneda_embeddings}).
So by the Yoneda lemma, any $X\in \Sch_S$ is uniquely determined by a representable functor $F\in \mathcal{C}$ of the form  by being its unique representative up to isomorphism.\\
We can now define group objects in $\Sch_S$ via the above embedding $h$. 
Let $X\in \Sch_S$, to give a group law on an object $X$ means that for each scheme $T$ over $S$, we have to specify a group law on the set $h_X(T) = \Hom_{\Sch_S}(T,X) = X(T)$ such that for every morphism $f:T_1\rightarrow T_2$, the induced map $h_X(f) : X(T_2) \rightarrow X(T_1)$ is a homomorphism of groups. 
Note that a category allows group objects if it has a terminal object and finite product, which is the case for $\Sch_S$ (a terminal object is the scheme $S$ and fibre products of schemes over $S$ are finite products).\\
So we can now define a functor $\widehat{h}_G = \Hom_{\Sch_S}(\_,G): \Sch \rightarrow \Grp$ sending any scheme $T$ over $S$ to the group $\widehat{h}_G(T) = G(T)$, by the above discussion this is a representable group-valued functor. 
And it gives rise to the functor $\widehat{h}$ from the category of group schemes over $S$ to the category of representable group-valued functors by sending $G\mapsto h_G$.\\
We now establish the equivalence : \\
Given a group scheme $(G,m,e,i)$ over $S$, its group axioms ensure that for every scheme $T\in \Sch_S$, the set $G(T)$ inherits a group structure as above, making the functor $\widehat{h}_G$ a representable group-valued functor.
Conversely, given a representable group-valued functor $\widehat{h}_G$, the group structure on $\widehat{h}_G(T)$ defines natural transformations  
$\widehat{m}_G :\widehat{h}_G\times \widehat{h}_G\rightarrow \widehat{h}_G$,
$\widehat{e}_G :\widehat{h}_S\rightarrow \widehat{h}_G$ and
$\widehat{i}_G :\widehat{h}_G\rightarrow \widehat{h}_G$.\\
And the Yoneda lemma tells us that these natural transformations correspond uniquely to $S$-morphisms $m_G:G\times_S G\rightarrow G$, $e_G:S\rightarrow G$ and $i_G:G\rightarrow G$. 
Moreover, because the embedding is fully faithful, the above natural transformations induced by group axioms over $\widehat{h}_G(T)$ for any schemes $T$ over $S$ translate directly into the commutative diagrams required for $G$ to be a group scheme.
\end{proof}

\begin{remark}
\label{affine_group_schemes_as_functors}
    Since the category of affine schemes is anti-equivalent to that of commutative $k$-algebras via the $\Spec$ functor, from the above proposition, 
    an affine group schemes $G$ over $k$ is equivalent to a group-valued representable functor from the category of $k$-algebras to the category of groups, represented by some $k$-algebra $A$ called the coordinate ring of $G$.
\end{remark}

\noindent Recall from algebraic geometry that there is an anti-equivalence between the category of affine schemes over $k$ and the category of commutative $k$-algebras. 
We can then restrict the above equivalence by an equivalence between affine group schemes over $k$ and representable group-valued functors on $\Alg_k$.
We continue by studying a few examples of group schemes.

\begin{example}
    Let $\Gamma$ be a finite group of order $n$. 
    We define the \textbf{constant group scheme} over $S$, denoted $\Gamma_S\rightarrow S$, to be the disjoint union of $n$ copies of $S$ indexed by $\Gamma$,
    equipped with the projective map given by the identity on each element, i.e.,
    $$\Gamma_S = \bigsqcup_{g\in \Gamma} S_g$$
    Then,
    $$\Gamma_S\times \Gamma_S = \bigsqcup_{g,h\in \Gamma\times\Gamma} S_{(g,h)}, \text{ with }S_{(g,h)} = S_g\times S_h$$
    The group operation on $\Gamma_S$ is induced by $\Gamma$ and defined by the multiplication map $m$ sending $S_{(g,h)}$ to $S_{gh}$.
\end{example}


\noindent In light of Theorem \ref{equivalence_catgory_etale_algebra}, we can ask ourselves if an equivalence of this type exists for finite \'etale group schemes over a field $k$.
The next lemma is necessary to answer this question. It is actually a result from Galois descent theory, and is due to \textit{Andreas Speiser} (it can be seen as a reinterpretation of Hilbert's Theorem $90$). For more on Galois descent, see \cite{galois_descent_Szamuely} $\S\: 2.3$.

\begin{lemma}
\label{descent_argument}
    Let $L|k$ be a Galois extension with Galois group $G=\Gal(L|k)$.
    Given a relative Galois representation $V$
    $$V \cong L \otimes V^G$$
\end{lemma}
\begin{proof}
    See \cite{galois_descent_Szamuely}, Lemma $2.3.8$.
\end{proof}

\begin{proposition}
    The category of finite \'etale group schemes over $k$ is equivalent to the category of finite groups carrying a continuous $\Gal(k)$-action.
\end{proposition}
\begin{proof}
    Assume the base field $S = \Spec\:k$ for $k$ a field.
    Then by Theorem \ref{etale_algebras}, a finite \'etale group scheme $G\rightarrow S$ is of the form $G = \Spec\: A$ with a $A$ finite \'etale $k$-algebra.\\
    And by Theorem \ref{equivalence_catgory_etale_algebra}, it corresponds to a finite set equipped with a continuous $\Gal(k)$-action.\\
    That finite set being the fibre $G_{k_s}$ of $G$ over the geometric point $\Spec \: k_s \rightarrow G$, it carries a group structure coming from that of $G$, meaning it is compatible with the Galois action.\\
    Conversely, the Galois-equivariant group operations on a finite continuous $\Gal(k)$-set, yield a group scheme structure on the corresponding  finite \'etale scheme over $k$.
    Over $k_s$, the corresponding constant group of such a $\Gal(k)$-set $X$ has a coordinate algebra $\prod_{x\in X} k_s$.
    Naturally, it is isomorphic to the vector space $V = \{f:X\rightarrow k_s\}$ where an element $a\in \prod k_s$ is of the form $(a_{x_1},...,a_{x_n})$ and is defined by a function $f(x) = a_x$. 
    The absolute Galois group has an action over this vector space, for any $\sigma\in \Gal(k)$, we have $\sigma\cdot (a_x)_{x\in X} = (\sigma(a_{\sigma^{-1}x}))_{x\in X}$.
    Notice that this equips $V$ with a semi-linear $\Gal(k)$-action.
    Now we can apply our above lemma for the descent argument. 
    Consider the subalgebra $A = V^G$ consisting of all $\Gal(k)$-equivariant functions in $V$.
    Since all the requirements of the lemma are verified, we have $A\otimes_k k_s \cong \prod k_s$.
    From Theorem \ref{etale_algebras}, this means that $A$ is a finite \'etale algebra, and so its spectrum gives a finite \'etale group scheme over $\Spec\: k$.\\
    To conclude that the resulting finite \'etale scheme admits a group structure making it a group scheme, we are going to check that its coordinate ring $A$ is a Hopf $k$-algebra.\\
    We define a map $\Delta:A\rightarrow A\otimes_k A$ as follow, 
    for an element $f\in A$, we set $\Delta(f)$ a function on $X\times X$ to be $\Delta(f)(x,y) = f(x\cdot y)$.
    Since the group operation on $X$ is compatible with the action of $\Gal(k)$, for $\sigma\in\Gal(k)$, we hve
    $$\sigma (\Delta(f)(x,y)) = \sigma (f(x\cdot y)) = f(\sigma (x\cdot y)) = f(\sigma(x)\cdot \sigma(y)) = \Delta(f)(\sigma(x),\sigma(y))$$
    so it is a well-defined map of $k$-algebras, induced by the group law on $X$.
    Similarly, we can define maps $\epsilon$ and $\Sigma$.\\
    Thus by Remark \ref{Hopf_algebra_group_scheme}, $\Spec\:A$ is a finite \'etale group scheme over $k$.
\end{proof}

\begin{remark}
\label{constant_group_scheme_etale}
The previous proposition actually tells us more.
Consider $G$ a finite \'etale group scheme over $k$, and $\Gamma = G(k_s)$ the group of $k_s$-valued points of $G$.
Then we can consider the constant group scheme $\Gamma_k$ over $k$, and the above proposition tells us that $G\otimes k_s \cong \Gamma_k \otimes k_s$.
This can be translated into the following, a finite group scheme $G$ over $k$ is \'etale if and only if its separable fibre $G_{k_s}$ is a constant group scheme.
\end{remark}

\begin{definition}
    Let $S$ be a connected scheme, and $G\rightarrow S$ a finite flat group scheme.
    $X$ is a (left) \textbf{$G$-torsor} or \textbf{$G$-principal homogeneous space} over $S$ is a finite locally free surjective morphism $X\rightarrow X$ together with a group action $\rho : G\times_S X\rightarrow X$ such that the map $(\rho,\id):G \times_S X\rightarrow X \times_S X$ is an isomorphism and the following diagrams commute 
    \[
    \begin{tikzcd}
        G\times_S G \times_S X \arrow[r,"\id\times \rho"] \arrow[d, "m\times \id"'] & G\times_S X \arrow[d,"\rho"] \\
        G\times_S X \arrow[r,"\rho"'] &  X
    \end{tikzcd}
    \hspace{1cm}
    \begin{tikzcd}
        S\times_S X \arrow[dr, "\cong"']\arrow[r,"\epsilon \times \id"]  & G\times_S X \arrow[d,"\rho"] \\
        &  X
    \end{tikzcd}
    \]
    
\end{definition}

\noindent A similar characterisation of $G$-torsors arise as follow:
\begin{lemma}
\label{other_charac_torsors}
    Let $S$ be a connected scheme, $G\rightarrow S$ a finite flat group scheme, and $X\rightarrow S$ a scheme over $S$ equipped with an action $\rho: G\times_SX \rightarrow X$ (on the left).
    These data define a $G$-torsor over $S$ if and only if there exists a finite locally free surjective morphism $Y\rightarrow S$ such that $X\times_SY\rightarrow Y$ is isomorphic (as a scheme over $Y$) with a $G\times_SY$-action, to $G\times_SY$ acting on itself by translations (on the left).
\end{lemma}
\begin{proof}
    Suppose such data define a $G$-torsor.
    By assumption, the $G$-torsor $X\rightarrow S$ is a finite locally free surjective morphism, and $X\times_SX\rightarrow X$ is isomorphic to $G\times_SX\rightarrow X$ (basicaly take $Y=X)$.\\
    For the other direction, notice that since $G\times_S Y\rightarrow Y$ and $Y\rightarrow S$ are finite locally free and surjective, the same is true for $X\rightarrow S$ (by base change).\\
    Now, the map $G\times_SX\rightarrow X\times_SX$ corresponds via Proposition \ref{affine_morph_quasi_coherent}
    to a morphism $\varphi :(\phi_*\mathcal{O}_X)\otimes_{\mathcal{O}_S}(\phi_*\mathcal{O}_X) \rightarrow (\phi_*\mathcal{O}_G)\otimes_{\mathcal{O}_S}(\phi_*\mathcal{O}_X)$.
    And since, by assumption, $G\times_S Y\cong X\times_SY$, we have that the map $(\rho \times \id_X\times \id_Y) : G\times_SX\times_SY\rightarrow X\times_S X\times_S Y$ is an isomorphism.
    This means that the morphism $\varphi$ becomes an isomorphism after tensoring with $\phi_*\mathcal{O}_Y$.
    But $\phi_*\mathcal{O}_Y$ is locally free (by assumption), so $\varphi$ must be an isomorphism.
\end{proof}

\noindent We finnish this section by giving a partial classification of $G$-torsors over $S$ for $G$ a finite \'etale group scheme over $S$.

\begin{theorem}
    Let $S$ be a connected scheme, and $G$ a finite \'etale group scheme over $S$.
    \begin{itemize}
        \item[$a)$] If $G$ is a constant group scheme $\Gamma_S$, then a $G$-torsor is the same as a finite Galois \'etale cover with group $\Gamma$.
        \item[$b)$] If there is a morphism $S\rightarrow \Spec\:k$ for $k$ a field, and $G$ arise from an \'etale group scheme $G_k$ over $k$ by base change to $S$, then every $G$-torsor $Y\rightarrow S$ is a finite \'etale cover of $S$.

        Moreover, there is a finite separable extension $L|k$ such that $Y\times_k\Spec\:L \rightarrow S\times_k \Spec\: L$ is a Galois \'etale cover.
    \end{itemize}
\end{theorem}
\begin{proof}
    The first statement follows by combining Proposition \ref{etale_cover_locally_trivial} and  Lemma \ref{other_charac_torsors}, we have
    $$Y\times_SY \cong \Gamma_S\times_SY = (\bigsqcup_{\Gamma} S)\times_SY = \bigsqcup_{\Gamma} Y$$
    By base change we obtain an action of $\Gamma_S$ on $Y_{\overline{s}}$ that is transitive (by definition of $G$-torsors) for $\overline{s}$ a geometric point of $S$, and since $\Gamma_S$ is the constant group scheme, the action comes from $\Gamma$.
    Thus the finite \'etale cover associated to the $\Gamma_S$-torsor is moreover Galois.\\
    For the second statement, consider a separable extension $L|k$ for which $G_k\times_k\Spec\:L$ is a constant group scheme (this is possible at least for $L=k_s$ from our discussion in Remark \ref{constant_group_scheme_etale}), 
    then we can apply $a)$ to $Y\times_k \Spec\:L \rightarrow S\times_k\Spec\:L$, and we can conclude similarly.
    The rest of the statement follows from another application of Proposition \ref{etale_cover_locally_trivial}.
\end{proof}

\subsubsection{\'Etale Fundamental Groups}

In previous sections we have illustrated a theory of finite \'etale coverings of schemes and shown that the classification of those covers is ``Galoisienne".
In continuation of a similar analogy with Chapter $1$, we now wish to define an object that behaves like the fundamental group but for connected schemes.\\
Like in Subsection \ref{universal_cover_section}, we will define this fundamental group via a fibre functor constructed as follows:    

\begin{construction}
    Let $S$ be a scheme, and consider the category $\Fet_S$.
    For the rest of the construction, we fix a geometric point $\overline{s}:\Spec\:\Lambda \rightarrow S$. 
    We now define a fibre functor on $\Fet_S$ as
    the functor $\Fib_{\overline{s}}$ sending an object $X\rightarrow S$ in $\Fet_S$ to the underlying set of its geometric fibre $X\times_S\Spec\:\Lambda$ at $\overline{s}$. 
    And sends a morphism $X\rightarrow Y$ from $\Fet_S$ to the induced set-theoretic map $X\times_S\Spec\:\Lambda \rightarrow Y\times_S\Spec\:\Lambda$.\\
    And finally we define an action on the fibres as follow:
    Consider $F:\mathcal{C}\rightarrow\mathcal{D}$ a functor between two categories.
    An automorphism of $F$ is a natural transformation $\omega:F\rightarrow F $ that has a two sided inverse, we write $\omega \in \Aut(F)$. Composition of natural transformations then equips the set $\Aut(F)$ with a group structure. 
    Notice that for any object $C\in \mathcal{C}$, any $\omega\in \Aut(F)$ which is a collection of bijections $\omega_C : F(C)\rightarrow F(C)$, and given a morphism $f:C\rightarrow C'$, we have the following commutative diagram,
    \[
    \begin{tikzcd}
        F(C) \arrow[r,"F(f)"] \arrow[d, "\omega_C"'] & F(C') \arrow[d,"\omega_{C'}"] \\
        F(C) \arrow[r,"F(f)"'] &  F(C').
    \end{tikzcd}
    \]
\end{construction}

\noindent In the literature the geometric fibre of a scheme say $X$ over $S$ at a geometric point say $\overline{s}:\Spec\:\Lambda \rightarrow S$ is often denoted by $X_{\overline{s}} : = X\times_S\Spec\:\Lambda$. 
In our case however, we want to stress that the main object of study is actually the functor $\Fib_{\overline{s}}$ and not the geometric point or the scheme.
So even though the commonly used notation would lighten our notations, we stick with writing the fibre functor.

\begin{definition}
    Given a scheme $S$ and a geometric point $\overline{s}:\Spec\:\Lambda\rightarrow S$, we define the \textbf{\'etale fundamental group} $\pi_1^e(S,\overline{s})$ as the automorphism group of the fibre functor $\Fib_{\overline{s}}$ on the category of finite \'etale covers of $S$.
\end{definition}

\begin{remark}
    This obviously implies that $\pi_1^e(S,\overline{s})$ acts naturally on $\Fib_{\overline{s}}(X)$ for any $X$ being a finite \'etale cover of $S$, so $\Fib_{\overline{s}}$ takes its values in the category of $\pi^e_1(S,\overline{s})$-sets. However this is not a monodromy action as it was the case for topological coverings.
\end{remark}

\noindent In order to understand the \'etale fundamental group of a connected scheme $S$ at a geometric point $\overline{s}$, we need to understand the automorphisms of the fibre functor. 
And just as it was the case for topological coverings (a key fact in proving Theorem \ref{covering_reformulation_1} and Theorem \ref{covering_reformulation_2}), a wise step would be to find a representation of the fibre functor. 
Awkwardly enough, the functor $\Fib_{\overline{s}}$ is not representable.
Which means there will not be a universal finite \'etale covering.
To address this issue, we work in a more general categorical setting, and introduce the weaker notion of pro-representability.
Many of the following proofs and concepts are inspired by \cite{fundamental_group_of_schemes_thesis}.

\begin{definition}
    Let $S$ be a connected scheme and $\overline{s}$ a geometric point. We define the cofiltered category $\mathcal{J}$ by 
    \begin{itemize}
        \item[$a)$] Objects of $\mathcal{J}$ are pairs $(Y_i,\overline{y}_i)$ where the $Y_i \in \Fet_S$ and $\overline{y}_i$ is an element of $\Fib_{\overline{s}}(Y_i)$ for any $Y_i \in \Fet_S$. We denote objects of $\mathcal{J}$ by $Y_i$ instead of $(Y_i,\overline{y}_i)$ when possible
        \item [$b)$] Morphisms of $\mathcal{J}$ are maps $f_{ij}: (Y_j,\overline{y}_j)\rightarrow (Y_i,\overline{y}_i)$ satisfying $\Fib_{\overline{s}}(f_{ij})(\overline{y}_j) = \overline{y}_i$. And we denote morphisms in $\mathcal{J}$ by the partial order $i\leq j$ if such a morphism exist.
    \end{itemize}
\end{definition}

\begin{lemma}
\label{making_Galois_cofinal}
    Let $S$ be a connected scheme and $\overline{s}$ a geometric point.
    Let $f:X\rightarrow S$ be in $\Fet_S$ with a fixed geometric point $\overline{x}\in \Fib_{\overline{s}}(X)$.
    Then there exists a unique finite \'etale cover $Y_k\in\mathcal{J}$ that is Galois over $S$ with a morphism $p:Y_k\rightarrow X$ such that $\Fib_{\overline{s}}(p)(\overline{y}_k) =\overline{x}$.
\end{lemma}
\begin{proof}
    Let $Z$ be the connected component $\overline{f}:Z\rightarrow X$ that lies over $\overline{x}$. 
    From Proposition \ref{Galois_embed_etale_cover}, we know there exists a Galois closure $Y_k \in \mathcal{J}$ with $p:Y_k\rightarrow Z$ such that $p_k$ factors through $\overline{f}$ by $\overline{f}\circ p$.
    $p$ being a Galois cover, it is connected and surjective, meaning that we can find an element $\overline{y}\in \Fib_{\overline{s}}(Y_k)$ such that $\Fib_{\overline{s}}(\overline{f}\circ p)(\overline{y}) = \overline{x}$.
    But again, since $\Aut(Y_k|S)$ is transitive, we can find $\sigma\in \Aut(Y_k|S)$ such that $\sigma(\overline{y}) = \overline{y}_k$ from the pair $(Y_k,\overline{y}_k)$ in the cofiltered category $\mathcal{J}$. 
    And by Proposition \ref{proporty_mono_etale_cover}, this is unique, giving us $\Fib_{\overline{s}}(\overline{f}\circ p \circ \sigma)(\overline{p}_k) = \overline{x}$ as we wanted.
\end{proof}

\noindent Note, we have proven above that we can find a Galois covering $Y_k$ in $\mathcal{J}$ with a morphism $p$ such that $\overline{x}$ is the pull-back of a distinguished point $\overline{y}_k$ in the geometric fibre of $Y_k$ by $p$.

\begin{lemma}
\label{direct_system_fibre}
    There is a covariant functor $F:\mathcal{J}\rightarrow \Fet_S$ with direct system $(Y_i,f_{ij})$ for $Y_i$ objects $\mathcal{J}$ and $f_{ij}$ morphisms in $\mathcal{J}$ which has the following direct limit for every $X\in \Fet_S$ : $$ \colim\:F = \varinjlim_{Y_i\in\mathcal{J}}\Hom(Y_i,X).$$
\end{lemma}
\begin{proof}
    Since $\mathcal{J}$ is a cofiltered category its underlying set is clearly partially ordered.
    We now show that it is also directed. 
    Let $Y_i,Y_j\in \mathcal{J}$, then $f:Y_i\times_S Y_j \rightarrow S$ is finite \'etale. 
    By applying the previous lemma to $f$, we can find an object $Y_k\in \mathcal{J}$ with a morphism $p:Y_k\rightarrow Y_i\times_SY_j$ such that $\Fib_{\overline{s}}(p)(\overline{y}_k) = (\overline{y}_i,\overline{y}_j)$. 
    We just have to composing $p$ with the projections of the fibre product onto the respective factors to get that $i\leq k$ and $j\leq k$.
\end{proof}

\begin{lemma}
\label{inverse_system_Galois}
    Let $\mathcal{K}$ be the subcategory of $\mathcal{J}$ restricting to the objects $Y_i$ that are Galois covers. $\mathcal{K}$ is then a cofinal set and we have 
    $$\varinjlim_{Y_i\in\mathcal{J}}\Hom(Y_i,X) \cong \varinjlim_{Y_i\in \mathcal{K}} \Hom(Y_i,X).$$
    The transition maps of the rightmost colimit $f_{ij}: Y_j\rightarrow Y_i$ are unique if they exist .
\end{lemma}
\begin{proof}
    For any $Y_i\in \mathcal{J}$, by Lemma \ref{making_Galois_cofinal}, we can find $Y_k\in \mathcal{K}$ such that $Y_i\leq Y_k$. 
    Uniqueness of the transition maps follows from Proposition \ref{proporty_mono_etale_cover}. 
    And the isomorphism of colimits follows from Theorem \ref{cofinal_identic_colim}.
\end{proof}

\noindent Note that in Lemmas \ref{direct_system_fibre} and \ref{inverse_system_Galois}, the index sets are actually implied. 
We can construct them by numbering the objects in the categories $\mathcal{J}$ and $\mathcal{K}$ respectively, so in a sense we can write them as $|\mathcal{J}|$ and $|\mathcal{K}|$.
Finally, we can prove what we wanted.

\begin{definition}
    A functor $F: \mathcal{C} \rightarrow Set$ is said to be \textbf{pro-representable} if there exist an inverse system $(Y_\alpha, f_{\alpha\beta})$ of objects of $\mathcal{C}$ indexed by a partially ordered set $I$, and a functorial isomorphism $\varinjlim \Hom (Y_\alpha, X) \cong F(X)$ for each $X\in \mathcal{C}$.
\end{definition}

\begin{proposition}
\label{fib_pro_representable}
    Let $S$ be a connected scheme and $\overline{s}:\Spec\:\Lambda \rightarrow S$ a geometric point. The functor $\Fib_{\overline{s}}$ is pro-representable, that is we have the following natural isomorphisms for every $X\in \Fet_S$.

    $$\varinjlim_{Y_k\in\mathcal{K}} \Hom(Y_k,X) \cong \Fib_{\overline{s}}(X).$$
\end{proposition}
\begin{proof}
    We have already shown in Lemmas \ref{direct_system_fibre} and \ref{inverse_system_Galois} that the direct limit on the left hand side is well-defined, it is only left to show the functorial isomorphism.
    For every $Y_k\in\mathcal{K}$ and $X\in \Fet_S$ consider the map of sets defined by 
    $$\Hom(Y_k, X) \rightarrow \Fib_{\overline{s}}(X) \: : \:\phi\mapsto\Fib_{\overline{s}}(\phi)(\overline{y}_k)$$
    
    \noindent And these maps are compatible with the transition maps of the inverse system by composition, yielding a functorial map $\Phi_X :\varinjlim\Hom(Y_k, X) \rightarrow \Fib_{\overline{s}}(X)$, sending $[\phi]$ to $\Fib_{\overline{s}}(\phi)(\overline{y}_k)$ where the representative of $[\phi]$ is $\phi:Y_k\rightarrow X$.\\
    The maps $\Hom(Y_k,X) \rightarrow \Fib_{\overline{s}}(X)$ are injective, indeed, assume $f,g:Y_k\rightarrow X$ are to Galois covers such that $\Fib_{\overline{s}}(f)(\overline{y}_k) = \Fib_{\overline{s}}(g)(\overline{y}_k)$, then by Lemma \ref{making_Galois_cofinal}, $f=g$.
    So the maps $\Phi_X$ are again injective for all $X\in\Fet_S$.
    And $\Phi_X$ is also surjective for all finite \'etale covering $X$ of $S$. Given $\overline{x}\in\Fib_{\overline{s}}(X)$, we can find an object $Y_k\in \mathcal{K}$ and a morphism $p:Y_k\rightarrow X$ over $S$ such that $\Fib_{\overline{s}}(p)(\overline{y}_k) = \overline{x}$, implying surjectivity.
    Finally if $f:X\rightarrow Y$ is a morphism in $\Fet_S$ then the induced maps defined by $G_k(f):\Hom(Y_k,X)\rightarrow\Hom(Y_k,Y)$ for all $Y_k\in \mathcal{K}$ gives a map $G(f)$ between the direct limits and the diagram below commutes :
    \[
    \begin{tikzcd}
        \varinjlim \Hom(Y_k,X) \arrow[r,"\Phi_X"] \arrow[d, "G(f)"'] & \Fib_{\overline{s}}(X) \arrow[d,"\Fib_{\overline{s}}(f)"] \\
        \varinjlim \Hom(Y_k,Y) \arrow[r,"\Phi_Y"'] &  \Fib_{\overline{s}}(Y)
    \end{tikzcd}
    \]
    The maps $\Phi_X$ of sets are thus functorial isomorphisms, so the functor $\Fib_{\overline{s}}$ is pro-representable. 
\end{proof}

\begin{theorem}
    The \'etale fundamental group of a connected scheme $S$ with geometric point $\overline{s}:\Spec\:\Lambda\rightarrow S$ is a profinite group with isomorphism 
    $$\pi_1^e(S,\overline{s}) \cong \varprojlim_{Y_k\in\mathcal{K}} \Aut(Y_k|S)^{op}.$$
    And its action on $\Fib_{\overline{s}}(X)$ is continuous for every $X\in\Fet_S$.
\end{theorem}
\begin{proof}
    By Corollary \ref{corollary_property_morphism_covers}, the groups $\Aut(Y_k|S)$ are all finite for all $Y_k\in \mathcal{K}$.
    We now construct maps $\varphi_{kl}$ such that $(\Aut(Y_k|S)^{op}, \varphi_{kl})$ forms an inverse system.
    Consider the map $f_{kl}: (Y_k,\overline{y}_k)\rightarrow(Y_l,\overline{y}_l)$ in $\mathcal{K}$ and recall that they are unique such that $\Fib_{\overline{s}}(f_{kl})(\overline{y}_l) = \overline{y}_k$ (unicity follows by Lemma \ref{making_Galois_cofinal} and the classification of finite \'etale covers).
    We have the bijections $\psi_k : \Aut(Y_k|S) \rightarrow \Fib_{\overline{s}}(Y_k)$ given by $\phi\mapsto \Fib_{\overline{s}}(\phi)(\overline{y}_k)$ 
    (indeed for $\Fib_{\overline{s}}(\phi)(\overline{y}_k) = \overline{y} \in \Fib_{\overline{s}}(Y_k)$ such a $\phi$ is unique by Lemma \ref{making_Galois_cofinal} and as in the proof of the previous proposition, $\psi_k$ is surjective).
    At last, we define the maps $\varphi_{kl} : \Aut(Y_l|S)\rightarrow\Aut(Y_k|S)$ as compositions following the below diagram
    \[
    \begin{tikzcd}
        \Aut(Y_l|S) \arrow[r,"\psi_l"] \arrow[d, dashed, "\varphi_{kl}"'] & \Fib_{\overline{s}}(Y_l) \arrow[d,"\Fib_{\overline{s}}(f_{lk})"] \\
        \Aut(Y_k|S) &  \Fib_{\overline{s}}(Y_k) \arrow[l, "\psi_k^{-1}"].
    \end{tikzcd}
    \]
    For $\phi\in\Aut(Y_l|S)$, the element $\varphi_{kl}(\phi)$ is determined by the automorphisms of $Y_k$ that satisfies the following commutative diagram,
    following the below diagram,
    \[
    \begin{tikzcd}
        Y_l  \arrow[d, "\phi"'] \arrow[r,"f_{lk}"] & Y_k \arrow[d,"\varphi_{kl}(\phi)"] \\
        Y_l \arrow[r,"f_{lk}"'] &  Y_k.
    \end{tikzcd}
    \]
    Now these elements are uniquely determined since the groups of automorphisms act transitively on the fibres meaning that  $\Fib_{\overline{s}}(\varphi_{kl}(\phi))(\overline{y}_k) = \Fib_{\overline{s}}(f_{lk}\circ\phi)(\overline{y}_k)$. The transition maps $\varphi_{kl}$ are clearly group homomorphisms as well as being surjective (by the above diagram we have that each $\varphi_{kl}(\phi)\circ f_{lk} = f_{lk}\circ \phi$).
    Therefore, we have a limit $\varprojlim_{Y_k\in\mathcal{K}} \Aut(Y_k|S)^{op}$.\\
    The contravariant Yoneda embedding $\mathcal{Y}^*: \mathcal{K}\rightarrow \Fun(\mathcal{K},\SET) :Y_k\mapsto \Hom(Y_k,\_)$ then gives a morphism 
    $$\varprojlim_{Y_k\in\mathcal{K}} \Aut(Y_k|S)^{op} \rightarrow \Aut(\varinjlim_{Y_k\in\mathcal{K}}\Hom(Y_k,\_))\cong \Aut(\Fib_{\overline{s}})$$
    which is seen as a bijective group homomorphism by the uniqueness of the transition maps for both (co)limits.

    \noindent Finally, for $\overline{x}\in \Fib_{\overline{s}}(X)$ recall from Lemma \ref{making_Galois_cofinal} that there is a Galois cover $Y_k\in\mathcal{K}$ such that $\Fib_{\overline{s}}(p)(\overline{y}_k)=\overline{x}$ (in that case we often say that the geometric point $\overline{x}$ is \textit{dominated} by a Galois cover $Y_k$). 
    Notice that there is a $1$ to $1$ correspondence between the action of $\pi_1^e(S,\overline{s})$ on $\Fib_{\overline{s}}$,  $\pi_1^e(S,\overline{s})\times \Fib_{\overline{s}}(X) \rightarrow\Fib_{\overline{s}}(X)$ and the group homomorphism $\pi_1^e(S,\overline{s})\rightarrow \Perm(\Fib_{\overline{s}}(X))$.
    This correspondence is quite natural, the group axioms make any action by an element of the group onto the finite set into a permutation of that finite set and the converse follows similarly by the group homomorphism axioms.
    And the action of $\pi_1^e(S,\overline{s})$ on $\overline{x}$ is factored through an action from $\Aut(Y_k|S)^{op}$ ($\pi_1^e(S,\overline{s}) \xrightarrow[]{\pr_k} \Aut(Y_k|S)\xrightarrow[]{\alpha}\Perm(\Fib_{\overline{s}}(X))$)
    As we have seen earlier, the first map is continuous because the topology on the profinite group is defined exactly so that each projection homomorphisms are continuous (see Remark \ref{prof_are_top}) and since $\Aut(Y_k|S)$ and $\Perm(\Fib_{\overline{s}}(X))$ are both finite discrete groups, any map between them is continuous, which concludes our proof.
\end{proof}

\noindent Having shown that the \'etale fundamental group is profinite, we will prove that there is again a classification of finite \'etale covers ‘‘\`a la Grothendieck".

\begin{theorem}
\label{main_theorem_etale_fundamental_group}
     Let $S$ be a connected scheme, and $\overline{s}:\Spec\: \Lambda \rightarrow S$ a geometric point and $G = \pi_1^e(S,\overline{s})$ the \'etale fundamental group.
     The functor $\Fib_{\overline{s}}$ induces an equivalence of category between $\Fet_S$ and finite $G$-sets with a continuous action. 
     Moreover, connected covers correspond to sets with a transitive $G$-action, 
     and Galois covers correspond to finite quotients of $G$.
\end{theorem}
\begin{proof}
    Our goal will be to prove that the functor $\Fib_{\overline{s}}$ is fully faithful and essentially surjective similar to the proof of theorem \ref{Grot_reformulation_Galois}.
    We first show fully faithfulness, so let $f,g:X\rightarrow Y$ be two morphisms of finite \'etale covers (in $\Fet_S$) such that $\Fib_{\overline{s}}(f) = \Fib_{\overline{s}}(g)$.
    Take $i: Z\hookrightarrow X$ a connected component, we have $\Fib_{\overline{s}}(f\circ i)= \Fib_{\overline{s}}(g\circ i)$ and by Proposition \ref{proporty_mono_etale_cover}, we have that $f\circ i = g\circ i$. We do the same for all connected components to obtain that $f = g$ which shows faithfulness.\\
    To complete the first step of this proof, we need to show that $\Fib_{\overline{s}}$ is full. For that, consider $X$ and $Y$ objects in $\Fet_S$ and $F: \Fib_{\overline{s}}(X)\rightarrow \Fib_{\overline{s}}(Y)$ a $G$-equivariant map between sets. 
    We want to construct a morphism in $\Fet_S$ which gets sent to $F$ by the fibre functor.
    The problem can be reduced to $X$ and $Y$ connected, indeed, let $X = \bigsqcup X_i$ be a decomposition of $X$ into connected components, then $\Hom(X,Y) \cong \bigsqcup \Hom(X_i,Y)$ and the same is true for $\Fib_{\overline{s}}$, indeed
    $$(X_1\sqcup X_2) \times_S \Spec\:\Lambda = (X_1\times_S\Spec\:\Lambda) \sqcup(X_2\times_S\Spec\:\Lambda).$$
    In a similar fashion we can assume that $Y$ is connected.\\
    Now, since $\pi_1^e(S,\overline{s})$ acts transitively on both $\Fib_{\overline{s}}(X)$ and $\Fib_{\overline{s}}(Y)$, the $G$-equivariant map $F$ will be determined by the image of $\overline{x}\in \Fib_{\overline{s}}(X)$ denoted by $\overline{y}:= F(\overline{x})$.
    Consider $Z\rightarrow S$ a Galois covering, since finite \'etale morphisms are closed under base change, 
    we can consider $p:Z\rightarrow X\times_S Y$ a morphism of finite \'etale covers that maps $\Fib_{\overline{s}}(p)(\overline{z})$ onto $(\overline{x},\overline{y})$ which exists by Lemma \ref{making_Galois_cofinal} for $\overline{z}\in\Fib_{\overline{s}}(Z)$.
    Let $\pr_X$ and $\pr_Y$ be the projection onto each factor, then $\Fib_{\overline{s}}(\pr_X)(\overline{z}) = \overline{x}$ and $\Fib_{\overline{s}}(\pr_Y)(\overline{z})=\overline{y}$.
    By Theorem \ref{Galois_correpondance_etale_cov}, notice that $X$ is a quotient $Z/\Aut(Z|X)$ and an intermediate cover of $Z\rightarrow Y$. 
    Therefore, we have a unique morphism $f:X\rightarrow Y$ (a Galois covering as well) such that the following diagram commutes,
    \[
    \begin{tikzcd}
        Z  \arrow[d, "\pr_X"'] \arrow[r,"\pr_Y"] & Y \\
        X \arrow[ur,"f"'] 
    \end{tikzcd}
    \]
Choosing an element $\overline{w}\in\Fib_{\overline{s}}(Z)$ with $\Fib_{\overline{s}}(\pr_Y)(\overline{w}) =\Fib_{\overline{s}}(\pr_X\circ f)(\overline{z})$ and an automorphism $\sigma\in\Aut(Z|S)$ that maps $\Fib_{\overline{s}}(\sigma)(\overline{z}) = \overline{w}$ then by Proposition \ref{proporty_mono_etale_cover}, we have that $\Fib_{\overline{s}}(f)(\overline{x}) = \overline{y}$ and thus that $\Fib_{\overline{s}}(f) = F$.\\
We are now going to show essential surjectivity.
Suppose $R$ is a finite set with a continuous $G$-action.
As above, we may assume that the action of $G$ on $R$ is transitive by decomposing $R$ into finitely many orbits and considering them separately since the fibre product of schemes preserves disjoint unions.
The action being continuous, we know that the stabiliser $U_r$ of the point $r\in R$ is an open subgroup of $G$ (Lemma \ref{continuous_open_stabiliser}) and by Corollary \ref{open_subgroup_profinite} it is also finite and closed. 
By Theorem \ref{Cameron_1_3}, just as in the Reformulation of Galois Theory $\ref{Grot_reformulation_Galois}$, we have that $R \cong G / U_r$ as a $G$-set.
The stabiliser $U_r$ being open, it also contains subgroups $V$ which are basis of open neighbourhoods of $1$ in $G$ arising as the kernel of the projection $G \xrightarrow[]{\pr_k}\Aut(Y_k|S)^{op}$, moreover these subgroups are normal in $G$ (see Remark \ref{prof_are_top}).\\
Consider $Y_k$ a Galois cover of $S$ as above, as we have seen in previous proofs, the action of $G$ on the fibre $\Fib_{\overline{s}}(Y_k)$ factors through $\Aut(Y_k|S)^{op}$ and of course since $Y_k$ is Galois, $\Aut(Y_k|S)$ acts transitively on $\Fib_{\overline{s}}(Y_k)$. 
Thus $\Fib_{\overline{s}}(Y_k)$ is a transitive $G$-set.
So once again, from Theorem \ref{Cameron_1_3}, we have $\Fib_{\overline{s}}(Y_k) \cong G/H$ for $H$ the stabiliser of a point in $\Fib_{\overline{s}}(Y_k)$.
But the stabiliser of a point of the fibre is precisely the subgroup $V$ above. 
Indeed, an element $g\in G$ acts trivially on a point of $\Fib_{\overline{s}}(Y_k)$ if and only if the image in the quotient group $\Aut(Y_k|S)^{op}$ is trivial.
That condition is precisely $g\in \ker(\pr_k)= V$, so $V = H$.\\
Now denote by $\overline{U}$ the image of $U_r$ by the group homomorphism $\pr_k$, which is a finite subgroup of $\Aut(Y_k|S)^{op}$ since $\pr_k$ is continuous as well, and notice that by the first isomorphism theorem, we also have that $\overline{U}\cong U_r /V$.
As in Proposition \ref{quotient_finite_etale_cover}, we construct the finite \'etale cover $X$ as the quotient of $Y_k$ by the action of $\overline{U}$, then we have 
$$\Fib_{\overline{s}}(X) = \Fib_{\overline{s}}(Y_K/\overline{U}) = \Fib_{\overline{s}}(Y_k)/ \overline{U} \cong (\pi_1^e(S,\overline{s})/ V) / (U_r/V) \cong \pi_1^e(S,\overline{s}) / U_r$$
concluding our proof.
\end{proof}

\noindent There is an interesting way of seeing sets equipped with a certain action that we describe now.

\begin{remark}
    Recall that a permutation representation of a group $G$ on a finite set $X$ is a homomorphism $\rho:G\rightarrow \Sym(X)$.
    There is an equivalence of categories between 
    finite sets equipped with a continuous action of a group $G$
    and finite continuous permutation representations of $G$.\\
    Indeed, in our case, for $S$ a connected scheme and $\overline{s}$ a geometric point of $S$, given a continuous action of $\pi^e_1(S,\overline{s})$ onto a finite set $X$, 
    we define a homomorphism $\rho : \pi^e_1(S,\overline{s}) \rightarrow \Sym(X)$ by $\rho(g)(x):= g\cdot x$.
    Then $\rho:$ is a group homomorphism, and it is a continuous map when $\Sym(X)$ is equipped with the discrete topology. 
    It is thus a finite continuous permutation representation.\\
    On the other hand, 
    given a finite continuous permutation representation $\rho : \pi^e_1(S,\overline{s}) \rightarrow \Sym(X)$, similarly as above we clearly have that $\pi^e_1(S,\overline{s})$ acts continuously on $X$ via the operation defined by $g\cdot x := \rho(g)(x)$, making it a continuous finite $\pi^e_1(S,\overline{s})$-set.\\
    Both constructions being literally the same data viewed differently, composing them gives back exactly what we started with, and the same is true for morphisms of the two categories. 
    Therefore, we have an equivalence of categories described below.
\end{remark}

\begin{corollary}
\label{corollary_representation}
    For $S$ a connected scheme with a geometric point $\overline{s}$, 
    we have an equivalence of categories between finite \'etale covers of $S$ 
    and finite continuous permutations representation of its \'etale fundamental group
\end{corollary}

\noindent Our goal was to reconcile the \'etale fundamental group with its topological counterpart. 
These results align perfectly with the classical properties of topological fundamental groups, confirming that our definition of $\pi_1^e$ is the appropriate algebraic analogue.
We can already compute a few examples.

\begin{example}
\label{fundamental_group_over_field}
    Let $S=\Spec\:k$ for $k$ a field and $\overline{s}:\Spec\Lambda \rightarrow S$ a geometric point with $k_s\subseteq \Lambda$.
    From Example \ref{example_etale_over_field}, connected finite \'etale covers are spectrums of separable field extensions $L_i$ of $k$.
    The fibre functor sends $\Spec\:L_i$ to the underlying set of $\Spec\:(L_i\otimes_k\Lambda)$.
    The image of the fibre functor is thus a finite set indexed by the $k$-algebra homomorphisms $L\rightarrow \Lambda$.
    Finally, since the image of each such homomorphism lies in the separable closure $k_s$ of $k$ in $\Lambda$, we have that $\Fib_{\overline{s}}(X) = \Hom_k(L,k_s)$ for all $X = \Spec\:L$, so by classical Galois theory, the \'etale fundamental group is the absolute Galois group $\pi_1^{e}(S,\overline{s}) = \Gal(k_s|k)$. 
    In particular, $\pi^e_1(S,\overline{x})$ is trivial for $k$ separably closed.
\end{example}

\begin{remark}
    Using the above example, we can justify our initial assumption: that the fibre functor is not representable. 
    We may think that if $\Fib_{\overline{s}}$ over $\Spec\: k$ were representable then it would be by $\Spec \: k_s$, although, $k_s$ is not a finite \'etale $k$-algebra. 
    So by Definition \ref{(pro)_representable} even over $\Spec\:k$ the fibre functor cannot be representable.
\end{remark}

\begin{example}
    Let $\mathbb{F}_q$ be a finite field with $q$ elements and denote $\overline{\mathbb{F}}_q$ its closure.
    From Galois theory we know that each extension of $\mathbb{F}_q$ is of the form $\mathbb{F}_{q^n}$ and that they are all Galois with $\Gal(\mathbb{F}_{q^n}|\mathbb{F}_q) \cong \mathbb{Z}/n\mathbb{Z}$ generated by the Frobenius morphism $x\mapsto x^q$ (see \cite{Lang_algebra}, Chapter $V.\:5$).\\
    By the previous example, we know that for fields, $\pi^e_1(\Spec\: \mathbb{F}_q)$ is isomorphic to the absolute Galois group  $\Gal(\overline{\mathbb{F}}_q)$ which is itself the profinite completion of $\mathbb{Z}/n\mathbb{Z}$ by Proposition \ref{Galois_groups_are_profinite}.\\
    Hence $\pi^e_1(\Spec\: \mathbb{F}_q) \cong \widehat{\mathbb{Z}}$.
\end{example}

\noindent These first two examples  already show us that for an algebraic space as ‘‘simple'' as a point, the \'etale fundamental group, while quite enlightening, can be far from being trivial.\\

\subsection{Comparison Theorems}
\label{comparaison_theorem}

In this section, our goal is to convince the reader that our definition of \'etale fundamental groups does make sense. 
We do this by comparing \'etale fundamental group of schemes over $\mathbb{C}$ with the topological fundamental group of complex analytic spaces.

\noindent A first issue is that the theorems used to compare such objects are very elaborate. 
In this thesis, we present no proof of these renowned theorems, but instead try to motivate and guide their usage. 
It is important to know that in the following we will frequently mention concepts or ideas arising from Riemann surfaces and complex analytic spaces which in themself are already very rich theories and certainly deserve another presentation like this one.
An interesting comment we can already make about (complex) analytic spaces is that they are topological spaces, which can usually be embedded with very nice topologies, this is important because it means that our previous discussion on topological covers is totally applicable here and it will not result in an irrelevant study as it was the case for schemes.\\

\noindent Before delving into deeper explanations, recall that we have presented in Section \ref{analytification} the \textit{analytification functor}, sending objects from 
the category of algebraic spaces (such as affine varieties, projective curves or schemes) to  objects of the category of analytic spaces (such as manifolds, Riemann surfaces or complex manifolds).



\noindent One of the first comparison theorems available is the following result by Chow from $1949$ \cite{Chow}.
\begin{theorem}[Chow]
    Let $\chi$ be a compact analytic subspace of the complex manifold $\mathbb{P}^n_{\mathbb{C}}$. 
    Then there is a subscheme $X\subset \mathbb{P}^n$ with $X^{an} = \chi$.
    Furthermore, holomorphic maps between compact analytic subspaces of $\mathbb{P}^n_{\mathbb{C}}$ induce regular rational morphisms between subschemes.
\end{theorem}
\begin{proof}
    See \cite{Chow} Theorem $V$.
\end{proof}

\noindent This theorem is often considered as the first bridge between complex analytical geometry and algebraic geometry. 
The next collection of results we mention are found in the more than influential paper written in French by J.P Serre in $1955$ called \textit{G\'eom\'etrie Analytique et G\'eom\'etrie Alg\'ebirque} or GAGA for short \cite{Serre}.
Which as the name suggests contains results deepening the connection between analytical spaces and algebraic ones. 
In particular, Serre showed that there is an equivalence between the category of coherent sheaves on $X$ and the category of coherent analytic sheaves on $X^{an}$ for a $X$ a projective variety over $\mathbb{C}$. 
From his results, he was able to prove Chow's theorem as a consequence of this, GAGA can now provide the backbone for comparing covers of schemes and covers of analytical spaces. A few of his results are listed in Appendix \ref{analytification}\\

\noindent For simplicity, from now on we will focus our study on a subcategory of complex analytic spaces, namely that of compact Riemann surfaces. 
\subsubsection{Overview of Riemann Surfaces}

\begin{definition}
    A \textbf{Riemann surfaces} $X$ is a $1$-dimensional complex manifold. 
    That is, $X$ is a Hausdorff topological space with an open covering $\{U_i\mid i\in I\}$, together with maps $f_i:U_i\rightarrow \mathbb{C}$ called \textit{complex charts} mapping $U_i$ homeomorphically onto an open subset $V\subset \mathbb{C}$ such that for each pair $i,j\in I$, the map $f_i\circ f_j^{-1} : f_j(U_i\cap U_j) \rightarrow \mathbb{C}$ is holomorphic. This is called a \textit{complex atlas} on $X$ and two complex atlases $\{U_i\mid i\in I\}$ and $\{ U'_i\mid i\in I\}$ on $X$ are equivalent if their union is also a complex atlas.
\end{definition}

\noindent From this definition, we can already see that Riemann surfaces inherits the topological property of locally compactness from $\mathbb{C}$.

\begin{definition}
    A \textbf{holomorphic (or analytic) map} $\varphi$ between two Riemann surfaces $X$ and $Y$ is a continuous map such for each pair of open subsets $U\subset X$, $V\subset Y$ satisfying $f(U)\subset V$ and complex charts $f:U\rightarrow \mathbb{C}$, $g:V\rightarrow \mathbb{C}$, the functions $g\circ\varphi\circ f^{-1}: g(V)\rightarrow \mathbb{C}$ are holomorphic.
\end{definition}

\noindent For more on this topic, we refer to the book \textit{Riemann Surfaces} \cite{Donaldson} of S.Donaldson and we suppose below that basic results of Riemann surfaces and complex analysis are known.
We will also assume that holomorphic maps between Riemann surfaces are non-constant and that the Riemann surfaces we consider are connected.

\begin{remark}
    Since compact Riemann surfaces are obviously topological spaces, one can apply our theory of topological covering spaces to them.
    But notice that we obtain much more information on these covers when they carry the additional structure of a complex manifold. 
    In particular, we can consider a whole new kind of covers for compact Riemann surfaces.
\end{remark}

\noindent The following results are well-known facts about holomorphic maps between Riemann surfaces.

\begin{proposition}
\label{holo_are_locally_ramified}
    Let $\varphi:Y\rightarrow X$ be a holomorphic map of Riemann surfaces. For each point $y\in Y$, there is a unique integer $k=k_y\ge 1$ such that we can find complex charts around $y$ in $Y$ and $\varphi(y)$ in $X$ in which $\varphi$ is represented by the complex map $z\mapsto z^k$.
\end{proposition}
\begin{proof}
    See Proposition $3$ of Chapter $4$ in \cite{Donaldson}.
\end{proof}

\noindent As a natural consequence, one sees that holomorphic maps between Riemann surfaces are open, indeed the maps $z\mapsto z^k$ is open.

\begin{definition}
    The above integer $k_y$ is called the \textbf{ramification index} of the map $\varphi$ at $y$. And the points $y$ with $k_y> 1$ are called \textbf{branch points}, we call the set of branch points of $\varphi$  \textbf{ramification locus} and denoted it by $B_{\varphi}$
\end{definition}

\begin{proposition}
    Let $\varphi: Y\rightarrow Y$ be a holomorphic map of Riemann surfaces. Then the set of branch points $B_{\varphi}$ is finite.
\end{proposition}
\begin{proof}
    See Proposition $4$ of Chapter $4$ from \cite{Donaldson}.
\end{proof}

\noindent We now restrict our attention to proper maps, recall that a map $f : S\rightarrow T$ between topological spaces $S,T$ is said to be \textbf{proper} if for any compact set $K\subset T$, the preimage $f^{-1}(K)$ is also compact.
Note that any map between compact topological sets is proper, and for $f:S\rightarrow T$ a proper map of topological spaces, its fibre at any point $t\in T$ is finite (Proposition $4$ Chapter $4$ \cite{Donaldson}).

\begin{lemma}
\label{proper_map_closed}
    A proper map $f:Y\rightarrow X$ to a compactly generated Hausdorff space is a closed map (a space $X$ is called \textit{compactly generated} if any subset $A$ of $X$ is closed when $A\cap K$ is closed in $K$ for each compact $K\subset X$).
\end{lemma}
\begin{proof}
    Let $C\subset Y$ be closed, and let $K$ be a compact subspace of $X$.
    Then $f^{-1}(K)$ is compact and so is $f^{-1}(K) \cap C$ that we denote as $B$.
    Then $f(B) = K\cap f(C)$ which is compact, 
    indeed $f(C\cap f^{-1}(K)) \subseteq K \cap f(C)$ and 
    suppose $y\in f(C)\cap K$, there exists $c\in C$ with $y= f(c)$ and for $y\in K$, we have that $c\in f^{-1}(K)$.
    And as $X$ is Hausdorff, $f(B)$ is closed. Finally by compact generation, $f(C)$ is closed in $Y$.
\end{proof}
\noindent Notice that locally compact spaces (in particular Riemann surfaces) are compactly generated, thus we can conclude that the finite set $\varphi(B_{\varphi})$ of critical values is also closed.
We can now state the main topological property of proper holomorphic maps.

\begin{proposition}
    Let $X$ be a Riemann surface (connected), and $\varphi:Y\rightarrow X$ a proper holomorphic map.
    The map $\varphi$ is surjective with finite fibres, and its restriction to $Y\:\backslash\:\varphi^{-1}(\varphi(B_{\varphi}))$ is a finite topological cover.
\end{proposition}
\begin{proof}
    Finiteness of fibres follows from the above discussion. And surjectivity follows from the fact that the image $\varphi(Y)$ is open and closed (proper maps are open and closed by Lemma \ref{proper_map_closed}) and that $X$ is connected.
    For the last statement, note that by Proposition \ref{holo_are_locally_ramified}, each of the finitely many preimages of $x\in X\:\backslash\: \varphi(B_{\varphi})$ has an open neighbourhood mapping homeomorphically onto some open neighbourhood of $x$ ; the intersection of these is a distinguished open neighbourhood of $x$ as in the definition of cover.
\end{proof}

\begin{definition}
    We call a proper surjective map of locally compact Hausdorff spaces that restricts to a cover outside a discrete closed subset a \textbf{finite branched cover}.
\end{definition}

\noindent We say that a finite branch cover is of degree $d$ or has $d$ many sheets if the covering obtained by restriction is $d$-sheeted. And since we have assumed Riemann surfaces to be connected, by Remark \ref{connected_constant_sheets}, the degree of finite branch covers is also constant. In that regard, for $X$ a Riemann surface, studying covering spaces above $X\backslash B_{\varphi}$ is done similarly as in Chapter $1$. 
We thus already know how to compute universal covers, automorphism groups of covers and fundamental groups of Riemann surfaces.\\

\noindent However, it is interesting to discuss that in the case of compact Riemann surfaces, fundamental groups are much more restricted. 
An interesting way to understand this is by looking at the classification of compact Riemann surfaces,
\begin{theorem}
    Every compact Riemann surface is homeomorphic to a torus with $g$ holes
\end{theorem}
\begin{proof}
    See $\cite{Donaldson}$ Theorem $3$ in Chapter $6$.
\end{proof}

\begin{remark}
The takeaway here is that one can construct a torus by identifying opposite sides of a square, with the same orientation. This generalises for tori with $g$ holes in the following way :\\
Take a regular $4g$-gon and label its sides clockwise by $b,a,b,a,d,c,d,c, ...$, considering the first two letters of each quadruplet with clockwise orientation and the last two having counter-clockwise orientation. 
By identifying each pair of same letters, respecting the orientation, one obtains a sphere with $g$ handles, and the sides of each quadruplet get mapped to closed paths all going through a common point $x$. 
The following picture from \cite{Hatcher} page $5$ is much more intuitive:
\begin{center}
    \includegraphics[scale=0.25]{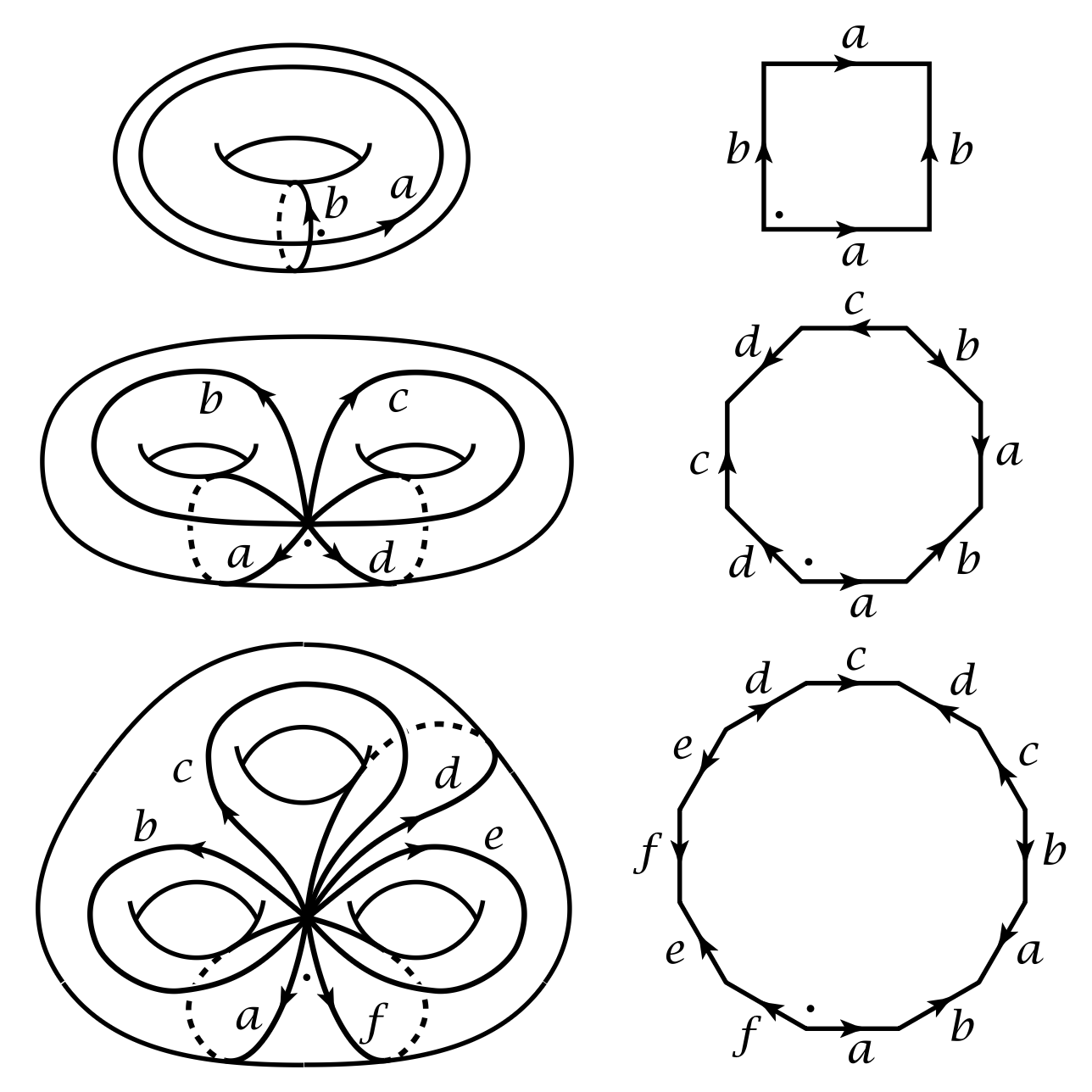}
\end{center}
\end{remark}

\noindent The above representation of tori with $g$ holes, together with the classification of compact Riemann surfaces with $g$ holes allows us to compute the fundamental group of a Riemann surface of genus $g$. 
We conveniently take the base point $x$ to be where paths coming from the first quadruplet $a,b,a,b$ meet. Then
\begin{theorem}
    The fundamental group of a compact Riemann surface $X$ of genus $g$ has a finite presentation of the form $\pi_1(X,x) = \langle a,b,c,d,...\mid [a,b][c,d]... = 1\rangle$, where the brackets $[a,b]$ denote the commutators $aba^{-1}b^{-1}$.
\end{theorem}
\begin{proof}
    The proof is an application of Van Kampen's theorem to cell complexes, see \cite{Hatcher} Chapter $1.2$
\end{proof}

\subsubsection{Main Results}

\noindent Next, we touch upon the main tool for this section, namely, non-singular projective curves over $\mathbb{C}$ and compact Riemann surfaces are equivalent. 
One direction is easier than the other.

\begin{construction}
    Let $X\subset \mathbb{CP}^n$ be a smooth projective curve over $\mathbb{C}$ given by $r$ homogeneous polynomials (with $r\geq n-1$).
    Consider $\{U_i \mid i\in I\}$ an affine covering of $\mathbb{CP}^n$, then passing to affine charts, the homogeneous polynomials become holomorphic (polynomial) functions $f_1,...,f_r :\mathbb{C}^n\rightarrow \mathbb{C}$.
    Recall that since $X$ is a smooth projective variety of dimension $1$ (so of codimension $n-1$), then its Jacobian at any (smooth) point $p$ of $X$ has full rank $n-1$. 
    So there exists a neighbourhood $V\subset\mathbb{C}^n$ of $p\in X$ and $n-1$ holomorphic functions $f_1,...,f_{n-1}:V\rightarrow \mathbb{C}$ such that the ideal of $X\cap V$ is generated by those $n-1$ functions.
    Let us define a map $F=(f_1,...,f_{n-1}) : V\rightarrow \mathbb{C}^{n-1}$.
    Since $F:\mathbb{C}^n\rightarrow\mathbb{C}^{n-1}$ is holomorphic and its differential at $p$ has rank $n-1$, by the holomorphic implicit function function theorem, near $p$, the zero locus $F^{-1}(0)$ is a complex submanifold of dimension $1$.
    Thus near every point $p\in X$, the set $X$ is biholomorphic to an open subset of $\mathbb{C}$, giving us local holomorphic charts on $X$.\\
    And what is left to check is that these local charts define a global complex structure. 
    Suppose we have two different systems of equations  $F=(f_1,...,f_{n-1})$ and $G=(g_1,...,g_{n-1})$ as above.
    By the implicit function theorem they each give a local parametrisation of $X$, 
    $\varphi : U\subset \mathbb{C}\rightarrow X\: ,\:t\mapsto (t,a_2(t),...,a_n(t))$
    and $\psi : V\subset \mathbb{C}\rightarrow X\: ,\: s\mapsto(s,b_2(s),...,b_n(s))$.
    It is now easy to check that the transition map $\psi^{-1}\circ \varphi : U\rightarrow V$ is holomorphic, because $\psi^{-1}$ is a biholomorphism onto its image (by the implicit function theorem) so its inverse is also holomorphic.
    And on overlaps $U_i\cap U_j$, the transition maps are restrictions of the standard projective coordinate changes $\mathbb{C}^n\rightarrow \mathbb{C}^n$.
    So $X$ naturally has the structure of a Riemann surface. 
    Moreover, since $X$ is embedded in a compact space $\mathbb{CP}^n$ with the usual topology,
    $X$ is actually a compact Riemann surface. 
    And we can identify the compact Riemann surface structure of a smooth projective curve $X$ over $\mathbb{C}$ by denoting it $X^{an}$.
\end{construction}

\noindent The other direction is significantly deeper.
A proof may require the introduction of notion of divisors and a thorough study of meromorphic functions, as well as the Riemann-Roch theorem (Theorem $VI, 1.3$ in \cite{Hartshorne}) and Chow's theorem.\\
A similar reasoning allows us to see smooth projective schemes over $\mathbb{C}$ to be compact complex manifolds.\\

\noindent Returning to branched coverings, we now study when finite branched coverings are \'etale. 
It is important to note that we are now not only considering finite branched coverings between Riemann surface, but as morphism of complex analytical spaces, so we also consider the associated structure sheaf.
And the structure sheaf of a Riemann surface $X$ is the sheaf of holomorphic functions, so for all $U\subset X$ open, $\mathcal{O}_X(U)=\{f:U\rightarrow \mathbb{C}\mid f \text{ holomorphic}\}$.
From complex analysis, one knows that holomorphic functions are locally given by convergent power series, so at any point $x\in U$, the stalk at $x$, $\mathcal{O}_{X,x}$ consists in germs of functions, where each germ corresponds to a convergent power series, at each point $x$, $\mathcal{O}_{X,x} = \mathbb{C}\{z\}$. 
From Proposition \ref{holo_are_locally_ramified}, around any point $x\in X$ there exist local coordinates such that a covering $f:Y\rightarrow X$ is of the form $z\mapsto z^k$, this is now also true for the induced map between the structure sheaves. 
In that regard, on every stalk, we can consider 
$\mathbb{C}\{z\}$ as a module over $\mathbb{C}\{w\}$ via $w = z^k$ (we get $\mathbb{C}\{w\}=\mathbb{C}\{z\}/(z^k-w)$), more precisely, $\mathbb{C}\{z\} = \mathbb{C}\{w\}\oplus z\cdot\mathbb{C}\{w\}\:\oplus\:...\oplus \:z^{k-1}\cdot\mathbb{C}\{w\}$ is a free and finite $\mathbb{C}\{w\}$-module of rank $k$.\\
So what we have just shown is that finite branch coverings are indeed finite flat morphisms.\\

\noindent Finally, we show that a finite branch covering is unramified if and only if it has no branch points.
Let us consider the sheaf of K\"ahler differentials of the Riemann surface $X$ over $\mathbb{C}$ locally. 
As we have seen above, $\mathcal{O}_{X,x} \cong \mathbb{C}\{z\}$ so locally we have $\Omega_{\mathbb{C}\{z\} /\mathbb{C}}\cong \mathbb{C}\{z\}\cdot dz$, the free module of rank $1$. 
Indeed, any derivations satisfies the chain rule, $d(f(z)) = f(z)dz$ and since $X$ is a smooth $1$-dimensional complex manifold, $dz$ is just a local basis.
With the same being true for $Y$,
we can now write $\Omega_X = \mathcal{O}_X\cdot dz$ and $\Omega_Y = \mathcal{O}_Y\cdot dw$.
From Proposition \ref{exact_seq_sheaves_differential}, for $f:Y\rightarrow X$, the exact sequence gives us the following isomorphism : $\Omega_{Y/X} = \Omega_Y / (\im f^*\Omega_X$).\\
So we now compute the pull-back of $\Omega_X$ locally. 
There is a natural map $f^*\Omega_X\rightarrow \Omega_Y$ coming from the universal property of differentials, which concretely gives $dw \mapsto d(f^\#(w))$. Recall that since locally $f(z) = z^k$, then $f^\#(w) = z^k$, so $dw\mapsto d(z^k) = kz^{k-1}dz$ (by the chain rule). Therefore, $f^*\Omega_X \subset \Omega_Y$ is generated by $kz^{k-1}dz$.
So, as above, we can write $f^*\Omega_Y = \mathcal{O}_X\cdot (kz^{k-1}dz)$. 
And we now compute $\Omega_{Y/X}$ in each cases : $k=1$, or $k>1$ ,\\
Case $k=1$ : we have $\Omega_{Y/X} \cong \frac{\mathcal{O}_X\cdot dz}{\mathcal{O}_X\cdot (kz^{k-1}dz)}\cong \frac{\mathcal{O}_X}{(kz^{k-1})}\cdot dz$ and,\\
since $k$ is invertible in $\mathbb{C}$, $\Omega_{Y/X} \cong \frac{\mathcal{O}_X}{(z^{k-1})} \cdot dz \cong \frac{\mathcal{O}_X}{(1)}\cdot dz \cong 0$.\\
Case $k>1$ : we have $\Omega_{Y/X} \cong \frac{\mathcal{O}_X}{(z^{k-1})}\cdot dz$ is supported at $z=0$, indeed, if we look at a point $x$ with $z=0$, on the local ring $\mathcal{O}_{X,x}\cong \mathbb{C}\{z\}$, $z^{k-1}$ is not a unit, so the quotient $\mathbb{C}\{z\}/(z^{k-1})$ is non-zero and in fact it is finite-dimensional over $\mathbb{C}$ with basis $1,z,...,z^{k-2}$. So it is ramified.

\begin{remark}
We have established that (non-constant) finite branched coverings with zero ramification locus (so finite topological covering spaces) are the same as finite \'etale coverings over compact Riemann surfaces. 
And from our previous discussion, we know that there is an equivalence (through the analytification functor) between compact Riemann surfaces and non-singular projective curves over $\mathbb{C}$.
Let $X$ be a projective curve over $\mathbb{C}$, the analytification functor thus, induces an equivalence of the category of finite \'etale covers of $X$ with that of finite topological coverings of $X^{an}$.
\end{remark}

\noindent Since we know that for a projective curve $X$ over $\mathbb{C}$ (also a scheme), the \'etale fundamental group is the automorphism group of the fibre functor $\Fib_{\overline{x}}$ at any geometric point $\overline{x}$ of $X$,
we now show that this fibre functor is actually identical to the topological fibre functor defined in Section \ref{universal_cover_section}, restricted to the finite case, which we now denote $\Fib^{top}_x$.
First, notice that over $\mathbb{C}$, for $Y\in \Fet_X$ the underlying set $\Fib_{\overline{x}}(Y) = Y\times_X\Spec\:\mathbb{C}$ consists of points $y\in Y$ whose image in $X$ must be $x$, together with a map $\kappa(x)\rightarrow \mathbb{C}$ (which is always the identity).
So we get that $|\Fib_{\overline{x}}(Y)| = \{y\in Y \mid f(y)=x\}$ for $f$ the associated finite \'etale morphism (as a morphism of sets). \\
Now, by ‘‘analytifying" the setup, we get a finite topological covering $f^{an}:Y^{an}\rightarrow X^{an}$.
Recall that points of $X^{an}$ are closed points of $X$ and since $Y$ is finite \'etale over $X$, as we have seen, the algebraic fibre of $Y$ is a disjoint union of copies of $\Spec\:\mathbb{C}$, so closed points.
Hence $\Fib^{top}_{(\im x)}(Y^{an}) = \{y\in Y \mid f(y)=x\}$. 
And we have that $\Fib_{\overline{x}}(Y) \cong \Fib^{top}_{(\im x)}(Y^{an})$ as sets.
Taking a morphism of covers $\phi:Y\rightarrow Z$, we get a morphism of sets $Y\times_X\Spec\:\mathbb{C} \rightarrow Z\times_X\rightarrow\Spec\:\mathbb{C}$ which sends $y\mapsto \phi(y)$. On the topological side, we get a morphism of sets $\Fib^{top}_{(\im x)}(Y^{an})\rightarrow \Fib^{top}_{(\im x)} (Z^{an})$ which also sends $y\mapsto \phi(y)$.
And the above identification of sets commutes with such morphisms of covers, yielding a natural isomorphism of the fibre functors.
Now, what would be the automorphism group of the topological fibre functor ?
By fully faithfulness of the Yoneda embedding, and since the topological fibre functor is represented by the universal cover, we get that the automorphisms of the topological fibre functor are precisely elements of $\Aut(\widetilde{X^{an}}|X^{an})$.
As usual, this forms a group under composition and from the reformulation Theorem \ref{covering_reformulation_2}, we actually have that $\Aut(\widetilde{X^{an}}|X^{an}) \cong \widehat{\pi_1(X^{an},x')}$ for $x'\in X^{an}$ the corresponding point of $\overline{x}$. And by chain of isomorphisms, we get that 
$$\pi^e_1(X,\overline{x}) \cong \widehat{\pi_1(X^{an},x')}.$$

\noindent Before ending this section, it is worthwhile pointing out that the equivalence of categories between compact Riemann surfaces and projective curves over $\mathbb{C}$ is deeper. 
Indeed, we have the following three-way equivalence of categories: 
\begin{itemize}
    \item[$a)$] Smooth projective algebraic curves over $\mathbb{C}$ and non-constant algebraic maps,
    \item[$b)$] Compact connected Riemann surfaces and non-constant holomorphic maps, and
    \item[$c)$] Finite field extensions of $\mathbb{C}$ of transcendence degree $1$ and morphisms of extensions of $\mathbb{C}$.
\end{itemize}
with the first two categories being anti-equivalent to the third one. This situation is explained in detail in Chapter $4$ of \cite{Szamuely}.\\
We end this section by mentioning that Grothendieck, in \cite{SGA} (Th\'eor\`em $XII.\:5.1$) proved a generalisation of Riemann's existence theorem:
\begin{theorem}
    Let $X$ be a scheme over $\mathbb{C}$ locally of finite type and $X^{an}$ the analytical space associated to $X$.
    The analytification functor which to any finite \'etale cover $Y$ associates $Y^{an}$ is an equivalence of categories between finite \'etale covers of $X$ and finite topological covers of $X^{an}$
\end{theorem}
\begin{proof}
    See \cite{SGA}, Th\'eor\`em $XII.\:5.1$
\end{proof}
\noindent Which has the following corollary ($XII.\:5.2$ in \cite{SGA}) who generalises our whole discussion to schemes of higher dimensions:
\begin{corollary}
    Let $X$ be a connected scheme over $\mathbb{C}$ locally of finite type, $x'$ a point in $X^{an}$ and $\overline{x}$ its corresponding point in $X$.
    We have the following canonical isomorphism $\pi^e_1(X,\overline{x}) \cong \widehat{\pi_1(X^{an},x')}$.
\end{corollary}

\begin{remark}\textbf{ }
\label{etale_fund_grp_fin_gen}
\begin{itemize}
    \item[$a)$] One might guess that such a result is quite powerful.
    An immediate application is that every smooth projective scheme over $\mathbb{C}$ has a topologically finitely generated \'etale fundamental group.
    Indeed, when $X^{an}$ is a compact complex manifold, its fundamental group is finitely presented (so also finitely generated) and by Remark \ref{finite_generated_prof_grp}, its profinite completion must be topologically finitely generated.
    \item[$b)$] It is also worthwhile mentioning that a similar result holds for the cohomology groups of coherent sheaves and analytic coherent sheaves. 
\end{itemize}
    
\end{remark}

\subsection{Properties of \'Etale Fundamental Groups}
In this section, we discuss a few interesting properties of the \'etale fundamental group. 
Namely, we prove its insensitivity to the change of geometric base points for connected schemes; this should not be surprising as it was already the case for the topological fundamental group.
Then, as the \'etale fundamental group is an object emanating from algebraic geometry, we will study how it behaves under functorial base change.
And finally, we finish this section by providing useful tools for computing \'etale fundamental groups, relying on powerful exact sequences.\\

\noindent Let us first show the following lemma.
\begin{lemma}
\label{fiber_isomorphism}
    Let $S$ be a connected scheme. Give, two geometric points $\overline{s}:\Spec\:\Lambda\rightarrow S$ and $s':\Spec\:\Lambda'\rightarrow S$, there exists an isomorphism of fibre functors $\Fib_{\overline{s}}\xrightarrow[]{\sim}\Fib_{\overline{s}'}$.
\end{lemma}
\begin{proof}
    By proposition \ref{fib_pro_representable}, both fibre functors are pro-representable. And from our discussion on the \'etale fundamental group we know that the representing inverse system have the same index set and objects $Y_k\in \mathcal{K}$ as defined in Lemma \ref{inverse_system_Galois}, only the transition maps may be different. 
    We denote them by $f_{ij}$ and $g_{ij}$ respectively.
    Constructing an isomorphism between the inverse systems $(Y_k,f_{ij})$ and $(Y_k,g_{ij})$, consists in defining a system of automorphisms $\phi_k\in\Aut(Y_k|S)$ mapping $f_{ij}$ onto $g_{ij}$ thus actually proving the lemma (since both functors are pro-representable).\\
    Assume that we have $i\leq j$ in $|\mathcal{K}|$ an automorphism $\phi_j\in \Aut(Y_j|S)$ and consider the fibre points $\overline{y}_i\in \Fib_{\overline{s}}(P_i)$ and $\overline{y}_j\in\Fib_{\overline{s}}(Y_j)$.
    Thus, by construction we have, $\Fib_{\overline{s}}(f_{ij})(\overline{y}_j) = \overline{y}_i$ and let us set $\overline{y}_i':= \Fib_{\overline{s}}(g_{ij})(\phi_j(\overline{y}_j))$.
    Since $Y_i$ is Galois, from Remark \ref{remark_Galois} $b)$, there exist a unique $\phi_i\in\Aut(Y_i|S)$ mapping $\overline{y}_i\mapsto \overline{y}_i'$.
    We can then apply Proposition \ref{proporty_mono_etale_cover} with $\overline{z} = \overline{y}_i$, $h_1= g_{ij}\circ \phi_j$ and $h_2=\phi_i\circ f_{ij}$, which then implies that the following diagram commutes 
    \[
    \begin{tikzcd}
        Y_j \arrow[r,"\phi_j"] \arrow[d,"f_{ij}"'] & Y_j\arrow[d,"g_{ij}"]\\
        Y_i \arrow[r, "\phi_i"'] & Y_i
    \end{tikzcd}
    \]
    Finally, we define a map $\psi_{ij} : \Aut(Y_j|S)\rightarrow \Aut(Y_i|S)$ to be the map sending each $\phi_j\in\Aut(Y_j|S)$ to the $\phi_i\in\Aut(Y_i|S)$ in the above diagram.
    It is a map of finite set (again by Corollary \ref{no_fixed_points}), thus we obtain an inverse system of finite sets : $(\Aut(Y_k|S), \psi_{ij})$ and by Proposition \ref{compact_inverse_system}, the inverse limit of such an inverse system is non-empty.
    Any element of it defines the required isomorphism.
\end{proof}

\begin{proposition}
    Let $S$ be a connected scheme. For any two geometric points $\overline{s}:\Spec\:\Lambda\rightarrow S$ and $\overline{s}':\Spec\:\Lambda'\rightarrow S$, there exists a continuous isomorphism of profinite groups $\pi^e_1(S,\overline{s}')\xrightarrow{\sim} \pi^e_1(S,\overline{s})$.
\end{proposition}
\begin{proof}
    From Lemma \ref{fiber_isomorphism}, we have an isomorphism $\varphi : \Fib_{\overline{s}}\xrightarrow{\sim}\Fib_{\overline{s}'}$ of fibre functors, which induces an isomorphism of their automorphism groups via $\phi\mapsto \varphi^{-1}\circ\phi\circ\varphi$ for $\phi$ an automorphism.
    Continuity of this isomorphism with respect to the profinite group structure follows from our construction in the previous proof. 
    Indeed, in the discrete topology, maps in $(\Aut(Y_k|S),\rho_{ij})$ are continuous, and since they satisfy the compatibility condition with the transition morphisms, they actually yield a continuous mapping between the inverse systems.
\end{proof}

\noindent Our analogy between the \'etale fundamental group and its topological counterpart checks out, on a connected scheme, the \'etale fundamental group is independent of the choice of base point.

\subsubsection{Base Change}

In order to study how the \'etale fundamental group behaves under base change, we need to define the correct notion of base change, such that it keeps track of the appropriate geometric points.
\begin{construction}
    Let $S$ and $S'$ be two connected schemes, equipped with the geometric points $\overline{s}:\Spec\:\Lambda\rightarrow S$ and $\overline{s}':\Spec\:\Lambda \rightarrow S'$ respectively. 
    Now consider a morphism $\phi$ between $S'$ and $S$ over $\Spec\:\Lambda$ : 
    \[
    \begin{tikzcd}[column sep=small]
        S'\arrow[rr,"\phi"] && S\\
        &\Spec\:\Lambda \arrow[ul, "\overline{s}'"]\arrow[ur, "\overline{s}"']
    \end{tikzcd}
    \]
    Then $\phi$ induces a \textit{base change functor} $\Bc_{S,S'} : \Fet_S\rightarrow \Fet_{S'}$ sending $X\in\Fet_S$ to $X\times_SS'\in\Fet_{S'}$, and a morphism of finite \'etale covers $X\rightarrow Y$ to the induced morphism $X\times_SS'\rightarrow Y\times_SS'$ 
    (this makes sense by Proposition \ref{base_change_etale} and the functor conditions are immediately verified).
    The condition $\phi\circ\overline{s}' = \overline{s}$ gives us the following equality, $\Fib_{\overline{s}}= \Fib_{\overline{s}'}\circ \Bc_{S,S'}$ and means that the geometric point $\overline{s}'$ lies above $\overline{s}$.
    As a consequence of this, every automorphism of $\Fib_{\overline{s}}$ is induced by an automorphism of $\Fib_{\overline{s}'}$ via composition with the functor $\Bc_{S,S'}$, and thus we obtain a map 
    $$\phi_* : \pi^e_1(S',\overline{s}')\rightarrow \pi^e_1(S,\overline{s}).$$
    The map $\phi_*$ is actually a continuous homomorphism of profinite groups. 
    The homomorphism part is easily verified since the morphism is defined functorially.
    And it is continuous because $\phi_*$ is compatible with finite quotients of the inverse system. 
    Indeed, any canonical projection $\pi^e_1(S,\overline{s})\rightarrow \Aut(Y_k|S)$ for $Y_k$ a Galois finite \'etale cover factors through a finite quotient of $\pi^e_1(S',\overline{s}')$.
\end{construction}

\noindent What we have now is a morphism $\phi_*$ of base change for \'etale fundamental groups. 
We study when this morphism is injective or surjective and we will fit such morphisms into an exact sequence.

\begin{proposition}
\label{image_base_change_induce}
\textbf{ }
\begin{itemize}
    \item[$a)$] The map $\phi_*$ is trivial if and only if for every connected finite \'etale cover $X\rightarrow S$, the base change $\Bc_{S,S'}(X)$ is a trivial cover.
    \item[$b)$] The map $\phi_*$ is surjective if and only if for every connected finite \'etale cover $X\rightarrow S$, the base change $\Bc_{S,S'}$ is connected as well.
\end{itemize}
\end{proposition}
\begin{proof} \textbf{ }
\begin{itemize}
    \item[$a)$] Suppose the map $\phi_*$ is trivial. 
    We have, for any finite \'etale connected cover $X$ of $S$ a corresponding finite $\pi^e_1(S,\overline{s})$-set, $F:= \Fib_{\overline{s}}(X)$.
    The finite set corresponding to the base change $\Bc_{S,S}(X)$ which is now a finite \'etale cover of $S'$ is again $F$ (by the above construction) and has a $\pi^e_1(S',\overline{s}')$-action, given by pull-back via $\phi_*$.
    Since $\phi_*$ is trivial, for every $x\in F$, $g\cdot x =x$ for all $g\in\pi^e_1(S',\overline{s}')$, so the action of $\pi^e_1(S',\overline{s}')$ on $F$ is trivial.
    And from the main Theorem \ref{main_theorem_etale_fundamental_group}, a finite \'etale cover is trivial if and only if it corresponds to a finite set with trivial action of the fundamental group. \\
    For the other direction, assume $\im \phi_*$ is non-trivial, and choose an open subgroup $U\subset\pi^e_1(S,\overline{s})$, not containing the whole of $\im \phi_*$.
    From Theorem \ref{Galois_correpondance_etale_cov}, we can obtain every finite \'etale cover via correspondence with a coset space $\pi^e_1(S,\overline{s}) / U$.
    Since $U$ does not contain the whole of $\im\phi_*$, the action of $\pi^e_1(S',\overline{s}')$ on $\pi^e_1(S,\overline{s})/U$ induced by $\phi_*$ is non-trivial.
    Thus the finite \'etale cover corresponding to the above coset space pulls back to a non-trivial cover of $S'$.

    \item[$b)$] Suppose $\phi_*$ is surjective. 
    Similarly as above, any connected finite \'etale cover $X$ of $S$, corresponds to a finite set $F$ which has a transitive $\pi^e_1(S,\overline{s})$-action and by base change,that same finite set has a $\pi^e_1(S',\overline{s}')$-action given by pull-back via $\phi_*$.
    But since $\phi_*$ is surjective, $\forall \:x,y\in F$ and for any $g\in\pi^e_1(S,\overline{s})$ such that $g\cdot x = y$ (transitivity), there exists $\sigma\in\pi^e_1(S',\overline{s}')$ such that $\sigma\cdot x=\phi_*(\sigma)\cdot x= g\cdot x = y$. 
    In that last equation, the first equality is obtained via the fibred functors equality (with composition of the base change functor from the above construction). 
    Indeed, $\phi_*(\sigma)$ acts on $\Fib_{\overline{s}}(X)$ via $\sigma$ acting on $\Fib_{\overline{s}'}(\Bc_{S,S'}(X))$.
    So the corresponding base change finite \'etale cover of $S'$ is also connected.\\
    And finally, again suppose that $\phi_*$ is not surjective.
    By compactness of profinite groups and Corollary \ref{closed_profinite_subgroup}, $\im\phi_*$ is a proper closed subgroup of $\pi^e_1(S,\overline{s})$.
    By Lemma \ref{lemma_profinite}, we can find an open subgroup $U\subset \pi^e_1(S,\overline{s})$ containing the closed subgroup $\im\phi_*$.
    Then $\pi^e_1(S',\overline{s}')$ acts trivially on $\pi^e_1(S,\overline{s})/U$ via $\phi_*$, which means that the connected cover corresponding to $\pi^e_1(S,\overline{s})$ pulls back to a trivial cover of $S'$ (different from $S'$) so not connected.
\end{itemize}
\end{proof}

\begin{remark}
\label{subgroup_stabiliser_etale}
    Observe that by Theorem \ref{main_theorem_etale_fundamental_group}, identifying a finite continuous transitive $\pi^e_1(S,\overline{s})$-set $F$ corresponds to choosing a base point $\overline{x}\in F$ (by transitive action). 
    Indeed, connected finite \'etale covers correspond to transitive $\pi^e_1(S,\overline{s})$-set,s and by Theorem \ref{Cameron_1_3}, any transitive $G$-set is isomorphic to the coset space $G/U$ for some subgroup $U\subset G$, where $U$ is actually the stabiliser of an element $\overline{x}\in F$.
    Also, notice here that the stabiliser $U$ is actually an open subgroup since the action is continuous (Lemma \ref{continuous_open_stabiliser}). 
\end{remark}

\begin{proposition}
\label{base_chage_image_last}
    Let $U\subset \pi^e_1(S,\overline{s})$ be an open subgroup, and let $X$ be the connected finite \'etale cover corresponding to the coset $\pi^e_1(S,\overline{s})/U$, let it have the base point $\overline{x}$ as above as well.\\
    The subgroup $U$ contains the image of $\phi_*$ if and only if the finite \'etale cover $\Bc_{S,S'}(X)$ of $S'$ has a section $S'\rightarrow \Bc_{S,S'}(X)$ sending $\overline{s}$ to $\overline{x}$.
\end{proposition}
\begin{proof}
    We have $\im \phi_*\subset U$ if and only if $\pi^e_1(S',\overline{s}')$ acts trivially on $\overline{x}$ via $\phi_*$ (by the previous Proposition \ref{image_base_change_induce}).
    Then this implies that the whole connected component of $\overline{x}$ in $X\times_SS'$ is fixed by $\pi^e_1(S',\overline{s}')$ (by transitivity).
    So this component is actually a one-sheeted trivial covering of $S'$. Thus mapped isomorphically onto $S'$ by $X\times_SS'\rightarrow S'$.
\end{proof}

\noindent We can now prove two propositions about the kernel of the map $\phi_*$.
\begin{proposition}
\label{base_change_kernel}
    Let $U'\subset \pi^e_1(S',\overline{s})$ be an open subgroup, and let $X'$ be the finite \'etale cover of $S'$ corresponding to the coset space $\pi^e_1(S',\overline{s}')/U'$.\\
    The subgroup $U'$ contains the kernel of $\phi_*$ if and only if there exists a finite \'etale cover $X$ of $S$ and a morphism $X_i\rightarrow X'$ over $S$, where $X_i$ is a connected component of $X\times_SS'$.
\end{proposition}
\begin{proof}
    Suppose there exists an \'etale cover $X$ of $S$ and a morphism $X_i\rightarrow X'$ as above, and let $U\subset \pi^e_1(S,\overline{s})$ be an open subgroup such that $X\cong \pi^e_1(S,\overline{s})/U$.
    By choosing an appropriate geometric base point $\overline{x}$ in $\Fib_{\overline{s}'}(X)$, we may identify the component $X_i$ of $X\times_SS'$ with the coset space $\pi^e_1(S',\overline{s}')/U''$ for some open subgroup $U''\subset \pi^e_1(S',\overline{s}')$.
    Importantly, note that $U''$ contains $\ker\phi_*$ since $\ker\phi_*$ stabilises $\overline{x}$ via pull-back of $\phi_*$ (by Remark \ref{subgroup_stabiliser_etale}).
    By hypothesis, there is a morphism $X_i\rightarrow X$, so we have the inclusion $U''\subset U'$, (this is Theorem \ref{Galois_correpondance_etale_cov}) so $U'$ contains $\ker\phi_*$ as well.\\
    Conversely, assume that the open subgroup $U'$ contains $\ker\phi_*$. 
    Since $H=\phi_*(\pi^e_1(S',\overline{s}'))$ is a closed subgroup of $\pi^e_1(S,\overline{s})$ ($\phi_*$ is continuous and so $H$ is compact), and $V'=\phi_*(U')$ is moreover open in $H$ (being compact and of finite index), so we may apply Lemma \ref{lemma_profinite}.
    That way, we find an open subgroup $V\subset \pi^e_1(S,\overline{s})$ with $V\cap H = V'$, which gives rise to a connected finite \'etale cover $X$ of $S$.
    Again consider a connected component $X_i$ of $X\times_SS'$ corresponding to some coset space $\pi^e_1(S',\overline{s}')/U''$ for $U''$ an open subgroup, and there is a morphism $X_i\rightarrow X$ if and only if $U''\subset U'$ by the classification of finite \'etale covers.
    And since both subgroups contain $\ker\phi_*$, the required inclusion is equivalent to $\phi_*(U'')\subset \phi_*(U')$, which is true by construction.
\end{proof}

\begin{corollary}
\label{criteria_injective_morphism_fund_group}
    The map $\phi_*$ is injective if and only if for every connected finite \'etale cover $X'$ of $S$, there exists a finite \'etale cover $X$ of $S$ and a morphism $X_i\rightarrow X'$ over $S'$, where $X_i$ is a connected component of $X\times_SS'$.
\end{corollary}
\begin{proof}
    By Lemma \ref{trivial_intersection_normal_subgroups}, the intersection of every normal open subgroup of $\pi^e_1(S',\overline{s}')$ is trivial.
    And by the above proposition, for any finite \'etale cover $X$ of $S$, there exists a morphism $X_i\rightarrow X'$ over $S'$, where $X_i$ is a connected component of the base change $\Bc_{S,S'}(X)$ if and only if $\ker \phi_* \subset U'$ for $U'$ an open subgroup of $\pi^e_1(S'\overline{s}')$.
    Thus $\ker\phi_*$ being included in all of those open normal subgroups of the above Lemma we refereed, it must be trivial as well.
\end{proof}

\noindent Notice that in the above corollary, if every connected finite \'etale cover $X'$ of $S'$ is actually a base change $\Bc_{S,S'}(X)$ of a finite \'etale cover $X$ of $S$, then the map $\phi: S'\rightarrow S$ is injective.

\begin{corollary}
\label{exact_sequence_middle}
    Let $S''\xrightarrow{\psi} S'\xrightarrow{\phi}$ be a sequence of morphisms of connected schemes, and let $\overline{s}$, $\overline{s}'$ and $\overline{s}''$ be geometric points of $S$, $S'$ and $S''$ respectively such that $\overline{s}= \phi\circ\overline{s}'$ ans $\overline{s}=\psi\circ \overline{s}''$. The sequence 
    $$\pi^e_1(S'',\overline{s}'')\rightarrow \pi^e_1(S',\overline{s}')\rightarrow \pi^e_1(S,\overline{s})$$
    is exact if and only if the two following  conditions are satisfied :
    \begin{itemize}
        \item[$a)$] For every connected cover $X$ of $S$, the base change $\Bc_{S,S''}(X)$ induced by the composition $\phi\circ\psi$ is a trivial cover of $S''$.
        \item[$b)$] Given a connected finite \'etale cover $X'$ of $S'$ such that $\Bc_{S',S''}(X')$ has a section over $S''$, there is a connected finite \'etale cover $X$ of $S$ and a morphism from a connected component of $\Bc_{S,S''}(X)$ onto $X'$ over $S'$.
    \end{itemize}
\end{corollary}
\begin{proof}
    This is a direct consequence of the previous, Proposition \ref{image_base_change_induce} $a)$, Proposition \ref{base_chage_image_last} and Proposition \ref{base_change_kernel}.
\end{proof}

\subsubsection{The Homotopy Exact Sequence}

In this section, we show very useful results that facilitate the computation of \'etale fundamental groups.
These are all more or less direct consequences of the homotopy exact sequence.
Recall from Section \ref{finite_etale_coverings} that $\overline{X}= X\times_kk_s$ and that a scheme $X$ over $k$ is geometrically integral if for all field extension $L|k$, $X\times_k\Spec\: L$ is integral.

\begin{theorem}[Homotopy Exact Sequence]
\label{first_homotopy_exact_sequence}
    Let $X$ be a quasi-compact and geometrically integral scheme over a field $k$. 
    Fix an algebraic closure $\overline{k}$ of $k$, and let $k_s$ be the separable closure.
    Consider the schemes $\overline{X}$ and $X$, with geometric points $\overline{x}':\Spec\:\Lambda \rightarrow \overline{X}$ and $\overline{x}:\Spec\:\Lambda\rightarrow X$ such that given the map $\phi:\overline{X}\rightarrow X$, $\phi\circ \overline{x}'=\overline{x}$.
    The sequence of profinite groups 
    $$1\rightarrow \pi^e_1(\overline{X},\overline{x}') \rightarrow \pi^e_1(X,\overline{x})\rightarrow \Gal(k_s|k)\rightarrow 1$$
    that is induced by the maps $\overline{X}\rightarrow X$ and $X\rightarrow \Spec\:k$ is exact.
\end{theorem}
\begin{proof}
    Injectivity of the first map $\pi^e_1(\overline{X}, \overline{x}')\rightarrow \pi^e_1(X,\overline{x})$ is given by the Corollary \ref{criteria_injective_morphism_fund_group} in view of Lemma \ref{etale_cover_closure_come_from_extension}.
    This may not be directly obvious. 
    By the Corollary, we start with a connected finite \'etale cover $\hat{Y}\rightarrow \overline{X}$, from the Lemma, there exist a separable extension $L$ of $k$ such that $Y_L\rightarrow X_L$ is a finite \'etale cover. 
    Moreover, for any finite separable extension $L$ of $k$, the projection $X_L\rightarrow X$ is a finite \'etale cover as well, this follows from Theorem \ref{etale_algebras} on \'etale algebras. 
    Hence we find a map $Y_L\rightarrow X$ that is a finite \'etale cover.
    Now the connected component in $Y_L\times_X\overline{X} \cong Y_L\times_k \Spec\:k_s$ is given by the connected finite \'etale cover from our initial assumption $\hat{Y}$. 
    Indeed, by the Lemma we referred to, $\hat{Y}\cong Y_L\times_L \Spec\:k_s$, and $Y_L\times_L \Spec\:k_s \hookrightarrow Y_L\times_k \Spec\:k_s$ is actually a closed immersion.\\
    Similarly, the map $\pi^e_1(X,\overline{x})\rightarrow \Gal(k_s|k)$ is surjective following Proposition \ref{image_base_change_induce}. This is more or less obvious by noticing that $\Gal(k_s|k)\cong \pi^e_1(\Spec\:k, x)$ for any geometric point $x$ of $\Spec\:k$ (this was Example \ref{fundamental_group_over_field}) and that the only connected finite \'etale covers of $\Spec\: k$ are $\Spec\:L$ where $L$ is a finite separable extension of $k$. We can conclude that since $X$ is geometrically an integral scheme, $X\times_k\Spec\:\overline{k}$ is integral, hence connected, any other scheme of the form $X\times_k\Spec\:L$ is also connected.\\
    What is left to show is exactness in the middle.
    For this we apply the Corollary \ref{exact_sequence_middle}.\\
    The condition $a)$ is satisfied as follows:
    we start with a connected cover $\Spec\:L$ of $\Spec\:k$, and then consider the base change cover $\Spec\:L \times_k \overline{X} \cong (\Spec\:L\times_k \Spec\:k_s) \times_{k_s} \overline{X}$.
    By Theorem \ref{etale_algebras}, $\Spec\:L\times_k \Spec\:k_s \cong \bigsqcup^n_{i=1} \Spec\:k_s$, so by distributivity of the fibre product over disjoint unions, 
    $$\Spec\:L \times_k \overline{X} \cong (\bigsqcup^n_{i=1} \Spec\:k_s) \times_{k_s} \overline{X} \cong \bigsqcup^n_{i=1} \overline{X}$$
    Condition $b)$ is the hard part of this proof.
    Here the idea is to construct a finite \'etale Galois cover $Y\rightarrow X$ such that $Y\times_X\overline{X} \cong Y\times_k\Spec\:k_s \rightarrow \overline{X}$ has a section, and show that it is isomorphic to $X_L$ over the generic point of $X$ (which is integral). This is skilfully done in \cite{Szamuely}, Proposition $5.6.1$.

\end{proof}

\noindent A more general statement of the above theorem is also true, but requires heavier algebraic geometry tools.

\begin{theorem}
    Let $S$ be a Noetherian integral scheme, and $\phi : X\rightarrow S$ a proper, flat morphism with geometrically integral fibres.
    Let $\overline{s}:\Spec\:k\rightarrow S$ be a geometric point of $S$ such that $k$ is the algebraic closure of the residue field of the image of $\overline{s}$ in $S$, and let $\overline{x}$ be a geometric point of the geometric fibre $X_{\overline{s}}$.
    The sequence 
    $$\pi^e_1(X_{\overline{s}},\overline{x})\rightarrow \pi^e_1(X,\overline{x}) \rightarrow \pi^e_1(S,\overline{s}) \rightarrow 1$$
    is exact.
\end{theorem}
\begin{proof}
    A proof of this theorem can be found in \cite{Szamuely} (Proposition $5.6.4$), it use the Stein factorisation.
    Another proof can be found in \cite{SGA}, (Th\'eor\`eme $IX.\: 6.1$) and it uses effective descent of morphisms.
\end{proof}

\noindent Interestingly enough, when $X$ is a scheme over $\Spec\:k$, the \'etale fundamental group of the base scheme $\pi^e_1(S,\overline{s})$ is isomorphic to the absolute Galois group $\Gal(k_s|k)$ (see Example \ref{fundamental_group_over_field}). 

\begin{corollary}
    Let $k$ be an algebraically closed field and $X$ and $Y$ Noetherian connected schemes over $k$.
    Assume moreover that $X$ is proper and geometrically integral, and choose geometric points $\overline{x}:\Spec \: k \rightarrow X$ and $\overline{y}:\Spec\:k\rightarrow Y$ of $X$ and $Y$ respectively, with values in $k$.
    Then the natural morphism
    $$\pi^e_1(X\times Y, (\overline{x},\overline{y})) \rightarrow \pi^e_1(X,\overline{x}) \times \pi^e_1(Y,\overline{y})$$
    induced by the projections of $X\times Y$ to $X$ and $Y$, respectively is actually an isomorphism.
\end{corollary}
\begin{proof}
    We can insert the above natural morphism in the following commutative diagram of exact sequences,
    \[
    \begin{tikzcd}
    1\arrow[r] & \pi^e_1(X,\overline{x})\arrow[r]\arrow[d] & \pi^e_1(X\times Y, (\overline{x},\overline{y}))\arrow[r]\arrow[d] & \pi^e_1(Y,\overline{y})\arrow[r]\arrow[d] & 1\\
    1\arrow[r] & \pi^e_1(X,\overline{x})\arrow[r] & \pi^e_1(X ,\overline{x})\times \pi^e_1(Y,\overline{y})\arrow[r] & \pi^e_1(Y,\overline{y})\arrow[r] & 1 
    \end{tikzcd}
    \]
    The upper row comes from applying the previous theorem to the morphism $X\times Y\rightarrow Y$ and the geometric point $\overline{y}$.
    Injectivity on the left comes from the fact that the projection $X\times Y\rightarrow X$ yields a section of the inclusion map $X\rightarrow X\times Y$, thus the homomorphism $\iota_*: \pi^e_1(X,\overline{x}) \rightarrow \pi^e_1(X\times Y,(\overline{x},\overline{y}))$ has a left inverse, making it injective.
    So the upper row is an exact sequence.\\
    The vertical maps on both sides are the identity maps.\\
    And the lower exact sequence directly follows from the splitting lemma with the inclusion $ \pi^e_1(X,\overline{x})\rightarrow \pi^e_1(X,\overline{x}) \times \pi^e_1(Y,\overline{y})$ and the projection $\pi^e_1(X,\overline{x}) \times \pi^e_1(Y,\overline{y}) \rightarrow \pi^e_1(Y,\overline{y})$.\\
    Hence the arrow in the middle must be an isomorphism.
\end{proof}

\noindent We end this section with an application of the previous corollary.

\begin{proposition}
\label{invariant_under_field_extension}
    Let $L$ be an extension of an algebraically closed field $k$ and $X$ be a proper integral scheme over $k$.
    The natural map $\pi^e_1(X_L,\overline{x}_L)\rightarrow \pi^e_1(X,\overline{x})$ is an isomorphism for every geometric point $\overline{x}$ of $X$.
\end{proposition}
\begin{proof}
    If $Y\rightarrow X$ is a connected finite \'etale cover, then $Y$ is reduced by Proposition \ref{reduced_carry_above} since $X$ is also reduced (Proposition \ref{integral_iff_irr_red}).
    Since $Y$ is reduced, so is the generic fibre $Y_{\eta}$ by Proposition \ref{reduced_carry_above}.
    And by Lemma \ref{finite_dimensional_direct_sum_reduced}, $Y_{\eta}$ is the spectrum of a direct sum of field extensions of $K(X)$, the function field of $X$. 
    The generic fibre also being connected, it must be a single finite field extension of the function field of $X$.
    Now $Y$ is also irreducible, it has $1$ minimal prime ideal (as does $Y_{\eta}$), then $Y$ is an integral scheme as well.
    And since $k$ is algebraically closed in $L$, the tensor product $K(Y)\otimes_k L$ is a field, where $K(Y)$ is the function field of $Y$.
    This shows that the base change $\Spec (K(Y)\otimes_kL)$, the part of $Y_L$ lying over its generic point $\eta$, is connected, and in turn, that $Y_L$ is connected.
    With Proposition \ref{image_base_change_induce} $b)$, we have that  $\pi^e_1(X_L,\overline{x}_L) \rightarrow \pi^e_1(X,\overline{x})$ is surjective.\\
    For injectivity, this is where the previous corollary intervenes. 
    We refer to Corollaire $X.\:1.8$ from \cite{SGA} or Proposition $5.6.7$ in \cite{Szamuely}.
\end{proof}

\noindent The previous proposition is actually quite useful.
Using the comparison theorems from the last section, we know that a smooth projective scheme $X$ over $\mathbb{C}$ has a topologically finitely generated \'etale fundamental group.
Now, if $X$ is also integral, we can apply Proposition \ref{invariant_under_field_extension} to the inclusion $\overline{\mathbb{Q}} \hookrightarrow \mathbb{C}$ and get that the \'etale fundamental group of $X$ over $\overline{\mathbb{Q}}$ is topologically finitely generated at every geometric point.
And by considering the inclusion $\overline{\mathbb{Q}}\hookrightarrow k$, for $k$ an algebraically closed field of characteristic $0$ we get the following result.
\begin{corollary}
\label{lats_thm}
    Let $k$ be an algebraically closed field of characteristic $0$, and let $X$ be a smooth connected projective scheme over $k$. For every geometric base point $\overline{x}$ the \'etale fundamental group $\pi^e_1(X,\overline{x})$ is topologically finitely generated.
\end{corollary}

\subsubsection{Some Examples}

We have already computed the \'etale fundamental groups of $X = \Spec\: k$ and $X = \Spec\: \mathbb{F}_p$.

\noindent The next examples illustrate the projective case. 
We know already that studying projective schemes is rather interesting as, over some number field $k$ whose algebraic closure can be embedded in $\mathbb{C}$,  their \'etale fundamental group is isomorphic to the profinite completion of the topological fundamental group of the analytification $X^{an}$. 

\begin{example}
    Let $k$ be an arbitrary field of characteristic $0$ and $X = \mathbb{P}^1_k$ the projective line over $k$.
    One way to compute the \'etale fundamental group of $X$ at a geometric point $\overline{x}$ of $X$ is to use the homotopy exact sequence from Theorem \ref{first_homotopy_exact_sequence}.
    Since every projective scheme over a field is compact, it is also quasi-compact.
    It is geometrically integral as well, indeed for every field extension $K$ of $k$, the scheme $\mathbb{P}^1_k\times_k\Spec\:K$ is isomorphic to $\mathbb{P}^1_K$ which is integral as it is irreducible and reduced ($K[x_1,x_2]$ is an integral domain).
    So we have the following exact sequence 
    $$0\rightarrow \pi^e_1(\mathbb{P}^1_{\overline{k}}, \overline{x}') \rightarrow \pi^e_1(\mathbb{P}^1_{k}, \overline{x}) \rightarrow \Gal(k_s|k) \rightarrow 0.$$
    And by Proposition \ref{invariant_under_field_extension} with the injection $\overline{k}\hookrightarrow \mathbb{C}$, we can compute $\pi^e_1(\mathbb{P}^1_{\overline{k}}, \overline{x}') \cong \pi^e_1(\mathbb{P}^1_{\mathbb{C}}, \overline{x}'_{\mathbb{C}})$, which is exactly the profinite completion of $\pi_1(S^2)$ for $\mathbb{P}^1_{\mathbb{C}}$ can be identified with the Riemann sphere.
    Now since any loop on $S^2$ is contractible to a point, $\pi_1(S^2) = 1$ and so $\pi^e_1(\mathbb{P}^1_{\overline{k}}, \overline{x}')$ is trivial as well.
    We can conclude that $\pi^e_1(\mathbb{P}^1_{k}, \overline{x}) \cong \Gal(k_s|k)$. 
    Which is trivial when $k$ is separably closed.
\end{example}

\noindent Using the above example we can compute the \'etale fundamental group of the following punctured projective lines over $\mathbb{Q}$.

\begin{examples} \textbf{ }
\begin{itemize}
    \item[$a)$] Suppose we consider the projective line over $\mathbb{Q}$ minus two points, say $\{0,\infty\}$. 
    As in the previous example, using a combination of Theorem \ref{first_homotopy_exact_sequence} and Proposition \ref{invariant_under_field_extension}, provided the embedding $\overline{\mathbb{Q}}\hookrightarrow \mathbb{C}$, we can compute its \'etale fundamental group. 
    We can conclude that the following sequence is exact
    $$0\rightarrow \widehat{\mathbb{Z}}\rightarrow \pi^e_1(\mathbb{P}^1_{\mathbb{Q}} \backslash\{0,\infty\}, \overline{x}) \rightarrow \Gal(\overline{\mathbb{Q}}|\mathbb{Q})\rightarrow 0.$$
    Indeed, notice that the Riemann sphere minus two points is homotopically equivalent to the circle which we showed has fundamental group $\mathbb{Z}$ at any point.
    
    \item[$b)$] We can proceed similarly to compute the \'etale fundamental group of the projective line over $\mathbb{Q}$ minus three points (say $\{0,1,\infty\}$).
    In this case, notice that the complex projective line minus three points is actually homotopically equivalent to a figure eight: the wedge sum of two circle $S^1\vee S^1$.
    Now to compute the fundamental group of a wedge sum of say two WC complexes (which is the case for $S^1$) one can use Van Kampen' theorem \cite{Hatcher}, (Theorem $1.20$). 
    It says that the fundamental group of the wedge is the free product of the fundamental groups of the two spaces. Thus $\pi_1(S^1\vee S^1) = \mathbb{Z}*\mathbb{Z}$, which is the free group on two generators $F_2$. This gives us the following exact sequence,
     $$0\rightarrow \widehat{F_2}\rightarrow \pi^e_1(\mathbb{P}^1_{\mathbb{Q}} \backslash\{0,1,\infty\}, \overline{x}) \rightarrow \Gal(\overline{\mathbb{Q}}|\mathbb{Q})\rightarrow 0.$$
\end{itemize}
\end{examples}

\noindent Grothendieck suggested that the most interesting group in mathematics is arguably $\pi^e_1(\mathbb{P}^1_{\mathbb{Q}} \backslash\{0,1,\infty\},\overline{x})$. 
Later on, he advised P. Deligne to study this group, which he did extensively, leading to a $200$-pages monograph \cite{droit_p_minus_3_pts}.
In his paper, Deligne introduced a Tannakian framework for the category of \textit{mixed motives}.
Without delving further into motivic theory, we can summarize Deligne’s objective in defining a ‘‘\textit{motivic}" version of the fundamental group through the following question:\\
There are many ways through which one can understand a (algebraic) space, this can be done through theories such as \textit{Betti} cohomology, \textit{de Rham} cohomology, \textit{$\ell$-adic} cohomology and \textit{crystalline} cohomology.
The problem is that these theories may feel like different ‘‘shadows" of the same object. 
Now Grothendieck hoped for a universal theory, sitting above them all, which he called the theory of motives.
And Deligne's approach consisted of finding a way to unify this perspective by applying it to objects richer than cohomology groups: non-abelian fundamental groups.\\
In \cite{droit_p_minus_3_pts}, he successfully defined what it meant for a non-abelian group to be motivic by working with a Tannakian realisation of the \'etale fundamental group of $\mathbb{P}^1_{\mathbb{Q}} \backslash \{0,1,\infty\}$ as a non-trivial example.
This the motivation behind the next chapter.

\newpage

\section{Tannakian Fundamental Groups}

In the last chapter, we established an equivalence of categories between the category of finite \'etale covers over a connected scheme and that of finite continuous representations of its \'etale fundamental group. 
In this chapter, we study a linearisation of this concept by replacing finite \'etale covers by linear algebraic data. 
We begin by studying representations of affine group schemes and their relation with comodules over Hopf algebras. 
Then introduce monoidal categories and develop the categorical language required for \textit{Tannaka duality}.
And finish with a specialisation of this reconstruction for neutral Tannakian categories, and define their Tannakian fundamental group.

\subsection{Representation of Group Schemes and Comodules}

Before defining comodules, one has to make sense of the kind of object we define comodules over. 
It is quite natural that we want comodules to be the dual notion of modules. 
And since we will study modules over algebras, we introduce below the notion of a coalgebra.
This will allow us to have a clean definition for coations of comodules.

\begin{definition}
    A coalgebra is a $k$-vector space $C$ endowed with a $k$-linear comultiplication : $\Delta:C\rightarrow C\otimes_k C$ and a $k$-linear counit : $\epsilon : C\rightarrow k$ satisfying the following coassociativity and counity axioms,
    \[
    \begin{tikzcd}
        C \arrow[r,"\Delta"] \arrow[d, "\Delta"'] & C \otimes C \arrow[d,"\Delta \otimes \id_C"] \\
        C \otimes C \arrow[r,"\id_C \otimes \Delta"'] & C \otimes C\otimes C
    \end{tikzcd}
    \hspace{1cm}
    \begin{tikzcd}
        C \arrow[rd, "\Delta"]
        \arrow[d, "\cong"']
        \arrow[r, "\cong"] &
        C \otimes k  \\
         k\otimes C &  
         C \otimes C
         \arrow[u,"\id_C\otimes \epsilon"']
         \arrow[l, "\epsilon\otimes \id_C"]
    \end{tikzcd}
    \]
    
\end{definition}

\noindent One can see from the above definition that coalgebras over $k$ are the dual notion of algebras over $k$. 
The next result tells us that under a finiteness condition, dualising a coalgebra, we get an algebra as well.

\begin{proposition}
\label{dualize_algebra_coalgebra}
    The contravariant dualisation functor $A\mapsto A^*$ induces an anti-equivalence between the category of finite-dimensional $k$-coalgebras and that of finite-dimensional $k$-algebras.
\end{proposition}
\begin{proof}
    Given a $k$-coalgebra $A$, the $k$-linear dual $A^* = \Hom_k(A,k)$ of the underlying finite-dimensional $k$-vector space carries additional structure.
    The comultiplication $\Delta$ induces a $k$-bilinear map $m:A^*\otimes_kA*\rightarrow A^*$ mapping a pair $(f,g)$ of maps in $\Hom_k(A,k)$ to the $k$-linear map $(f\otimes g)\circ \Delta : A\rightarrow k$.
    The counit $\epsilon$ has a $k$-linear dual $e :k\rightarrow A^*$ determined by the image of $1\in k$ in $A^*$.
    By the coassociativity and counity axioms, the maps $m$ and $e$ define a multiplication and a unit satisfying associativity and unity, making $A^*$ a $k$-algebra.\\
    Conversely, given a finite-dimensional $k$-algebra $B$, the multiplication map $m$ induces a map $B^*\rightarrow (B\otimes B)^*$ on the $k$-linear duals. 
    Here enters our crucial assumption that algebras must be finite-dimensional.
    Indeed, the natural map $B^*\otimes B^*\rightarrow (B\otimes B)^*$ sending $(f,g)$ to the $k$-linear map $a\otimes b\mapsto f(a)g(b))$ is an isomorphism if $B$ is finite-dimensional over $k$, in that case it is an injective map between vector spaces of the same dimension.
    And by reversing the above procedure, we can endow $B^*$ with a structure of $k$-coalgebra.
\end{proof}

\begin{remark}
\label{dual_algebara_no_coalgebra}
    The above equivalence restricts to the finite-dimensional case, because whereas it is true that arbitrary $k$-coalgebras dualise into algebras, the converse is false for infinite-dimensional $k$-algebras.
    For example, the $k$-algebra $A = k[x]$ does not dualise into a coalgebra.
    Denote $A^* = (k[x])^*$,
    then the natural map $A^*\otimes A^*\rightarrow (A\otimes A)^*$ is not surjective anymore.
    Indeed, elements of $A^*\otimes A^*$ are finite sum tensors (they have finite rank) while elements of $(A\otimes A)^*$ are infinite sequences of evaluation on the basis.
\end{remark}

\begin{definition}
    A \textbf{bialgebra} $H$ is a vector space that is at the same time an algebra with multiplication $m$ and unit $u$ and a coalgebra with comultiplication $\Delta$ and counit $\epsilon$, such that $\Delta$ and $\epsilon$ are algebra morphisms (or equivalently that $m$ and $u$ are coalgebra morphisms).\\
    A \textbf{Hopf algebra} is a bialgebra that possesses an additional map $\Sigma: H\rightarrow H$, which makes the following diagram commute:
    \[
    \begin{tikzcd}
        & H\otimes H
        \arrow[rr,"\Sigma\otimes \id"] &&
        H\otimes H
        \arrow[dr,"m"]\\
        H
        \arrow[ur,"\Delta"]
        \arrow[rr, "\epsilon"]
        \arrow[dr,"\Delta"']&&
        k
        \arrow[rr,"u"]&&
        H\\
        &H\otimes H
        \arrow[rr,"\id\otimes \Sigma"']&&
        H\otimes H
        \arrow[ur,"m"']   
    \end{tikzcd}
    \]
\end{definition}

\begin{definition}
\label{def_hopf_alg}
    Let $C$ be a $k$-coalgebra.
    A (right) \textbf{comodule} for $C$ is a vectorspace $M$ endowed with a linear map (a coaction) $\rho:M\rightarrow M \otimes C$ such that the following diagrams commute 
    \[
    \begin{tikzcd}
        M \arrow[r,"\rho"] \arrow[d, "\rho"'] & M \otimes C \arrow[d,"\rho \otimes \id_C"] \\
        M \otimes C \arrow[r,"\id_M \otimes \Delta"'] & M\otimes C\otimes C
    \end{tikzcd}
    \hspace{1cm}
    \begin{tikzcd}
        M \arrow[r, "\rho"]\arrow[dr, "\cong"']  &  M\otimes C \arrow[d,"\id_M\otimes \epsilon"] \\
         &  M \otimes k
    \end{tikzcd}
    \]
    A \textbf{subcomodule} of a comodule $M$ is a $k$-subspace $N$ such that $\rho(N)\subset N\otimes C$.
\end{definition}

\begin{remark}
    Similarly to Proposition \ref{dualize_algebra_coalgebra}, there is an equivalence between (left) finite-dimensional $C$-comodules and (right) finite-dimensional $C^*$-modules. 
    For $A$ a finite-dimensional $k$-algebra, we even have the following equivalence of categories: $\Mod^{\text{f.d.}}_A \simeq \Comod^{\text{f.d.}}_{A^*}$.
\end{remark}

\noindent We now state the fundamental theorem of comodules, 

\begin{theorem}
\label{fundamental_theorem_comodule}
    Let $C$ be a $k$-coalgebra and $M$ a right $C$-comodule.
    The any element $m\in M$ is contained in a $C$-subcomodule that is finite-dimensional over $k$. Consequently, any finite-dimensional subspace $V\subset M$ is contained in a $C$-subcomodule of $M$ that is finite-dimensional over $k$.
\end{theorem}
\begin{proof}
    See \cite{Milne_alg_group}, Proposition $4.7$, where the statement can be generalised to $\mathcal{O}(G)$ an arbitrary $k$-coalgebra
\end{proof}

\begin{remark}
\label{representation_of_group_as_functor}
    Recall that in the last chapter, we identified group schemes with group-valued representable functors (Proposition \ref{group_scheme_are_representable_functors}).
    And that in particular, we can see affine group schemes $G$ over $k$ as representable group-valued functors $R\mapsto \Hom_k(\Spec\:R,G)$ from the category of $k$-algebras to that of groups, that is represented by the coordinate ring $A$ of $G$ which is a Hopf algebra (Remark \ref{Hopf_algebra_group_scheme}).
    In that case, a representation $(V,\rho)$ of an affine group scheme $G$ consists of giving for each $k$-algebra $R$ a $G(R)\rightarrow \End(V\otimes_kR)$. Functorially, this amounts to defining a natural transformation $\rho:G\rightarrow \GL$.
\end{remark}

\begin{proposition}
\label{representations_iso_comodules}
    There is a bijection between right comodules over the commutative Hopf algebra $A$ and left representations of the corresponding affine group scheme.
\end{proposition}
\begin{proof}
    Let $M$ be an $A$-comodule where $A$ is a commutative Hopf algebra.
    Given a $k$-algebra $R$, an element of $G(R)$ corresponds to a $k$-algebra homomorphism $g: A\rightarrow R$.
    The composite map $\pi_R : M\xrightarrow{\rho}M\otimes_kA\xrightarrow{\id\otimes g} M\otimes_kR$ induces an $R$-linear map $M\otimes_kR\rightarrow M\otimes_kR$, so an object of $\GL(M\otimes_kR)$ that depends on $R$ functorially.
    The maps $\pi_R$ are actually actions, as in remark \ref{representation_of_group_as_functor}.
    The identity $e\in G(R)$ is given by the counit $\epsilon:A\rightarrow k \rightarrow R$ and for $g,h\in G(R)$, the product $g\cdot h$ in $G(R)$ is $m_R\circ (g\otimes h) \circ \Delta$.
    By the counit and coassociativity of the comodule $M$, we have $\pi_R(e) = \id$ and $\pi_R(g\cdot h) = \pi_R(g)\circ\pi_R(h)$. 
    So this is a left representation of the affine group scheme $G$.\\
    On the other hand, let $(V,\rho)$ be a left representation of an affine group scheme $G$, so it has an action $\pi_A$, such that for $\id_A\in G(A)$ the identity element, $\pi_A(\id_A):V\otimes A\rightarrow V\otimes A$ is $A$-linear.
    The composition of $\pi_A(\id_A)$ with $v\mapsto v\otimes 1$ defines a map $\nu:V\rightarrow V\otimes A$, which can be seen as a coaction on the $A$-comodule $V$.
    The two construction being inverse to one another, we have a bijection.\\
    
\end{proof}

What we have seen so far is that for an affine group scheme $G$, one can consider its coordinate ring, which is a Hopf algebra.
Then studying its category of representations is the same as studying comodules over its coordinate ring.
Now we want to know if the reverse computation is possible.
By that, we mean if it is possible to recover the affine group scheme $G$ from its category of representations (which is equivalent to a certain category of comodules).
This problem has been extensively studied and is largely understood through the lens of \textit{Tannaka-Krein duality}. 
We shall give a somewhat motivation for this.
We start with the theory of \textit{Morita equivalence}, which states that one can determine an algebra $A$ from its category of modules $\Mod_A$ up to Morita equivalence (see \cite{Etingof}, Chapter $7.8$).
But in the case of finite-dimensional, by considering the associated forgetful functor $\omega:\Mod_A\rightarrow \Vect_k$, our $k$-algebra can be recovered up to \textit{isomorphism} as the endomorphism algebra of the forgetful functor.
If $A$ is a finite-dimensional $k$-algebra, then the category of finite-dimensional $A$-modules is a $k$-linear abelian category. 
The other way around, for $\mathcal{C}$ a $k$-linear abelian category, then one can take $A=\End(P)^{op}$ for $P$ a progenerator of $\mathcal{C}$. That same progenerator in $\mathcal{C}$ represents a forgetful functor $\omega :\mathcal{C}\rightarrow \Vect^{\text{f.d.}}_k$, giving us $A = \End(P)^{op} \cong \End(\omega)$ (see \cite{Etingof}, Chapter $1.8$).
And a similar statement holds for arbitrary $k$-algebras.
What we are interested in is what happens in the case of comodules (corresponding to representations), we want to dualise the finite-dimensional statement to obtain a similar reconstruction but for coalgebras from a category of comodules over a coalgebra.
However, as we saw in Remark \ref{dual_algebara_no_coalgebra} a similar procedure in the infinite-dimensional case fails in general.
Still, it is possible to recover an arbitrary $k$-coalgebra $C$ from the category $\Comod^{\text{f.d.}}_C$ of finite-dimensional right $C$-comodules with the associated forgetful functor $\omega :\Comod^{\text{f.d.}}_C\rightarrow \Vect^{f.d.}_k$.  \\
Our first result in that direction comes from an application of the fundamental theorem of comodules. 
We call the forgetful functor $\omega :\Comod_C\rightarrow \Vect_k$ a \textit{fibre functor} and we define its restriction $\omega^f: \Comod^{\text{f.d.}}_C\rightarrow \Vect^{\text{f.d.}}_k$.

\begin{lemma}
    Let $\omega$ and $\omega^f$ be fibre functors as above, then
    $$\Hom(\omega,\omega) =\Hom(\omega^f,\omega^f)$$
\end{lemma}
\begin{proof}
    Let $M$ be a $C$-comodule and let $N$ be a subcomodule of $M$ via the inclusion $\iota :N \hookrightarrow M$. Then, for $\lambda\in\End(\omega)$, by naturality, the restriction of $\lambda_M$ is given exactly by $\lambda_N$, i.e.
    \[
    \begin{tikzcd}[column sep = large]
        N \arrow[r,"\lambda_N"]\arrow[d,"\iota"'] & N\arrow[d,"\iota"]\\
        M\arrow[r,"\lambda_M"'] & M.
    \end{tikzcd}
    \]
    From the fundamental theorem of comodules \ref{fundamental_theorem_comodule}, $M$ is the sum of all its finite-dimensional subcomodules.
    Hence, $\lambda_M$ is completely determined by the natural transformation $\lambda$ in each of these finite-dimensional subcomodules.
\end{proof}

\begin{remark}
    In light of Proposition \ref{representations_iso_comodules},
    the previous lemma states that in order to study arbitrary representations of an affine group scheme, it suffices to study the finite-dimensional ones.
\end{remark}

\noindent The following lemma will help us recover a coalgebra from its category of comodules with the fibre functor $\omega$. 
We denote $\omega\otimes V$, the functor from $\Comod^{\text{f.d.}}_C\rightarrow \Vect_k$ sending $M\mapsto \omega(M)\otimes V$.

\begin{lemma}
\label{useful_lemma_underlying_coalgebra}
    Let $A$ be a $k$-coalgebra and $\omega^f$ the fibre functor as above.
    The underlying $k$-vector space of $A$ represents the functor $V\mapsto \Hom(\omega, \omega\otimes V)$ on the category $\Vect_k$. In particular, we have 
    $$\Hom(C,V) \cong \Hom(\omega, \omega\otimes V).$$
\end{lemma}
\begin{proof}
    See \cite{Schauenburg}, Lemma $2.2.1$. 
\end{proof}


\begin{corollary}
    A $k$-coalgebra $C$ is determined up to unique isomorphism by the category of $\Comod^{\text{f.d.}}_C$ and the functor $\omega$.
\end{corollary}
\begin{proof}
    By the previous proposition, the underlying vector space of $C$ is uniquely determined up to isomorphism as the representing object of the functor $V\mapsto \Hom(\omega, \omega\otimes V)$.
    W endow this underlying vector space with a coalgebra structure as follows:\\
    We recover the comultiplication $\Delta : C \rightarrow C\otimes_k C$, by considering $\Hom(C,V)\xrightarrow{\sim} \Hom(\omega, \omega\otimes V)$, setting $V= C\otimes_k C$. 
    We see that $\Delta$ corresponds to a canonical natural transformation $\omega\mapsto \omega \otimes (C\otimes_k C)$.
    By considering the canonical natural transformation $\alpha : \omega\rightarrow \omega \otimes C$ induced by the coaction of comodules over $C$, and taking the composition $(\alpha \otimes \id) \circ \alpha :\omega\rightarrow \omega \otimes (C\otimes_k C)$, we obtained the natural transformation we wanted.\\
    Similarly, the counit $\epsilon : C\rightarrow k$ corresponds to the natural isomorphism $\omega\xrightarrow{\sim}\omega\otimes k$.

\end{proof}

\begin{remark}
\label{remark_abelian_category}
    In the literature, it is often shown that the coalgebra $C$ is recovered as the vector space of coendomorphisms of the fibre functor $\omega$, with a coalgebra structure (see \cite{Schauenburg}, Lemma $2.1.9$).
    Moreover, one can show that for a $k$-linear abelian category $\mathcal{C}$, and $\omega$ a $k$-linear fibre functor, there exists a $k$-coalgebra such that $\mathcal{C}$ is equivalent to the category of (right) $C$-comodules of finite-dimension over $k$ (see \cite{Schauenburg}, Theorem $2.2.8$).
\end{remark}

\noindent Assume now that $C$ has an additional structure of a Hopf algebra. 
Then, we have a multiplication map $m: C\otimes_k C\rightarrow C$, and it is a $k$-coalgebra homomorphism.
Using this, we can define a $C$-comodule structure on the vector space $\omega(M)\otimes_k \omega(N)$, for $M,N\in \Comod_C^{\text{f.d.}}$, by defining a coaction as:
$$M\otimes_k N\xrightarrow{\rho_M\otimes\rho_N} M\otimes_k C \otimes_k N \otimes_k C \xrightarrow{\sim} M \otimes_k N \otimes_k  C \otimes_k C \xrightarrow{\id\otimes \id \otimes m} M \otimes_k N\otimes_k C.$$
In the same way, each $k$-linear map $C\otimes_k C\rightarrow V$ induces a $k$-linear map $\omega(M)\otimes_k \omega(N) \rightarrow \omega(M)\otimes_k \omega(N)\otimes_k V$. 
We then have the following analogue of Lemma \ref{useful_lemma_underlying_coalgebra}.

\begin{lemma}
\label{useful_lemma_underlying_coalgebra_2}

    The underlying $k$-vector space of $C\otimes_kC$ represents the functor $V\mapsto \Hom_{\Vect_k}(\omega\otimes\omega, \omega\otimes\omega\otimes V)$.
    In particular, we have 
    $$\Hom(C\otimes_kC, V)\cong \Hom_{\Vect_k}(\omega\otimes\omega , \omega\otimes\omega \otimes V).$$
\end{lemma}
\begin{proof}
    Since we have just seen that each $k$-linear map $C\otimes_k C\rightarrow V$ induces a $k$-linear map $\omega(M)\otimes_k \omega(N) \rightarrow \omega(M)\otimes_k \omega(N)\otimes_k V$ for each $M,N\in\Comod^{\text{f.d.}}_C$.
    This yields a natural transformation $\Hom(C\otimes_kC,\_)\rightarrow \Hom(\omega\otimes\omega ,\omega\otimes \omega\_)$.
    And the proof that it is an isomorphism is analogous to thath of Lemma \ref{useful_lemma_underlying_coalgebra}.
    (again see \cite{Schauenburg} Lemma $2.2.1$).
\end{proof}

\noindent Tannaka-Krein duality tells us that the coalgebra we recover from its category of finite-dimensional comodules and $\omega$ has extra structure when $\Comod^{\text{f.d.}}$ and the functor $\omega$ also have extra configurations. 
We now introduce the language necessary to understand those extra conditions.\\

\subsection{Monoidal Categories and Tannaka-Krein Theorem}

Monoidal categories, also often called \textit{tensor categories}, can be seen as the generalisation of categories carrying a tensor product, such as the categories of vector spaces, abelian groups, $R$-modules, or $R$-algebras.
We will use the theory developed in the previous section to prove that the category of finite-dimensional representations of an affine group scheme over $k$ is equivalent to a monoidal category under necessary and sufficient conditions. 
And we will end this section by introducing the Tannaka-Krein Duality theorem.

\subsubsection{Monoidal Categories}

\begin{definition}
    A \textbf{monoidal category} is 
    a category $\mathcal{C}$ together with 
    a bifunctor $\_\otimes\_:\mathcal{C}\times \mathcal{C}\rightarrow \mathcal{C}$ called the \textbf{tensor product}, 
    a \textbf{unit object} $1$ in $\mathcal{C}$ together with natural isomorphisms $\Phi$ of functors from $\mathcal{C}\times\mathcal{C}\times\mathcal{C}$ to $\mathcal{C}$ given on a triple $(X,Y,Z)$ of objects by 
    $$\Phi_{X,Y,Z} : (X\otimes Y)\otimes Z \xrightarrow{\sim} X\otimes (Y\otimes Z),$$
    and natural isomorphisms $l$ from $\otimes\circ (I\times \id)$ to $\id$ 
    and $r$ from $\otimes \circ (\id\times I)$ to $\id$.
    Such that the following diagrams commute for each four-tuple $(X,Y,Z,W)$ of objects in $\mathcal{C}$,
    \[
    \begin{tikzcd}[row sep = large]
        (X\otimes (Y\otimes Z)) \otimes W 
        \arrow[rr, "\Phi_{X,Y,Z}\otimes\id_W"] \arrow[d,"\Phi_{X,Y\otimes Z,W}"'] && ((X\otimes Y) \otimes Z) \otimes W \arrow[d,"\Phi_{X\otimes Y, Z, W}"]\\
        X\otimes ((Y\otimes Z)\otimes W)
        \arrow[dr, "\id_X\otimes \Phi_{Y,Z,W}"'] && (X\otimes Y)\otimes (Z\otimes W)
        \arrow[dl, "\Phi_{X,Y, Z\otimes W}"]\\
        & X\otimes (Y\otimes (Z\otimes W))
    \end{tikzcd}
    \]
    \[
    \begin{tikzcd}
        (X\otimes I)\otimes Y 
        \arrow[rr,"\Phi_{X,I,Y}"]
        \arrow[dr,"r_X\otimes \id_Y"']&&
        X\otimes (I\otimes Y)
        \arrow[dl,"\id_X\otimes l_Y"]\\
        & X\otimes Y
    \end{tikzcd}
    \]
    The first diagram is called the associativity constraint, while the second is referred to as the unitary constraint.
    A monoidal category $\mathcal{C}$ is called \textbf{strict} if $\nu$ and $\Phi$ are identities.
\end{definition}

\begin{remark}
    A consequence of Mac Lane's coherence theorem is that everything that can be proven in a strict monoidal category is also valid in a monoidal category. 
    For that reason, we will usually denote monoidal categories by tuples $(\mathcal{C},\otimes, I)$, omitting associativity and unitary constraints.
    We will formulate this coherence theorem later but assume its validity from now on.
\end{remark}

\begin{examples}\textbf{ }
\begin{itemize}
    \item[$a)$] Let $k$ be a field and $R$ a commutative ring or a $k$-algebra. The tuples $(\Vect_k, \otimes, k)$ and $(\Mod_R, \otimes_R,R)$ are clearly monoidal categories, and so is $(\SET, \times, \{*\})$
    
    \item[$b)$] Let $G$ be a group. Then the category $(\Rep^{\text{f.d.}}_k(G), \otimes, k)$ is a monoidal category.\\
    For two representations $(V,\alpha)$ and $(W,\beta)$, we define
    $$(V,\alpha) \otimes (W,\beta) = (V\otimes W, \alpha\otimes \beta)$$
    where $V\otimes W$ is the usual tensor product of $k$-vector spaces, and $\alpha\otimes\beta:G\rightarrow \GL(V\otimes W)$ is defined as 
    $$(\alpha\otimes\beta)(g)(\sum_i v_i\otimes w_i) = \sum_i \alpha(g)(v_i)\otimes \beta(g)(w_i).$$
    The unit being the trivial representation on $k$, i.e. $(k,\epsilon)$, where
    $$\epsilon:G\rightarrow \GL(k) = k, \: \epsilon(g)= 1.$$
\end{itemize}
\end{examples}

\begin{definition}
    A monoidal category $(\mathcal{C}, \otimes, I)$ is called \textbf{braided} if there exists an isomorphism $\gamma$ of functors from $\mathcal{C}\times \mathcal{C}$ to $\mathcal{C}$ given on a pair $(X,Y)$ of objects by 
    $\gamma_{X,Y} : X\otimes Y\xrightarrow{\sim} Y\otimes X$
    such that the following diagrams commute for all $X,Y,Z\in \mathcal{C}$,
    \[
    \begin{tikzcd}[column sep = large]
        (X\otimes Y)\otimes Z  
        \arrow[r,"\Phi_{X,Y,Z}"] 
        \arrow[d, "\gamma_{X,Y}\otimes\id_Z"'] & X\otimes (Y\otimes Z)
    \arrow[d,"\gamma_{X,Y\otimes Z}"]\\
        (Y\otimes X)\otimes Z
        \arrow[d,"\Phi_{Y,X,Z}"']&
        (Y\otimes Z)\otimes X
        \arrow[d,"\Phi_{Y,Z,X}"]\\
        Y\otimes (X\otimes Z) 
        \arrow[r, "\id_Y\otimes\gamma_{X,Z}"'] &
        Y\otimes(Z\otimes X)
    \end{tikzcd}
    \hspace{1cm}
    \begin{tikzcd}[column sep = large]
        X\otimes (Y\otimes Z) 
        \arrow[r, "\Phi_{X,Y,Z}^{-1}"] \arrow[d,"\id_X\otimes \gamma_{Y,Z}"']  & 
        (X\otimes Y)\otimes Z
        \arrow[d,"\gamma_{X\otimes Y, Z}"]\\
        X\otimes (Y\otimes Z) \arrow[d,"\Phi_{X,Z,Y}^{-1}"']  
        &  Z\otimes (X\otimes Y)
        \arrow[d,"\Phi_{Z,X,Y}^{-1}"]\\
        (X\otimes Z)\otimes Y
        \arrow[r,"\gamma_{X,Z}\otimes \id_Y"']&
        (Z\otimes X)\otimes Y
    \end{tikzcd}
    \]
    And if $\gamma_{X,Y} = \gamma^{-1}_{Y,X}$ for all $X,Y\in \mathcal{C}$, then $\mathcal{C}$ is a \textbf{symmetric} monoidal category.
\end{definition}

\begin{definition}
    We say that a symmetric monoidal category $(\mathcal{C},\otimes,I)$ is \textbf{left rigid} if each object $X$ in $\mathcal{C}$ is left rigid.
    And we say that an object $X$ is left rigid if there exists an object $X^*$ together with morphisms $\ev:X\otimes X^*\rightarrow I$ and
    $\coev:I\rightarrow X^*\otimes X$ such that the following diagrams commute,
    \[
    \begin{tikzcd}
        X 
        \arrow[r,"\cong"] 
        \arrow[d, "\cong"'] & 
        I\otimes X \\
        X\otimes I
        \arrow[r,"\id_X\otimes \coev"'] &  X\otimes X^*\otimes X 
        \arrow[u,"\ev\otimes\id_X"']
    \end{tikzcd}
    \hspace{1cm}
    \begin{tikzcd}
        X^* 
        \arrow[r, "\cong"] 
        \arrow[d,"\cong"']  & 
        X^*\otimes I\\
        I \otimes X^* 
        \arrow[r,"\coev\otimes\id_{X^*}"']  &  
        X^*\otimes X\otimes X^*
        \arrow[u,"\id_{X^*}\otimes \ev"']
    \end{tikzcd}
    \]
    We define symmetrically \textbf{right dual} objects in a monoidal category and say that a monoidal category is right dual if all of its objects admit right duals.
    If every object of a monoidal category admits simultaneously right and left duals, we say that the category is \textbf{rigid}.
    Notice that if a monoidal category is symmetric, it is left rigid if and only if it is right rigid and its left and right duals coincide.
\end{definition}

\begin{remark} \textbf{ }
\begin{itemize}
    \item[$a)$] In a rigid monoidal category,
    the dual of $(X\otimes Y)$ is $Y^*\otimes X^*$.\\
    The evaluation and coevaluation maps are respectively given by the composites \\
    $(X\otimes Y) \otimes (Y^* \otimes X^*) \xrightarrow[]{\id\otimes \ev_Y\otimes \id} X\otimes I\otimes X^*\xrightarrow[]{\sim} X\otimes X^*\xrightarrow[]{\ev_X} I$, and \\
    $I\xrightarrow[]{\coev_Y} Y^*\otimes Y \xrightarrow[]{\sim} Y^*\otimes I\otimes Y \xrightarrow[]{\id\otimes \coev\id} Y^*\otimes X^*\otimes X\otimes Y$.
    
    \item[$b)$] In a rigid monoidal category, every morphism $f:X\rightarrow Y$ has a \textbf{transpose} :
    it is the morphism $f^t: Y^*\rightarrow X^*$ defined as the composite
    $$Y^*\xrightarrow[]{\sim} I\otimes Y^*\xrightarrow[]{\coev\otimes\id} X^*\otimes X\otimes Y^* \xrightarrow[]{\id\otimes f\otimes \id} X^*\otimes Y\otimes Y^* \xrightarrow[]{\id\otimes \ev} X^*\otimes I \xrightarrow[]{\sim} X^*.$$
    And the map $f\mapsto f^t$ induces a bijection $\Hom(X,Y)\xrightarrow[]{\sim}\Hom(Y^*,X^*)$.
\end{itemize}
    
\end{remark}

\begin{examples}\textbf{ }
\begin{itemize}
    \item[$a)$] The category $(\Vect^{\text{f.d.}}_k, \otimes, k)$ is rigid.\\
    Let $V$ be a finite-dimensional $k$-vector space with basis $\{e_1,...,e_n\}$ and consider its dual vector space $\Hom(V,k)$ with corresponding dual basis $\{e^1,...,e^n\}$.
    We define the maps $\ev: V\otimes \Hom(V,k)\rightarrow k$ mapping $v\otimes f\mapsto f(v)$ 
    and $\coev : k\rightarrow \Hom(V,k)\otimes V$ sending $r\mapsto r\sum e^i\otimes e_i$
    Let $v = \sum_{i=1}^n v_ie_i\in V$, the following equation satisfy rigidity constraints :
    \begin{equation*}
    \begin{split}
        v\mapsto v\otimes 1 & \mapsto v\otimes \Bigl(\sum^n_{i=1}e^i\otimes e_i\Bigr)\\
        & = v \otimes \Bigl(\sum^n_{i=1}e^i \otimes \sum^n_{i=1}e_i\Bigr)\\
        & = \Bigl(\sum^n_{i=1} v_ie_i \otimes \sum^n_{i=1} e^i\Bigr)\otimes \sum^n_{i=1} e_i\\
        & \mapsto \sum^n_{i=1} v_i \otimes \sum^n_{i=1} e_i \mapsto \sum^n_{i=1} v_ie_i = v.
    \end{split}
    \end{equation*}
    And the second constraint is similarly verified. 
    Since in the category of finite-dimensional vector space over $k$, we have a canonical isomorphism $V\otimes V^*\cong V^*\otimes V$, the above equation is also symmetrically true, making $\Vect_k$ a rigid category.

    \item[$b)$] For $G$ a group, the category $(\Rep(G), \otimes, k)$ is rigid.\\
    Let $(V,\rho)$ be a representation of $G$, we define its dual representation to be the pair $(\Hom(V,k),\rho_{V^*})$ where for all $v\in V$, $\rho_V^*(g)(f(v)) = f(\rho_V(g^{-1})(v))$.
    
    The evaluation and coevaluation maps are defined similarly as above, here we just have to check that they are $G$-linear :
    \begin{equation*}
    \begin{split}
        \ev(\rho_{V\otimes V^*}(g)(v\otimes f)) &= \ev (\rho_V(g)(v)\otimes \rho_{V^*}(g)(f))\\
        & = \rho_{V^*}(g)(f(\rho_V(g)(v))) \\
        & = f(\rho_V(g^{-1})(\rho_V(g)(v))) \\
        & = f(\rho_V(g^{-1}g)(v))\\
        & = f(v) = \ev (v\otimes f)
    \end{split}
    \end{equation*}
    And since it is sufficient to check for a basis element, we get that for $1\in k$
    \begin{equation*}
    \begin{split}
        \rho_{V^*\otimes V}(g)(\coev(1)) = \rho_{V^*\otimes V}(g)\Bigl(\sum^n_{i=1} e^i\otimes e_i\Bigr) & = \sum^n_{i=1}\rho_{V^*}(g)(e^i)\otimes \rho_V(g)(e_i)\\
        & =\sum^n_{i=1} \Bigl( \sum_{j=1}^n g^{-1}_{ji} e^n\Bigr)\otimes \Bigl(\sum^n_{l=1}g_{il}e_j\Bigr) \\
        &= \sum_{i=1}^n\sum_{j=1}^n\sum_{l=1}^n g^{-1}_{ji}g_{il}(e^j\otimes e_l)\\
        & = \sum^n_{j=1}\sum^n_{l=1}(\delta_{jl}(e^j\otimes e_l)) = \sum^n_{i=1} e^i\otimes e_i
    \end{split}
    \end{equation*}
    We could have proved the same by considering right duals objects, hence $\Rep(G)$ is a rigid category.
    
\end{itemize}
\end{examples}

\begin{definition}
    Let $(\mathcal{C}, \otimes, I)$ and $(\mathcal{D}, \boxtimes, J)$ be monoidal categories.
    A \textbf{monoidal functor} from $\mathcal{C}$ to $\mathcal{D}$ is a functor $F$ together with a natural transformation $\Omega$ of functors from $\mathcal{C}\times \mathcal{C}$ to $\mathcal{D}$ given a pair $(X,Y)$ of objects in $\mathcal{C}$ by 
    $\Omega_{X,Y} : F(X\otimes Y)\rightarrow F(X) \boxtimes F(Y)$
    such that the following diagram commute for all $X,Y,Z\in \mathcal{C}$,
    \[
    \begin{tikzcd}[row sep = large, column sep = huge]
        F((X\otimes Y)\otimes Z) \arrow[r, "F(\Phi_{X,Y,Z})"] \arrow[d, "\Omega_{X\otimes Y, Z}"] & 
        F(X\otimes (Y\otimes Z))
        \arrow[d, "\Omega_{X,Y\otimes Z}"]\\
        F(X\otimes Y) \boxtimes F(Z)
        \arrow[d,"\Omega_{X,Y}\boxtimes \id_{F(Z)}"] & 
        F(X)\boxtimes F(Y\otimes Z)
        \arrow[d, "\id_{F(X)}\boxtimes \Omega_{Y,Z}"]\\
        (F(X)\boxtimes F(Y))\boxtimes F(Z)
        \arrow[r,"\Psi_{F(X),F(Y),F(Z)}"'] &
        F(X)\boxtimes (F(Y)\boxtimes F(Z))
    \end{tikzcd}
    \]
    Moreover, we require that there is a morphism $\Omega_0 : J\rightarrow F(I)$ such that the following diagrams commute : 
    \[
    \begin{tikzcd}[column sep = large]
        F(X)\boxtimes J
        \arrow[r,"\id_{F(X)}\boxtimes \Omega_0"] 
        \arrow[d, "\cong"'] & 
        F(X)\otimes F(I) 
        \arrow[d,"\Omega_{X,I}"]\\
        F(X) 
        \arrow[r,"\cong"'] &  
        F(X\otimes I) 
    \end{tikzcd}
    \hspace{1cm}
    \begin{tikzcd}[column sep = large]
        J \boxtimes F(X)
        \arrow[r," \Omega_0 \boxtimes \id_{F(X)}"] 
        \arrow[d, "\cong"']& 
        F(I)\otimes F(X) 
        \arrow[d,"\Omega_{I,X}"]\\
        F(X) 
        \arrow[r,"\cong"'] &  
        F(I\otimes X) 
    \end{tikzcd}
    \]
    We say that a monoidal functor is \textbf{strong} if $\Omega_0$ and $\Omega_{X,Y}$ are isomorphisms for all $X,Y\in \mathcal{C}$, and that it is \textbf{strict} if those maps are identities.
\end{definition}

\begin{example} \textbf{ }
\begin{itemize}
    \item[$a)$] Let $G$ be a group. The forgetful functor $U:\Rep_k^{\text{f.d.}}(G)\rightarrow \Vect_k$ is monoidal.\\
    And the maps $\Omega_{V,W}: U(V\otimes W,\alpha\otimes\beta) \xrightarrow[]{\sim} U(V,\alpha)\otimes U(W,\beta)$ and $k\mapsto U(k,\epsilon) =k$ are identities making it a strict monoidal functor.
\end{itemize}
\end{example}

\begin{definition}
    Let $(F,\Omega)$ and $(G, \Lambda)$ be monoidal functors from $(\mathcal{C},\otimes,I)$ to $(\mathcal{D},\boxtimes,J)$.
    A \textbf{monoidal natural transformation} $\alpha:F\rightarrow G$ is a natural transformation such that for $X,Y\in \mathcal{C}$, we have the following compatibility conditions,
    \[
    \begin{tikzcd}
        F(X\otimes Y) 
        \arrow[r, "\alpha_{X\otimes Y}"] \arrow[d, "\Omega_{X,Y}"'] & 
        G(X\otimes Y) 
        \arrow[d, "\Lambda_{X,Y}"']\\
        F(X)\boxtimes F(Y) 
        \arrow[r,"\alpha_X\boxtimes \alpha_Y"'] & 
        G(X) \boxtimes G(Y) 
    \end{tikzcd}
    \hspace{1cm}
    \begin{tikzcd}
        F(I) 
        \arrow[rr, "\alpha_I"] \arrow[dr,"\cong"']  && 
        G(I) 
        \arrow[dl,"\cong"]\\
        & J
    \end{tikzcd}
    \]
    We say that two monoidal categories are \textbf{monoidally equivalent} if there exists an equivalence of categories that is a strong monoidal functor.
\end{definition}
\begin{remark}
    If a monoidal functor is an equivalence of categories, then it is automatically a strong monoidal functor
\end{remark}

\noindent The next theorem is Mac Lane's coherence theorem. It can be understood as a strictification result and the original statement along with its original proof can be found in \cite{McLane_Sau}, Chapter $VII.\:2$ of 
\begin{theorem}
\label{Mac_Lane_coherence}
    Any monoidal category is monoidally equivalent to a strict monoidal category.
\end{theorem}

\noindent Some consequences of Mac Lane's coherence theorem are the following proposition.

\begin{proposition}
    Strong monoidal functors between rigid monoidal categories preserve duals.    
\end{proposition}
\begin{proof}
    Let $(F,\Omega_{\_,\_},\Omega_0) : (\mathcal{C},\otimes, I) \rightarrow (\mathcal{D},\boxtimes, J)$ be a strong monoidal functor between rigid categories.
    Due to Mac Lane's coherence theorem, we can freely restrict ourselves to the case of strict monoidal categories.
    Consider the following evaluation and coevaluation maps for $F(X)$ and $F(X^*)$:
    \[
    \begin{tikzcd}
        \ev_{\mathcal{D}} : F(X)\boxtimes F(X^*)
        \arrow[r,"\Omega_{X,X^*}"] &
        F(X\otimes X^*) 
        \arrow[r,"F(\ev_{\mathcal{C}})"] &
        F(I)
        \arrow[r,"\Omega_0^{-1}"]&
        J\\
        \coev_{\mathcal{D}} : J 
        \arrow[r,"\Omega_0"]&
        F(I)
        \arrow[r,"F(\coev_{\mathcal{C}})"]&
        F(X^*\otimes X) 
        \arrow[r,"\Omega_{X^*,X}^{-1}"]&
        F(X^*)\boxtimes F(X)
    \end{tikzcd}
    \]
    Then for all $X\in \mathcal{C}$, with $X^*$ its dual, the conditions $$ (\id_{F(X^*)}\boxtimes \ev_{\mathcal{D}}) \circ (\coev_{\mathcal{D}} \boxtimes \id_{F(X^*)}) : J\boxtimes F(X^*) \xrightarrow[]{\sim} F(X^*)\boxtimes J$$ and 
    $$(\ev_{\mathcal{D}} \boxtimes\id_{F(X)})\circ (\id_{F(X)}\boxtimes \coev_{\mathcal{D}}) : F(X)\boxtimes J \xrightarrow[]{\sim} J\boxtimes F(X)$$
    are satisfied since such conditions hold as for evaluation and coevaluation for $X$ and $X^*$ in $\mathcal{C}$
\end{proof}

\begin{proposition}
\label{rigid_monoidal_are_iso}
    A natural transformation of strong monoidal functors between rigid monoidal categories is always an isomorphism.
\end{proposition}
\begin{proof}
    Let $F$ and $G$ be two strong monoidal functors as in the previous proof, then we have $F(X^*)\cong F(X)^*$ and $G(X^*)\cong G(X)^*$ for every $X\in\mathcal{C}$.
    Consider $\alpha : F\rightarrow G$ a monoidal natural transformation.
    The morphism $\alpha_{X^*}$ induces a morphism $F(X)^*\rightarrow G(X)^*$ which has a transpose $G(X)\rightarrow F(X)$.
    By applying this to all $X$, we obtain a natural transformation $\alpha^t:G\rightarrow F$.\\
    By considering the following compatibility commutative diagram,
    \[
    \begin{tikzcd}[column sep = large]
        F(Y^*\otimes X^*)
        \arrow[r, "\alpha_{(X\otimes Y)^*}"]
        \arrow[d] &
        G(Y^*\otimes X^*)
        \arrow[d]\\
        F(Y^*)\boxtimes F(X^*)
        \arrow[r, "(\alpha_{X^*}\boxtimes\alpha_{Y^*})^*"'] &
        G(Y^*)\boxtimes G(X^*)
    \end{tikzcd}
    \]
    by transposing the maps, considering the isomorphism $(X\otimes Y)^*\cong Y^*\otimes X^*$ and noticing that the dual of the identity element is itself, we obtain the compatibility conditions for $\alpha^t$ to be a monoidal natural transformation.\\
    Moreover, since $\alpha^t$ is constructed by considering transposes of $\alpha$ on the dual of $X$, $(\alpha_{X^*})^t\circ \alpha_X = \id_{F(X)}$ and $\alpha_{X^*}\circ (\alpha_{X})^t = \id_{G(X)^*}$ for all $X\in\mathcal{C}$.
    Making $\alpha$ a natural isomorphism.
\end{proof}


\noindent Since finite-dimensional representation of an affine group scheme over $k$ are equivalent to finite comodules over a commutative Hopf algebra.
For $G$ an affine group scheme, we can consider the fibre functor $\omega$ above as a functor from the category $\Rep^{\text{f.d.}}(G)$ onto $\Vect^{\text{f.d.}}_k$.
And $\omega$ induces a monoidal functor $\omega\otimes R :  V\rightarrow V\otimes_k R$
\begin{definition}
\label{automorphism_fibre_functor}
    We define three set-valued functors, 
    $\End(\omega): \Alg_k\rightarrow \SET$ sending $k$-algebras $R$ to the sets of $R$-linear natural transformations $\Hom_R(\omega\otimes R,\omega\otimes R)$,
    $\End^{\otimes}(\omega)$ sending $R$ to the set monoidal natural transformation $\Hom^{\otimes}_R(\omega\otimes R,\omega\otimes R)$, 
    and $\Aut^{\otimes}(\omega)$ sending $R$ to the set of monoidal natural isomorphisms $\Isom^{\otimes}_R(\omega\otimes R,\omega\otimes R)$
\end{definition}

\begin{remark}
    Notice that $\Isom^{\otimes}_R(\omega\otimes R,\omega\otimes R)$ is actually carrying a group structure, making $\Aut^{\otimes}(\omega)$ a group-valued functor.
    Indeed, the group operation is defined by composition, component wise,
    for $\alpha,\beta \in \Isom_R^{\otimes}(\omega\otimes R,\omega\otimes R)$, $\alpha_M\circ\beta_M = (\alpha\circ\beta)_M$.\\
    The naturality condition for an element $\alpha\circ\beta$ still holds:
    $$\alpha_N\circ\beta_N\circ(f\otimes\id_A) = \alpha_N\circ(f\otimes \id_A)\circ\beta_M= (f\otimes\id_A)\circ\alpha_M\circ\beta_M$$
    The identity natural transformation $\id_{\omega\otimes R}$ is the identity element, for each $M$, $(\id)_M = \id_{M\otimes A}$
    $$\id_N\circ(f\otimes\id_A) = (f\otimes\id_A)\circ\id_M$$
    And since each component $\alpha_M$ is an isomorphism, it has a unique inverse linear map $(\alpha_M)^{-1}$.
    $$\alpha_N^{-1}\circ(f\otimes\id_A) = (f\otimes \id_A)\circ \alpha_M^{-1}$$
\end{remark}

\begin{proposition}
\label{automorphism_group_scheme}
    Let $G$ be an affine group scheme.
    Then there is a canonical isomorphism of group-valued functors $G\xrightarrow{\sim} \Aut^{\otimes}(\omega)$. Consequently, $\Aut^{\otimes}(\omega)$ is an affine group scheme.
\end{proposition}
\begin{proof}
    We first check that we have $\End(\omega)(R)\cong\Hom_{\Mod_R}(A\otimes_kR,R)$ functorially in $R$, for $A$ the coordinate ring of $G$.
    Notice that we have an isomorphism 
    $\Hom_{\Mod_R}(A\otimes_kR,R) \cong \Hom_{\Vect_k}(A,R)$
    induced by composition with extension of scalars $\varphi:A\rightarrow A\otimes_kR,\:a\mapsto a\otimes 1$ and restriction of scalars $\psi :A\otimes_kR\rightarrow R,\:a\otimes r\mapsto ar$.\\
    Similarly, we have the following isomorphism $\Hom_R(\omega\otimes R,\omega\otimes R) \cong \Hom_k(\omega,\omega\otimes R)$,
    induced by morphisms $\Phi: \Hom_R(\omega\otimes R,\omega\otimes R) \rightarrow \Hom_k(\omega, \omega\otimes R)$ defined as $\Phi(\alpha_V):\omega(V)\rightarrow\omega(V)\otimes R$ with $\Phi(\alpha_V)(v):= \alpha_V(v\otimes 1)$ for $\alpha\in \Hom_R(\omega\otimes R,\omega\otimes R)$,
    and $\Psi:\Hom_k(\omega, \omega\otimes R)\rightarrow \Hom_R(\omega\otimes R,\omega\otimes R)$ defined as $\Psi(\beta_V):\omega(V)\otimes R\rightarrow \omega(V)\otimes R$ with $\Psi(\beta_V)(v\otimes r) :=:=\beta_V(v)\cdot r$ for $\beta \in \Hom_k(\omega, \omega\otimes R)$.\\
    And, by Lemma \ref{useful_lemma_underlying_coalgebra}, we have $\FEnd(\omega)(R) \cong \Hom_{\Mod_R}(A\otimes_kR, R)$.\\
    Notice that we have shown that $\End(\omega)(R)$ and $G(R)$ are functorially isomorphic in $R$ as sets.\\
    Next, we show that elements in $\End(\omega\otimes R)$ who are moreover monoidal correspond to $k$-algebra homomorphisms $A\rightarrow R$.
    By definition, a morphism $\varphi_R \in \End(\omega\otimes R)$ is a monoidal natural transformation, if for all $V,W\in \Rep^{\text{f.d.}}(G)$, the following diagram commutes 
    \[
     \begin{tikzcd}
        \omega(V\otimes W) \otimes R 
        \arrow[r, "\varphi_{R}"] \arrow[d, ""] & 
        \omega(V\otimes W) \otimes R 
        \arrow[d, ""]\\
        \omega(V)\otimes\omega(W)\otimes R
        \arrow[r,"\varphi_R \otimes \varphi_R"'] & 
        \omega(V)\otimes\omega(W)\otimes R
    \end{tikzcd}
    \]
    where the vertical arrows are the isomorphisms $\Omega_{V,W}$ in the definition of monoidal functors.\\
    Using our discussion above, by Propositions \ref{useful_lemma_underlying_coalgebra} and \ref{useful_lemma_underlying_coalgebra_2}, the $R$-linear maps $\varphi_R$ and $\varphi_R\otimes \varphi_R$ correspond to the $k$-linear maps $\lambda:A\rightarrow R$ and $\lambda\circ m: A\otimes_kA\rightarrow R$ respectively, where $m:A\otimes_kA\rightarrow A$ is the multiplication of $A$ (a $k$-algebra).
    Then replacing the maps in the above diagrams and the corresponding objects, the resulting commutative diagram implies the condition for $\lambda$ to be a $k$-algebra homomorphism.
    Thus $\End^{\otimes}(\omega)(R) \cong \Hom_{\Alg_k}(A,R)$ .\\
    And finally, by Proposition \ref{rigid_monoidal_are_iso}, $\Aut^{\otimes}(\omega) \cong \End^{\otimes}(\omega)$.
\end{proof}

\subsubsection{Tannaka-Krein Duality}

Before finishing this section, we have to talk about the Tannaka-Krein Duality.
This result can be understood as a reconstruction result. 
Usually, its statement holds for any $k$-linear abelian category (see Remark \ref{remark_abelian_category}), and it automatically gives that $\Comod^{\text{f.d.}}$ is a $k$-linear abelian category (which we have already reviewed, see \ref{comodules_abelian}). Thus, below we state a restriction of the famous Tannaka Duality.

\begin{theorem}[Tannaka Duality]
\label{Tannaka_duality}
    Let $C$ be a $k$-coalgebra for $k$ a field and $\omega:\Comod^{\text{f.d.}}_C\rightarrow \Vect^{\text{f.d.}}_k$ the fibre functor. We have the following equivalences :
    \begin{itemize}
        \item[$a)$] The category $\Comod^{\text{f.d.}}_C$ is monoidal category and $\omega$ is a monoidal functor if and only if there is a canonical structure of $k$-bialgebra on $C$.
        \item[$b)$] $\Comod^{\text{f.d.}}_C$ is moreover a rigid monoidal category if and only if $C$ has the structure of a Hopf algebra.
        \item[$c)$] $\Comod_C^{\text{f.d.}}$ is a rigid symmetric monoidal category and $\omega$ is a  symmetric monoidal functor if and only if $C$ has the structure of a commutative Hopf algebra.
    \end{itemize}
\end{theorem}
\begin{proof}
    See \cite{Schauenburg} Chapter $2.3$ and $2.4$.\\ 
\end{proof}

\noindent We say that the above functor $\omega$ is \textit{symmetric monoidal} if the following diagram commutes 

\[
\begin{tikzcd}
    \omega(V\otimes W)
    \arrow[r,"\Omega_{V,W}"]
    \arrow[d, "\omega(\gamma'_{V,W})"']
    &
    \omega(V)\otimes \omega(W)
    \arrow[d, "\gamma_{\omega(V),\omega(W)}"] \\
    \omega(W\otimes V) \arrow[r,"\Omega_{W,V}"'] & 
    \omega(W)\otimes \omega(V),
\end{tikzcd}
\]
where $\gamma$ and $\gamma'$ are respectively, the symmetry in $\Comod_C^{\text{f.d.}}$ and the flip map in $\Vect^{\text{f.d.}}_k$.

\subsection{Neutral Tannakian Categories and Tannakian Fundamental Groups}

We are finally able to introduce \textit{Tannakian categories} and the \textit{Tannakian fundamental group}.

\begin{definition}
    A \textbf{neutral Tannakian category} over a field $k$ is a rigid, $k$-linear, abelian, symmetric monoidal category $\mathcal{C}$ whose unit $1$ satisfies $\End(1) \cong k$, and is moreover equipped with an exact faithful tensor functor $\omega :\mathcal{C}\rightarrow \Vect^{f.d.}_k$. The functor $\omega$ is called a (neutral) \textbf{fibre functor}
\end{definition}

\begin{example}
    Let $G$ be an affine group scheme over $k$.
    The category $\Rep^{\text{f.d.}}(G)$ of finite-dimensional representations of $G$ with its usual tensor product and the forgetful functor $\omega:\Rep^{\text{f.d.}}(G)\rightarrow \Vect^{\text{f.d.}}_k$ as the fibre functor is a neutral Tannkian category.
\end{example}

\noindent The converse is a consequence of Tannaka duality \ref{Tannaka_duality} and we get the following theorem,

\begin{theorem}
    Neutral Tannakian categories over $k$ are equivalent to categories of finite-dimensional representations of an affine group schemes over $k$.
\end{theorem}
\begin{proof}
    This is a direct application of Theorem \ref{Tannaka_duality} and Proposition \ref{representations_iso_comodules}.
\end{proof}

\begin{definition}
    Consider $(\mathcal{C},\omega)$ a Tannakian category. The associated functor $\Aut^{\otimes}(\omega)$ is called the \textbf{Tannakian fundamental group} of $(\mathcal{C},\omega)$
\end{definition}

\noindent Notice that with the previous theorem and Proposition \ref{automorphism_group_scheme}, the Tannakian fundamental group is always isomorphic to some group scheme as a functor.\\

\noindent Our next goal is to give a result identifying somewhat quotients of Tannakian fundamental groups with affine subgroup schemes of $\GL_n$.
But first, let us give two constructions:
\begin{construction}
\label{construction_algebraic_hull}
    Let $G$ be an arbitrary group and $k$ a field. 
    Then the category of finite-dimensional representations of such an arbitrary group is still a neutral Tannakian with its usual forgetful functor.
    The Tannakian fundamental group of its category of finite-dimensional representations together with the forgetful functor is usually called the \textit{algebraic hull} .
    One can also consider the algebraic hull of a topological group by taking its category of finite-dimensional continuous representations.
\end{construction}

\noindent The next construction tells us how to consider quotient of the Tannakian fundamental group,
\begin{construction}
    Given a Tannakian category $(\mathcal{C},\omega)$ and an object $X\in \mathcal{C}$, we denote $\langle X\rangle_{\otimes}$ the \textit{smallest full Tannakian subcategory of $\mathcal{C}$ containing $X$}.
    Its objects are subquotients of finite direct sums of objects of the form $X^{\otimes r}\otimes (X^*)^{\otimes s}$ for some $r,s\geq 0$.
\end{construction}

\noindent And we have the following result,

\begin{proposition}
\label{Tannaka_subcat_Zariski_closure}
    Let $G$ be an arbitrary group, $k$ a field and $(V,\rho)$ finite-dimensional representation of $G$ over $k$.
    Then the Tannakian fundamental group of the full Tannakian subcategory $\langle V\rangle_{\otimes}$ of $\Rep^{\text{f.d.}}_k(G)$ is canonically isomorphic to the Zariski closure of the image of $\rho$ in $\GL_n$.
\end{proposition}

\noindent Notice that $\overline{\im\:\rho}$ is actually an affine group scheme in $\GL_n$.
It is a closed subvariety and since $G$ is a group, using the fact $\rho(\overline{G})\subset \overline{\rho(G)}$, then $\overline{\rho(G)}$ has a group structure.\\
In order to prove the above proposition, we need a lemma that is equivalent to the following Theorem in \cite{waterhouse}
\begin{theorem}[Waterhouse]
    Let $k$ be a field and $G$ a closed subgroup scheme of $\GL_n$. Every finite-dimensional representation of $G$ can be reconstructed from its original representation on $k^n$ by the process of forming tensor products, direct sums, subrepresentations, quotients and duals.
\end{theorem}
\begin{proof}
    See \cite{waterhouse} Chapter $3.5$.
\end{proof}

\begin{corollary}
    Let $k$ be a field and $G$ a closed subgroup scheme of $\GL_n$ over $k$.
    Denote by $X$ the $n$-dimensional representation of $G$ corresponding to its action on $k^n$ via the inclusion $G\rightarrow \GL_n$.
    Then the full Tannakian subcategory $\langle X\rangle_{\otimes} \subset \Rep_k(G)$ is the whole of $\Rep_k(G)$.
\end{corollary}
\begin{proof}
    Since Tannakian categories are rigid monoidal and abelian,
    they are closed under taking the tensor product, dual, direct sum and taking quotients and subrepresentations since both last are instance of kernels and cokernels.
\end{proof}

\begin{proof}[Proof of Proposition \ref{Tannaka_subcat_Zariski_closure}]
    Denote $G$ the affine group scheme over $k$ defined by $\overline{\im\:\rho}$ in $\GL_n$.
    The representation $\rho:V\rightarrow \GL_n(V)$ give rise to a representation of the affine group scheme $G$ in a natural way by taking the inclusion $G\rightarrow \GL_n(V)$, we denote it by $\overline{V}$.
    Similarly, any object of $\langle V\rangle_{\otimes}$ give rise to an object of $\Rep_G$ and e obtain an equivalence between $\langle V\rangle_{\otimes}$ and the full Tannakian subcategory $\langle \overline{V}\rangle_{\otimes}$ of $\Rep_G$.
    In particular the two categories have isomorphic Tanakian fundamental group.
    But by the above corollary, $\langle\overline{V}\rangle_{\otimes}$ is the whole of $\Rep_G$, proving our proposition.\\
\end{proof}

\begin{example}
    One interesting example of Tannakian categories, is that of complex local systems.
    Consider $X$ a connected and locally simply-connected topological space and fix $x\in X$ a base point.
    The category $\LS_X$ of local systems on $X$ is a $\mathbb{C}$-linear abelian category (since $\LCS_{\mathbb{C}}(X) \simeq \Mod_{\mathbb{C}[\pi_1(X,x)]}$, and a category of $R$-modules is abelian). 
    With the usual tensor product and dual construction from linear algebra, $\LS_X$ is endowed with a symmetric rigid monoidal structure.
    Taking the stalk of a local system at $x$ yields a fibre functor with values in $\Vect^{\text{f.d.}}_{\mathbb{C}}$.
    This means that $\LS_X$ with the functor taking stalks of local systems at $x$ is a neutral Tannakian category.\\
    Recall that in Corollary \ref{local_systems_representations} we have established an equivalence of categories between complex local systems on $X$ and finite-dimensional representations of the fundamental group $\pi_1(X,x)$. 
    Notice that this is an equivalence of (neutral) Tannakian categories, thus the Tannakian fundamental group of $\LS_X$ is isomorphic to the algebraic hull of $\pi_1(X,x)$ over $\mathbb{C}$ as in Construction \ref{construction_algebraic_hull}.\\
    Finally, using our previous Proposition \ref{Tannaka_subcat_Zariski_closure}, the Tannakian fundamental group of each full Tannakian subcategory $\langle \mathcal{S}\rangle_{\otimes}$ of $\LS_X$ is isomorphic to the Zariski closure of the image of $\rho_{\mathcal{S},x} : \pi_1(X,x)\rightarrow \GL_n(\mathbb{C})$.
\end{example}

\subsubsection{Nori's Fundamental Group Scheme}

\noindent We end this chapter with a correspondence between the above Tannakian formalism and the \'etale fundamental group of schemes from the previous chapter presented by Nori in his PhD thesis \cite{Nori}.
For $S$ a proper integral scheme over a field $k$ and $s:\Spec\: k\rightarrow S$ a $k$-rational point, Nori showed that 
the full subcategory of essentially finite vector bundles over $S$ denoted $\EF_S$ of the category of locally free sheaves on $S$ spanned by the essentially finite sheaves together with the usual tensor product of sheaves and the functor $\mathcal{E}\rightarrow s^*\mathcal{E}$, is a neutral Tannakian category over $k$
(see section $2.3$, page $83$, in \cite{Nori}).

\begin{definition}
    For $S$ and $s$ as above, we define the \textbf{Nori fundamental group scheme} of $S$ at the base point $s$ to be the Tannakian fundamental group of the neutral Tannakian category $(\EF_S, s^*)$ mentionned above.
    We denote it by $\pi_1^N(S,s)$.
\end{definition}

\begin{proposition}
\label{Nori_invers_limit}
    The fundamental group scheme $\pi_1^N(S,s)$ is an inverse limit of finite $k$-group schemes.
\end{proposition}
\begin{proof}
    See \cite{Nori}, Conclusion $I.\:3.$
\end{proof}

\noindent The majority of \cite{Nori} focuses on what happens to that fundamental group, in the case where $S$ is defined over fields of positive characteristic.
However, what is really surprising is that over a field of characteristic $0$, we can compare the Nori fundamental group scheme and the \'etale fundamental group scheme.
Indeed, this is due to the fact that in characteristic $0$, every finite group scheme is \'etale and combining this with the above Proposition \ref{Nori_invers_limit}, gives us the following corollary.

\begin{corollary}
    When $k$ is an algebraically closed field of characteristic $0$.
    For $S$ a proper integral scheme over $k$ and $s:\Spec\:k \rightarrow S$ a $k$-rational point, there is a canonical isomorphism
    $\pi_1^N(S,s) \xrightarrow[]{\sim} \pi_1^e(S,\overline{s})$ for each $k$-valued geometric point $s = \overline{s}$ of $S$.
\end{corollary}

\appendix

\newpage
\section{Selected Results from Commutative Algebra and Groups Theory}

For this thesis, most results from (non-)commutative algebra that can be found in \cite{Lang_algebra} and \cite{Eisenbud} are assumed to be known by the reader. 
However, for the sake of keeping this document self-contained, we recall a few important and necessary results.

\begin{theorem}
\label{Lang_1_10}
    Let $A$be a subring of $B$, let $\mathfrak{p}$ be a prime ideal in $A$ and assume that $B$ is integral over $A$. Then $\mathfrak{p}B\neq B$ and there exists a prime ideal $\mathfrak{q}$ in $B$ lying above $\mathfrak{p}$
\end{theorem}
\begin{proof}
    See Proposition $1.10$ in Chapter $VII, \S1$ of \cite{Lang_algebra}.
\end{proof}

\noindent The following lemmas are the famous Nakayama Lemmas. 
They have many different statements, we give two of them here and the latter can be seen as a consequence of the former.
\begin{lemma}[Nakayama]
\label{Nakayama_first_lemma}
     Let $I$ be an ideal in $R$, and $M$ a finitely generated module over $R$.
    If $IM=M$, then there exists $r\in R$ with $r\equiv 1 \mod{I}$ such that $rM=0$.
\end{lemma}

\begin{lemma}[Nakayama]
\label{Nakayama_lemma}
    Let $I$ be an ideal in $R$, and $M$ a finitely generated $R$-module. 
    If $I=IM$, then there exists an $i\in I$ such that $im = m$ for all $m\in M$. 
\end{lemma}

\subsection{Krull Dimension and Artinian Rings}

\begin{definition}
    The \textbf{Krull dimension} of a commutative ring $R$ is the supremum of the lengths of all increasing chain of prime ideals in $R$
\end{definition}

\begin{definition}
    A ring $R$ is called \textbf{Noetherian} if there is no infinite increasing sequence of ideals $I_1\subsetneq I_2 \subsetneq \dots\subsetneq I_n \subsetneq \dots$, while it is called \textbf{Artinian} if there is no infinite descending sequence of ideals $I_1\supsetneq I_2 \supsetneq \dots\supsetneq I_n \supsetneq \dots$.
\end{definition}

\begin{example} \textbf{ }
\label{f_d_k_alg_artinian}
\begin{itemize}
    \item[$a)$] Let $k$ be a field, then finite-dimensional $k$-algebras are Artinian.
    \item[$b)$] The ring of polynomials with finite variables over a field, such as $k[x_1,\dots , x_n]$ is Noetherian.
    \end{itemize}
\end{example}

\noindent In the following we wish to show that ant Artinian ring has zero Krull dimension.

\begin{lemma}
Let $A$ be an Artinian ring which is also an integral domain. Then $A$ is a field.
\end{lemma}
\begin{proof}
    Let $x\in A$ be non-zero. We hae a descending chain $1\supseteq (x) \supseteq (x^2) ...$, which must stabilise, so $(x^n) = (x^{n+1})$ for some $n$.
    This implies that $x^n = yx^{n+1}$ for some $y\in A$. 
    Since $A$ is an integral domain and $x\neq 0$, we can cancel $x^n$ to get $xy = 1$ making $x$ a unit. 
    This holds for all non-zero $x\in A$, hence the ring $A$ is a field.
\end{proof}
\begin{lemma}
    In an Artinian ring, every prime ideal is maximal.
\end{lemma}
\begin{proof}
    If $A$ is Artinian and $\mathfrak{p}\subset A$ is a prime ideal, then $A/\mathfrak{p}$ is still Artinian.
    Since $A/\mathfrak{p}$ is an integral domain, by the previous lemma it is a field, so $\mathfrak{p}$ is actually a maximal ideal.
\end{proof}
\begin{proposition}
    A ring $A$ has Krull dimension zero if and only if every prime ideal is maximal.
\end{proposition}
\begin{proof}
    If $\mathfrak{p}\subset A$ is a non-maximal prime ideal, then there is some maximal ideal $\mathfrak{m}$ with $\mathfrak{p}\subsetneq\mathfrak{m}$, giving a chain of length $1$, and so $\dim A \geq 1$.
    Conversely, if $\dim A \geq 1$, then there is some chain of prime ideals $\mathfrak{p}_0 \supsetneq \mathfrak{p}_1\supsetneq ...$, whereby $\mathfrak{p}_1$ is not a maximal prime ideal. 
\end{proof}

\begin{corollary}
\label{Artin_zero_dim}
    Artinian rings have zero Krull dimension.
\end{corollary}

\subsection{Classification of Transitive $G$-Spaces}

The result we mention in this subsection is used extensively in the first chapter.
It is quite powerful as it gives a classification of transitive $G$-spaces by means of stabiliser.
\begin{theorem}\textbf{ }
\label{Cameron_1_3}
    \begin{itemize}
        \item[$a)$] Let $\Omega$ be a transitive $G$-space. Then $\Omega$ is isomorphic to the coset space $G/H$ where $H = G_{\alpha}$, the stabiliser for $\alpha\in\Omega$.
        \item[$b)$] Two coset spaces $G/H$ and $G/K$ are isomorphic if and only if $H$ and $K$ are conjugate subgroups of $G$.
    \end{itemize}
\end{theorem}
\begin{proof}
    See Theorem $1.3$ of \cite{Cameron_groups}.\\
\end{proof}

\section{Review of Category Theory}

As category theory provides the fundamental mathematical framework for this work and to ensure a self-contained treatment, we provide below a brief overview of the category-theoretic foundations utilised throughout this document.

\subsection{Definitions and Basic Properties}

\begin{definition}
    A \textbf{category} consists of the following data:
    \begin{itemize}
        \item[$a)$] a class $|\mathcal{C}|= ob(\mathcal{C}) = \mathcal{C}$ of objects, denoted by $X,Y,Z,...$,
        \item[$b)$] for any two objects $X,Y$, a set $\Hom_{\mathcal{C}}(X,Y)$ of morphisms,
        \item[$c)$] for any three objects $X,Y,Z$ a composition law for the morphisms : \\$\Hom(X,Y) \times \Hom(Y,Z) \rightarrow \Hom(X,Z)$, $(f,g) \rightarrow g\circ f$, and
        \item[$d)$] for any object $X$ a unit morphism on $X$, denoted by $1_X$ or $X$ for short.
    \end{itemize}
   
    These data are subjected to the following compatibility conditions :
    \begin{itemize}
        \item[$1)$] for all objects $X,Y,Z,U$, and all morphisms $f \in \Hom(X,Y), g\in \Hom(Y,Z)$ and $h\in \Hom(Z,U)$, we have $h\circ (g \circ f) = (h\circ g)\circ f$, and
        \item[$2)$] for all objects $X,Y,Z$, and all morphisms $f \in \Hom(X,Y)$ and $g \in \Hom(Y,Z)$, we have $1_Y \circ f= f $ and $g\circ 1_Y= g$.
    \end{itemize}
\end{definition}

\begin{definition}
    A (covariant) \textbf{functor} $F$ between two categories $\mathcal{C}$ and $\mathcal{D}$ consists of a rule $X \mapsto F(X)$ on objects and a map on sets of morphisms $\Hom_{\mathcal{C}}(X, Y)\rightarrow \Hom_{\mathcal{D}}(F(X), F(Y))$ which sends identity morphisms to identity morphisms and preserves composition. And a contravariant functor $F :\mathcal{C}\rightarrow\mathcal{D}$ is a covariant functor $F :\mathcal{C}^{op}\rightarrow\mathcal{D}$.
\end{definition}

\begin{definition}
    Let $F,G : \mathcal{C} \rightarrow\mathcal{D}$ be two functors. A \textbf{natural transformation} or morphism of functors $\alpha : F \rightarrow G$ $($sometimes denoted by $\alpha : F \Rightarrow G)$, assigns to every object $X \in \mathcal{C}$ a morphism $\alpha_X : F(X) \rightarrow G(X)$ in $\mathcal{D}$ rendering for every $f : X \rightarrow Y$ in $\mathcal{C}$ the following diagram in $\mathcal{D}$ commutative,
    \[
    \begin{tikzcd}
        F(X) \arrow[r,"F(f)"] \arrow[d, "\alpha_X"']
        & F(Y) \arrow[d,"\alpha_Y"] \\
        G(X) \arrow[r,"G(f)"']
        &  G(Y)
    \end{tikzcd}
    \]
    The morphism $\alpha$ is an \textbf{natural isomorphism} if each $\alpha_X$ is an isomorphism, in that case we write $\mathcal{C}\cong\mathcal{D}$.
\end{definition}

\begin{definition}
    A \textbf{subcategory} $\mathcal{D}$ of a category $\mathcal{C}$ consists of:
    \begin{itemize}
        \item[$a)$] a subclass $\mathcal{D}$ of the class of objects of $\mathcal{C}$, and
        \item[$b)$] for every pair $X, X'\in\mathcal{D}$, we have a subset $\Hom_{\mathcal{D}}(Y,Y')$ of $\Hom_{\mathcal{C}}(X,X')$, such that:\\ $1)$ for all $X,X',X''\in\mathcal{D}$, and all $f \in \Hom_{\mathcal{D}}(X,X')$ and $g \in \Hom_{\mathcal{D}}(X',X'')$, then $g\circ f \in \Hom_{\mathcal{D}}(X,X'')$, and\\
        $2)$ for all $X\in\mathcal{D}$, $1_X\in \Hom_{\mathcal{D}}(X,X)$.

    \end{itemize}
\end{definition}

\begin{definition}
    Consider a functor $F : \mathcal{C}\rightarrow\mathcal{D}$. For any two objects $X,Y\in\mathcal{C}$ consider the map of sets $F_{XY} : \Hom_{\mathcal{C}} (X,Y) \rightarrow \Hom_{\mathcal{D}}(F(X),F(Y))$ with  $F_{XY}(f) = F(f)$.
    The functor $F$ is called
    \begin{itemize}
        \item[-] \textbf{faithful} if and only if $F_{XY}$ is injective for all objects $X,Y\in\mathcal{C}$,
        \item[-] \textbf{full} if and only if $F_{XY}$ is surjective for all objects $X,Y\in\mathcal{C}$,
        \item[-] \textbf{fully faithful} if and only if $F_{XY}$ is bijective for all objects $X,Y\in\mathcal{C}$,
        \item[-] an \textbf{isomorphism of categories} when $F$ is fully faithful and $F$ induces a bijection on the classes of objects of $\mathcal{C}$ and $\mathcal{D}$, and
        \item[-] \textbf{essentially surjective} if every object $\mathcal{D}$ is isomorphic to some object of the form $F(X)$.
    \end{itemize}
\end{definition}

\noindent Clearly, a subcategory $\mathcal{D}$ of $\mathcal{C}$ is itself a category and there is a faithful functor $\mathcal{D}\rightarrow\mathcal{C}$. 
A subcategory is said to be \textit{full} if this functor is moreover a full functor.

\begin{definition}
\label{(pro)_representable}
    A functor $F:\mathcal{C}\rightarrow Set$ is said to be \textbf{representable} if there is an object $X\in\mathcal{C}$ and an isomorphism of functors $F\simeq \Hom_{\mathcal{C}}(X,\_)$. 
    $F: \mathcal{C} \rightarrow Set$ is said to be \textbf{pro-representable} if there exist an inverse system $(P_\alpha, \phi_{\alpha\beta})$ of objects of $\mathcal{C}$ indexed by a partially ordered set $I$, and a functorial isomorphism $\varinjlim \Hom (P_\alpha, X) \cong F(X)$ for each $X\in \mathcal{C}$ (it is a filtered limit).
\end{definition}

\noindent The following well-known lemma is of pivotal importance. Observe that if $X$ and $Y$ are objects of $\mathcal{C}$, every morphism $Y \rightarrow X$ induces a morphism of
functors $\Hom(X,\_) \rightarrow \Hom(Y,\_)$ via composition.

\begin{lemma}[Yoneda]
    If $F$ and $G$ are functors $\mathcal{C}\rightarrow \textbf{Sets}$ represented by objects $X\in\mathcal{C}$ and $Y\in\mathcal{D}$, respectively $($so $F\simeq \Hom_{\mathcal{C}}(X,\_)$ and $G\simeq \Hom_{\mathcal{C}}(Y,\_))$, every natural transformation $\phi : F \rightarrow G$ is induced by a unique morphism $Y \rightarrow X$ as above.
\end{lemma}
\begin{proof}
    See \cite{McLane_Sau}, chapter $III. \: 2.$
\end{proof}

\begin{definition}
    Let $\mathcal{C}$ be a small category. Denote $\Fun(\mathcal{C},\SET)$ the category of covariant functors from $\mathcal{C}$ to $\SET$ and $\CFun(\mathcal{C},\SET)$ the category of contravariant functors from $\mathcal{C}$ to $\SET$. Then the functor
    $$\mathcal{Y}^* : \mathcal{C}\rightarrow \Fun(\mathcal{C},\SET) \:, \: \mathcal{Y}^*(C) = \Hom_{\mathcal{C}}(C, \_) \: , \: \mathcal{Y}^*(f) = \Hom_{\mathcal{C}}(f,\_)$$
    is called the \textbf{contravariant Yoneda embedding functor}. And
    $$\mathcal{Y}_* : \mathcal{C}\rightarrow \CFun(\mathcal{C},\SET) \:, \: \mathcal{Y}_*(C) = \Hom_{\mathcal{C}}(\_,C) \: , \: \mathcal{Y}_*(f) = \Hom_{\mathcal{C}}(\_,f)$$
    is called the \textbf{covariant Yoneda embedding functor}.
\end{definition}

\begin{proposition}
\label{yoneda_embeddings}
    The Yoneda embedding functors are full and faithful.
\end{proposition}
\begin{proof}
    See \cite{McLane_Sau}, chapter $III. \: 2.$
\end{proof}

\subsection{Limits}
\begin{definition}
    Let $F : \mathcal{Z} \rightarrow\mathcal{C}$ be a covariant functor. A \textbf{cone} on $F$ is couple $(M,\mu)$, where $M\in \mathcal{C}$ and $\mu_Z \in \Hom_{\mathcal{C}}(M, F(Z))$ for every object $Z\in \mathcal{Z}$, such that $F(f) \circ \mu_Z = \mu_{Z'}$ for all morphisms $f : Z \rightarrow Z'$.\\
    A morphism of cones between the cones $(M,\mu)$ and $(N,\nu)$ on the functor $F : \mathcal{Z}\rightarrow\mathcal{C}$ is a morphism $f : M \rightarrow N$ in $\mathcal{C}$ such that $\nu_Z \circ f=\mu_Z$.
\end{definition}

\begin{definition}
    Let $F : \mathcal{Z}\rightarrow\mathcal{C}$ be a covariant functor. A \textbf{limit} of $F$ is a cone $(L,\lambda)$ on $F$, such that for any (other) cone $(T,\tau)$ on $F$, there exists a unique morphism of cones $u: (T,\tau) \rightarrow (L,\lambda)$. We denote this as $\Lim F = (L,\lambda)$.
\end{definition}

\noindent Similarly we define a cocones, morphisms of cocones and a colimits (by reversing the arrows).

\begin{theorem}
    If $(L,\lambda)$ and $(L',\lambda')$ are two limits of a functor $F : \mathcal{Z}\rightarrow\mathcal{C}$, then $(L,\lambda)\cong(L,\lambda')$ as cones. In other words, if the limit of a functor exists, it is unique up to isomorphism.
\end{theorem}
\begin{proof}
    See \cite{McLane_Sau}
\end{proof}

\noindent We give below a few examples of (co)-limits.

\subsubsection*{Filtered limits}

Let $(I,\leq)$ be a directed set, i.e. $I$ is a poset such that for all $i,j\in I$, there exist an element $k\in I$ such that if $i\leq k$ and $j\leq k$.
And let $(C_i, f_{ij})$ be a \textbf{direct system} of objects and morphisms in a category $\mathcal{C}$.
That is $\{C_i\mid i\in I\}$ is a family of objects indexed by $I$ in $\mathcal{C}$ and $f_{ij}:C_i\rightarrow C_j$ are morphisms in $\mathcal{C}$ for all $i\leq j$ such that $f_{ii}$ is the identity on $C_i$ and $f_{ik} = f_{jk}\circ f_{ij}$ for all $i\leq j \leq k$.
An \textbf{inverse system} is a direct system but with reversed arrows.

\begin{lemma}
\label{lemma_joost_direct_limit}
    Let $\mathcal{Z}$ be the category associated to the directed set $(I,\leq)$ and $\mathcal{C}$ be any category.
    There is a bijective correspondence between covariant functors $F : \mathcal{Z}\rightarrow \mathcal{C}$ and directed systems in $\mathcal{C}$.
    There is a bijective correspondence between contravariant functors $F : \mathcal{Z}\rightarrow \mathcal{C}$ and inverse systems in $\mathcal{C}$.
\end{lemma}

\begin{definition}
\label{def_direct_inverse_limit}
    Thus the \textbf{direct limit} of a covariant functor $F : \mathcal{Z}\rightarrow \mathcal{C}$ with corresponding direct system $(C_i, f_{ij})$ is the following colimit $(C, c_i:C_i\rightarrow C)$ :  
    \[
    \begin{tikzcd}[column sep=small]
        C_i \arrow[rr,"f_{ij}"] \arrow[dr, "c_i"'] \arrow[ddr, bend right, "t_i"'] && C_j \arrow[dl, "c_j"] \arrow[ddl, bend left, "t_j"]\\
        & C \arrow[d, dashed, "u"] &\\
        &  T 
    \end{tikzcd}
    \]
    It is denoted as $C=\varinjlim_{i\in I} C_i$.\\
    Similarly, the \textbf{inverse limit} of a contravariant functor $F : \mathcal{Z}\rightarrow \mathcal{C}$ with corresponding inverse system $(C_i, f_{ji})$ is the following limit $(L, l_i: L\rightarrow C_i)$ :
    \[
    \begin{tikzcd}[column sep=small]
        & T \arrow[d, dashed, "u"] \arrow[ddr, bend left, "t_i"] \arrow[ddl, bend right, "t_j"']&\\
        & L \arrow[dr, "l_i"] \arrow[dl, "l_j"'] &\\
        C_j \arrow[rr,"f_{ji}"']   && C_i 
    \end{tikzcd}
    \]
    It is denoted as $L=\varprojlim_{i\in I} C_i$.
\end{definition}

\begin{proposition}
\label{compact_inverse_system}
    An inverse system $(C_i,f_{ij})$ of non-empty compact topological spaces is non-empty
\end{proposition}
\begin{proof}
    Consider the subsets $C_{\alpha\beta} \subset \prod C_{i}$ consisting of sequences $(c_i)$ satisfying $f_{\alpha\beta}(c_{\beta}) = c_{\alpha}$ for a fixed pair $\alpha\leq \beta$.
    These are closed subsets of the product, and their intersection is precisely $\varprojlim C_i$.
    Furthermore, since the index set is directed, then finite intersections of the $C_{\alpha\beta}$ are non-empty.
    Since $\prod C_i$ is compact by Tychonoff's theorem (see \cite{Intro_manifolds}, Proposition $IV.\:4.36.$), it ensures that $\varprojlim C_i$ is non-empty.
\end{proof}

\begin{remark}
\label{direct_limit_modules}
    Direct and inverse limits are usually referred to as \textit{filtered colimit/limit} respectively. 
    Often filtered (co)limits preserve a lot of structure, notably, taking the direct limit of $R$-module gives again a $R$-module. This is checked in the following theorem.
\end{remark}

\begin{theorem}
\label{direct_limit_R_module}
    Let $R$ be a ring and let $\{M_i, f_{ij}\}$ be a directed system of $R$-modules over some directed set $(I\leq)$.
    Let $M=\bigoplus_{i\in I}M_i$ and for every $i\in I$, let $\rho_i:M_i\rightarrow M$ be the natural injection map.
    Define the $N$ the $R$-submodule of $M$ by 
    $$N = \langle \rho_jf_{ij}(x_i) - \rho_i(x_i) \mid i\leq j,x_i\in M_i\rangle.$$
    Then $\varinjlim M_i = M/N$.
\end{theorem}
\begin{proof}
    See \cite{Lang_algebra}, Theorem $III.\:10.1$
\end{proof}

\noindent An important construction we can make from filtered limits are \textit{pro-objects}. They allow one to work with systems that behave like a limit even if a limit does not exist in a category. 
We enumerate the following interesting pro-objects :

\begin{definition}
\label{profinite_group}
A \textbf{profinite group} is defined to be an inverse limit of a system of finite groups.
\end{definition}

\begin{definition}
\label{pro_representable}
    Let $\mathcal{C}$ be a category, and $F$ a set-valued functor on $\mathcal{C}$.
    We say that $F$ is \textbf{pro-representable} if there exists an direct system $(C_i , f_{ij})$ of objects and morphisms in $ \mathcal{C}$ indexed by the directed set $(I,\leq)$ and a functorial isomorphism $\varinjlim \Hom(C_i, X) \cong F(X)$ for each object $X\in \mathcal{C}$.
\end{definition}

\noindent A way of generalising the notion of directed system is by considering the directed system to be a category called \textit{(co)-filtered} in the following way :

\begin{definition}
    A category $\mathcal{J}$ is said to be \textbf{filtered} when 
    it is non-empty,
    for every two objects $J$ and $J'$ in $\mathcal{J}$, there exists an object $K$ and two morphisms $f:J\rightarrow K$ and $f':J'\rightarrow K$ in $\mathcal{J}$, and
    for every two morphisms $u,v:I\rightarrow J$ in $\mathcal{J}$, there exists an object $K$ and a morphism $w:J\rightarrow K$ such that $w\circ u = w\circ v$.
    Dually, a category $\mathcal{J}$ is said to be \textbf{cofiltered} if the opposite category $\mathcal{J}^{op}$ is filtered.
\end{definition}

\noindent The category $\mathcal{Z}$ in Lemma \ref{lemma_joost_direct_limit} and Definition \ref{def_direct_inverse_limit} is first a filtered category and secondly a cofiltered category and the functor $F:\mathcal{Z}\rightarrow \mathcal{C}$ is thus often called a filtered diagram and dually a cofiltered diagram.

\begin{theorem}
\label{cofinal_identic_colim}
    Let $\mathcal{K}$ be a subcatecory of a filtered category $\mathcal{J}$ such that it is a cofinal set.
    Suppose that the functors $F:\mathcal{J}\rightarrow \mathcal{C}$ and $F':\mathcal{K}\rightarrow \mathcal{C}$ have colimits.
    Then there is a canonical isomorphism between there colimits.
\end{theorem}
\begin{proof}
    See \cite{McLane_Sau}, Theorem $IX.3.1$ 
\end{proof}
\noindent A subset $I \subset J$ of a poset $J$ is said to be \textit{cofinal} if for every $j\in J$, there exists some $i\in I$ such that $j\leq i$.

\subsubsection*{Products and Coproducts}

Let $\mathcal{C}$ be a category, $X_1$ and $X_2$ objects of $\mathcal{C}$.
A \textbf{product} of $X_1$ and $X_2$ is an object $X$ often denoted as $X_1\times X_2$, equipped with a pair of morphisms $\pr_1:X \rightarrow X_1$ and $\pr_2:X\rightarrow X_2$ satisfying the following universal property:
For every object $Y\in \mathcal{C}$ and every pair of morphisms $f_1:Y\rightarrow X_1$ and $f_2:Y\rightarrow X_2$, there exists a unique morphism $f:Y\rightarrow X_1\times X_2$ such that the following diagram commutes
\[
\begin{tikzcd}[column sep=large, row sep = large]
        & Y \arrow[d, dashed, "f"] \arrow[dr, "f_2"] \arrow[dl, "f_1"']&\\
        X_1 & X_1\times X_2 \arrow[l, "\pr_1"]\arrow[r, "\pr_2"'] & X_2
\end{tikzcd}
\]
Similarly, one defines a \textbf{coproduct} of $X_1$ and $X_2$ as an object written as $X_1 \sqcup X_2$ with a pair of morphisms $i_1: X_1\rightarrow X_1\sqcup X_2$ and $i_2: X_2\rightarrow X_1\sqcup X_2$ satisfying the same universal property as above but with the arrows reversed 
\[
\begin{tikzcd}[column sep=large, row sep = large]
        & Y   &\\
        X_1\arrow[ur, "f_1"] \arrow[r, "i_1"'] & 
        X_1\sqcup X_2 \arrow[u, dashed, "f"'] 
         & 
        X_2 \arrow[ul, "f_2"'] \arrow[l, "i_2"]
\end{tikzcd}
\]
 
\noindent Whether a product or a coproduct exists may depend on $\mathcal{C}$ or on $X_1$ and $X_2$, however if it does exists it is unique up to canonical isomorphism.\\
\noindent In a category $\mathcal{C}$ with finite product an object $X$ is said to be a \textbf{group object} if it has as morphisms, a group multiplication and inverse operation, as well as an identity such that the multiplication is associative and the inverse and the identity are two-sided.

\subsubsection*{Kernels and Cokernels}

Let $\mathcal{C}$ be a category with $f:X\rightarrow Y$ a morphism in $\mathcal{C}$.\\
An \textbf{equaliser} consists of an object $E\in \mathcal{C}$ and a morphism $\eq:E\rightarrow X$ in $\mathcal{C}$ such that for two parallel morphisms $f,g:X\rightarrow Y$ and for any morphism $m:Z\rightarrow X$ in $\mathcal{C}$, there exists a unique morphism $u:Z\rightarrow E$ such that the following diagram commutes 
\[
\begin{tikzcd}[column sep=large, row sep = large]
        E \arrow[r,"\eq"] &
        X  \arrow[r, shift left, "f"] \arrow[r, shift right, "g"']  & 
        Y\\
        Z \arrow[u, dashed, "u"] \arrow[ur, "m"']
\end{tikzcd}
\]
In that casen $(E,\eq)$ is said to be an equaliser for $f$ and $g$.
Similarly we define a \textbf{coequaliser} $(Q,q)$ in $\mathcal{C}$ as the dual of the equaliser.\\
A \textbf{kernel} of a morphism $f:X\rightarrow Y$ in $\mathcal{C}$ is an equaliser of $f$ and the zero morphism from $X$ to $Y$.\\
A \textbf{cokernel} of $f$ in $\mathcal{C}$ is defined as the coequaliser of $f$ and the zero morphism from $X$ to $Y$.

\subsubsection*{Initial and terminal objects:}

\begin{proposition}
Let 
$\mathcal{Z} =  \emptyset$ be the empty category, and let 
$F : \mathcal{Z}\rightarrow \mathcal{C}$
be the unique functor. Then
\begin{itemize}
    \item[$i)$] the limit of $F$ is an object $T$ in $\mathcal{C}$ with the property that for any (other) object $C\in\mathcal{C}$ there exists a unique morphism $C\rightarrow T$. In particular, $\Hom(T,T)$ is a singleton. If $T$ exists, it is called a (the) \textbf{terminal object} of $\mathcal{C}$.
    
    \item[$ii)$]the colimit of $F$ is an object $I\in \mathcal{C}$ with the property that for any object $C\in\mathcal{C}$, there exists a unique morphism $I\rightarrow C$. In particular, $\Hom(I,I)$ is a singleton. If $I$ exists, it is called a (the) \textbf{initial object} of $\mathcal{C}$.
\end{itemize}
\end{proposition}

\noindent Remark that if they exist, the initial and terminal object are unique up to isomorphism.\\

\subsection{Adjoint Functors}

\begin{definition}
    Let $\mathcal{C}$ and $\mathcal{D}$ be categories; let $L : \mathcal{C}\rightarrow\mathcal{D}$ and $R : \mathcal{D}\rightarrow\mathcal{C}$ be functors. 
    We say that $(L,R)$ is a pair of \textbf{adjoint functors}, or $L$ is a \textbf{left adjoint} to $R$, or $R$ is a \textbf{right adjoint} to $L$ if and only if for any objects $C \in\mathcal{C}$ and $D\in\mathcal{D}$, there is an isomorphism
$$ \theta_{C,D} : \Hom_D(LC,D)  \rightarrow \Hom_C (C,RD) $$
that is natural in both arguments $C$ and $D$. 
If there exists an adjoint pair of functors between the categories $\mathcal{C}$ and $\mathcal{D}$, we refer to this situation as an \textbf{adjunction} and denote it by $(L,R):\mathcal{C} \rightleftarrows \mathcal{D}$
\end{definition} 

\begin{examples}
\textbf{ }
    \begin{itemize}
        \item[$a)$] Let $\phi : R \rightarrow S$ be a morphism of rings. Then this induces a functor $R_{\phi} : M_S \rightarrow M_R$, from (right) $S$-modules to (right) $R$-modules, called the \textbf{restriction of scalars}, that defines on a right $S$-module $M$ a right $R$-module structure by the formula $m\cdot r = m\cdot \phi(r)$, $m\in M, r\in R$.
        This functor has a left adjoint, called the \textbf{extension of scalars}, given by 
        $-\otimes_R S :M_R \rightarrow M_S$.
        
        \item[$b)$] The \textbf{pull-back} and \textbf{pushout} functors defined above are respectively left and right adjoint one to another.
    \end{itemize}
\end{examples}

\begin{theorem}
    Consider the functor $L:\mathcal{C}\rightarrow \mathcal{D}$ and $R:\mathcal{D}\rightarrow \mathcal{C}$. 
    Then there is a bijective correspondence between the natural isomorphism $\theta_{C,D}$ (as above) that turn $(L,R)$ into a pair of adjoint functors, and pair of natural transformations $(\eta,\varepsilon)$ defined by 
    $\eta_C: C\rightarrow RLC$ and $\varepsilon_D:LRD\rightarrow D$ that satisfy $R\epsilon\circ\eta R = R$ and $\varepsilon L\circ L\eta=L$.
\end{theorem}
\begin{proof}
    See \cite{McLane_Sau}.
\end{proof}

\noindent In that case, we call the above natural transformations of the adjunction $(\eta,\varepsilon)$, respectively \textit{unit} and \textit{counit} .

\begin{theorem}
    Let $(L,R) : \mathcal{C}\rightleftarrows \mathcal{D}$ be an adjoint pair of functors. Then $R$ preserves all limits that exist in $\mathcal{D}$; $L$ preserves all colimits that exist in $\mathcal{D}$.
\end{theorem}
\begin{proof}
    See \cite{McLane_Sau}, Theorem $V.\:5.1$.\\
\end{proof}

\subsection{Abelian Categories}

\begin{definition}
    A category is \textbf{preadditive} if every $\Hom_{\mathcal{C}}$-sets are abelian groups and the composition of morphisms is bilinear.\\
    A preadditive category is \textbf{additive} if every finite set of objects has a biproduct, so we can form firect sums and direct products (which are isomorphic) and zero object (which are empty biproducts).\\
    An additive category is \textbf{preabelian} if every morphism has both a kernel and a cokerel.\\
    Finally, a preablian category is \textbf{abelian} if every monomorphism is a kernel of some morphism and every epimorphism is a cokernel of some morphism.\\
    An additive category \textbf{$k$-linear} if every of its $\Hom$-sets are $k$-vector space.
\end{definition}

\begin{definition}
    A functor $F:\mathcal{C}\rightarrow\mathcal{D}$ between two preadditive categories is called \textbf{additive} if $F(f+g) = F(f)+F(g)$ for any two object $A,B\in\mathcal{C}$ and any two morphisms $f,g\in\Hom_{\mathcal{C}}(A,B)$.\\
    A functor $F:\mathcal{C}\rightarrow\mathcal{D}$ between two preabelian categories is called  \textbf{exact} if for any exact sequence of the form
    $0\rightarrow A \xrightarrow{f} B\xrightarrow{g}C\rightarrow 0$
    in $\mathcal{C}$, there is an exact sequence $0\rightarrow F(A) \xrightarrow{F(f)} F(B)\xrightarrow{F(g)}F(C)\rightarrow 0$ in $\mathcal{D}$.
\end{definition}

\begin{examples}\textbf{ }
\begin{itemize}
    \item[$a)$] Trivially, the category of abelian groups is an abelian category.
    \item[$b)$] Let $R$ be a ring, then the category of $R$-modules is an abelian category.
    \item[$c)$] Let $G$ be a group. The category of finite-dimensional representation of $G$ over $k$ being equivalent to the category of $k[G]$-modules makes it an abelian category.
    It is even a $k$-linear abelian category.
\end{itemize}
\end{examples}

\begin{proposition}
\label{comodules_abelian}
    Let $C$ be a coalgebra over a field $k$. Then the category of finite-dimensional $C$-comodules is $k$-linear and abelian.
\end{proposition}
\begin{proof}
    See \cite{DNR} Corollary $2.1.19$.\\
\end{proof}

\subsection{Equivalence of Categories}
We have introduced the notion of an isomorphism of categories (a pair of functors $F :\mathcal{C}\rightarrow\mathcal{D}, G: \mathcal{D}\rightarrow\mathcal{C}$ such that $F\circ G= 1_D$ and $G\circ F = 1_C$). This definition is however very rigid and many natural examples of “equivalent theories” do not fall under this definition.\\
Hence, it would be very reasonable to introduce a new definition where we replace the above identities by natural isomorphisms
\begin{definition}
    Two categories $\mathcal{C}$ and $\mathcal{D}$ are \textbf{equivalent} if there exist two functors $F : \mathcal{C}\rightarrow\mathcal{D}$ and $G: \mathcal{D}\rightarrow\mathcal{C}$, and two natural isomorphisms $\alpha:F\circ G \xrightarrow{\sim} 1_D$ and $\beta : G\circ F \xrightarrow{\sim} 1_C$. In this situation we say that the functor $G$ is a quasi-inverse for $F$ (and $F$ is a quasi-inverse for $G$).
    Finally, we say that $\mathcal{C}$ and $\mathcal{D}$ are anti-equivalent if $\mathcal{C}$ is equivalent to $\mathcal{D}^{op}$.
\end{definition}

\noindent We say that an adjunction $(L,R):\mathcal{C}\rightleftarrows\mathcal{D}$ is an \textbf{adjoint equivalence} if both the unit $\eta$ and counit $\varepsilon$ are natural isomorphisms.

\begin{theorem}
    The following properties of a functor $F:\mathcal{C}\rightarrow \mathcal{D}$ are equivalent :
    \begin{itemize}
        \item[$a)$] $F$ is an equivalence of categories,
        \item[$b)$] $F$ is part of an adjoint equivalence $\mathcal{C}\rightleftarrows\mathcal{D}$.
        \item[$c)$] $F$ is fully faithful and essentially surjective
    \end{itemize}
\end{theorem}
\begin{proof}
    See \cite{McLane_Sau}, Theorem $IV.\:4.1.$
\end{proof}

\subsection{Grothendieck Topologies}
\label{Groth_topo}
We finish our review on categories by an interesting way of endowing a category with a sort of topology called a Grothendieck topology.
These kinds of topologies allow to axiomatise the notion of an open cover for a category.

\begin{definition}
    A \textbf{Grothendieck topology} $\mathcal{G}$ on a category $\mathcal{C}$ is a collection of maps $\mathcal{G}(X) = \{(U_i\rightarrow X)_{i\in I}\}$ for all $X\in \mathcal{C}$ called coverings, such that,
    \begin{itemize}
        \item If $X'\rightarrow X$ is an isomorphism, then $\{X'\rightarrow X\}\in \mathcal{G}(X)$.
        \item If $\{U_i\rightarrow X\}_{i\in I} \in \mathcal{G}(X)$ and $W\rightarrow X$, then $\{U_i\times_XW \rightarrow W\} \in \mathcal{G}(X)$.
        \item If $\{U_i \rightarrow X\}_{i\in I}\in\mathcal{G}(X)$ and $\{V_{ij} \rightarrow U_i\}_{j\in J}\in \mathcal{G}(U_i)$ for all $i\in I$, then $\{V_{ij} \rightarrow X\}_{(i,j)\in I\times J}\in \mathcal{G}(X)$.
    \end{itemize}
\end{definition}

\noindent We can now define the analogue of topological spaces but for categories,

\begin{definition}
    A \textbf{site} is a pair $(\mathcal{C},\mathcal{G})$, such that $\mathcal{C}$ is a category and $\mathcal{G}$ is a Grothendieck topology on $\mathcal{C}$.
\end{definition}

\noindent Now it becomes possible to define sheaves on a category and so it is even possible to define their sheaf cohomology.

\begin{remark}
    As it is the case for topological spaces, categories can be endowed with many different Grothendieck topologies.
    The category of schemes is of particular interest to us here.\\
    The category $\Sch$ can be endowed with the elementary \textit{Zariski topology}, in that case a Zariski covering is a family of morphisms $\{(f_i:U_i \rightarrow X)_{i\in I}\}$ of schemes such that each $f_i$ is an open immersion and such that $X =\bigcup_{i\in I} f_i(U_i)$.\\
    We gave an introduction to \textit{\'etale topology} on the category of schemes in Chapter \ref{nex_topology}.\\
    And as other classes of Grothendieck topologies we may find flat topologies. 
    Among them we have the \textit{fppf topology} (as in fid\`element plate de pr\'esentation finie) and the  \textit{fpqc topology} (as in fid\`element plate et quasi-compacte).
    In these topologies, the morphisms of affine schemes that are covering morphisms are respectively  \textit{faithfully flat, of finite presentation, and is quasi-finite}, and \textit{faithfully flat and quasi-compact}.\\
\end{remark}

\section{Review of Algebraic Geometry}

Often, before introducing schemes, one defines the notion of an algebraic variety.
It can be seen as the set of common solutions of a system of polynomial equations over a field (usually algebraically closed). 
Here, we take the liberty of referring to the first chapter of \cite{Hartshorne}, and immediately start by introducing the theory of sheaves, then schemes.\\

\subsection{Sheaves}
\label{review_sheaves}
\begin{definition}
    A \textbf{presheaf} $\mathcal{F}$ of sets $($resp. rings, abelian groups, $R$-modules, ...$)$ on a topological space $X$ consists of the following data:
    \begin{itemize}
        \item[$a)$] for every open set $U \subset X$, a set $($resp. ring, abelian group, $R$-module, ...$)$ $\mathcal{F}(U)$, and
        \item[$b)$] for every inclusion $V\subset U$ of open sets in $X$ a morphism of sets $($resp. rings, abelian groups, $R$-modules, ...$)$ $\rho_{UV}:\mathcal{F}(U)\rightarrow\mathcal{F}(V)$ called the restriction map, such that
        \begin{itemize}
            \item[$1)$] $\mathcal{F}(\emptyset) = 0$,
            \item[$2)$] $\rho_{UU}$ is the identity map on $\mathcal{F}(U)$ for all open sets $U\subset X$, and 
            \item[$3)$] for any inclusion $W\subset V\subset U$ of opens of $X$, we have $\rho_{UV}\circ\rho_{VW} = \rho_{UW}$.
        \end{itemize}
    \end{itemize}
    The elements of $\mathcal{F}(U)$ are called sections of $\mathcal{F}$ over $U$.
\end{definition}

\noindent We can reformulate the definition of a presheaf in the language of category, for any topological space $X$, we define a category $\Top(X)$ whose objects are open subsets of $X$ and where the morphisms are inclusion maps. So if $V\subset U$ in $X$, $\Hom(U,V)=\emptyset$ and $\Hom(V,U)$ has just one element.
Now a presheaf is just a contravariant functor from the category $\Top(X)$ to the category of sets (resp. rings, abelian groups, $R$-modules, ...).

\begin{definition}
    Now a presheaf $\mathcal{F}$ is a called a \textbf{sheaf} if it satisfies the following \textit{glueing property} : \\
    For $U\subset X$ a non-empty open subset, $\{U_i : i\in I\}$ a non-empty open cover of $U$ and $s_i\in\mathcal{F}(U_i)$ such that $s_i\vert_{U_i\cap U_j} = s_j\vert_{U_i\cap U_j}$ for all $i,j\in I$, then there is a unique $s\in\mathcal{F}(U)$ such that $s\vert_{U_i} = s_i$ for all $i\in I$.
\end{definition}

\begin{definition}
    Let $\mathcal{F}$ be a presheaf on $X$, and let $U \subset X$ be an open subset. Then the restriction of $\mathcal{F}$ to $U$ is defined to be the presheaf $\mathcal{F}\vert_U$ on $U$ with $\mathcal{F}\vert_U(V) = \mathcal{F}(V)$ for every open subset $V\subset U$, and with the restriction maps taken from $\mathcal{F}$ . Note that if $\mathcal{F}$ is a sheaf then so is $\mathcal{F}\vert_U$
\end{definition}

\begin{examples}\textbf{ }
    \begin{itemize}
        \item[$a)$] Let $X$ be a complex manifold. 
        Due to the maximum modulus principle from complex analysis, we know that the only holomorphic functions $f:X\rightarrow \mathbb{C}$ are the constant functions. 
        To resolve this challenge, we work ‘‘locally" and define the following sheaf $\mathcal{O}_X$ to keep track of the holomorphic structure on the underlying topological space of $X$ on arbitrary open subsets. 
        For every open subset $U\subset X$, we define $\mathcal{O}_X(U)$ to be the ring (commutative) of holomorphic functions from $U$ to $\mathbb{C}$ with the restriction being the restriction of domain of functions.
        
        \item[$b)$] Let $S$ be a topological space (or abelian group, ...), then we can define a sheaf $\mathcal{F}_S$ on any topological space $X$ by setting $\mathcal{F}_S(U)$ to be the set (or abelian group, ...) of continuous maps $U\rightarrow S$ and taking the restriction maps to be the usual restriction of functions. Accordingly, this construction defines a sheaf. Moreover, if $S$ has the discrete topology, every continuous map $f:U\rightarrow S$ is constant on each connected component of $U$, so $\mathcal{F}_S(U) = S$. And in that case, we call $\mathcal{F}_S$ the constant sheaf on $X$ with value $S$.
    \end{itemize} 
\end{examples}

\begin{definition}
    If $\mathcal{F}$ and $\mathcal{G}$ are presheaves on $X$, a \textbf{morphism of presheaves} $\varphi : \mathcal{F}\rightarrow\mathcal{G}$ consists of a morphism of sets (resp. rings, abelian groups, $R$-modules, ...) $\varphi_U : \mathcal{F}(U)\rightarrow\mathcal{G}(U)$ for each open set $U$, such that for every inclusion $V\subset U$, the following diagram commutes,\\
    \[
    \begin{tikzcd}
        \mathcal{F}(U) \arrow[r,"\varphi_U"] \arrow[d, "\rho^{\mathcal{F}}_{UV}"']
        & \mathcal{G}(U) \arrow[d,"\rho^{\mathcal{G}}_{UV}"] \\
        \mathcal{F}(V) \arrow[r,"\varphi_V"']
        &  \mathcal{G}(V)
    \end{tikzcd}
    \]
\end{definition}

\noindent This forms a category of presheaves of sets (resp. rings, abelian groups, $R$-modules, ...) on a fixed topological space $X$, denoted $\PSh(X)$. And the category of sheaves of sets (resp. rings, abelian groups, $R$-modules, ...) on $X$, denoted $\Sh(X)$ is defined as the full subcategory of the corresponding presheaf category. This means that a morphism of sheaves is just a morphism of the underlying presheaves.

\begin{definition}
    Let $\mathcal{F}$ be a presheaf of sets (resp. rings, abelian groups, $R$-modules, ...) on $X$ and $x\in X$. 
    The \textbf{stalk} $\mathcal{F}_x$ of $\mathcal{F}$ at $x$ is defined as the direct limit of the sets (resp. rings, abelian groups, $R$-modules, ...) $\mathcal{F}(U)$ for all open subsets $U\subset X$ containing $x$, via the restriction maps $\rho$
\end{definition}

\begin{remark}
    An element of $\mathcal{F}_x$ is represented by a pair $(U,s)$, where $U$ is an open neighbourhood of $x$ and $s\in\mathcal{F}(U)$. And two such elements $(U,s)$ and $(U',s')$ are equivalent in $\mathcal{F}_x$ if and only if there is an open neighbourhood $W$ of $x$, with $W\subset U\cap U'$ such that $s\vert_W = s'\vert_W$. The elements of $\mathcal{F}_x$ are called germs of $\mathcal{F}$ at $x$.
    Note that a morphism $\varphi: \mathcal{F}\rightarrow\mathcal{G}$ of presheaves on $X$ induces a morphism $\varphi_x:\mathcal{F}_x\rightarrow\mathcal{G}_x$ on the stalks, for any point $x\in X$.
\end{remark}

\begin{definition}
    Let $f:X\rightarrow Y$ be a continuous map of topological spaces. 
    For any sheaf $\mathcal{F}$ on $X$, we define the \textbf{direct image sheaf} or \textbf{push forward sheaf} $f_{*}\mathcal{F}$ on $Y$ by $(f_*\mathcal{F})(V) = \mathcal{F}(f^{-1}(V))$ for any open set $V\subset Y$. 
\end{definition}

\begin{construction}
\label{almost_pullback_sheaf}
    For $f: X\rightarrow Y$ a continuous map of topological spaces, the direct image functor $f_* : \Sh(X) \rightarrow \Sh(Y)$ has an adjoint functor called the \textbf{inverse image functor} $f^{-1}$. 
    Let $\mathcal{G}$ be a sheaf on $Y$,
    unlike the direct image sheaf, the inverse image sheaf cannot be defined as simply as $f^{-1}\mathcal{G}(U) = \mathcal{G}(f(U))$ for each open set $U\subset X$ as unfortunately $f(U)$ is not necessarily open.
    What we do then is to define $f^{-1}\mathcal{G}$ to be the sheaf associated to the presheaf $U\mapsto \varinjlim_{V\supset f(U)} \mathcal{G}(V)$.\\
    The process of considering a sheaf associated to a presheaf is called sheafification and is described in \cite{Hartshorne} Proposition-Definition $II.\:1.2$.\\
\end{construction}

\subsection{Schemes}
\label{Annexe_scheme}

In classical algebraic geometry, a scheme is a generalisation of a variety (it allows to consider non-reduced rings and thus it allows for nilpotents).
We begin with the affine case.

\begin{definition}
    Let $R$ be a ring. The set of all prime ideals of $R$ is called the \textbf{spectrum} of $R$ or the \textbf{affine scheme} associated to $R$, and is denoted by $\Spec\:R$.
\end{definition}

\begin{definition}
    Let $R$ be a ring and $P\in\Spec\:R$ a prime ideal of $R$ (a point in the corresponding affine scheme).
    \begin{itemize}
        \item[$a)$] We denote $\kappa(P)$ the quotient field of the integral domain $R/P$. It is called the \textbf{residue field} of $\Spec\:R$ at $P$.
        \item[$b)$] For any $f\in R$, we define the value of $f$ at $P$, to be the image of $f$ under the composite ring homomorphism $R\rightarrow R/P\rightarrow \kappa(P)$.
    \end{itemize}
\end{definition}

\begin{definition}
    An affine scheme $\Spec\: R$ is called \textbf{integral} if it is non-empty and the ring $R$ is an integral domain.
\end{definition}
\begin{definition}
    An affine scheme $\Spec\: R$ is called \textbf{reduced} if the ring $R$ is also reduced, i.e. it contains no nilpotent elements.
\end{definition}

\noindent Let $R$ be a ring, similarly as in the case of affine varieties, we define the \textit{zero locus} $V(S)$ of $S\subset R$ and the \textit{ideal} $I(X)$ of $X\subset \Spec\: R$ to be :
\begin{equation*}
\begin{split}
    V(S) &:= \{P\in\Spec\:R \mid f(P)=0\:,\forall \:f\in S\} \subset \Spec\:R, \\
    I(X) &:=\{f\in R\mid f(P)=0,\:\forall\: P\in X\} \trianglelefteq R
\end{split}
\end{equation*}

\noindent We can now endow an affine scheme with a topology.
\begin{definition}
    We define the \textbf{Zariski topology} on an affine scheme $\Spec\:R$ to be the topology whose closed sets are exactly the sets of the form $V(S)$ for some $S\subset R$.
\end{definition}
\noindent We have a scheme-theoretic version of the Nullstellensatz, 
\begin{proposition}
    Let $R$ be a ring. 
    
    $a)$ For any closed subset $X\subset \Spec\:R$, we have $V(I(X))=X$.
    
    $b)$ For any ideal $J\trianglelefteq R$, we have $I(V(J))=\sqrt{J}$.\\
    In particular there is an inclusion-reversing bijection
    \begin{equation*}
        \{\text{closed subsets of }\Spec\:R\} \xleftrightarrow{1:1} \{\text{radical ideals in } R\}.
    \end{equation*}
\end{proposition}
\begin{proof}
    See \cite{Eisenbud}, Theorem $I.\:1.6$.
\end{proof}

\begin{remark}
    A particular property of schemes is that it allows points that are not necessarily closed, meaning that their closure is a bigger subscheme containing that point.
    We call non-closed point $P$ in a scheme $X$ a \textbf{generic point} if its closure is the entire scheme.
\end{remark}

\noindent We now cite a few topological properties of affine schemes :

\begin{definition}
    A topological space $X$ is called \textbf{reducible} if it can be written as $X = X_1 \cup X_2$, for closed subsets $X_1,X_2 \subsetneq X$. Otherwise it is called \textbf{irreducible}.
\end{definition}

\begin{proposition}
    An affine variety $\Spec\: R$ is irreducible if and only if $R$ is an integral domain.
\end{proposition}
\begin{proposition}
\label{integral_iff_irr_red}
    An affine scheme $X$ is said to be integral if and only if it is irreducible and reduced.
\end{proposition}

\begin{definition}
    An affine scheme $X$ is called \textbf{disconnected} if it can be written as $X=X_1\cup X_2$ for closed subsets $X_1,X_2\subsetneq X$ with $X_1\cap X_2 = \emptyset$. Otherwise it is called \textbf{connected}
\end{definition}

\begin{definition}
    A topological space is called \textbf{Noetherian} if there is no infinite strictly decreasing chain of closed subsets.
\end{definition}

\noindent Thus an affine scheme $\Spec\: R$ is Noetherian if and only if $R$ is a Noetherian ring.

\begin{definition}
    The \textbf{dimension} $\dim X\in \mathbb{N}\cup \{\infty\}$ of a non-empty topological space $X$ is the supremum of the length of chains of strictly increasing irreducible closed subsets of $X$.
\end{definition}

\begin{proposition}
    For a ring $R$ and an element $f\in R$, we call $D(f) := \Spec\:R \backslash V(f)$ the \textbf{distinguished open subset} of $f$ in $\Spec\: R$.
\end{proposition}
\begin{remark}
    The distinguish open sets form a basis of the Zariski topology of an affine scheme $\Spec\:R$ in the sense that every open subset $U\subset \Spec\:R$ is an arbitrary union of distinguished opens.
\end{remark}

\noindent We now define the notion of morphism of schemes. 
Since schemes usually carry more structure than typical topological spaces, it is natural that morphism of schemes need to satisfy more conditions than just being continuous. 
It does not just have to move points around, it needs to map sections of structure sheaves in a meaningful way. 
This is defined below.

\begin{definition}
    A \textbf{locally ringed space} is a ringed space $(X,\mathcal{O}_X)$ such that each stalk $\mathcal{O}_{X,P}$ for $P\in X$ is a local ring.
    A \textbf{morphism} of locally ringed spaces from $(X,\mathcal{O}_X)$ to $(Y,\mathcal{O}_Y)$ is given by a continuous map $f:X\rightarrow Y$ and for every open subset $U\subset Y$ a ring homomorphism $f^*_U:\mathcal{O}_Y(U)\rightarrow O_X(f^{-1}(U))$ called the \textit{pull-back} on $U$, such that the two following conditions hold :
    \begin{itemize}
        \item[$a)$] The pull-back maps must be compatible with restriction (this implies there are ring homomorphism $f^*_P:\mathcal{O}_{Y,f(P)}\rightarrow \mathcal{O}_{X,P}$ on the stalks for every $P\in X$).
        \item[$b)$] The ring homomorphism induced by $f^*$ between the stalks of $Y$ and $X$ must be a local homomorphism, i.e.
        for every $P\in X$, we have $(f^*_P)^{-1}(I_P) = I_{f(P)}$ where $I_P$ and $I_{f(P)}$ denote maximal ideals of $\mathcal{O}_{X,P}$ and $\mathcal{O}_{Y,f(P)}$ respectively.
    \end{itemize}
\end{definition}

\noindent Any affine scheme is a locally ringed space with its structure sheaf (see Proposition $II. \:2.3.$ in \cite{Hartshorne}).
Let $X$ be an affine scheme and $\mathcal{O}_X$ its structure sheaf.
For any open subset $U\subset X$, then $\mathcal{O}_X(U)$ is the ring of regular functions on $U$. This is not to be confused with the ring of \textit{rational functions} on $U$.

\begin{definition}
    A \textbf{scheme} is a locally ringed space that has an open cover by affine schemes. Morphisms of schemes are just morphisms as locally ringed spaces. This makes a category of schemes denoted $\Sch$.
\end{definition}

\noindent In the literature, there is often a distinction between separated and non-separated scheme (the notion corresponds to that of Hausdroff spaces in topology). For sake of simplicity, all the schemes we consider are assumed to be separated.

\begin{remark}
    All the above topological concepts defined for affine schemes immediately apply to schemes.
\end{remark}

\noindent We can now define \textit{open} and \textit{closed subschemes} :
\begin{construction}
    Let $X$ be a scheme.
    We say that $Y$ is a \textbf{open subscheme} of $Y$ whose underlying space is the subspace $Y$ of $X$ together with an isomorphism of structure sheaf $\mathcal{O}_Y$ with the restriction $\mathcal{O}_{X\vert_Y}$. Similarly, an \textbf{open immersion} is a morphism $X\rightarrow Y$ which induces an isomorphism of $X$ with an open subscheme of $Y$.\\
    On the other hand, $Z$ is a \textbf{closed subscheme} of $X$ if there is a morphism of locally ringed spaces $i:Z\rightarrow X$ that is a closed immersion.
    This means it has to satisfy the following criteria :
    \begin{itemize}
        \item[$\cdot$] the map $i$ is a homeomorphism of $Z$ onto its image,
        \item[$\cdot$] the associated sheaf map $\mathcal{O}_X\rightarrow i_*\mathcal{O}_Z$ is surjective with kernel $\mathcal{I}$,
        \item[$\cdot$] the kernel $\mathcal{I}$ is locally generated by sections as an $\mathcal{O}_X$-module.
    \end{itemize}
    This is further explained in the next subsection.
\end{construction}

\noindent Even though they are an important class of schemes, we give a quick definition of \textit{projective schemes}.
\begin{construction}
    Let $R$ be a commutative graded ring, where $R = \oplus_{d\ge 0} R_d$ is its gradation. 
    The elements of $R_d\backslash \{0\}$ are said to be homogeneous of degree $d$.
    An ideal in a graded ring is called \textbf{homogeneous} if it is generated by homogeneous elements.
    The ringed space $(\Spec\: R, \mathcal{O}_{\Spec\:R})$ is an \textbf{affine projective scheme} if $R$ is a graded ring as above and $\Spec\:R := \{P \trianglelefteq R \mid \text{ with } P \text{ a homogeneous prime ideals} \}$, with the usual structure sheaf.
    Similarly as above, one define a \textbf{projective scheme} $(X,\mathcal{O}_X)$.
\end{construction}

\begin{definition}
    Let $X$ be a scheme. The \textbf{ring of rational functions} on $X$ is the ring $R(X)$ whose elements are rational functions of $X$.
\end{definition}

\begin{remark}
    We can also define the residue field $\kappa(P)$ for $P$ a point of a locally ringed space $(X,\mathcal{O}_X)$. It is defined as follows $\kappa(P) := \mathcal{O}_{X,P}/\mathfrak{m}_P$ where $\mathfrak{m}_P$ is the unique maximal ideal of $\mathcal{O}_{X,P}$.
\end{remark}

\begin{lemma}
\label{function_field_residue_field}
    Let $X$ be an integral scheme, then the function field $k(X) = \mathcal{O}_{X,\eta} = \kappa(\eta)$ for $\eta$ the unique generic point of $X$ (see Lemma $\ref{irreducible_generic_point}$).
\end{lemma}
\begin{proof}
    See \cite{Stacks}\href{https://stacks.math.columbia.edu/tag/01RV}{(Tag 01RV)}.
\end{proof}
\begin{definition}
    Let $X$ be an integral scheme. The \textbf{function field} or \textbf{field of rational functions} of $X$ is the field $R(X)$, that we now denote $k(X)$.
\end{definition}

\noindent One powerful aspects of scheme theory is that it allows us to consider varieties/schemes over particular fields.
For example, we talk about a scheme $X$ over a field $k$, when we consider the map $X\rightarrow \Spec\:k$.
Over certain fields, schemes may have different properties which are called relative properties, since they are now relative to a base scheme.
This was the case for most concepts presented in Section \ref{morphisms_of_schemes}. 
In that sense relative properties of schemes can be thought of as properties of morphisms of scheme.

\begin{definition}
    A morphism of schemes $f:X\rightarrow Y$ is said to be \textbf{quasi-compact} if $Y$ can be covered by open affine subschemes $(V_i)_{i\in I}$ such that the pre-images $f^{-1}(V_i)$ are compact (in the usual topological sense) for all $i\in I$
\end{definition}

\begin{definition}
    A morphisms $f:X\rightarrow Y$ of schemes is called \textbf{proper} if it is (separated) of finite type and universally closed. 
\end{definition}

\noindent One of the main important tool in that flavour is the ability to change base schemes.

\begin{definition}
    Let $X$ and $Y'$ be schemes over a scheme $Y$, we define the \textbf{fibre product} of $X$ and $Y'$ over $Y$ to be the scheme  $X\times_YY'$ such that the following diagram commutes:
    \[
    \begin{tikzcd}
        X\times_Y Y' \arrow[r,"p_{Y'}"] \arrow[d, "p_X"']
        & Y' \arrow[d,"g"] \\
        X \arrow[r,"f"']
        &  Y
    \end{tikzcd}
    \]
    If we consider the scheme $X$ over $Y$ and $g:Y'\rightarrow Y$ the change of basis, the \textbf{base change} of $X$ is the scheme $X' :=X\times_YY'$ over $Y'$ and we call $p_{X'}: X' \rightarrow Y'$ the \textbf{base change map}.
\end{definition}

\noindent We say that a morphism of schemes is \textbf{universally closed} if for any base change, its base change map is a closed map.

\begin{remark}
    The disjoint union distributes over the fibre product of a scheme, meaning that for any schemes $X_i$ and $Y$ over a base scheme $Z$, there is an isomorphism of schemes :
    $$(\bigsqcup_i X_i) \times_Z Y \cong \bigsqcup_i(X_i\times_Z Y) .$$
\end{remark}

\noindent An interesting use of the base change is the ability to consider \textit{geometric notions}.
Usually, we want to work over algebraically closed fields, but not all properties of schemes are invariant under base change. 
So we can “force" a scheme to have a certain property over an algebraically closed field by considering it having that property over its closure.
As an illustration we give the following definition
\begin{definition}
\label{geometrically_integral}
    A scheme $X$ over a field $k$ is said to be \textbf{geometrically integral} if for all field extension $L|k$, $X\times_k \Spec\: L$ is integral, then of course $X\times_k \Spec\: \overline{k}$ is also integral.
\end{definition}

\noindent There are many ways to define the fibre of a morphism in algebraic geometry

\begin{definition}
    Let $f:X\rightarrow Y$ be a morphism of schemes, let $y\in Y$ and $\kappa(y)$ the residue field of $y$ with the inclusion morphism being $\iota: \Spec\:\kappa(y)\rightarrow Y$.
    We define the \textbf{fibre of $f$ at $y$} to be the scheme $X_y := X\times_Y \Spec\:\kappa(y)$
\end{definition}

\subsection{Quasi-coherent Sheaves}
\label{annexe_quasi_coherent}
So far we have introduced schemes and morphisms between them without mentioning any sheaves other than the structure sheaves, our review of scheme theory would not be complete without mentioning sheaves of modules.

\begin{definition}
    Let $X$ be a scheme and $\mathcal{O}_X$ its structure sheaf.
    A \textbf{sheaf of modules} on $X$, also called a sheaf of $\mathcal{O}_X$-modules is a sheaf $\mathcal{F}$ on $X$ such that 
    $\mathcal{F}(U)$ is an $\mathcal{O}_X(U)$-module for all open subsets $U\subset X$ and all restriction maps are $\mathcal{O}_X$-modules homomorphisms in the sense that 
    $$(\varphi +\psi)\vert_U = \varphi\vert_U + \psi\vert_U \: \text{ and }\: (\lambda\varphi)\vert_U = \lambda\vert_U\cdot \varphi\vert_U$$
    for all open subsets $U\subset V$ of $X$, $\lambda\in \mathcal{O}_X(V)$ and $\varphi,\psi\in\mathcal{F}(V)$.
\end{definition}

\begin{examples} \textbf{ }
\begin{itemize}
    \item[$a)$] Let $X$ be a scheme, then $\mathcal{O}_X$ is a sheaf of modules on $X$.
    
    \item[$b)$]  Let $f:\mathcal{F}\rightarrow \mathcal{G}$ be a morphisms of sheaves on a scheme $X$.
    For any open subset $U\subset X$, we set 
    $(\Ker f)(U):= \ker( f_U : \mathcal{F}(U) \rightarrow \mathcal{G}(U) )$. $\Ker f$ is a sheaf of modules on $X$, we call it the \textbf{kernel sheaf} of $f$.
    
    \item[$c)$] For a morphism $f:\mathcal{F}\rightarrow \mathcal{G}$ of sheaves on a scheme $X$ and any open subset $U\subset X$, we set $(\Imm' f)(U):= \im( f_U : \mathcal{F}(U) \rightarrow\mathcal{G}(U)).$
    $\Imm' f$ defines a presheaf of modules but not yet a sheaf. We define its sheafification $\Imm f$ the \textbf{image sheaf} of $f$.

    \item[$d)$] Let $f:\mathcal{F}\rightarrow \mathcal{G}$ be a morphism of sheaves on a scheme $X$.
    We define the \textbf{tensor presheaf} $\mathcal{F}\otimes'\mathcal{G}$ on $X$ by $(\mathcal{F}\otimes'\mathcal{G})(U):= \mathcal{F}(U) \otimes_{\mathcal{O}_X(U)}\mathcal{G}(U)$ for any open subset $U\subset X$.
    We define its sheafification to be the \textbf{tensor sheaf} $\mathcal{F}\otimes \mathcal{G}$.

    \item[$e)$] Let $f:X\rightarrow Y$ be a morphism of schemes and let $\mathcal{F}$ be a sheaf on $X$. Consider the direct image sheaf $f_*\mathcal{F}$ on $Y$, it is a sheaf of $\mathcal{O}_Y$-modules.
\end{itemize} 
\end{examples}

\begin{lemma}
\label{exact_at_level_of_stalk}
    Let $(X, \mathcal{O}_X)$ be a ringed space. The category $\Mod(\mathcal{O}_X)$ is an abelian category. Moreover a complex
    $\dots \rightarrow\mathcal{F}_1\rightarrow \mathcal{F}_2 \rightarrow \mathcal{F}_3 \rightarrow\dots$ is exact at $\mathcal{F}_2$ if and only if for all $x\in X$, the complex
    $\dots \rightarrow(\mathcal{F}_1)_x\rightarrow (\mathcal{F}_2)_x \rightarrow (\mathcal{F}_3)_x \rightarrow\dots$ is exact at $(\mathcal{F}_2)_x$.
\end{lemma}
\begin{proof}
    See \cite{Stacks} \href{https://stacks.math.columbia.edu/tag/01AG}{(Tag 01AG)}
\end{proof}

\noindent We now describe exactly how an $R$-module $M$ determines a sheaf of modules $\widetilde{M}$ on the affine scheme $X=\Spec\:R$
\begin{definition}
    Let $X=\Spec\: R$ be an affine scheme, and let $M$ be an $R$-module.
    For any open subset $U\subset X$, we set
    \begin{equation*}
    \begin{split}
        \widetilde{M}(U) := \{\varphi = (\varphi_P)_{P\in U} \mid& \varphi_P \in M_P,\:\forall P\in U,\\
        &\forall P, \:\exists U_P \text{ an open neighbourhood of } P \text{ in } U\\
        &\text{and } g\in M, f\in R \text{ with } \varphi_Q = \frac{g}{f}, \:\forall Q\in U_P
    \end{split}
    \end{equation*}
    From the definition, $\widetilde{M}$ is a sheaf, moreover $\widetilde{M}(U)$ is a module over $\widetilde{R}(U) = \mathcal{O}_X(U)$, hence $\widetilde{M}$ is a sheaf of modules on $X$. We call it the \textbf{sheaf associated to $M$}. 
\end{definition}

\begin{proposition}
    Let $R$ be a ring, let $M$ be an $R$-module, and let $\widetilde{M}$ be the sheaf on $X : \Spec \: R$ associated to $M$. Then :
    \begin{itemize}
        \item[$a)$] $\widetilde{M}$ is an $\mathcal{O}_X$-module,
        \item[$b)$] for each $P\in X$, the stalk $(\widetilde{M})_P$ of the sheaf $\widetilde{M}$ at $P$ is isomorphic to the localized module $M_P$ and
        \item[$c)$] in particular the global sections $\widetilde{M}(X) = M$.
    \end{itemize}
\end{proposition}
\begin{proof}
    See \cite{Hartshorne} Proposition $5.1$.
\end{proof}

\begin{definition}
    A sheaf of modules $\mathcal{F}$ on a scheme $X$ is said to be \textbf{locally free} if there is an affine open cover $\{U_i |i\in I\}$ of $X$ such that on every open $U_i = \Spec\:R_i$ the restricted sheaf $\mathcal{F}\vert_{U_i}$ is isomorphic to the sheaf $\widetilde{M}_i$ associated to a free $R_i$-module $M_i$ of finite rank.
    If this rank is the same for all $i\in I$, it is also called the rank of $\mathcal{F}$.
\end{definition}

\noindent But not every sheaf of modules on an affine scheme $\Spec\: R$ is associate to an $R$-module. For that reason, we have to define additional properties.

\begin{definition}
    A sheaf of modules $\mathcal{F}$ on a scheme $X$ is called \textbf{quasi-coherent} if there is an affine cover $\{ U_i |i\in I\}$ of $X$ such that on every open $U_i=\Spec\:R_i$, the restricted sheaf $\mathcal{F}\vert_{U_i}$ is isomorphic to the sheaf associated to $M_i$ and $R_i$-module.
    Moreover, if $M_i$ is a finitely generated $R_i$-module for all $i$, we say that $\mathcal{F}$ is a $\textbf{coherent}$ sheaf
\end{definition}

\begin{example}
    The structure sheaf $\mathcal{O}_X$ is quasi-coherent on any scheme $X$, since for any open subset $U=\Spec\:R\subset X$ is in fact of the form $\mathcal{F}\vert_U \cong \widetilde{M}$ for any $R$-module $M$
\end{example}

\noindent We define the \textbf{tensor product} $\mathcal{F}\otimes_{\mathcal{O}_X}\mathcal{G}$ of two $\mathcal{O}_X$-modules to be the sheaf associated to the presheaf $U\mapsto \mathcal{F}(U) \otimes_{\mathcal{O}_X(U)} \mathcal{G}(U)$.

Taking the tensor product of two quasi-coherent sheaves gives a quasi-coherent sheaf again

\begin{proposition}
    Let $f:X\rightarrow Y$ be a morphism of schemes with $X$ a Noetherian scheme.
    Then is $\mathcal{F}$ is a quasi-coherent sheaf of $\mathcal{O}_X$-modules, then its direct image $f_*\mathcal{F}$ is a quasi-coherent sheaf of $\mathcal{O}_Y$-modules.
\end{proposition}
\begin{proof}
    See \cite{Hartshorne}, Proposition $II.\:5.8$.
\end{proof}

\noindent Unfortunately, the adjoint functor $f^{-1}$ from Remark \ref{almost_pullback_sheaf} does not take quasi-coherent sheaves on $X$ to quasi-coherent sheaves on $Y$.
For that reason we define the \textit{pull-back sheaf} $f^*$ as follow
\begin{definition}
    Let $f:(X,\mathcal{O}_X)\rightarrow (Y,\mathcal{O}_Y)$ be a morphism of ringed spaces and $\mathcal{G}$ a sheaf on $Y$. 
    We define the \textbf{pull-back sheaf} $f^*\mathcal{G}$ to be the tenor product $f^{-1}\mathcal{G}\otimes_{f^{-1}\mathcal{O}_Y} \mathcal{O}_X$
\end{definition}

\noindent And now the pull-back sheaf of a quasi-coherent sheaf is again quasi-coherent by Proposition $II.\:5.8$ from \cite{Hartshorne}.

\begin{construction}
    Let $\mathcal{F}$ be a sheaf on a scheme $X$. 
    For a closed point $P\in X$ and the inclusion map $i:P\rightarrow X$, the sheaf pull-back sheaf $i^*\mathcal{F}$ on $P$ has only one non-trivial space of sections $i^*\mathcal{F}(P)$.
    The geometric meaning of this space is better seen in the affine case.
    Let $X = \Spec\: R$ so that $P\trianglelefteq R$ is a maximal ideal and $\mathcal{F} = \widetilde{M}$ for an $R$-module $M$.
    We have $i^*\mathcal{F}(P) = M\otimes_R R/P = M/PM$ as a vector space over the residue field $\kappa(P)$.
    This vector space $i^*\mathcal{F}(P)$ over $\kappa(P)$ is called the \textbf{fibre} of $\mathcal{F}$ at $P$.
\end{construction}

\begin{theorem}
    The category of quasi-coherent sheaves on schemes is an abelian category.
\end{theorem}
\begin{proof}
    See \cite{Hartshorne}, $III.\: 1.0.5$.
\end{proof}

\noindent In particular it is stable under kernels, images, cokernels and quotients.\\
Regarding Lemma \ref{exact_at_level_of_stalk}, we have the following proposition.
\begin{proposition}
    If $0\rightarrow\mathcal{F}_1\rightarrow \mathcal{F}_2 \rightarrow \mathcal{F}_3\rightarrow 0$ is a short exact sequence of quasi-coherent sheaf on an affine scheme scheme $X$, then 
    $0\rightarrow\mathcal{F}_1(X)\rightarrow \mathcal{F}_2(X) \rightarrow \mathcal{F}_3(X)\rightarrow 0$ is a short exact sequence. 
    In other words, the global section functor is flat
\end{proposition}
\begin{proof}
    See \cite{Hartshorne}, Proposition $5.6$.
\end{proof}

\noindent However it is not the case that the global section functor is flat for non-affine schemes, it is only left exact.
Indeed the category of quasi-coherent sheaves has \textit{enough injectives}, which means that all of its objects admit an \textit{injective resolution} which basically mean that the above global section functor is left exact. This follows from the following result
\begin{proposition}
     Let $(X,\mathcal{O}_S$ be a ringed space. Then the category of sheaves of $\mathcal{O}_X$-modules has enough injectives.
\end{proposition}
\begin{proof}
    \cite{Hartshorne} Proposition $III.\: 2.2$.
\end{proof}

\noindent Nakayama's Lemma \ref{Nakayama_first_lemma} is a useful tool in algebraic geometry. 
It has the following interpretation in terms of coherent sheaves 
\begin{corollary}[Coherent Nakayama]
    If the stalk of a coherent sheaf is zero at a point $P$, then the sheaf restricts to zero in an open neighbourhood of $P$.
\end{corollary}

\begin{lemma}
\label{ideal_sheaf}
    Let $i: X \rightarrow Y$ be a closed embedding. Then
    \begin{itemize}
        \item[a)] If $\mathcal{F}$ is a quasi-coherent sheaf on $X$, then the direct image $i_*\mathcal{F}$ is a quasi-coherent sheaf on $Y$.
        \item[b)] $0 \rightarrow \mathcal{I}_{X/Y} \rightarrow\mathcal{O}_{Y}\rightarrow i_*\mathcal{O}_X\rightarrow0$ is an exact sequence of quasi-coherent sheaves with $\mathcal{I}_{X/Y}$ the kernel of $i^{\#}:\mathcal{O}_Y\rightarrow i_*\mathcal{O}_X$ called the ideal sheaf of $X$ in $Y$.
    \end{itemize}
\end{lemma}
\begin{proof}
    See \cite{Hartshorne}, Proposition $5.9$.
\end{proof}

\begin{proposition}
\label{affine_morph_quasi_coherent}
    There is an anti-equivalence between the category of affine morphisms $X\rightarrow Y$ and the category of quasi-coherent sheaves of $\mathcal{O}_X$-algebras.
\end{proposition}
\begin{proof}
    See \cite{Stacks} \href{https://stacks.math.columbia.edu/tag/01SA}{(Tag 01SA)}.\\
\end{proof}

\subsection{Quasi-Coherent Sheaf Cohomology}
\label{Sheaf_cohomology}

We have seen above that the functor of global sections (which can be thought of as a derived functor here) sends a short exact sequence 
$$0\rightarrow \mathcal{F}_1 \rightarrow \mathcal{F}_2 \rightarrow \mathcal{F}_3 \rightarrow 0$$
of quasi-coherent sheaves on a scheme $X$ to the following left exact sequence 
$$0\rightarrow \mathcal{F}_1(X) \rightarrow \mathcal{F}_2(X) \rightarrow \mathcal{F}_3(X) \rightarrow \dots$$
Sheaf cohomology is basically the study of when the above sequence fails to be exact.

\begin{definition}
    Let $\mathcal{F}$ be a quasi-coherent sheaf on $X$ and $\{U_1,...,U_n\}$ an affine open cover of $X$.
    \begin{itemize}
        \item[$a)$] For all $p\in \mathbb{N}$, we define the sets 
        $$C^p(\mathcal{F}) = \bigoplus_{i_0<\dots< i_p} \mathcal{F}(U_{i_0} \cap \dots \cap U_{i_p})$$
        An element $\varphi \in C^p(\mathcal{F})$ is just a collection of sections $\varphi_{i_0,\dots ,ip}\in \mathcal{F}(U_{i_0} \cap \dots \cap U_{i_p})$ for all intersections of $p+1$ sets taken from the chosen affine open cover. 

        \item[$b)$] For every $p\in \mathbb{N}$ we define a linear \textbf{boundary map} $d^p: C^p(\mathcal{F}) \rightarrow C^{p+1}(\mathcal{F})$ by 
        $$(d^p\varphi)_{i_0,\dots,i_{p+1}} = \sum^{p+1}_{k=0} (-1)^k \varphi_{i_0,\dots,\hat{i_k},\dots, i_{p+1} \vert_{U_{i_0}\cap\dots \cap U_{i_{p+1}}}}$$
    \end{itemize}
\end{definition}

\begin{lemma}
    The composition of two consecutive boundary maps is $0$.
\end{lemma}
\begin{proof}
    This is a straight up computation.
\end{proof}

\begin{definition}
    Let $\mathcal{F}$ be a coherent sheaf on a scheme $X$ and let $r\in \mathbb{N}$.
    \begin{itemize}
    \item[$a)$]
    The following chain complex is called  the \textbf{$\check{\textbf{C}}$ech complex} (the fact that it is a complex follows from the previous lemma),
    $$C^0(\mathcal{F}) \xrightarrow[]{d^0} C^1(\mathcal{F})\xrightarrow[]{d^1} C^2(\mathcal{F})\xrightarrow[]{d^2} \dots.$$
    \item[$b)$] We define the $r$-th \textbf{$\check{\text{C}}$ech cohomology group} of $\mathcal{F}$ to be $H^r(X,\mathcal{F}) := \ker d^r / \im d^{r-1}$.
    With the convention that $C^{-1}(\mathcal{F})$ and $d^{-1}$ are zero so that $H^0(X,\mathcal{F}) = \ker d^0$.
    \item[$c)$] Suppose we work over a field $k$, we ca define $h^r(X,\mathcal{F})$ the dimension of $H^r(X,\mathcal{F})$ as a $k$-vector space.
    \end{itemize}
\end{definition}

\noindent The $\check{\text{C}}$ech cohomology is usually defined for sheaves of abelian groups, but since quasi-coherent sheaves are a special kind of sheaf of abelian groups, we make no distinction between quasi-coherent sheaf cohomology and the $\check{\text{C}}$ech cohomology.
It is also important to note that as we have defined it, the quasi-coherent sheaf cohomology is independent of the choice of affine open covering.
We have the following first few properties. 

\begin{proposition}
    Let $\mathcal{F}$ be a quasi-coherent sheaf on a scheme $X$. Then
    \begin{itemize}
        \item[$a)$] $H^0(X,\mathcal{F}) = \mathcal{F}(X)$,
        \item[$b)$] if $X$ is affine, then $H^r(X,\mathcal{F}) = 0$ for all $r> 0$.
    \end{itemize}
\end{proposition}
\begin{proof}
    See \cite{Hartshorne}, for $a)$, Lemma $III.\: 4.1$ and for $b)$, 
\end{proof}

\begin{proposition}
    Let $0\rightarrow \mathcal{F}_1\rightarrow\mathcal{F}_2\rightarrow\mathcal{F}_3\rightarrow 0$ be a short exact sequence of quasi-coherent sheaves on a scheme $X$.
    Then there exists a long exact sequence of cohomology groups 
    \begin{equation*}
    \begin{split}
        0&\rightarrow H^0(X,\mathcal{F_\text{1}}) \rightarrow H^0(X,\mathcal{F_\text{2}}) \rightarrow H^0(X,\mathcal{F_\text{3}})\\
        &\rightarrow H^1(X,\mathcal{F_\text{1}}) \rightarrow H^1(X,\mathcal{F_\text{2}}) \rightarrow H^1(X,\mathcal{F_\text{3}})\\
        &\rightarrow H^2(X,\mathcal{F_\text{1}}) \rightarrow \dots
    \end{split}
    \end{equation*}
\end{proposition}
\begin{proof}
    See \cite{Vakil} Chapter $18.1$, $(iii)$.\\
\end{proof}

\noindent We only mention it, but the cohomology of sheaves gives us powerful invariants, useful for the classification of schemes.
Such as, the cohomology groups themselves, the \textit{Euler characteristics}, the genus of a scheme and much more.

\subsection{Analytification}
\label{analytification}

A smooth scheme over $\mathbb{C}$ can be seen as a complex manifold. 
Since we can apply techniques of complex analysis and differential geometry to complex manifolds, 
provided we have a nice dictionary between the language of abstract algebraic geometry and complex
manifolds, we could translate those analytical results into results about schemes.

\noindent We first need to define what we mean by an analytic space (complex in our case) 

\begin{definition}

    Let $U\subset \mathbb{C}^n$ be the Cartesian product of open discs : $\{|z_i| <1 \mid i=1,...,n\}$,
    let $f_1,...,f_m$ be holomorphic functions on $U$ and
    let $Y\subset U$ be the closed subset (for the usual Euclidean topology) consisting of the zero locus $V(f_1,...,f_m)$, 
    with its structure sheaf the restriction $\mathcal{O}_Y = \mathcal{O}_U/(f_1,...,f_m)$ (for $\mathcal{O}_U$ the sheaf of holomorphic functions on $U$).\\
    A \textbf{complex analytic space} is a locally ringed space $(X,\mathcal{O}_X)$ which can be covered by affine open sets, each of which biholomorphic (as locally ringed spaces) to a space $Y$ as above. (Notice that $Y$ is open in $X$ but closed in $U\subset \mathbb{C}^n$)
\end{definition}

\noindent Notice that the sheaf structure is actually a structure of $\mathbb{C}$-modules.

\begin{remark}
    Theses spaces are called \textit{analytic} because the topology on them is the topology induced from the usual one on $\mathbb{C}^n$ open sets are small, and compatible with analysis (limits, continuity, holomorphic functions, ...)
\end{remark}

\noindent Similarly as in the algebraic case, complex analytic spaces are a generalisation of complex manifolds, indeed, they also allow singularities.
For example, let $X = \{(u,v)\in\mathbb{C}^2 \mid uv= 0\}$ be a complex analytic space, then it is not a complex manifold since there is no neighbourhood at the origin that is biholomorphic to an open subset of $\mathbb{C}$.

\begin{definition}
    If $X$ and $Y$ are complex analytical spaces, a morphism $\varphi:X\rightarrow Y$ of locally $\mathbb{C}$-ringed spaces is called a \textbf{holomorphic map}. 
\end{definition}

\noindent As the name of this section suggests, the goal is to associate to an algebraic scheme $X$ over $\mathbb{C}$ such a complex analytic space that we will denote by $X^{an}$,

\begin{construction}
    Let $X$ be a scheme of finite type over $\mathbb{C}$. 
    Consider a covering of $X$ by affine open subsets $Y_i=\Spec\:A_i$, recall that for all $i$, $A_i$ is of finite type, so $A_i \cong \mathbb{C}[x_1, ..., x_n]/(f_1, ..., f_m)$ with $f_1,...,f_m$ polynomials in $x_1,...,x_n$.
    We can regard those polynomials as holomorphic functions on $\mathbb{C}^n$ and their zero locus form a complex analytic closed subspace $Y_i^{an}$ in $\mathbb{C}^n$.
    Since the scheme $X$ is obtained by glueing the open sets $Y_i$, we can use the same glueing data to glue the analytic spaces $Y_i^{an}$ into an analytic space $X^{an}$. This is the associated complex analytic space of $X$.
\end{construction}

\begin{remark}
    Notice that the above construction is functorial thus we obtain a functor $\_^{an}$ from the category of finite schemes over $\mathbb{C}$ to the category of complex analytical spaces.
    There is also a continuous map $\varphi : X^{an}\rightarrow X$ of the underlying topological spaces, which sends $X^{an}$ bijectively onto closed points of $X$.
\end{remark}

\noindent In his groundbreaking paper \cite{Serre}, Serre proved a plethora of results about the relationship between a scheme $X$ ans its associated analytic space $X^{an}$, we enumerate them below:

\begin{proposition}
\label{analytification_space}
    Let $X$ be a scheme over $\mathbb{C}$, locally of finite type and $X^{an}$ the associated analytical space. 
    Then the scheme $X$ is 
    \begin{itemize}
        \item[$a)$] non-empty,
        \item[$b)$] discrete,
        \item[$c)$] regular,
        \item[$d)$] normal, 
        \item[$e)$] reduced, or
        \item[$f)$] of dimension $n$
    \end{itemize}
    if and only if the associated $X^{an}$ is as well.
\end{proposition}
\begin{proof}
    See \cite{SGA}, Proposition $XII.\:2.1$.
\end{proof}

\begin{proposition}
\label{analytification_morphisms}
    Let $X$ and $Y$ be two schemes over $\mathbb{C}$ locally of finite type, $f:X\rightarrow Y$ a morphism of finite type and 
    $f^{an}: X^{an}\rightarrow Y^{an}$ the associated analytical morphism induced by $f$. 
    Then $f$ is 
    \begin{itemize}
        \item[$a)$] surjective,
        \item[$b)$] a closed immersion,
        \item[$c)$] an immersion,
        \item[$d)$] proper, or
        \item[$e)$] finite
    \end{itemize}
    if and only if $f^{an}$ is as well.
\end{proposition}
\begin{proof}
    See \cite{SGA}, Proposition $XII.\:3.2$.
\end{proof}

\noindent 
We say that a complex analytical space $(X,\mathcal{O}_X)$ is actually a \textit{Riemann surface} if and only if it is $1$-dimensional and smooth (or non-singular). 
Thus, notice that the above notions also make sense for Riemann surfaces when considering their structure sheaf of holomorphic functions. 
\begin{remark}
    It is interesting to note that in \cite{Serre}, for $X$ a projective scheme over $\mathbb{C}$,
    Serre also showed that the analytification functor yields an equivalence of categories between coherent sheaves on $X$ and coherent analytic sheaves on $X^{an}$.
\end{remark}

\bibliographystyle{alpha}
\bibliography{bibliothec}
\nocite{*}

\end{document}